\documentclass[a4paper,10pt]{article}
\usepackage{amscd}
\usepackage{amsmath}
\usepackage{bbding}
\usepackage{amsfonts}
\usepackage{caption}
\usepackage{mathrsfs}
\usepackage{graphicx}
\usepackage{subfigure}
\usepackage{amssymb}
\usepackage{float}
\usepackage{cases}
\usepackage{graphicx,latexsym,amsmath,amssymb,amsthm,enumerate}
\usepackage{color}
\usepackage{stmaryrd}
\allowdisplaybreaks[2]   %允许公式分页显示
9.8in\textwidth 6.9in \hoffset -2.5cm \DeclareFontEncoding{OT2}{}{}
\DeclareTextSymbol{\cyrsftsn}{OT2}{126}
\DeclareTextSymbol{\textnumero}{OT2}{125}

\renewcommand{\theequation}{\thesection.\arabic{equation}}
\theoremstyle{plain}
\newtheorem*{remark}{Remark}
\newtheorem{theorem}{Theorem}[section]
\newtheorem{lemma}{Lemma}[section]
\newtheorem{corollary}{Corollary}[section]
\newtheorem{proposition}{Proposition}[section]
\newtheorem{definition}{Definition}[section]
\newtheorem{example}{Example}[section]
\begin{document}
\title{{\LARGE\bf{Stochastic Integration on Stochastic Sets of Interval Type and Applications to Mathematical Finance}\thanks{This work was supported by the National Natural Science Foundation of China (12171339).}}}
\author{Jia Yue$^a$, Ming-Hui Wang$^a$, Nan-Jing Huang$^b$\footnote{Corresponding author.  E-mail addresses: nanjinghuang@hotmail.com, njhuang@scu.edu.cn} \\
{\small\it a. School of Mathematics, Southwestern University of Finance and Economics,}\\
{\small\it Chengdu, Sichuan 610074, P.R. China}\\{\small\it b. Department of Mathematics, Sichuan University, Chengdu,
Sichuan 610064, P.R. China}}
\date{ }
\maketitle
\begin{flushleft}
\hrulefill\\
\end{flushleft}
 {\bf Abstract}.
In the works of \cite{He, Jacod}, stochastic sets $\mathbb{B}$ of interval type, along with $\mathbb{B}$-stochastic processes, were introduced within the framework of stochastic analysis. In this paper, we undertake the construction of $\mathbb{B}$-stochastic integration by exploring three novel types of $\mathbb{B}$-stochastic integrals: Stieltjes integrals of $\mathbb{B}$-predictable processes with respect to $\mathbb{B}$-adapted processes with finite variation, stochastic integrals of $\mathbb{B}$-predictable processes with respect to $\mathbb{B}$-inner local martingales, and stochastic integrals of $\mathbb{B}$-predictable processes with respect to $\mathbb{B}$-inner semimartingales.
These $\mathbb{B}$-stochastic integrals are exclusively defined on subsets $\mathbb{B}$ of $\llbracket{0,+\infty}\llbracket$, with values outside the scope of $\mathbb{B}$ being deemed irrelevant.
Additionally, we present several notable consequences, including the relationship between $\mathbb{B}$-stochastic integrals and existing stochastic integrals, as well as It\^{o}'s formula for $\mathbb{B}$-inner semimartingales.
In the context of models pertaining to uncertain time-horizons in mathematical finance, we establish essentials of mathematical finance for general markets characterized by sudden-stop horizons. This is achieved by defining self-financing strategies, admissible strategies, and no-arbitrary conditions.
In such financial markets, the exclusivity characteristic inherent in $\mathbb{B}$-stochastic integrals offers investors a viable alternative approach. This approach enables them to effectively filter out unnecessary information pertaining to asset price dynamics and portfolio strategies that extend beyond the predefined time-horizons.



\section{Introduction}
\subsection{Discussion and problems}
Unlike conventional stochastic processes, $\mathbb{B}$-stochastic processes, as delineated in \cite{He, Jacod}, are exclusively defined on a stochastic set $\mathbb{B}$ of interval type, which constitutes a particular subset within $\llbracket{0,+\infty}\llbracket$. Within this framework, values outside the scope of $\mathbb{B}$ are considered inconsequential.
The exclusivity property inherent in $\mathbb{B}$-stochastic processes offers a potent tool within the domain of stochastic analysis.
Jacod and Shiryaev (see Definition 2.46 of \cite{Jacod}) define a local martingale on $\mathbb{B}=\llbracket{0,T}\llbracket$  by employing a predictable time $T$ and a sequence of local martingales. They note that the values of such a $\llbracket{0,T}\llbracket$-local martingale beyond $\llbracket{0,T}\llbracket$ do not matter. This definition is subsequently utilized to investigate semimartingales' characteristics and exponential formula (see Theorem 2.47 of \cite{Jacod}).
Under the condition of local absolute continuity between two measures, He et al. (see Subsection 12.2 of \cite{He}) apply $\mathbb{B}$-stochastic processes to the study of Girsanov's theorems for local martingales and semimartingales.
However, an overarching theory that sheds light on stochastic integration on stochastic sets of interval type
is missing in the existing literature. This deficiency poses significant challenges in applying $\mathbb{B}$-stochastic processes to more general stochastic analysis and practical domains. For instance, without a well-defined theory of stochastic integration on stochastic sets of interval type, it is exceedingly difficult to establish the change-of-variable formula (i.e., It\^{o}'s Formula) for $\mathbb{B}$-stochastic processes.
In this paper, we endeavor to develop such a theory. Our objective is to formalize the definition of stochastic integration on stochastic sets of interval type and to demonstrate its applications in various areas of applied probability. One notable application lies in the field of mathematical finance, where stochastic integration on stochastic sets of interval type holds substantial importance in problems that involve uncertain time-horizons. We shall delve deeper into this aspect later on.

To elucidate the central problem under investigation, we shall initiate our discussion by revisiting the definitions of stochastic sets $\mathbb{B}$ of interval type and $\mathbb{B}$-stochastic processes:
\begin{itemize}
  \item [(1)] According to Definition 8.16 of \cite{He}, a set $\mathbb{B}\subseteq \Omega\times \mathbf{R}^+$ is called a {\bf stochastic set of interval type} if there is a non-negative random variable $T$ such that for each $\omega\in \Omega$ the section $\mathbb{B}_\omega=\{t: (\omega,t)\in \mathbb{B}\}$ is either $[0,T(\omega)[$ or $[0,T(\omega)]$, and $\mathbb{B}_\omega\neq\emptyset$. Furthermore, if $\mathbb{B}$ is also an optional (resp. predictable) set, then it is termed an {\bf optional} (resp. {\bf predictable}) {\bf set of interval type}.
  \item [(2)] From Definition 8.19 of \cite{He}, it follows that, given a class $\mathcal{D}$ of processes possessing the property $P$, a real-valued function $X$ defined on an optional set $\mathbb{B}$ of interval type is referred to as a {\bf $\mathbb{B}$-process} having the property $P$, if there exist an increasing sequence $(T_n)_{n\in\mathbf{N}^+}$ (in short: $(T_n)$) of stopping times and a sequence $(X^{(n)})\subseteq \mathcal{D}$ of processes such that $T_n\uparrow T$ ($T$ is the debut of $\mathbb{B}^c$),
$\bigcup\limits_{n=1}^{{+\infty}}\llbracket{0,T_n}\rrbracket\supseteq \mathbb{B}$, and for each $n\in \mathbf{N}^+$, $(XI_\mathbb{B})^{T_n}=(X^{(n)}I_\mathbb{B})^{T_n}$.
The collection of all $\mathbb{B}$-processes having the property $P$ is denoted by $\mathcal{D}^\mathbb{B}$, and in this case, $(T_n,X^{(n)})_{n\in\mathbf{N}^+}$ (in short: $(T_n,X^{(n)})$) is called a {\bf fundamental coupled sequence} (in short: FCS) for $X\in\mathcal{D}^\mathbb{B}$.
\end{itemize}
If $\mathcal{D}$ denotes the class encompassing all processes, then $X\in\mathcal{D}^\mathbb{B}$ is called a $\mathbb{B}$-process, and any FCS for $X\in\mathcal{D}^\mathbb{B}$ is also called a {\bf coupled sequence} (in short: CS) for $X$.
By choosing different instances of $\mathcal{D}$, we can define various classes of $\mathbb{B}$-processes, including $\mathcal{P}^\mathbb{B}$ ($\mathbb{B}$-predictable processes), $\mathcal{V}^\mathbb{B}$ ($\mathbb{B}$-adapted processes with finite variation), $(\mathcal{M}_{\mathrm{loc}})^\mathbb{B}$ ($\mathbb{B}$-local martingales), and $\mathcal{S}^\mathbb{B}$ ($\mathbb{B}$-semimartingales).

The principal aim of this paper is to propose a constructive approach to resolving the following problem concerning $\mathbb{B}$-stochastic integration:

\noindent\textbf{Problem (SI).} {\it Let $\mathbb{B}$ be an optional (resp. predictable) set of interval type, and assume that $H\in \mathcal{P}^\mathbb{B}$ and $X\in \mathcal{S}^\mathbb{B}$. Under what integrability conditions imposed on $H$ and $X$ can we define the $\mathbb{B}$-stochastic integral $H_{\bullet}X$ of $H$ with respect to $X$?}

The study of stochastic integrals boasts a rich historical background. For a comprehensive understanding, one may consult \cite{Doleans-dade,Ito,Jacod1979,Kardaras,Kunita} and \cite{Jarrow,Kuo,Meyer2009} for detailed expositions.
It is a logical and natural requirement that the $\mathbb{B}$-stochastic integrals we establish should reduce to the existing stochastic integrals of predictable processes with respect to semimartingales when $\mathbb{B}$ is identical to $\llbracket{0,+\infty}\llbracket$. Consequently, the construction of a viable $\mathbb{B}$-stochastic integration framework necessitates a program that is analogous to the current methodologies for defining stochastic integrals of predictable processes with respect to semimartingales.

To tackle Problem (SI), let us revisit the well-established theory of stochastic integration with respect to semimartingales (see, e.g., Definition 9.13 of \cite{He}, or Definition 12.3.19 of \cite{Cohen}).
Let $\widetilde{H}$ be a predictable process and $\widetilde{X}$ be a semimartingale. We say that $\widetilde{H}$ is $\widetilde{X}$-integrable if there exists a decomposition $\widetilde{X}=\widetilde{M}+\widetilde{A}$ such that $\widetilde{H}\in \mathcal{L}_m(\widetilde{M})\cap \mathcal{L}_s(\widetilde{A})$, where $\widetilde{M}$ is a local martingale, $\widetilde{A}\;(\widetilde{A}_0=0)$ is an adapted process with finite variation, $\mathcal{L}_m(\widetilde{M})$ and
$\mathcal{L}_s(\widetilde{A})$ represent the classes of predictable processes that are integrable with respect to $\widetilde{M}$ and $\widetilde{A}$, respectively. And in this scenario, the stochastic integral of $\widetilde{H}$ with respect to $\widetilde{X}$, denoted by $\widetilde{H}.\widetilde{X}$, is defined as $\widetilde{H}.\widetilde{X}=\widetilde{H}.\widetilde{A}+\widetilde{H}.\widetilde{M}$. Note that for $t\geq 0$, $(\widetilde{H}.\widetilde{X})_t=\int_{[0,t]}\widetilde{H}_sd\widetilde{X}_s=\widetilde{H}_0\widetilde{X}_0
+\int_{0}^t\widetilde{H}_sd\widetilde{X}_s$.
The collection of all predictable processes which are $\widetilde{X}$-integrable is denoted by $\mathcal{L}(\widetilde{X})$.
Given this background, in order to effectively address Problem (SI), it is of utmost importance to formulate the following two sub-problems:

\noindent\textbf{Problem (SI-1).} {\it Let $\mathbb{B}$ be an optional (resp. predictable) set of interval type, and assume that $H\in \mathcal{P}^\mathbb{B}$ and $A\in \mathcal{V}^\mathbb{B}$. Under what integrability conditions posed on $H$ and $A$ can we define the $\mathbb{B}$-stochastic integral $H_{\bullet}A$ of $H$ with respect to $A$?}

\noindent\textbf{Problem (SI-2).}  {\it Let $\mathbb{B}$ be an optional (resp. predictable) set of interval type, and assume that  $H\in \mathcal{P}^\mathbb{B}$ and $M\in (\mathcal{M}_{\mathrm{loc}})^\mathbb{B}$. Under what integrability conditions posed on $H$ and $M$ can we define the $\mathbb{B}$-stochastic integral $H_{\bullet}M$ of $H$ with respect to $M$?}

With regard to Problem (SI-1), one might initially contemplate defining the $\mathbb{B}$-stochastic integral $H_{\bullet}A$ through Stieltjes integrals as follows: For $(\omega,t)\in \mathbb{B}$, $(H_{\bullet}A)(\omega,t)=\int_{[0,t]}H_s(\omega)dA_s(\omega)$, under the condition that for all $(\omega,t)\in \mathbb{B}$, $\int_{[0,t]}|H_s(\omega)||dA_s(\omega)|<{+\infty}$. Here, $\int_{[0,t]}|dA_s(\omega)|$ denotes the variation of the process $A(\omega,\cdot)$ over the interval $[0,t]$. However, this proposed definition of $H_{\bullet}A$ may fall short of satisfying a fundamental property. Specifically, it may fail to ensure that $H_{\bullet}A$ remains a $\mathbb{B}$-adapted process with finite variation. This property is a crucial characteristic of the existing stochastic integrals (for instance, refer to Theorem 3.46(2) of \cite{He}).

We shall present an example to elucidate the aforementioned point. Let $T$ be a random variable with the cumulative distribution function defined as $G(t)=\frac{t}{2}I_{\{0\leq t<1\}}+I_{\{t\geq 1\}}$. Denote by $\mathbb{F}=(\mathcal{F}_t,t\geq 0)$ the natural filtration of the process $A=I_{\llbracket{T,{+\infty}}\llbracket}$. We define $\widehat{\mathbb{B}}=\llbracket{0,T_F}\llbracket\cap\llbracket{0,T_{F^c}}\rrbracket$, where $F=[T=1]$. It can be deduced that $T$ is an $\mathbb{F}$-stopping time with $T>0$, while $T_F$ is an $\mathbb{F}$-stopping time but not an $\mathbb{F}$-predictable time. This can be verified by Theorem 3.29(6) in \cite{He}. Consequently, $\widehat{\mathbb{B}}$ constitutes an optional set of interval type, although it is not a predictable set of interval type.
Next, we introduce the following functions defined on ${\widehat{\mathbb{B}}}$: $\widehat{H}(\omega,t)=\frac{1}{1-t}$, $\widehat{A}(\omega,t)=t$, $L(\omega,t)=\int_{[0,t]}\widehat{H}_s(\omega)d\widehat{A}_s(\omega)$, for $(\omega,t)\in\widehat{\mathbb{B}}$.
It is evident that $\widehat{H}\in \mathcal{P}^{\widehat{\mathbb{B}}}$ with the FCS $(T,HI_{\llbracket{0,1}\llbracket})$ and $\widehat{A}\in\mathcal{V}^{\widehat{\mathbb{B}}}$ with the FCS $(T,\widetilde{A})$, where $\widetilde{A}_t=t,\; t\in \mathbf{R}^+$. Through direct calculation, we can infer that $\int_{[0,t]}|\widehat{H}_s(\omega)||d\widehat{A}_s(\omega)|<{+\infty}$ and $L(\omega,t)=\ln\frac{1}{1-t}$ for all $(\omega,t)\in \widehat{\mathbb{B}}$.
Now, assume that $L$ is a $\mathbb{B}$-adapted process with finite variation, and $(T_n,L^{(n)})$ is its FCS. Given that $T_F$ is not an $\mathbb{F}$-predictable time, according to Definition 6.2.1 in \cite{Cohen}, there exists a set $B\subseteq[T=1]$ with $\mathbf{P}(B)>0$ and an integer $m\in\mathbf{N}^+$ such that $T_m(\omega)=T(\omega)=1$ for $\omega\in B$.
By the equation $L^{(m)}I_{{\widehat{\mathbb{B}}}\llbracket{0,T_m}\rrbracket}=LI_{{\widehat{\mathbb{B}}}\llbracket{0,T_m}\rrbracket}$, it is ensured that $L^{(m)}(\omega,t)=L(\omega,t)$ for $\omega\in B$ and $0\leq t<1$. However, this contradicts the property $L^{(m)}\in\mathcal{V}$, since $L^{(m)}$ cannot be c\`{a}dl\`{a}g. Consequently, we have conclusively shown that $L\notin \mathcal{V}^{\widehat{\mathbb{B}}}$. Given this result, it is not appropriate to directly define $\widehat{H}_\bullet\widehat{A}$ as follows: For $(\omega,t)\in \mathbb{B}$, $(\widehat{H}_\bullet\widehat{A})(\omega,t)=\int_{[0,t]}\widehat{H}_s(\omega)d\widehat{A}_s(\omega)$.

Regarding Problem (SI-2), given that $\widetilde{H}\in \mathcal{P}$ and $\widetilde{M}\in\mathcal{M}_{\mathrm{loc}}$, we revisit the notion of integrability. Specifically, $\widetilde{H}$ is deemed integrable with respect to $\widetilde{M}$ (see, e.g., Definition 9.1 of \cite{He}) if there exists a (unique) local martingale $L$ such that
$[L,\widetilde{N}]=\widetilde{H}.[\widetilde{M},\widetilde{N}]$
holds for every $\widetilde{N}\in\mathcal{M}_{\mathrm{loc}}$, where $[L,\widetilde{N}]$ and $[\widetilde{M},\widetilde{N}]$ represent the quadratic covariations of local martingales.
In light of this definition, as well as the construction of stochastic integration in \cite{Kardaras,Kunita}, it is crucial to introduce the concept of a $\mathbb{B}$-quadratic covariation for any processes $M,N\in (\mathcal{M}_{\mathrm{loc}})^\mathbb{B}$. According to Theorem 8.25 of \cite{He}, the $\mathbb{B}$-quadratic covariation can be characterized by a $\mathbb{B}$-process of $\mathcal{V}^\mathbb{B}$, denoted by $[M,N]$, which adheres to the following conditions: $MN-[M,N]\in (\mathcal{M}_{\mathrm{loc,0}})^\mathbb{B}$ and $\Delta [M,N]=\Delta M\Delta N$.
Here, $\Delta [M,N]$, $\Delta M$ and $\Delta N$ are $\mathbb{B}$-jump process (see Definition \ref{X+-}).
However, the practical application of such a definition presents a significant challenge. Notably, $[M,N]$ is not necessarily the unique $\mathbb{B}$-process $V\in \mathcal{V}^\mathbb{B}$ that satisfies both $MN-V\in (\mathcal{M}_{\mathrm{loc,0}})^\mathbb{B}$ and $\Delta V=\Delta M\Delta N$. This observation stands in stark contrast to the fundamental property of the quadratic covariation of local martingales, as articulated in sources such as Theorem 7.31 of \cite{He}.

We now proceed to elucidate the aforementioned point through an example that is grounded in Remark 8.24 of \cite{He}.
Let $T$ be a random variable following a unit exponential distribution, and let $\mathbb{F}=(\mathcal{F}_t,t\geq 0)$ denote the natural filtration of the process $I_{\mathbb{B}^*}$, where $\mathbb{B}^*=\llbracket{0,T}\llbracket$. According to Example 6.2.5 of \cite{Cohen} and Lemma 2.1 of \cite{Aksamit}, $T$ constitutes an $\mathbb{F}$-totally inaccessible time with $T>0$, and ${\mathbb{B}^*}$ is an optional set of interval type. We define the ${\mathbb{B}^*}$-process $\widehat{M}$ as $\widehat{M}=A^p\mathfrak{I}_{\mathbb{B}^*}$ (see Section 2.1), where $A=I_{\llbracket{T,{+\infty}}\llbracket}$, and $A^p_t=T\wedge t$ serves as the compensator of $A$ (see Proposition 2.4 of \cite{Aksamit}).
It is straightforward to verify that $\widehat{M}\in (\mathcal{M}_{\mathrm{loc}})^{\mathbb{B}^*}$ with the FCS $(T_n=T,M^{(n)}=A^p-A)$. On one hand, we observe that $\widehat{M}^{2}-\widehat{M}^{2}=0\mathfrak{I}_{\mathbb{B}^*}\in(\mathcal{M}_{\mathrm{loc},0})^{\mathbb{B}^*}$, with $\widehat{M}^{2}\in\mathcal{V}^{\mathbb{B}^*}$ and $\Delta \widehat{M}^{2}=(\Delta \widehat{M})^{2}=0\mathfrak{I}_{\mathbb{B}^*}$. On the other hand, given that $A^p-A\in\mathcal{V}$, we deduce that
$[A^p-A]=\sum_{s\leq\cdot}(\Delta(A^p-A)_s)^2=\sum_{s\leq\cdot}(\Delta A_s)^2=\sum_{s\leq\cdot}\Delta A_s=A$.
Then for each $n\in \mathbf{N}^+$, the relations
$\widehat{M}^{ 2}I_{\mathbb{B}^*\llbracket{0,T_n}\rrbracket}=((M^{(n)})^{ 2}-A)I_{\mathbb{B}^*\llbracket{0,T_n}\rrbracket}$ and $(M^{(n)})^{2}-A=(M^{(n)})^{2}-[M^{(n)}]\in\mathcal{M}_{\mathrm{loc},0}$
jointly imply that $\widehat{M}^{2}\in(\mathcal{M}_{\mathrm{loc},0})^{\mathbb{B}^*}$. This subsequently yields $\widehat{M}^{2}-0\mathfrak{I}_{\mathbb{B}^*}\in(\mathcal{M}_{\mathrm{loc},0})^{\mathbb{B}^*}$ and $\Delta (0\mathfrak{I}_{\mathbb{B}^*})=(\Delta M)^{2}=0\mathfrak{I}_{\mathbb{B}^*}$.
Consequently, both $0\mathfrak{I}_\mathbb{B}$ and $\widehat{M}^{2}$ ($\widehat{M}^{ 2}\neq 0\mathfrak{I}_{\mathbb{B}^*}$) satisfy the criteria for being the process $V\in \mathcal{V}^{\mathbb{B}^*}$ such that $\widehat{M}^{2}-V\in (\mathcal{M}_{\mathrm{loc},0})^{\mathbb{B}^*}$ and $\Delta V=(\Delta M)^2$.
This observation highlights the inappropriateness of directly defining the $\mathbb{B}$-quadratic covariation $[\widehat{M},\widehat{M}]$ as a $\mathbb{B}$-process of $\mathcal{V}^\mathbb{B}$ satisfying the conditions $\widehat{M}^2-[\widehat{M},\widehat{M}]\in (\mathcal{M}_{\mathrm{loc,0}})^\mathbb{B}$ and $\Delta [\widehat{M},\widehat{M}]=(\Delta \widehat{M})^2$.

\subsection{Contribution}
This study focuses on the construction of $\mathbb{B}$-stochastic integration by providing affirmative resolutions to Problems (SI), (SI-1), and (SI-2), among which Problem (SI-2) assumes a pivotal role. To achieve this objective, we introduce the concept of $\mathbb{B}$-quadratic covariations, which serve as the cornerstone for the development of $\mathbb{B}$-stochastic integration. To provide a preliminary insight into the implementation of this framework, we shall first revisit the construction of quadratic covariations in the classical setting (see, e.g., \cite{Cohen,He}).
In the context where $\widetilde{M}$ and $\widetilde{N}$ are both square-integrable martingales, the predictable quadratic covariation $\langle \widetilde{M},\widetilde{N}\rangle$ can be uniquely characterized as the predictable process with integrable variation, such that the process $\widetilde{M}\widetilde{N}-\langle \widetilde{M},\widetilde{N}\rangle$ constitutes a uniformly integrable martingale with a null initial value. Subsequently, the quadratic covariation $[\widetilde{M},\widetilde{N}]$ of two local martingales $\widetilde{M}$ and $\widetilde{N}$ is defined as $[\widetilde{M},\widetilde{N}]=\widetilde{M}_0\widetilde{N}_0+\langle \widetilde{M}^c,\widetilde{N}^c\rangle+\sum_{s\leq \cdot}(\Delta \widetilde{M}_s\Delta \widetilde{N}_s)$, where $\widetilde{M}^c$ and $\widetilde{N}^c$ are continuous martingale parts of $\widetilde{M}$ and $\widetilde{N}$, respectively.
In contrast to this conventional approach, our investigation commences with the scenario where $M$ and $N$ are both $\mathbb{B}$-continuous local martingales. We define the $\mathbb{B}$-predictable quadratic covariation $\langle M,N\rangle$ as the unique $\mathbb{B}$-process belonging to $(\mathcal{A}_{\mathrm{loc}}\cap \mathcal{C})^\mathbb{B}$ (see the notation of Section \ref{section1.4}), such that the process $MN-\langle M,N\rangle$ remains a $\mathbb{B}$-continuous local martingale with a null initial value. The rationale behind this definition stems from the observation that $\mathbb{B}$-square integrable martingales may not satisfy the uniqueness property inherent in the classical definition (refer to aforementioned $\mathbb{B}^*$ and $\widehat{M}$).
Subsequently, by leveraging $\mathbb{B}$-jump processes, we analogously define the $\mathbb{B}$-quadratic covariation $[M,N]$ of two $\mathbb{B}$-local martingales as $[M,N]=M_0N_0\mathfrak{I}_\mathbb{B}+\langle M^c,N^c\rangle+\Sigma (\Delta M\Delta N)$, where $\Sigma (\Delta M\Delta N)$ represents the $\mathbb{B}$-summation process of $\Delta M\Delta N$ (see Theorem \ref{restriction}(2)), and $M^c$ and $N^c$ denote the continuous martingale parts of $M$ and $N$ (see Theorem \ref{Lem-M}(1)), respectively. This definition extends the classical notion of quadratic covariation to the $\mathbb{B}$-stochastic framework, thereby facilitating the development of $\mathbb{B}$-stochastic integration.

Despite the well-defined nature of the $\mathbb{B}$-quadratic covariation $[M,N]$ for two $\mathbb{B}$-local martingales, it may not necessarily be the unique $\mathbb{B}$-process $V\in\mathcal{V}^\mathbb{B}$ satisfying the conditions $MN-V\in (\mathcal{M}_{\mathrm{loc,0}})^\mathbb{B}$ and $\Delta V=\Delta M\Delta N$. As previously discussed (refer to aforementioned $\mathbb{B}^*$ and $\widehat{M}$), this non-uniqueness arises from the existence of a $\mathbb{B}$-local martingale whose FCS must incorporate values that transcend the scope of $\mathbb{B}$. To address this limitation, we introduce the concept of $\mathbb{B}$-inner local martingales.
A $\mathbb{B}$-inner local martingale is characterized by an FCS that inherently excludes values lying outside the domain of $\mathbb{B}$. This definition aligns seamlessly with the exclusivity property inherent in $\mathbb{B}$-stochastic processes. More crucially, it guarantees the uniqueness in defining $\mathbb{B}$-stochastic integration (refer to the Remark following Definition \ref{HM}), thereby providing a robust framework for further theoretical developments and applications.

By leveraging the concepts of $\mathbb{B}$-quadratic covariation and $\mathbb{B}$-inner local martingales, we expand the scope of integrands and integrators in stochastic integrals to encompass $\mathbb{B}$-predictable processes and $\mathbb{B}$-inner semimartingales, respectively. This extension enables the establishment of stochastic integration on stochastic sets of interval type. More precisely, the following aspects are elaborated:
\begin{itemize}
  \item [(1)] Firstly, the $\mathbb{B}$-stochastic integral $H_{\bullet}A$ is constructed through the Stieltjes integral by paths, defined as $L(\omega,t)=\int_{[0,t]}H(\omega,s)dA(\omega,s)$ for $(\omega,t)\in \mathbb{B}$. In contrast to its classical counterpart, an additional condition $L\in\mathcal{V}^\mathbb{B}$ is imposed to ensure that $H_{\bullet}A$ remains a $\mathbb{B}$-adapted process with finite variation. The necessity of this condition is substantiated through equivalent integrability conditions (refer to Theorem \ref{HA-equivalent}).

  \item [(2)] Secondly, the $\mathbb{B}$-stochastic integral $H_{\bullet}M$ is developed under the integrability condition that $M$ is a $\mathbb{B}$-inner local martingale. A notable feature of this type of $\mathbb{B}$-stochastic integral is that the processes generated by integration retain the property of being $\mathbb{B}$-inner local martingales, thereby preserving the exclusivity characteristic inherent in $\mathbb{B}$-stochastic processes. More significantly, the introduction of the class of $\mathbb{B}$-inner local martingales ensures the uniqueness of the $\mathbb{B}$-stochastic integral.

  \item [(3)] Thirdly, under the integrability condition that $X$ is a $\mathbb{B}$-inner semimartingale, the $\mathbb{B}$-stochastic integral $H_{\bullet}X$ is established by combining the aforementioned two types of $\mathbb{B}$-stochastic integrals and utilizing an inner decomposition of $X$. The resulting stochastic integral is independent of the specific choice of inner decomposition for $X$ and remains a $\mathbb{B}$-inner semimartingale. Most importantly, the It\^{o} Formula for $\mathbb{B}$-inner semimartingales is derived, which facilitates the practical applications of $\mathbb{B}$-stochastic integration.
\end{itemize}

The above approach offers pedagogical advantages, as it does not necessitate prior knowledge of stochastic analysis on stochastic sets of interval type. Not only does this method construct $\mathbb{B}$-stochastic integrals in a manner analogous to existing stochastic integration frameworks, but it also establishes a comprehensive connection with a series of established stochastic integrals. On the one hand, when $\mathbb{B}$ coincides with the stochastic interval $\llbracket{0,+\infty}\llbracket$, the $\mathbb{B}$-stochastic integration reverts to the conventional stochastic integration, thereby demonstrating the consistency and generality of our approach.
On the other hand, one can leverage the integrability conditions inherent in existing stochastic integration theories to assess the integrability of $\mathbb{B}$-stochastic integrals. Furthermore, a $\mathbb{B}$-stochastic integral can be characterized as the sum of a sequence of existing stochastic integrals, thereby providing a bridge between the new framework and the established literature. This dual perspective enhances the accessibility and applicability of $\mathbb{B}$-stochastic integration within the broader context of stochastic analysis.

\subsection{Applications to mathematical finance}
Uncertain time-horizons are pivotal in addressing problems that incorporate randomness within real financial markets. Illustrative examples of such problems encompass the exit of a stock from the market (see, e.g., \cite{Bayraktar}), the default of a security (see, e.g., \cite{Okhrati}), and the demise of an investor (see, e.g., \cite{Yaari}), among numerous other events that occur on uncertain dates.
However, when mathematically modeling these problems, it is often unavoidable that extraneous information, which lies beyond the uncertain time-horizons, is inadvertently incorporated. This inclusion of irrelevant data can obscure the analysis and lead to convoluted results.
It is therefore more judicious to formulate these problems exclusively on the basis of the necessary information encapsulated within the specified time-horizons.
Nevertheless, this poses a challenge when attempting to apply the current stochastic integral framework, as it typically relies on information that may extend beyond the necessary bounds of the time-horizons.

Let us delve deeper into the issue of unnecessary information relative to uncertain time-horizons within the following scenario. Consider a simplified financial market where an investor allocates capital to a risky stock. We denote the time span of the market by $\llbracket{0,T}\rrbracket$, where $T$ is a positive real number, and let $\tau$ represent the uncertain time at which the stock exits the market, modeled as a positive stopping time. It is evident that, if all stock information throughout the entire time span is deemed viable, the portfolio optimization problem can be addressed by applying optimal investment strategies with an uncertain time-horizon (see, e.g., \cite{Blanchet}).
However, there are two pivotal aspects worthy of discussion:
\begin{itemize}
  \item [(1)] Firstly, is it imperative for the investor to model stock information subsequent to the exit time? From the investor's perspective, investments should be made strictly prior to the exit time $\tau$ and within the terminal time $T$. It is more reasonable to characterize the investor's time-horizon using a stochastic set of interval type, denoted as $\mathbb{B}=\llbracket{0,T}\rrbracket\llbracket{0,\tau}\llbracket$. Ideally, the stock price dynamics within $\mathbb{B}$ should suffice for the investor to make informed investment decisions, and portfolio rules should be developed accordingly within this time-horizon. Consequently, if the investor can derive an optimal portfolio rule using only the stock price information within $\mathbb{B}$, then stock information outside $\mathbb{B}$ can be deemed extraneous, and the financial market should ``stop suddenly" prior to the exit time.
  \item [(2)] Secondly, does incorporating stock information outside the time-horizon $\mathbb{B}$ influence the investment decision? When the default time $\tau$ is considered as a component of the terminal time, the investor's portfolio strategy and/or wealth may indeed be affected by stock information outside $\mathbb{B}$ (see, e.g., \cite{Blanchet,Landriault,Lv,Yu}), particularly in scenarios where a significant jump in stock price occurs at the exit time $\tau$. However, when the current date $t$ ($t\leq T$) is sufficiently distant from the exit time $\tau$ (i.e., the probability $\mathbf{P}(\tau\leq t)$ is negligible), the investor would reasonably believe that the stock will not exit the market within the time-horizon $[0,t]$. In such cases, the portfolio strategy and wealth should remain independent of stock information outside $\mathbb{B}$. Therefore, stock information outside the time-horizon retains the potential to alter investment decisions, especially as the exit time approaches.
\end{itemize}

In the existing literature, scenarios analogous to the one described above are discussed by augmenting stochastic processes and integrals with extraneous information outside the designated time-horizons, utilizing the current construction of stochastic integrals (see, e.g., \cite{Blanchet, Kostyunin, Landriault, Lin, Lopez-Barrientos,Lv, Marin-Solano,Yu}).
Specifically, the stock price process is constructed over the time span $\llbracket{0,T}\rrbracket$ or $\llbracket{0,\infty}\llbracket$ through the application of stochastic integrals. Subsequently, the uncertain time $\tau$ is incorporated into the investment framework, resulting in a terminal time defined as $\tau\wedge T$ or $\tau$. However, as previously elucidated, the stock information pertaining to the time-horizon $\rrbracket{\tau\wedge T,T}\rrbracket$ or $\rrbracket{\tau,\infty}\llbracket$ may be superfluous within the context of portfolio strategies and wealth management.
Given this exclusivity property, an alternative approach to addressing problems with uncertain time-horizons involves the utilization of stochastic integrals on stochastic sets of interval type. This method offers a more focused analysis by excluding irrelevant information, thereby enhancing the applicability of the results within the realm of mathematical finance.

In Section \ref{section6}, we construct financial markets characterized by time-horizons defined as stochastic sets of interval type, which we refer to as financial markets with sudden-stop horizons. We demonstrate how to efficiently eliminate extraneous information beyond these time-horizons by employing stochastic integrals defined on such stochastic sets. The crux of our approach lies in two key aspects:
\begin{itemize}
  \item [(1)] Firstly, we address the issue of determining the terminal time. Let $\mathbb{B}$ denote the sudden-stop horizon, and assume that the stock price $S$ is a $\mathbb{B}$-inner semimartingale with an appropriate FCS $(T_n,S^{(n)})$, where $\mathbb{B}=\llbracket{0,T_F}\llbracket\;\cap\;\llbracket{0,T_{F^c}}\rrbracket$  (see Lemma \ref{B}), $T_n\uparrow T$, and $T$ represents the debut of $\mathbb{B}^c$. The sudden-stop horizon is then partitioned into a sequence of time-horizons $\rrbracket{T_{n-1},T_n}\rrbracket$ (with $T_0=0$ and $n\in \mathbf{N}^+$), each of which can be regarded as the $n$-th investment period. Consequently, if appropriate conditions are met (for instance, when $T_n<T_F$ holds or $\mathbf{P}(T_n<T_F)$ is sufficiently high), the date $T_n$ can be selected as the terminal time over the interval $\rrbracket{T_{n-1},T_n}\rrbracket$. By doing so, the default event can be effectively excluded from the current time horizon.

  \item [(2)] Secondly, we focus on excluding unnecessary information beyond $\mathbb{B}$. In scenarios where $T_n<T_F$ holds, the uncertain time $T_F$ lies beyond $\mathbb{B}$, and the essential information regarding the stock price is determined by $S^{(n)}$ over the time-horizon $\llbracket{0,T_n\wedge T_{F^c}}\rrbracket$. Thus, the investor can effectively disregard extraneous information outside $\mathbb{B}$. Conversely, in cases where $T_n<T_F$ does not hold, the uncertain time still lies beyond $\mathbb{B}$ with probability $\mathbf{P}(T_n<T_F)$. Nevertheless, the necessary stock price information is predominantly determined by $(M^{(n)})^{T_n\wedge (T_F-)}$ (which serves as the local martingale exponent of $S^{(n)}$), excluding information beyond $\mathbb{B}$ in theory. In practice, investors in real financial markets can only make a finite number of investment decisions. Therefore, the investor can select an integer $N$ such that $\mathbf{P}(T_N<T_F)\geq p$ holds for a sufficiently high $p\in(0,1)$. By choosing $\{T_1\wedge T_{F^c},T_2\wedge T_{F^c},\cdots,T_N\wedge T_{F^c}\}$ as the terminal dates, the investor can exclude unnecessary information outside $\mathbb{B}$ with a sufficiently high probability.
\end{itemize}
To elucidate these two points, we present a simplified market model with a sudden-stop horizon and investigate the investor's optimal portfolio rules. Our findings reveal that the optimal strategy excludes information regarding the stock price after the uncertain time and does not convey any information about portfolio strategies beyond the time-horizon. This result confirms the exclusivity of the market with a sudden-stop horizon, thereby differentiating it from conventional markets with an uncertain time-horizon  (see, e.g., \cite{Blanchet,Lv,Yu}).

We also establish the essentials of mathematical finance for a generalized market framework denoted as $(S,\mathbb{F},\mathbb{B})$. Here, $\mathbb{B}$ represents a stochastic set of interval type, encapsulating the investor's sudden-stop horizon, while $S$ signifies a $\mathbb{B}$-inner semimartingale, modeling the stock price dynamics. By leveraging $\mathbb{B}$-stochastic integrals, we formally define self-financing strategies and admissible strategies within the context of $(S,\mathbb{F},\mathbb{B})$. Additionally, we delineate the conditions under which the market adheres to the principle of no-arbitrage. These definitions seamlessly degenerate into the conventional notions of self-financing strategies and admissible strategies, as well as the no-arbitrage conditions, within existing financial markets (see, e.g., \cite{Aksamit}).
More significantly, we elucidate the intrinsic relationship between the market $(S,\mathbb{F},\mathbb{B})$ and existing financial markets. Specifically, for each $n\in \mathbf{N}^+$, the market $(S,\mathbb{F},\mathbb{B})$ over the time-horizon $\mathbb{B}\llbracket{0,T_n}\rrbracket$ is equivalent to a conventional market $(S^{n},\mathbb{F})$, where $(T_n,S^{n})$ constitute an inner FCS for $S\in\mathcal{S}^{i,\mathbb{B}}$. This equivalence implies that the investor does not necessitate prior knowledge of a financial market characterized by a sudden-stop horizon. Within the confines of the specified time-horizon, the market can invariably be perceived as a conventional financial market, thereby facilitating a more intuitive and tractable analysis.

\subsection{Notation}\label{section1.4}
Let $(\Omega,\mathcal{F},\mathbf{P})$ denote a probability space, and let $\mathbb{F}=(\mathcal{F}_t,t\geq 0)$ represent a given filtration on this space that satisfies the usual conditions.
Unless otherwise explicitly specified, our analysis is grounded in the filtered probability space $(\Omega,\mathcal{F},\mathbb{F},\mathbf{P})$, which serves as the fundamental starting point for our investigation. The notation presented below is formulated in accordance with the research conducted by \cite{He}.

Denote the interval $\{x:a\leq x\leq b\}$, where $-\infty\leq a<b\leq {+\infty}$, by $[a,b]$, and analogous notation is used for $[a,b[$, $]a,b]$ and $]a,b[$. Let $\mathbf{R}$ represent the set of all real numbers, $\mathbf{R}^+=[0,{+\infty}[$ denote the set of all non-negative real numbers, and $\mathbf{N}^+=\{1,2,\cdots\}$ signify the set of all positive integers. For any $a\in\mathbf{R}$, we define
$a^+=\max\{a,0\}$ and $a^-=\max\{-a,0\}$.
Let $A$ and $B$ be two subsets of $\Omega$ (resp. $\Omega\times \mathbf{R}^+$).
The union and intersection of $A$ and $B$ are denoted by $A\cup B$ and $A\cap B$ (or simply $AB$), respectively, while the complement of $A$ is denoted by $A^c$.
The indicator function of the set $A$ is defined as follows:
\begin{equation*}
I_A(\omega)=\left\{
\begin{aligned}
1,&\quad \omega\in A,\\
0,&\quad \omega\in A^c.
\end{aligned}
\right.
\quad \left(\text{resp.}\quad
I_A(\omega,t)=\left\{
\begin{aligned}
1,&\quad (\omega,t)\in A,\\
0,&\quad (\omega,t)\in A^c.
\end{aligned}
\right.
\right)
\end{equation*}
For the sake of notational simplicity, the set comprising all elements that satisfy the property $P$, namely $\{\omega\in\Omega: P(\omega)\}$ (resp. $\{(\omega,t)\in\Omega\times \mathbf{R}^+: P(\omega,t)\}$), is denoted by $[P]$, provided that such notation introduces no ambiguity.

For two stopping times $S$ and $T$, we denote their minimum by $T\wedge S=\min\{T,S\}$. Subsequently, we define four distinct types of stochastic intervals as follows:
\[
\left\{
\begin{aligned}
\llbracket{S,T}\rrbracket&=\left\{(\omega,t)\in \Omega\times \mathbf{R}^+: S(\omega)\leq t\leq T(\omega)\right\},\\
\llbracket{S,T}\llbracket&=\left\{(\omega,t)\in \Omega\times \mathbf{R}^+: S(\omega)\leq t< T(\omega)\right\},\\
\rrbracket{S,T}\rrbracket&=\left\{(\omega,t)\in \Omega\times \mathbf{R}^+: S(\omega)< t\leq T(\omega)\right\},\\
\rrbracket{S,T}\llbracket&=\left\{(\omega,t)\in \Omega\times \mathbf{R}^+: S(\omega)< t< T(\omega)\right\}.
\end{aligned}
\right.
\]
In particular, we adopt the shorthand notation $\llbracket{T}\rrbracket=\llbracket{T,T}\rrbracket$, which represents the graph of $T$. Furthermore, for a sequence $(T_n)$ of stopping times, the notation $T_n\uparrow T$ signifies that $(T_n)$ is an increasing sequence of stopping times satisfying the limit condition $\lim\limits_{n\rightarrow+\infty}T_n=T$.

A stochastic process $(X_t)_{t\in \mathbf{R}^+}$ (or simply a process, which refers to a family of real random variables indexed by $\mathbf{R}^+$) is also denoted by $X$. By convention, for any c\`{a}dl\`{a}g process $X$, we set $X_{0-}=X_0$. Two processes are considered identical if they are indistinguishable: specifically, for two processes $X$ and $Y$, the relation $X=Y$ indicates that $X$ and $Y$ are indistinguishable.
For two subsets $C$ and $\widetilde{C}$ of $\Omega\times \mathbf{R}^+$ and a mapping $X: \widetilde{C}\rightarrow \mathbf{R}$, the equality $C=\widetilde{C}$ means $I_C=I_{\widetilde{C}}$, and the term $XI_C$ is introduced to streamline notation, which is defined as follows:
\begin{equation*}
(XI_C)(\omega,t)=\left\{
\begin{aligned}
&X(\omega,t),\quad&&(\omega,t)\in C\cap \widetilde{C},\\
&0,\quad&&\text{otherwise}.
\end{aligned}
\right.
\end{equation*}

Let $\mathcal{D}$ be a class of processes. We denote by $\mathcal{D}_0$ the subclass of $\mathcal{D}$ comprising all processes of $\mathcal{D}$ with null initial values. According to Definition 7.1 in \cite{He}, the localized class of $\mathcal{D}$, denoted by $\mathcal{D}_{\mathrm{loc}}$, is the collection of all processes $X$ satisfying the follows: $X_0$ is $\mathcal{F}_0$-measurable, and there exists a sequence $(T_n)$ of stopping times with $T_n\uparrow {+\infty}$ such that  $X^{T_n}-X_0\in\mathcal{D}$ for each $n\in \mathbf{N}^+$. The sequence $(T_n)$ is referred to as a localizing sequence for $X$ (w.r.t. $\mathcal{D}$).
$\mathcal{D}$ is called stable under stopping if $X^T\in\mathcal{D}$ for any $X\in \mathcal{D}$ and any stopping time $T$. Additionally, $\mathcal{D}$ is deemed stable under localization if it satisfies the equality $\mathcal{D}=\mathcal{D}_{\mathrm{loc}}$.

Throughout this paper, we use the following notations:
\begin{itemize}
  \item  $\mathcal{P}$ (resp. $\mathcal{C}$, resp. $\mathcal{R}$) --- the class of all predictable (resp. continuous, resp. c\`{a}dl\`{a}g) processes;
  \item  $\mathcal{V}$ (resp. $\mathcal{A}$) --- the class of all adapted (resp. adapted integrable) processes with finite variation;
  \item  $\mathcal{V}^+$ (resp. $\mathcal{A}^+$) --- the class of all adapted (resp. adapted integrable) increasing processes;
  \item  $\mathcal{M}_{\mathrm{loc}}$ (resp. $\mathcal{M}^c_{\mathrm{loc}}$, resp. $\mathcal{M}^d_{\mathrm{loc}}$) --- the class of all (resp. continuous, resp. purely discontinuous) local martingales;
  \item  $\mathcal{W}_{\mathrm{loc}}$ --- the class of all locally integrable variation martingale;
  \item  $\mathcal{S}$ --- the class of all semimartingales.
\end{itemize}
Note that $\mathcal{M}^d_{\mathrm{loc}}=\mathcal{M}^d_{\mathrm{loc},0}$. {\bf We emphasize that all elements of the set $\mathcal{S}$ are presumed to be c\`{a}dl\`{a}g.} Furthermore, the following classes are stable under stopping and localization: $\mathcal{P}$, $\mathcal{V}$, $\mathcal{A}_{\mathrm{loc}}$, $\mathcal{V}^+$, $\mathcal{A}^+_{\mathrm{loc}}$, $\mathcal{M}_{\mathrm{loc}}$, $\mathcal{M}^c_{\mathrm{loc}}$, $\mathcal{M}^d_{\mathrm{loc}}$, $\mathcal{W}_{\mathrm{loc}}$, and $\mathcal{S}$, where $\mathcal{A}^+_{\mathrm{loc}}=(\mathcal{A}^+)_{\mathrm{loc}}$. For a detailed proof, one may refer to, for instance, \cite{He,Jacod}.

This paper is organised as follows.
In the forthcoming section, we lay the groundwork for our study by introducing pertinent preliminaries. Specifically, we present the fundamental properties of $\mathbb{B}$-processes and delve into an investigation of $\mathbb{B}$-jump processes.
In Section \ref{section3}, we tackle Problem (SI-1) by formally defining and scrutinizing Stieltjes integrals of $\mathbb{B}$-predictable processes with respect to $\mathbb{B}$-adapted processes with finite variation.
Proceeding to Section \ref{section4}, we construct $\mathbb{B}$-quadratic covariations and define $\mathbb{B}$-inner local martingales. By leveraging the concept of $\mathbb{B}$-quadratic covariation, we establish stochastic integrals of $\mathbb{B}$-predictable processes with respect to $\mathbb{B}$-inner local martingales. This development offers a definitive and affirmative solution to Problem (SI-2).
In Section \ref{section5}, we integrate the two types of $\mathbb{B}$-stochastic integrals introduced in the preceding sections. Specifically, we construct stochastic integrals of $\mathbb{B}$-predictable processes with respect to $\mathbb{B}$-inner semimartingales, which provides an affirmative resolution to Problem (SI). Additionally, we present It\^{o}'s formula tailored for $\mathbb{B}$-inner semimartingales.
Finally, in Section \ref{section6}, we apply the $\mathbb{B}$-stochastic integrals developed in the preceding sections to the analysis of financial markets characterized by sudden-stop horizons.

\section{Preliminaries}\label{section2}\noindent
\setcounter{equation}{0}
In this section, we delineate the preliminaries to our study. Initially, we expound upon the fundamental properties of general $\mathbb{B}$-processes, which serve as the bedrock for the construction of $\mathbb{B}$-stochastic integration. Building upon this, we then delve into an examination of a significant class of $\mathbb{B}$-jump processes. These processes constitute a potent analytical instrument, particularly in the context of $\mathbb{B}$-semimartingales.

In this paper, our focus is directed towards optional sets of interval type and predictable sets of interval type.  Whenever stochastic sets of interval type are referenced, it is specifically in the context of optional sets of interval type or predictable sets of interval type. {\bf For the sake of simplicity, we employ the notations $\mathbb{B}$ and $\mathbb{C}$ throughout the rest of our paper to represent a stochastic set of interval type and a predictable set of interval type, respectively.}

\subsection{Fundamental properties of $\mathbb{B}$-processes}
The subsequent result, originated from Theorem 8.17 in \cite{He}, provides a characterization of stochastic sets of interval type by expressing them in the form of stochastic intervals.

\begin{lemma}\label{th8.17}
\begin{description}
  \item [$(1)$] $\mathbb{B}$ is an optional (resp. predictable) set of interval type if and only if
$I_\mathbb{B}=I_FI_{\llbracket{0,T}\llbracket}+I_{F^c}I_{\llbracket{0,T}\rrbracket}$, i.e.,
      \begin{equation}\label{B}
         \mathbb{B}=\llbracket{0,T_F}\llbracket\;\cap\;\llbracket{0,T_{F^c}}\rrbracket,
      \end{equation}
      where $T$ is a stopping time and the debut of $\mathbb{B}^c$,
      and $F\in \mathcal{F}_{T}$, and $T_F=TI_{F}+(+\infty)I_{F^c}>0$ is a (resp. predictable) stopping time.
  \item  [$(2)$] $\mathbb{C}$ is a predictable set of interval type if and only if $\mathbb{C}=\bigcup\limits_{n=1}^{{+\infty}} \llbracket{0,\tau_n}\rrbracket$, where $(\tau_n)$ is an increasing sequence of stopping times. The sequence $(\tau_n)$ is called a {\bf fundamental sequence} (in short: FS) for $\mathbb{C}$.
\end{description}
\end{lemma}

\begin{remark}\normalfont
Let $\mathbb{B}$ be defined as specified in \eqref{B}. The stopping time $T$, representing the debut of $\mathbb{B}^c$, is uniquely determined by $\mathbb{B}$. Conversely, the set $F$ may not be uniquely specified. For instance, consider the following equality:
\[
\llbracket{0,+\infty}\llbracket=\llbracket{0,T_\Omega}\llbracket\;\cap\;\llbracket{0,T_{\Omega^c}}\rrbracket
=\llbracket{0,T_\emptyset}\llbracket\;\cap\;\llbracket{0,T_{\emptyset^c}}\rrbracket,
\]
where $T=+\infty$. Nonetheless, the stopping time $T_F$ remains uniquely determined by $\mathbb{B}$. To elucidate this point, suppose there exists another set $G\in \mathcal{F}_{T}$ such that $T_G>0$ and
\begin{equation*}
    \mathbb{B}=\llbracket{0,T_G}\llbracket\;\cap\;\llbracket{0,T_{G^c}}\rrbracket.
\end{equation*}
It then follows that $F\cap[T<+\infty]=G\cap[T<+\infty]$. Consequently, it is rigorously established that $T_F=T_G$.
\end{remark}

In order to enhance the applicability of stochastic processes on stochastic sets of interval type, we hereby introduce the following concepts:
\begin{itemize}
  \item [(1)] Let $X$ and $Y$ be two $\mathbb{B}$-processes. $Y$ is said to be a $\mathbb{B}$-\textbf{modification} of $X$, if $XI_\mathbb{B}$ is a modification of $YI_\mathbb{B}$. $X$ and $Y$ are said to be $\mathbb{B}$-\textbf{indistinguishable}, if $XI_\mathbb{B}$ and $YI_\mathbb{B}$ are indistinguishable.
      As usual, two indistinguishable $\mathbb{B}$-processes are regarded as the same. Specifically, the relation $X=Y$ is equivalently to $XI_\mathbb{B}=YI_\mathbb{B}$.
  \item [(2)] Let $\widetilde{\mathbb{B}}$ be another stochastic sets of interval type, with the property that $\mathbb{B}\subseteq \widetilde{\mathbb{B}}$. Suppose $X$ is a $\widetilde{\mathbb{B}}$-process. We introduce the concept of the $\mathbb{B}$-\textbf{restriction} of $X$, denoted by $X\mathfrak{I}_\mathbb{B}$, which is defined as a real-valued function on $\mathbb{B}$ that satisfies the condition $(X\mathfrak{I}_\mathbb{B})I_\mathbb{B}=XI_\mathbb{B}$. As a consequence of this definition, for any class $\mathcal{D}$ and any process $Y\in \mathcal{D}$, it holds that $Y\mathfrak{I}_\mathbb{B}\in \mathcal{D}^\mathbb{B}$ with the FCS $(T_n=T,Y)$, where $T$ is the debut of $\mathbb{B}^c$.
   \item [(3)]  A stopping time $\tau$ is called a $\mathbb{B}$-\textbf{inner stopping time} if $\llbracket{0,\tau}\rrbracket\subseteq \mathbb{B}$. Let $X$ be a $\mathbb{B}$-process, and let $T$ and $S$ denote two $\mathbb{B}$-inner stopping times.
       Analogous to conventional stopped processes, we introduce the stopped process $X^T$, which is defined as follows:
       \begin{equation}\label{stop}
        X^T=XI_{\llbracket{0,T}\rrbracket}+X_TI_{\rrbracket{T,{+\infty}}\llbracket}.
       \end{equation}
It can be effortlessly verified that the relation $(X^T)^S=X^{T\wedge S}=(X^S)^T$ holds true.
\end{itemize}

To enhance the utilization of general $\mathbb{B}$-processes, we present a summary of their fundamental properties in the subsequent theorem and corollary.
\begin{theorem}\label{process}
Let $\mathbb{B}$ be given by \eqref{B}, $\mathcal{D}$ be a class of processes, and $X,Y\in \mathcal{D}^\mathbb{B}$. Suppose that $(T_n,X^{(n)})$ is an FCS for $X\in \mathcal{D}^\mathbb{B}$  (resp. a CS for $X$), and that $(S_n)$ is an increasing sequence of stopping times with $S_n\uparrow T$ and $\bigcup\limits_{n=1}^{{+\infty}}\llbracket{0,S_n}\rrbracket\supseteq \mathbb{B}$.
\begin{itemize}
  \item[$(1)$] $X=Y$ if and only if $XI_{\mathbb{B}\llbracket{0,S_n}\rrbracket}=YI_{\mathbb{B}\llbracket{0,S_n}\rrbracket}$ for each $n\in \mathbf{N}^+$.
  \item[$(2)$] $X=X^{(k)}=X^{(l)}$ on $\mathbb{B}\llbracket{0,T_k}\rrbracket$ for any $k,\;l\in \mathbf{N}^+$ with $k\leq l$, i.e.,
      \begin{equation}\label{xkl}
      XI_{\mathbb{B}\llbracket{0,T_k}\rrbracket}=X^{(k)}I_{\mathbb{B}\llbracket{0,T_k}\rrbracket}
      =X^{(l)}I_{\mathbb{B}\llbracket{0,T_k}\rrbracket}.
      \end{equation}
      Specially, $X^{(k)}I_{\llbracket{0}\rrbracket}=XI_{\llbracket{0}\rrbracket}$.
  \item[$(3)$] $(\tau_n,X^{(n)})$ is an FCS for $X\in \mathcal{D}^\mathbb{B}$ (resp. a CS for $X$), where $\tau_n=T_n\wedge S_n$ for each $n\in \mathbf{N}^+$.
  \item[$(4)$] Suppose that $\mathcal{D}$ satisfies the linearity: $aU+bV\in \mathcal{D}$ holds for all $U,V\in \mathcal{D}$ and all $a,b\in \mathbf{R}$.
      Then $aX+bY\in \mathcal{D}^\mathbb{B}$ holds for all $a,b\in \mathbf{R}$.
  \item[$(5)$] $X$ can be expressed as
      \begin{equation}\label{x-expression}
      X=\left(X_0I_{\llbracket{0}\rrbracket}+\sum\limits_{n=1}^{{+\infty}}X^{(n)}I_{\rrbracket{T_{n-1},T_n}
      \rrbracket}\right)\mathfrak{I}_\mathbb{B},\quad T_0=0.
      \end{equation}
      Furthermore, if $(S_n,\widetilde{X}^{(n)})$ is also an FCS for $X\in \mathcal{D}^\mathbb{B}$ (resp. a CS for $X$), then $X=\widetilde{X}$ where $\widetilde{X}$ is given by
      \begin{equation*}
      \widetilde{X}=\left(X_0I_{\llbracket{0}\rrbracket}+\sum\limits_{n=1}^{{+\infty}}\widetilde{X}^{(n)}I_{\rrbracket{S_{n-1},S_n}
      \rrbracket}\right)\mathfrak{I}_\mathbb{B},\quad S_0=0.
      \end{equation*}
      In this case, we say the expression of \eqref{x-expression} is independent of the choice of the FCS $(T_n,X^{(n)})$ for $X\in \mathcal{D}^\mathbb{B}$ (resp. the CS $(T_n,X^{(n)})$ for $X$).
      \end{itemize}
\end{theorem}
\begin{corollary}\label{process-FS}
Let $(\tau_n)$ be an FS for $\mathbb{C}$, $\mathcal{D}$ be a class of processes, and $X,Y\in \mathcal{D}^\mathbb{C}$. Suppose that $(T_n,X^{(n)})$ is an FCS for $X\in \mathcal{D}^\mathbb{C}$  (resp. a CS for $X$). Put $S_n=T_n\wedge \tau_n$ for each $n\in \mathbf{N}^+$.
\begin{itemize}
  \item [$(1)$] $X=Y$ if and only if for each $n\in \mathbf{N}^+$, $XI_{\llbracket{0,\tau_n}\rrbracket}=YI_{\llbracket{0,\tau_n}\rrbracket}$, or equivalently, $X^{\tau_n}=Y^{\tau_n}$.
  \item [$(2)$] $(S_n)$ is also an FS for $\mathbb{C}$, and $(S_n,X^{(n)})$ is also an FCS for $X\in \mathcal{D}^\mathbb{C}$ (resp. a CS for $X$) satisfying
       \begin{equation*}
       X^{S_n}=(X^{(n)})^{S_n}, \quad n\in \mathbf{N}^+.
       \end{equation*}
  \item [$(3)$] $X$ can be expressed as
      \begin{equation}\label{x-expression-FS}
      X=\left(X_0I_{\llbracket{0}\rrbracket}+\sum\limits_{n=1}^{{+\infty}}X^{\tau_n}I_{\rrbracket{\tau_{n-1},\tau_n}
      \rrbracket}\right)\mathfrak{I}_\mathbb{C},\quad \tau_0=0.
      \end{equation}
      Furthermore, if $(\widetilde{\tau}_n)$ is also an FS for $\mathbb{C}$, then $X=\widetilde{X}$ where  $\widetilde{X}$ is given by
      \begin{equation*}
      \widetilde{X}=\left(X_0I_{\llbracket{0}\rrbracket}+\sum\limits_{n=1}^{{+\infty}}X^{\widetilde{\tau}_n}
      I_{\rrbracket{\widetilde{\tau}_{n-1},\widetilde{\tau}_n}
      \rrbracket}\right)\mathfrak{I}_\mathbb{C},\quad \widetilde{\tau}_0=0.
      \end{equation*}
      In this case, we say the expression of \eqref{x-expression-FS} is independent of the choice of FS $(\tau_n)$.
\end{itemize}
\end{corollary}

Leveraging the concept of $\mathbb{B}$-restriction, we are able to derive the following results. Notably, the third result is drawn from Theorem 8.22 in \cite{He}.
\begin{theorem}\label{restriction}
\begin{itemize}
  \item [$(1)$] $H$ is a $\mathbb{B}$-predictable (resp. $\mathbb{B}$-locally bounded predictable) process if and only if there exists a predictable (resp. locally bounded predictable) process $\widetilde{H}$ satisfying
\begin{equation}\label{couple}
H=\widetilde{H}\mathfrak{I}_\mathbb{B}.
\end{equation}
  The process $\widetilde{H}$ is called a {\bf coupled  predictable} (resp. {\bf locally bounded predictable}) {\bf process} for $H$.
  \item [$(2)$] Let $X$ be a $\mathbb{B}$-optional process. Then $X$ is a $\mathbb{B}$-thin process if and only if $XI_\mathbb{B}$ is a thin process. Provided that for all $t>0$,
      $\sum_{s\leq t}|(XI_\mathbb{B})_s|$ is a.s. finite, the $\mathbb{B}$-restriction of $\sum_{s\leq \cdot}(XI_\mathbb{B})_s$ is called the $\mathbb{B}$-{\bf summation process} of $X$, denoted by $\Sigma X$, i.e.,
   \[
   \Sigma X=\left(\sum_{s\leq \cdot}(XI_\mathbb{B})_s\right)\mathfrak{I}_\mathbb{B}.
   \]
  \item [$(3)$] Let $\mathcal{D}$ be a class of processes which is stable under stopping, and suppose that $X\in \mathcal{D}^\mathbb{B}$ with the FCS $(T_n,X^{(n)})$. Define the predictable set $\mathbb{C}$ of interval type as
\[
\mathbb{C}=\bigcup\limits_{n}\llbracket{0,T_n}\rrbracket.
\]
Then there exists $\widetilde{X}\in\mathcal{D}^\mathbb{C}$ with the FCS $(T_n,\widetilde{X}^{T_{n}})$ satisfying $XI_\mathbb{B}=\widetilde{X}I_\mathbb{B}$ and
\begin{equation}\label{continuation}
\widetilde{X}^{T_1}=(X^{(1)})^{T_1},\quad \widetilde{X}^{T_{n+1}}=\widetilde{X}^{T_{n}}+(X^{(n+1)})^{T_{n+1}}-(X^{(n+1)})^{T_n}, \; n\in \mathbf{N}^+.
\end{equation}
  The sequence $((T_n,X^{(n)}),\widetilde{X})$ (or simply $(T_n,X^{(n)},\widetilde{X})$) is called a {\bf continuation} for $X\in \mathcal{D}^\mathbb{B}$.
  \item [$(4)$] Let $\mathbb{B}$ be given by \eqref{B}, and let $\mathcal{D}$ be a class of processes. Suppose that $(X^{(n)})\subseteq \mathcal{D}$ denotes a sequence of processes, and that $(T_n)$ is an increasing sequence of stopping times, satisfying $T_n\uparrow T$ and $\bigcup\limits_{n=1}^{{+\infty}}\llbracket{0,T_n}\rrbracket\supseteq \mathbb{B}$. If for any $k, l\in \mathbf{N}^+$ with $k\leq l$,
      \begin{equation*}
      X^{(k)}I_{\mathbb{B}\llbracket{0,T_k}\rrbracket}=X^{(l)}I_{\mathbb{B}\llbracket{0,T_k}\rrbracket},
      \end{equation*}
then $X\in \mathcal{D}^\mathbb{B}$ with the FCS $(T_n,X^{(n)})$, where $X$ is defined by
\begin{equation*}
      X=\left(X^{(1)}_0I_{\llbracket{0}\rrbracket}+\sum\limits_{n=1}^{{+\infty}}X^{(n)}I_{\rrbracket{T_{n-1},T_n}
      \rrbracket}\right)\mathfrak{I}_\mathbb{B},\quad T_0=0.
      \end{equation*}
\end{itemize}
\end{theorem}

The subsequent theorem, as derived from Theorem 8.20 in \cite{He}, highlights the crucial significance of the stopped process corresponding to \eqref{stop}.
\begin{theorem}\label{fcs}
Let $S$ be a $\mathbb{B}$-inner stopping time, and let $\mathcal{D}$ be a class of processes. Suppose that $X\in \mathcal{D}^\mathbb{B}$ with an FCS $(T_n,X^{(n)})$. If the class $\mathcal{D}$ is stable under stopping and localization, then $X^S\in \mathcal{D}$, and $(T_n,(X^{(n)})^S)$ is an FCS for $X^S\mathfrak{I}_\mathbb{B}\in \mathcal{D}^\mathbb{B}$.
\end{theorem}

\begin{corollary}\label{fcs-p}
Let $\mathcal{D}$ be a class of processes, and $X\in \mathcal{D}^\mathbb{C}$. Suppose that $(\tau_n)$ is an FS for $\mathbb{C}$.
\begin{enumerate}
  \item [$(1)$] $(\tau_n,X^{\tau_n})$ forms a CS for $X$.
  \item [$(2)$] If $\mathcal{D}$ is stable under stopping and localization, then $(\tau_n,X^{\tau_n})$ forms an FCS for $X\in\mathcal{D}^\mathbb{C}$.
\end{enumerate}
\end{corollary}

\begin{corollary}\label{cD=DB}
Let $\mathcal{D}$ be a class of processes. If $\mathcal{D}$ is stable under stopping and localization, then $\mathcal{D}^{\llbracket{0,+\infty}\llbracket}=\mathcal{D}$.
\end{corollary}

The subsequent theorem elucidates two key properties pertaining to $\mathbb{B}$-summation processes, which play a pivotal role in the definition and examination of $\mathbb{B}$-quadratic covariations (see Section \ref{section4.1}).

\begin{theorem}\label{thin}
\begin{itemize}
  \item [$(1)$] Let $X$ be a $\mathbb{B}$-thin process with an FCS $(T_n,X^{(n)})$. If for each $n\in \mathbf{N}^+$, $\Sigma X^{(n)}$ is well-defined,
       then $(T_n,\Sigma X^{(n)})$ is an FCS for $\Sigma X\in \mathcal{V}^\mathbb{B}$.
  \item [$(2)$] Let $\tau$ be a $\mathbb{B}$-inner stopping time. If $X$ be a $\mathbb{B}$-thin process and $\Sigma X$ is well-defined,  then $XI_{\llbracket{0,\tau}\rrbracket}$ is a thin process satisfying
      \begin{equation}\label{sigmaX}
      \Sigma (XI_{\llbracket{0,\tau}\rrbracket})=(\Sigma X)^\tau.
      \end{equation}
\end{itemize}
\end{theorem}

\subsection{Jump processes of $\mathbb{B}$-c\`{a}dl\`{a}g processes}
In this subsection, we set $X_{0-}=X_0$ for any process $X$ that we consider.
Following the left-hand limit process (or simply, its left-limit process) of a c\`{a}dl\`{a}g process, we embark on defining left-limit processes for general $\mathbb{B}$-processes, assuming their existence, and subsequently delve into the exploration of their fundamental properties.

\begin{definition}\label{X+-}\normalfont
Let $X$ be a $\mathbb{B}$-process.
If for all $(\omega,t)\in \mathbb{B}$ with $t>0$, the left-hand limits $X(\omega,t-)$ exist, then the {\bf left-limit process} of $X$, denoted by $X_{-}$, is defined by
  \[
   X_{-}(\omega,t)=\left\{
   \begin{aligned}
    &X(\omega,0),&&\quad \omega\in \Omega,\;t=0,\\
    &X(\omega,t-)=\lim\limits_{s<t,s\uparrow t}X(\omega,s),&&\quad (\omega,t)\in \mathbb{B},\; t>0.
    \end{aligned}
    \right.
  \]
\end{definition}

Let $X$ be a $\mathbb{B}$-process. Analogously, we say the left-limit process $X_{-}$ exists, if $X_{-}$ is well-defined according to Definition \ref{X+-}. It is evident that $X_{-}$ remains a $\mathbb{B}$-process. More importantly, the definition of $X_{-}$ bears a close resemblance to that of a conventional left-limit process.
In brief, the left-limit process $X_{-}$ exists if, for each $\omega\in \Omega$, the path $X_{.}(\omega)$ admits finite left-hand limits on the section $\mathbb{B}_\omega=\{t: (\omega,t)\in \mathbb{B}\}$.

The following two results delineate sufficient conditions for the existence of the left-limit process of a $\mathbb{B}$-process, further elucidating the relationship between this left-limit process and conventional left-limit processes.
\begin{theorem}\label{X-left}
Let $X$ be a $\mathbb{B}$-process.
If there exists a CS $(T_n,X^{(n)})$ for $X$ such that $(X^{(n)})_{-}$ exists for each $n\in \mathbf{N}^+$, then $X_{-}$ exists, and $(T_n,(X^{(n)})_{-})$ is a CS for $X_{-}$.
\end{theorem}

\begin{corollary}\label{deltaX}
\begin{itemize}
  \item [$(1)$] If $X\in \mathcal{R}^\mathbb{B}$ with the FCS $(T_n,X^{(n)})$, then $X_{-}$ exists, and $(T_n,(X^{(n)})_{-})$ is a CS for $X_{-}$.
  \item [$(2)$] If $X$ is a $\mathbb{B}$-adapted c\`{a}dl\`{a}g process with the FCS $(T_n,X^{(n)})$, then $X_{-}$ is a $\mathbb{B}$-locally bounded predictable process, and $(T_n,(X^{(n)})_{-})$ is an FCS for $X_{-}$ (a $\mathbb{B}$-locally bounded predictable process).
  \item [$(3)$] Let $X$ be a $\mathbb{C}$-process, and $(\tau_n)$ be an FS for $\mathbb{C}$. If $(X^{\tau_n})_{-}$ exists for each $n\in \mathbf{N}^+$, then $X_{-}$ exists, and $(\tau_n,(X^{\tau_n})_{-})$ is a CS for $X_{-}$. Furthermore, if $X$ is a $\mathbb{C}$-adapted c\`{a}dl\`{a}g process, then $(\tau_n,(X^{\tau_n})_{-})$ is an FCS for $X_{-}$ (a $\mathbb{C}$-locally bounded predictable process).
\end{itemize}
\end{corollary}

Drawing upon Definition \ref{X+-} and Corollary \ref{deltaX}(1), we are poised to introduce the jump process of a $\mathbb{B}$-c\`{a}dl\`{a}g process.
Given $X\in \mathcal{R}^\mathbb{B}$, we adhere to the conventional notation by denoting the {\bf jump process} of $X$ as
 \begin{equation}\label{X-X-}
  \Delta X=X-X_{-}.
 \end{equation}
Corollary \ref{deltaX}(1) assures us that the jump process $\Delta X$ is well-defined. The forthcoming theorem is dedicated to presenting the fundamental properties of $\mathbb{B}$-jump processes.

\begin{theorem}\label{delta}
Let $X,Y\in \mathcal{R}^\mathbb{B}$, and $Z\in\mathcal{R}$.
\begin{itemize}
  \item [$(1)$] If $(T_n,X^{(n)})$ is an FCS for $X\in \mathcal{R}^\mathbb{B}$, then $(T_n,\Delta X^{(n)})$ is a CS for $\Delta X$. Furthermore, if $X\in \mathcal{R}^\mathbb{C}$ and $(\tau_n)$ is an FS for $\mathbb{C}$, then $(\tau_n,\Delta X^{\tau_n})$ is a CS for $\Delta X$.

  \item [$(2)$] For all $a\in \mathbf{R}$,
\begin{align}\label{cadlag-linear}
\Delta(aX)=a \Delta X, \quad \Delta(X+Y)= \Delta X+\Delta Y.
\end{align}

  \item [$(3)$] If $\widetilde{\mathbb{B}}$ is another stochastic set of interval type satisfying $\widetilde{\mathbb{B}}\subseteq \mathbb{B}$, then
  \begin{equation}\label{eqXT0}
  \Delta (X\mathfrak{I}_{\widetilde{\mathbb{B}}})=(\Delta X)\mathfrak{I}_{\widetilde{\mathbb{B}}}.
  \end{equation}
  Specially, $\Delta (Z\mathfrak{I}_{\mathbb{B}})=(\Delta Z)\mathfrak{I}_{\mathbb{B}}$.

  \item [$(4)$] If $T$ is a $\mathbb{B}$-inner stopping time, then
  \begin{align}\label{eqXT}
\Delta X^T=\Delta XI_{\llbracket{0,T}\rrbracket},\quad \Delta X^{T-}=\Delta XI_{\llbracket{0,T}\llbracket},
\end{align}
  where $X^{T-}$ is defined by
 \begin{equation*}
     X^{T-}=XI_{\llbracket{0,T}\llbracket}+X_{T-}I_{\llbracket{T,{+\infty}}\llbracket}.
  \end{equation*}
  Specially, $\Delta Z^S=\Delta ZI_{\llbracket{0,S}\rrbracket}$ and $\Delta Z^{S-}=\Delta ZI_{\llbracket{0,S}\llbracket}$, where $S$ is a stopping time.

  \item [$(5)$] $\Delta X=0\mathfrak{I}_\mathbb{B}$ if and only if $X\in \mathcal{C}^\mathbb{B}$.
  \item [$(6)$] Assume further that $X,Y\in \mathcal{S}^\mathbb{B}$. Then $\Delta X\Delta Y$ is a $\mathbb{B}$-thin process satisfying $\Sigma (\Delta X\Delta Y)\in \mathcal{V}^\mathbb{B}$. Furthermore, if $(T_n,X^{(n)})$ and $(T_n,Y^{(n)})$ are FCSs for $X\in \mathcal{S}^\mathbb{B}$ and $Y\in \mathcal{S}^\mathbb{B}$ respectively, then $(T_n,\Delta X^{(n)}\Delta Y^{(n)})$ is an FCS for $\Delta X\Delta Y$ (a $\mathbb{B}$-thin process), and $(T_n,\Sigma(\Delta X^{(n)}\Delta Y^{(n)}))$ is an FCS for $\Sigma (\Delta X\Delta Y)\in \mathcal{V}^\mathbb{B}$.
\end{itemize}
\end{theorem}

\section{Stieltjes integrals of $\mathbb{B}$-predictable processes with respect to \\
$\mathbb{B}$-adapted processes with finite variation}\label{section3}
\setcounter{equation}{0}
In this section, we address Problem (SI-1) by extending the existing framework of Stieltjes integrals to accommodate the examination of integrals of $\mathbb{B}$-predictable processes with respect to $\mathbb{B}$-adapted processes of finite variation.
Subsequently, we proceed to investigate the fundamental properties inherent to these $\mathbb{B}$-Stieltjes integrals. Unless otherwise explicitly specified, we shall consistently assume, throughout this section, that $H\in \mathcal{P}^\mathbb{B}$ and $A\in \mathcal{V}^\mathbb{B}$.

\subsection{Definition of $H_{\bullet}A$}
In response to Problem (SI-1), we present the following affirmative resolution, which is grounded in the existing framework of Stieltjes integrals (for instance, as delineated in Definition 3.45 of \cite{He}).

\begin{definition}\normalfont\label{HA}
We say that $H$ is $\mathbb{B}$-integrable w.r.t. $A$, if the following two conditions are satisfied:
\begin{itemize}
  \item [$(1)$] For all $(\omega,t)\in \mathbb{B}$, it holds that
   \[
   \int_{[0,t]}|H_s(\omega)||dA_s(\omega)|<{+\infty};
   \]
  \item [$(2)$] $L\in \mathcal{V}^\mathbb{B}$, where $L$ is defined by
  \begin{equation}\label{HA-de}
   L(\omega,t)=\int_{[0,t]}H_s(\omega)dA_s(\omega),\quad (\omega,t)\in \mathbb{B}.
  \end{equation}
  \end{itemize}
In this case, the $\mathbb{B}$-process $L$, denoted by $H_{\bullet}A$, is called the Stieltjes integral of $H$ w.r.t $A$. The collection of all $\mathbb{B}$-predictable processes which are $\mathbb{B}$-integrable w.r.t. $A$ is denoted by $\mathcal{L}^\mathbb{B}_s(A)$.
\end{definition}

\begin{remark}\normalfont\normalfont
Let $H\in\mathcal{L}^\mathbb{B}_s(A)$.
\begin{itemize}
  \item [$(1)$] As exemplified by the instances of $\widehat{H}\in \mathcal{P}^{\widehat{\mathbb{B}}}$ and $\widehat{A}\in \mathcal{V}^{\widehat{\mathbb{B}}}$ in Section 1.1, it is essential to incorporate condition (2) into Definition \ref{HA}. This inclusion ensures that the $\mathbb{B}$-stochastic integral $H_{\bullet}A$ possesses the requisite fundamental property.

  \item [$(2)$]
  Suppose that $(T_n, H^{(n)})$ and $(T_n,A^{(n)})$ are FCSs for $H\in \mathcal{P}^\mathbb{B}$ and $A\in \mathcal{V}^\mathbb{B}$, respectively.
  For any given $(\omega,t)\in \mathbb{B}$, there exists an integer $m$ such that $(\omega,t)\in {\mathbb{B}\llbracket{0,T_m}\rrbracket}$. Consequently, it holds that $H(\omega,s)=H^{(m)}(\omega,s)$ and $A(\omega,s)=A^{(m)}(\omega,s)$ for all $s\in [0,t]$. Therefore, the integral $H_{\bullet}A$, as defined by \eqref{HA-de}, can also be interpreted as the conventional Stieltjes integral within the domain $\mathbb{B}$.

  \item [$(3)$]
  Analogous to the definition of $H_{\bullet}A$ provided in \eqref{HA-de},
  we can also define the Stieltjes integral of a $\mathbb{B}$-measurable process w.r.t.  a $\mathbb{B}$-process with finite variation. However, this extension will not be utilized within the scope of this paper.

  \item [$(4)$]
  From Corollary \ref{cD=DB}, it follows that the $\mathbb{B}$-Stieltjes integral $H_{\bullet}A$, as defined by \eqref{HA-de}, simplifies to the conventional Stieltjes integral $H.A$ when $\mathbb{B}=\llbracket{0,+\infty}\llbracket=\Omega\times\mathbf{R}^+$. To elaborate, the precise relationship between these integrals is as follows:
  \begin{itemize}
  \item [] Let $H\in \mathcal{P}^{\llbracket{0,+\infty}\llbracket}=\mathcal{P}$ and $A\in \mathcal{V}^{\llbracket{0,+\infty}\llbracket}=\mathcal{V}$. If $H\in\mathcal{L}_s^{\llbracket{0,+\infty}\llbracket}(A)$, then $H\in\mathcal{L}_s(A)$ and $H_{\bullet}A=H.A$.
  \end{itemize}
\end{itemize}
\end{remark}

\subsection{Fundamental properties of $H_{\bullet}A$}
Utilizing Definition \ref{HA}, one can readily deduce the following properties pertaining to the $\mathbb{B}$-Stieltjes integral $H_{\bullet}A$. These properties are consistent with the well-established characteristics inherent in conventional Stieltjes integrals.

\begin{theorem}\label{HAproperty}
Let $V\in\mathcal{V}^\mathbb{B}$, and $a,b\in \mathbf{R}$. Suppose that $H\in \mathcal{L}^\mathbb{B}_s(A)\cap\mathcal{L}^\mathbb{B}_s(V)$ and $K\in \mathcal{L}^\mathbb{B}_s(A)$.
\begin{itemize}
\item[$(1)$] $aH+bK\in \mathcal{L}^\mathbb{B}_s(A)$, and in this case, we have
  \begin{equation}\label{ab}
   (aH+bK)_{\bullet}A=a(H_{\bullet}A)+b(K_{\bullet}A).
  \end{equation}
\item[$(2)$] $H\in \mathcal{L}^\mathbb{B}_s(aA+bV)$, and in this case, it holds that
  \begin{equation}\label{ab2}
   H_{\bullet}(aA+bV)=a(H_{\bullet}A)+b(H_{\bullet}V).
  \end{equation}
\item[$(3)$] Let $\widetilde{H}\in \mathcal{P}^\mathbb{B}$. Then $\widetilde{H}\in \mathcal{L}^\mathbb{B}_s(H_{\bullet}A)$ if and only if $\widetilde{H}H\in \mathcal{L}^\mathbb{B}_s(A)$.
  Furthermore, if $\widetilde{H}\in \mathcal{L}^\mathbb{B}_s(H_{\bullet}A)$ (or equivalently, $\widetilde{H}H\in \mathcal{L}^\mathbb{B}_s(A)$), then
  \begin{equation}\label{ab3}
  (\widetilde{H}H)_{\bullet}A=\widetilde{H}_{\bullet}(H_{\bullet}A).
  \end{equation}
  \end{itemize}
\end{theorem}

The following two theorems elucidate the intricate relationship between the $\mathbb{B}$-Stieltjes integral $H_{\bullet}A$ and the conventional Stieltjes integrals. Specifically, Theorem \ref{HA-equivalent} elaborates on the equivalent conditions governing the existence of $H_{\bullet}A$, thereby unveiling that the integrability of such $\mathbb{B}$-integrals is fundamentally contingent upon the integrability of the traditional Stieltjes integrals. Moreover, Theorem \ref{HA-FCS} provides a characterization of $H_{\bullet}A$ as the summation of a sequence of conventional Stieltjes integrals, thereby offering a deeper insight into the structural composition of the $\mathbb{B}$-Stieltjes integral.

\begin{theorem}\label{HA-equivalent}
The following assertions are equivalent:
\begin{description}
  \item [(A1)] $H\in\mathcal{L}^\mathbb{B}_s(A)$.
  \item [(A2)] There exist a coupled predictable process $\widetilde{H}$ for $H\in\mathcal{P}^\mathbb{B}$ and an FCS $(T_n,A^{(n)})$ for $A\in\mathcal{V}^\mathbb{B}$ such that for each $n\in \mathbf{N}^+$, $\widetilde{H}\in \mathcal{L}_s(A^{(n)})$.
  \item [(A3)] There exist FCSs $(T_n, H^{(n)})$ for $H\in\mathcal{P}^\mathbb{B}$ and $(T_n,A^{(n)})$ for $A\in\mathcal{V}^\mathbb{B}$ such that for each $n\in \mathbf{N}^+$, $H^{(n)}\in \mathcal{L}_s(A^{(n)})$.
\end{description}
\end{theorem}

\begin{remark}\normalfont
By appealing to Theorem \ref{process}(3), the statement labeled as (A3) in Theorem \ref{HA-equivalent} can be rephrased into an equivalent condition, as stated below:
\begin{description}
  \item[(A3$'$)]
  There exist FCSs $(T_n, H^{(n)})$ for $H\in\mathcal{P}^\mathbb{B}$ and $(S_n,A^{(n)})$ for $A\in\mathcal{V}^\mathbb{B}$ such that for each $n\in \mathbf{N}^+$, $H^{(n)}\in \mathcal{L}_s(A^{(n)})$.
\end{description}
\end{remark}

\begin{theorem}\label{HA-FCS}
Suppose that $(T_n, H^{(n)})$ and $(T_n,A^{(n)})$ are FCSs for $H\in\mathcal{P}^\mathbb{B}$ and $A\in\mathcal{V}^\mathbb{B}$ respectively such that for each $n\in \mathbf{N}^+$, $H^{(n)}\in \mathcal{L}_s(A^{(n)})$. Then $(T_n,H^{(n)}.A^{(n)})$ is an FCS for $H_{\bullet}A\in\mathcal{V}^\mathbb{B}$, and $H_{\bullet}A$ can be expressed as
      \begin{equation}\label{HA-expression2}
      H_{\bullet}A=\left((H_0A_0)I_{\llbracket{0}\rrbracket}+\sum\limits_{n=1}^{{+\infty}}(H^{(n)}.A^{(n)})
      I_{\rrbracket{T_{n-1},T_n}\rrbracket}\right)\mathfrak{I}_\mathbb{B},\quad T_0=0.
      \end{equation}
Furthermore, if $(S_n, \widetilde{H}^{(n)})$ and $(\widetilde{S}_n,\widetilde{A}^{(n)})$ are also FCSs for $H\in\mathcal{P}^\mathbb{B}$ and $A\in\mathcal{V}^\mathbb{B}$ respectively such that for each $n\in \mathbf{N}^+$, $\widetilde{H}^{(n)}\in \mathcal{L}_s(\widetilde{A}^{(n)})$, then $H_{\bullet}A=X$ where $X$ is given by
      \begin{equation*}
      X=\left((H_0A_0)I_{\llbracket{0}\rrbracket}+\sum\limits_{n=1}^{{+\infty}}(\widetilde{H}^{(n)}.\widetilde{A}^{(n)})
      I_{\rrbracket{\widetilde{T}_{n-1},\widetilde{T}_n}\rrbracket}\right)\mathfrak{I}_\mathbb{B},\quad \widetilde{T}_0=0,
      \end{equation*}
and $\widetilde{T}_n=S_n\wedge \widetilde{S}_n,\;n\in \mathbf{N}^+$. In this case, we say that the expression of \eqref{HA-expression2} is independent of the choice of FCSs $(T_n, H^{(n)})$ for $H\in\mathcal{P}^\mathbb{B}$ and $(T_n,A^{(n)})$ for $A\in\mathcal{V}^\mathbb{B}$.
\end{theorem}

\begin{remark}\normalfont
Within the framework of expression \eqref{HA-expression2}, the FCS $(T_n, H^{(n)})$ for $H\in\mathcal{P}^\mathbb{B}$ can be substituted by a coupled predictable process $\widetilde{H}$ for $H\in\mathcal{P}^\mathbb{B}$. This substitution is valid since $(T_n,\widetilde{H}$) also constitutes an FCS for $H\in\mathcal{P}^\mathbb{B}$. Specifically, if $\widetilde{H}$ is a coupled predictable process for $H\in\mathcal{P}^\mathbb{B}$, and if $(T_n,A^{(n)})$ is an FCS for $A\in\mathcal{V}^\mathbb{B}$ such that for each $n\in \mathbf{N}^+$, $\widetilde{H}\in \mathcal{L}_s(A^{(n)})$, then $(T_n,\widetilde{H}.A^{(n)})$ is an FCS for $H_{\bullet}A\in\mathcal{V}^\mathbb{B}$, and $H_{\bullet}A$ can be equivalently expressed as
      \begin{equation}\label{HA-expression1}
      H_{\bullet}A=\left((H_0A_0)I_{\llbracket{0}\rrbracket}+\sum\limits_{n=1}^{{+\infty}}(\widetilde{H}.A^{(n)})
      I_{\rrbracket{T_{n-1},T_n}\rrbracket}\right)\mathfrak{I}_\mathbb{B},\quad T_0=0.
      \end{equation}
\end{remark}

\begin{corollary}\label{bound-HA}
Let $H$ further be a $\mathbb{B}$-locally bounded predictable process. If $(T_n,H^{(n)})$ is an FCS for $H$ (a $\mathbb{B}$-locally bounded predictable process), and if $(T_n,A^{(n)})$ is an FCS for $A\in\mathcal{V}^\mathbb{B}$, then $(T_n,H^{(n)}.A^{(n)})$ is an FCS for $H_{\bullet}A\in\mathcal{V}^\mathbb{B}$.
\end{corollary}

\begin{corollary}\label{HAp-equivalent}
Let $H\in \mathcal{P}^\mathbb{C}$ and $A\in \mathcal{V}^\mathbb{C}$.
\begin{itemize}
  \item [$(1)$] $H\in\mathcal{L}^\mathbb{C}_s(A)$ if and only if
      there exist an FS $(\tau_n)$ for $\mathbb{C}$ such that for each $n\in \mathbf{N}^+$, $H^{\tau_n}\in \mathcal{L}_s(A^{\tau_n})$.
  \item [$(2)$] Suppose that $H\in\mathcal{L}^\mathbb{C}_s(A)$ and $(\tau_n)$ is an FS for $\mathbb{C}$. Then $(\tau_n,H^{\tau_n}.A^{\tau_n})$ is an FCS for $H_{\bullet}A\in\mathcal{V}^\mathbb{C}$, and $H_{\bullet}A$ can be expressed as
      \begin{equation}\label{HA-expression0}
      H_{\bullet}A=\left((H_0A_0)I_{\llbracket{0}\rrbracket}+\sum\limits_{n=1}^{{+\infty}}
      (H^{\tau_n}.A^{\tau_n})I_{\rrbracket{\tau_{n-1},\tau_n}
      \rrbracket}\right)\mathfrak{I}_\mathbb{C},\quad \tau_0=0.
      \end{equation}
      Furthermore, if $(\widetilde{\tau}_n)$ is also an FS for $\mathbb{C}$, then $H_{\bullet}A=X$ where $X$ is given by
      \begin{equation*}
      X=\left((H_0A_0)I_{\llbracket{0}\rrbracket}+\sum\limits_{n=1}^{{+\infty}}
      (H^{\widetilde{\tau}_n}.A^{\widetilde{\tau}_n})I_{\rrbracket{\widetilde{\tau}_{n-1},\widetilde{\tau}_n}
      \rrbracket}\right)\mathfrak{I}_\mathbb{C},\quad \widetilde{\tau}_0=0.
      \end{equation*}
      In this case, we say that the expression of \eqref{HA-expression0} is independent of the choice of FS $(\tau_n)$ for $\mathbb{C}$.
\end{itemize}
\end{corollary}

Let $H\in\mathcal{L}^\mathbb{B}_s(A)$.
According to Theorems \ref{HA-equivalent} and \ref{HA-FCS}, the $\mathbb{B}$-Stieltjes integral $H_{\bullet}A$ is fundamentally characterized by a sequence of Stieltjes integrals that pertain to FCSs for $H\in \mathcal{P}^\mathbb{B}$ and $A\in \mathcal{V}^\mathbb{B}$. Conversely, if $(T_n, H^{(n)})$ and $(T_n,A^{(n)})$ are FCSs for $H\in\mathcal{P}^\mathbb{B}$ and $A\in\mathcal{V}^\mathbb{B}$, respectively, it is not necessarily guaranteed that $H^{(n)}\in \mathcal{L}_s(A^{(n)})$ for each $n\in \mathbf{N}^+$. This is due to the fact that the processes $H^{(n)}$ and $A^{(n)}$ encompass information beyond the scope of $\mathbb{B}$, resulting in $H^{(n)}$ potentially not being integrable w.r.t. $A^{(n)}$ for some $n\in \mathbf{N}^+$. To illustrate this point, we provide a straightforward example.

\begin{example}\label{example_A}\normalfont
Let us define $\mathbb{C}=\llbracket{0,1}\rrbracket$, $H=1\mathfrak{I}_\mathbb{C}$ and $A(\omega,t)=t$ for $(\omega,t)\in \mathbb{C}$.
For each $n\in \mathbf{N}^+$ and $(\omega,t)\in \Omega\times \mathbf{R}^+$, we set $T_n=1$, $A^{(n)}(\omega,t)=t$ and
\[
H^{(n)}(\omega,t)=I_{\llbracket{0,1}\rrbracket}(\omega,t)+\frac{1}{2-t}I_{\rrbracket{1,2}\llbracket}(\omega,t)
+I_{\llbracket{2,+\infty}\llbracket}(\omega,t).
\]
It follows that $(T_n,H^{(n)})$ and $(T_n,A^{(n)})$ are FCSs for $H\in \mathcal{P}^\mathbb{C}$ and $A\in \mathcal{V}^\mathbb{C}$, respectively. However, it is noted that $H^{(n)}\notin \mathcal{L}_s(A^{(n)})$ for each $n\in \mathbf{N}^+$. Nevertheless, $H\in\mathcal{L}^\mathbb{C}_s(A)$.
To substantiate this claim, we verify the condition stipulated in Corollary \ref{HAp-equivalent}(1). For this purpose, we set $\tau_n=1$ for each $n\in \mathbf{N}^+$. The sequence $(\tau_n)$ constitutes an FS for $\mathbb{C}$ that satisfies $H^{\tau_n}\in \mathcal{L}_s(A^{\tau_n})$ for each $n\in \mathbf{N}^+$, which is precisely the requirement needed to establish our claim.
\end{example}

\section{Stochastic integrals of $\mathbb{B}$-predictable processes with respect to $\mathbb{B}$-inner local martingales}\label{section4}
\setcounter{equation}{0}
In this section, we provide an affirmative response to Problem (SI-2) by conducting a rigorous examination of stochastic integrals of $\mathbb{B}$-predictable processes with respect to $\mathbb{B}$-inner local martingales, along with an in-depth analysis of their fundamental properties.
To begin with, we construct the $\mathbb{B}$-quadratic covariation for two $\mathbb{B}$-local martingales, and the concept of $\mathbb{B}$-inner local martingales. This construction serves as a crucial foundation for our subsequent discussions.
Building upon the established $\mathbb{B}$-quadratic covariations, we then proceed to develop the theory of stochastic integrals of $\mathbb{B}$-predictable processes with respect to $\mathbb{B}$-inner local martingales.
Finally, we undertake a comprehensive study of the fundamental properties of these $\mathbb{B}$-stochastic integrals. In doing so, we elucidate the intricate relationship between $\mathbb{B}$-stochastic integrals and existing stochastic integrals, thereby providing a deeper understanding of their interplay and significance within the broader framework of stochastic analysis.

\subsection{$\mathbb{B}$-quadratic covariations and $\mathbb{B}$-inner local martingales}\label{section4.1}
To begin, we introduce the predictable quadratic covariations of $\mathbb{B}$-continuous local martingales, a key component in defining $\mathbb{B}$-quadratic covariations. To lay the groundwork for introducing this definition, we now present the following lemma.

\begin{lemma}\label{ex-quad-p}
Let $M,N\in(\mathcal{M}^c_{\mathrm{loc}})^\mathbb{B}$. Then there exists a unique process $V\in (\mathcal{A}_{\mathrm{loc}}\cap \mathcal{C})^\mathbb{B}$ such that $MN-V\in (\mathcal{M}^c_{\mathrm{loc},0})^\mathbb{B}$.
\end{lemma}

\begin{definition}\normalfont\label{de-quad-p}
Let $M,N\in(\mathcal{M}^c_{\mathrm{loc}})^\mathbb{B}$. The unique process $V\in (\mathcal{A}_{\mathrm{loc}}\cap \mathcal{C})^\mathbb{B}$ in Lemma \ref{ex-quad-p}, denoted by $\langle M,N\rangle$, is called the $\mathbb{B}$-{\bf predictable quadratic covariation} of $M$ and $N$. Furthermore, if $M=N$, then $\langle M,M\rangle$ (or simply, $\langle M\rangle$) is called the $\mathbb{B}$-{\bf predictable quadratic variation} of $M$.
\end{definition}

The fundamental properties of $\mathbb{B}$-predictable quadratic covariations are expounded upon in the subsequent theorem.
\begin{theorem}\label{property-qr-p}
Let $M,\; N,\;\widetilde{M}\in (\mathcal{M}^c_{\mathrm{loc}})^\mathbb{B}$.
\begin{itemize}
  \item [$(1)$] If $(T_n,M^{(n)})$ and $(T_n,N^{(n)})$ are FCSs for $M\in(\mathcal{M}^c_{\mathrm{loc}})^\mathbb{B}$ and $N\in(\mathcal{M}^c_{\mathrm{loc}})^\mathbb{B}$ respectively, then $(T_n,\langle M^{(n)},N^{(n)}\rangle)$ is an FCS for $\langle M,N\rangle\in (\mathcal{A}_{\mathrm{loc}}\cap \mathcal{C})^\mathbb{B}$, and $(T_n,\langle M^{(n)}\rangle)$ is an FCS for $\langle M\rangle\in (\mathcal{A}^{+}_{\mathrm{loc}}\cap \mathcal{C})^\mathbb{B}$.
  \item [$(2)$] For $a,b\in \mathbf{R}$, we have
      \[
       \langle M,N\rangle=\langle N,M\rangle,\quad \langle aM+b\widetilde{M},N\rangle=a\langle M,N\rangle+b\langle\widetilde{M},N\rangle.
      \]
  \item [$(3)$] If $\tau$ is a $\mathbb{B}$-inner stopping time, then we have
      \begin{equation}\label{Mc-tau1}
       \langle M^{\tau},N^{\tau}\rangle
       =\langle M,N\rangle^{\tau}
      \end{equation}
  and
      \begin{equation}\label{Mc-tau2}
       \langle M^{\tau}\mathfrak{I}_\mathbb{B},N^{\tau}\mathfrak{I}_\mathbb{B}\rangle
      =\langle M^{\tau},N^{\tau}\rangle\mathfrak{I}_\mathbb{B}
      =\langle M,N\rangle^{\tau}\mathfrak{I}_\mathbb{B}=\langle M^{\tau}\mathfrak{I}_\mathbb{B},N\rangle.
      \end{equation}
\end{itemize}
\end{theorem}

To facilitate the expansion of the concept of $\mathbb{B}$-predictable quadratic covariations, we will now delve into the decomposition of a $\mathbb{B}$-local martingale.
\begin{theorem}\label{Lem-M}
Let $M\in (\mathcal{M}_{\mathrm{loc}})^\mathbb{B}$.
\begin{itemize}
  \item [$(1)$] $M$ admits a unique decomposition
\begin{equation}\label{con-M}
M=M_0\mathfrak{I}_\mathbb{B}+M^c+M^d,
\end{equation}
where $M^c\in (\mathcal{M}^c_{\mathrm{loc},0})^\mathbb{B}$ and $M^d\in (\mathcal{M}^d_{\mathrm{loc}})^\mathbb{B}$. $M^c$ is called the {\bf continuous martingale part} of $M$, and $M^d$ is called the {\bf purely discontinuous martingale part} of $M$.
  \item [$(2)$] If $(T_n,M^{(n)})$ is an FCS for $M\in (\mathcal{M}_{\mathrm{loc}})^\mathbb{B}$, then $(T_n,(M^{(n)})^c)$ and $(T_n,(M^{(n)})^d)$ are FCSs for $M^c\in(\mathcal{M}^c_{\mathrm{loc},0})^\mathbb{B}$ and $M^d\in(\mathcal{M}^d_{\mathrm{loc}})^\mathbb{B}$, respectively.
  \item [$(3)$] If $\tau$ is a $\mathbb{B}$-inner stopping time, then we have
  \begin{equation}\label{McS}
  (M^c)^\tau=(M^\tau)^c,\quad (M^d)^\tau=(M^\tau)^d
  \end{equation}
  and
  \begin{equation}\label{McSB}
  (M^c)^\tau\mathfrak{I}_\mathbb{B}=(M^\tau)^c\mathfrak{I}_\mathbb{B}=(M^\tau\mathfrak{I}_\mathbb{B})^c,\quad (M^d)^\tau\mathfrak{I}_\mathbb{B}=(M^\tau)^d\mathfrak{I}_\mathbb{B}=(M^\tau\mathfrak{I}_\mathbb{B})^d.
  \end{equation}
\end{itemize}

\end{theorem}

By leveraging the concept of $\mathbb{B}$-predictable quadratic covariations, alongside the continuous martingale parts of $\mathbb{B}$-local martingales, we are able to formulate the definition of $\mathbb{B}$-quadratic covariations for $\mathbb{B}$-local martingales.

\begin{definition}\normalfont\label{QC}
Let $M,N\in(\mathcal{M}_{\mathrm{loc}})^\mathbb{B}$, and $M^c$ and $N^c$ be their continuous martingale parts. Define
\begin{equation}\label{[M,N]}
[M,N]=M_0N_0\mathfrak{I}_\mathbb{B}+\langle M^c,N^c\rangle+\Sigma (\Delta M\Delta N).
\end{equation}
Then $[M,N]$ is called the $\mathbb{B}$-{\bf quadratic covariation} of $M$ and $N$. In particular, $[M,M]$ (or simply $[M]$) is called the $\mathbb{B}$-{\bf quadratic variation} of $M$.
\end{definition}

Given that $M,N\in(\mathcal{M}_{\mathrm{loc}})^\mathbb{B}$, Definition \ref{de-quad-p} and Theorem \ref{delta}(6) guarantee that the quadratic covariation $[M,N]$ as defined in \eqref{[M,N]} is well-defined. The formulation of \eqref{[M,N]} is also consistent with the conventional form of the quadratic covariation of local martingales, as delineated in Definition 7.29 of \cite{He}.
However, in order to address Problem (SI-2), we encounter two major challenges. As will be demonstrated later in Section 4.2, these challenges may result in the non-uniqueness of the construction of $\mathbb{B}$-stochastic integration.
Firstly, the construction of the $\mathbb{B}$-quadratic covariation may fail to satisfy two fundamental properties that are pivotal in the context of the existing quadratic covariation. For instance, let $\mathbb{B}^*$ and $\widehat{M}\in(\mathcal{M}_{\mathrm{loc}})^{\mathbb{B}^*}$ be as defined in Section 1.1. Then:
\begin{itemize}
  \item [$(1)$] $[\widehat{M}]$ is not the unique $\mathbb{B}^*$-process $V\in \mathcal{V}^{\mathbb{B}^*}$ such that $\widehat{M}^2-V\in (\mathcal{M}_{\mathrm{loc,0}})^\mathbb{B}$ and $\Delta V=(\Delta \widehat{M})^2$.
  \item [$(2)$] The condition $[\widehat{M}]=0\mathfrak{I}_{\mathbb{B}^*}$ does not imply $\widehat{M}=0\mathfrak{I}_{\mathbb{B}^*}$.
\end{itemize}
Secondly, it is a well-known fact that $\mathcal{M}_{\mathrm{loc}}\cap\mathcal{C}=\mathcal{M}^c_{\mathrm{loc}}$. Nevertheless, the analogous statement for $\mathbb{B}$-local martingales, i.e., $(\mathcal{M}_{\mathrm{loc}})^\mathbb{B}\cap\mathcal{C}^\mathbb{B}=(\mathcal{M}^c_{\mathrm{loc}})^\mathbb{B}$, does not hold in general. This can be illustrated once again using $\mathbb{B}^*$ and $\widehat{M}\in(\mathcal{M}_{\mathrm{loc}})^\mathbb{B}$ from Section 1.1.
On the one hand, it is evident that $\widehat{M}\in(\mathcal{M}_{\mathrm{loc}})^{\mathbb{B}^*}$ with the FCS $(T_n=T,M^{(n)}=A^p-A)$, and $\widehat{M}\in\mathcal{C}^\mathbb{B}$ with the FCS $(T_n=T,N^{(n)}=A^p)$. On the other hand, $\widehat{M}$ is not a $\mathbb{B}^*$-continuous local martingale. This is because every $\mathbb{F}$-continuous local martingale $N$ must satisfy $N=a$ with a constant $a\in \mathbf{R}$ (as stated in Proposition 2.7 of \cite{Aksamit}). Consequently, the relation $\widehat{M}\in(\mathcal{M}_{\mathrm{loc}})^{\mathbb{B}^*}\cap\mathcal{C}^{\mathbb{B}^*}$ does not imply $M\in(\mathcal{M}^c_{\mathrm{loc}})^{\mathbb{B}^*}$. Therefore, we proceed to define a subclass of $\mathbb{B}$-local martingales to overcome these two issues.

\begin{definition}\normalfont\label{inner}
Let $\mathbb{B}$ be given by \eqref{B}, and $M\in(\mathcal{M}_{\mathrm{loc}})^\mathbb{B}$. We say $M$ is an {\bf essentially inner local martingale} on $\mathbb{B}$ (simply, a $\mathbb{B}$-{\bf inner local martingale}), if there exists an FCS $(T_n,M^{(n)})$ for $M\in(\mathcal{M}_{\mathrm{loc}})^\mathbb{B}$ such that
\begin{equation}\label{inner-eq}
(M^{(n)})^{T_n\wedge (T_F-)}\in\mathcal{M}_{\mathrm{loc}},\quad n\in \mathbf{N}^{+}.
\end{equation}
The collection of all $\mathbb{B}$-inner local martingales is denoted by $(\mathcal{M}_{\mathrm{loc}})^{i,\mathbb{B}}$, and we also say that $(T_n,M^{(n)})$ is an {\bf inner FCS} for $M\in(\mathcal{M}_{\mathrm{loc}})^{i,\mathbb{B}}$.
\end{definition}

In the context of Definition \ref{inner}, $(T_n,(M^{(n)})^{T_n\wedge (T_F-)})$ also constitutes an inner FCS for $M\in(\mathcal{M}_{\mathrm{loc}})^{i,\mathbb{B}}$. Hence, the definition of a $\mathbb{B}$-inner local martingale implies that $M$ admits an FCS which is determined exclusively by values on $\mathbb{B}$.
There exist three notable classes of $\mathbb{B}$-inner local martingales:
\begin{itemize}
  \item [$(1)$] $(\mathcal{M}_{\mathrm{loc}})^\mathbb{C}=(\mathcal{M}_{\mathrm{loc}})^{i,\mathbb{C}}$. If $M\in(\mathcal{M}_{\mathrm{loc}})^\mathbb{C}$ and $(\tau_n)$ is an FS for $\mathbb{C}$, then $(\tau_n,M^{\tau_n})$ is an inner FCS for $M\in(\mathcal{M}_{\mathrm{loc}})^{i,\mathbb{C}}$.
  \item [$(2)$] $(\mathcal{M}^c_{\mathrm{loc}})^\mathbb{B}\subseteq(\mathcal{M}_{\mathrm{loc}})^{i,\mathbb{B}}$. If $M\in(\mathcal{M}^c_{\mathrm{loc}})^\mathbb{B}$, then each FCS $(T_n,M^{(n)})$ for $M\in(\mathcal{M}^c_{\mathrm{loc}})^\mathbb{B}$ is an inner FCS for $M\in(\mathcal{M}_{\mathrm{loc}})^{i,\mathbb{B}}$.
  \item [$(3)$] Let $\mathbb{B}$ be given by \eqref{B}, and $\widetilde{M}\in\mathcal{M}_{\mathrm{loc},0}$ satisfying the condition $\widetilde{M}^{T_F-}\in\mathcal{M}_{\mathrm{loc},0}$. Suppose that $\widetilde{M}$ has the strong property of predictable representation, i.e., it holds that $\mathcal{M}_{\mathrm{loc},0}=\{\widetilde{H}.\widetilde{M}: \widetilde{H}\in \mathcal{L}_m(\widetilde{M})\}$ (see, e.g., Definition 13.1 in \cite{He}). Then it follows that $(\mathcal{M}_{\mathrm{loc}})^\mathbb{B}=(\mathcal{M}_{\mathrm{loc}})^{i,\mathbb{B}}$, and
      each FCS $(T_n,M^{(n)})$ for $M\in(\mathcal{M}_{\mathrm{loc}})^\mathbb{B}$ constitutes an inner FCS for $M\in(\mathcal{M}_{\mathrm{loc}})^{i,\mathbb{B}}$. Indeed, for each $n\in \mathbf{N}^+$, there exists $H^{(n)}\in \mathcal{L}_m(\widetilde{M})$ such that $M^{(n)}=M^{(n)}_0+H^{(n)}.\widetilde{M}$. Consequently, the relation $(M^{(n)})^{T_n\wedge (T_F-)}=M^{(n)}_0+H^{(n)}.\widetilde{M}^{T_n\wedge (T_F-)}\in\mathcal{M}_{\mathrm{loc}}$ implies that $(T_n,M^{(n)})$ is indeed an inner FCS for $M\in(\mathcal{M}_{\mathrm{loc}})^{i,\mathbb{B}}$.
\end{itemize}
The classes delineated in (1) and (2) are straightforward to comprehend. Furthermore, we present a specific instance of class (3).

\begin{example}\normalfont
Let $T$ be a discrete random variable satisfying $\mathbf{P}(T=1)=\mathbf{P}(T=2)=\frac{1}{2}$, and let $\mathbb{F}=(\mathcal{F}_t,t\geq 0)$ be the natural filtration of the process $A=I_{\llbracket{T,{+\infty}}\llbracket}$. Then $T$ is an $\mathbb{F}$-stopping time with $T>0$. Define $\tau=2I_{[T=2]}+(+\infty)I_{[T=1]}$, $\mathbb{B}=\llbracket{0,\tau}\llbracket$, and $\widetilde{M}=A^p-A$, where $A^p=f^T$ is the compensator of $A$ (see Proposition 2.4 in \cite{Aksamit}) with
\begin{equation}\label{ff}
f(t)=\left\{
\begin{aligned}
0,\quad &0\leq t<1,\\
\frac{1}{2},\quad &1\leq t<2,\\
\frac{3}{2},\quad &t\geq 2.
\end{aligned}
\right.
\end{equation}
It follows that $\tau>0$ is a stopping time, and $\mathbb{B}$ is an optional set of interval type. By applying \eqref{ff}, it becomes evident that
\[
\Delta \widetilde{M}_{\tau}I_{[\tau<+\infty]}=\Delta \widetilde{M}_2I_{[T=2]}=(\Delta f(2)-\Delta A_T)I_{[T=2]}=0,
\]
which yields $\widetilde{M}^{\tau-}=\widetilde{M}^{\tau}\in\mathcal{M}_{\mathrm{loc},0}$. According to Proposition 2.7 in \cite{Aksamit}, $\widetilde{M}$ has the strong property of predictable representation. Consequently, it holds that $(\mathcal{M}_{\mathrm{loc}})^\mathbb{B}=(\mathcal{M}_{\mathrm{loc}})^{i,\mathbb{B}}$.
\end{example}

In the subsequent two theorems, we present the fundamental properties of $\mathbb{B}$-inner local martingales and $\mathbb{B}$-quadratic covariations.

\begin{theorem}\label{MC}
Let $M,N\in(\mathcal{M}_{\mathrm{loc}})^{i,\mathbb{B}}$.
\begin{itemize}
  \item [$(1)$]
       If $(S_n,N^{(n)})$ is an FCS for $M\in(\mathcal{M}_{\mathrm{loc}})^\mathbb{B}$, and if $(T_n,M^{(n)})$ is an inner FCS for $M\in(\mathcal{M}_{\mathrm{loc}})^{i,\mathbb{B}}$, then $(\tau_n=S_n\wedge T_n,N^{(n)})$ is also an inner FCS for $M\in(\mathcal{M}_{\mathrm{loc}})^{i,\mathbb{B}}$.
  \item [$(2)$]  $(\mathcal{M}_{\mathrm{loc}})^{i,\mathbb{B}}\cap\mathcal{C}^\mathbb{B}=(\mathcal{M}^c_{\mathrm{loc}})^\mathbb{B}$.
  \item [$(3)$] $aM+bN\in(\mathcal{M}_{\mathrm{loc}})^{i,\mathbb{B}}$ for all $a,b\in \mathbf{R}$. If $(T_n,M^{(n)})$ and $(T_n,N^{(n)})$ are inner FCSs for $M\in (\mathcal{M}_{\mathrm{loc}})^{i,\mathbb{B}}$ and $N\in (\mathcal{M}_{\mathrm{loc}})^{i,\mathbb{B}}$ respectively, then $(T_n,aM^{(n)}+bN^{(n)})$ is an inner FCS for $aM+bN\in (\mathcal{M}_{\mathrm{loc}})^{i,\mathbb{B}}$.
  \item [$(4)$] $M^d\in(\mathcal{M}^{d}_{\mathrm{loc}})^\mathbb{B}\cap(\mathcal{M}_{\mathrm{loc}})^{i,\mathbb{B}}$. If $(T_n,M^{(n)})$ is an inner FCS for $M\in (\mathcal{M}_{\mathrm{loc}})^{i,\mathbb{B}}$, then $(T_n,(M^{(n)})^d)$ is an inner FCS for $M^d\in(\mathcal{M}^{d}_{\mathrm{loc}})^\mathbb{B}\cap(\mathcal{M}_{\mathrm{loc}})^{i,\mathbb{B}}$. We denote by $(\mathcal{M}^d_{\mathrm{loc}})^{i,\mathbb{B}}$ the collection of all $\mathbb{B}$-purely discontinuous local martingales with inner FCSs.
  \item [$(5)$] If $(T_n,M^{(n)})$ is an inner FCS for $M\in(\mathcal{M}_{\mathrm{loc}})^{i,\mathbb{B}}$, and if $(T_n,M^{(n)},\widetilde{M})$ is the continuation for $M\in(\mathcal{M}_{\mathrm{loc}})^{i,\mathbb{B}}$, then $(T_n,\widetilde{M}^{T_n})$ is an inner FCS for $M\in(\mathcal{M}_{\mathrm{loc}})^{i,\mathbb{B}}$. And $(T_n,M^{(n)},\widetilde{M})$ is called an {\bf inner continuation} for $M\in(\mathcal{M}_{\mathrm{loc}})^{i,\mathbb{B}}$.
\end{itemize}
\end{theorem}

\begin{theorem}\label{[M]-property}
Let $M,N\in(\mathcal{M}_{\mathrm{loc}})^\mathbb{B}$.
\begin{itemize}
  \item [$(1)$] If $(T_n,M^{(n)})$ and $(T_n,N^{(n)})$ are FCSs for $M\in(\mathcal{M}_{\mathrm{loc}})^\mathbb{B}$ and $N\in(\mathcal{M}_{\mathrm{loc}})^\mathbb{B}$ respectively, then $(T_n,[M^{(n)},N^{(n)}])$ is an FCS for $[M,N]\in \mathcal{V}^\mathbb{B}$, and $(T_n,[M^{(n)}])$ is an FCS for $[M]\in (\mathcal{V}^+)^\mathbb{B}$, and $(T_n,\sqrt{[M^{(n)}]})$ is an FCS for $\sqrt{[M]}\in (\mathcal{A}^+_{\mathrm{loc}})^\mathbb{B}$.

  \item [$(2)$] For $\widetilde{M}\in(\mathcal{M}_{\mathrm{loc}})^\mathbb{B}$ and $a,b\in \mathbf{R}$,
      \[
       [M,N]=[N,M],\quad [aM+b\widetilde{M},N]=a[M,N]+b[\widetilde{M},N].
      \]

  \item [$(3)$]  If $\tau$ is a $\mathbb{B}$-inner stopping time, then
      \begin{equation}\label{MN}
       [M^{\tau},N^{\tau}]=[M,N]^{\tau}
      \end{equation}
       and
      \begin{equation}\label{MNB}
      [M^{\tau}\mathfrak{I}_\mathbb{B},N^{\tau}\mathfrak{I}_\mathbb{B}]
      =[M^{\tau},N^{\tau}]\mathfrak{I}_\mathbb{B}
      =[M,N]^{\tau}\mathfrak{I}_\mathbb{B}=[M^{\tau}\mathfrak{I}_\mathbb{B},N].
      \end{equation}

  \item [$(4)$] Further suppose $M,N\in(\mathcal{M}_{\mathrm{loc}})^{i,\mathbb{B}}$. Then $[M,N]$ is the unique $\mathbb{B}$-adapted process with finite variation such that $MN-[M,N]\in (\mathcal{M}_{\mathrm{loc,0}})^{i,\mathbb{B}}$ and $\Delta [M,N]=\Delta M\Delta N$.

  \item [$(5)$] Further suppose $M\in(\mathcal{M}_{\mathrm{loc}})^{i,\mathbb{B}}$. Then $[M]=0\mathfrak{I}_\mathbb{B}$ if and only of $M=0\mathfrak{I}_\mathbb{B}$.
\end{itemize}
\end{theorem}

\begin{remark}\normalfont
\begin{itemize}
  \item [(1)] A $\mathbb{B}$-local martingale is not necessarily a $\mathbb{B}$-inner local martingale. For instance, consider $\mathbb{B}^*$ and $\widehat{M}\in(\mathcal{M}_{\mathrm{loc}})^{\mathbb{B}^*}$ in Section 1.1. Then $\widehat{M}\in (\mathcal{M}_{\mathrm{loc}})^{\mathbb{B}^*}$ is indeed a $\mathbb{B}^*$-local martingale, yet it does not qualify as a $\mathbb{B}^*$-inner local martingale. Specifically, it is straightforward to observe that $[\widehat{M}]=0\mathfrak{I}_{\mathbb{B}^*}$ while $\widehat{M}\neq 0\mathfrak{I}_{\mathbb{B}^*}$. This observation, in conjunction with Theorem \ref{[M]-property}(5), leads to the conclusion that $\widehat{M}\notin (\mathcal{M}_{\mathrm{loc}})^{i,\mathbb{B}^*}$.
  \item [(2)] Drawing upon Theorem \ref{MC}(2), Theorem \ref{[M]-property}(4) and Theorem \ref{[M]-property}(5), it can be established that $\mathbb{B}$-inner local martingales indeed fulfill the aforementioned fundamental properties, which are widely recognized as characteristic of local martingales.
  \end{itemize}

\end{remark}

\subsection{Definition of $H_{\bullet}M$}
Unless otherwise explicitly specified, we will uniformly assume, for the entirety of the remaining content within this section, that $H\in \mathcal{P}^\mathbb{B}$ and $M\in (\mathcal{M}_{\mathrm{loc}})^{i,\mathbb{B}}$.

Drawing upon the quadratic covariations of $\mathbb{B}$-local martingales, we extend the definition of stochastic integrals of predictable processes with respect to local martingales. This extension enables us to rigorously present the following affirmative resolution of Problem (SI-2).

\begin{definition}\normalfont\label{HM}
If there exists a $\mathbb{B}$-process $L\in(\mathcal{M}_{\mathrm{loc}})^{i,\mathbb{B}}$ such that
\begin{equation}\label{deHM}
[L,N]=H_{\bullet}[M,N]
\end{equation}
holds for every process $N\in(\mathcal{M}_{\mathrm{loc}})^\mathbb{B}$ (this naturally implies $H\in\mathcal{L}_s^\mathbb{B}([M,N])$), then we say that $H$ is $\mathbb{B}$-integrable w.r.t. $M$. In this case, the $\mathbb{B}$-process $L$, denoted by $H_{\bullet}M$, is called the stochastic integral of $H$ w.r.t. $M$, and the collection of all $\mathbb{B}$-predictable processes which are $\mathbb{B}$-integrable w.r.t. $M$ is denoted by $\mathcal{L}_m^\mathbb{B}(M)$.
\end{definition}

When considering the stochastic integral $H_{\bullet}M$ as defined in Definition \ref{HM} and $H_{\bullet}A$ as per Definition \ref{HA}, we consistently specify the classes to which $M$ and $A$ belong, thereby eliminating any potential ambiguity between the two notations $H_{\bullet}M$ and $H_{\bullet}A$. Furthermore, it is noteworthy that if the $\mathbb{B}$-process $L$ in (\ref{deHM}) exists, it is necessarily unique. To elaborate, suppose $\widetilde{L}\in(\mathcal{M}_{\mathrm{loc}})^{i,\mathbb{B}}$ is another $\mathbb{B}$-process satisfying the condition $[\widetilde{L},N]=H_{\bullet}[M,N]$ for every $N\in(\mathcal{M}_{\mathrm{loc}})^\mathbb{B}$. By setting $N=L-\widetilde{L}\in(\mathcal{M}_{\mathrm{loc}})^{i,\mathbb{B}}$, we obtain the relation $[L-\widetilde{L}]=0\mathfrak{I}_\mathbb{B}$. Invoking Theorem \ref{[M]-property}(5), we can infer that $L=\widetilde{L}$, thereby establishing the uniqueness of $L$.

\begin{remark}\normalfont
Let us delve deeper into the rationality of Definition \ref{HM}.
\begin{itemize}
\item [$(1)$] It is of considerable importance to observe that when $\mathbb{B}=\llbracket{0,+\infty}\llbracket=\Omega\times\mathbf{R}^+$, the $\mathbb{B}$-stochastic integral $H_{\bullet}M$, as defined by \eqref{deHM}, simplifies to the existing stochastic integral $H.M$. Specifically, the following relationship can be established:
    \begin{itemize}
      \item [] Let $H\in \mathcal{P}^{\llbracket{0,+\infty}\llbracket}$ and $M\in (\mathcal{M}_{\mathrm{loc}})^{\llbracket{0,+\infty}\llbracket}$. If $H\in\mathcal{L}_m^{\llbracket{0,+\infty}\llbracket}(M)$, then $H\in\mathcal{L}_m(M)$ and $H_{\bullet}M=H.M$.
    \end{itemize}
    Indeed, by invoking Corollary \ref{cD=DB}, we can deduce $\mathcal{P}=\mathcal{P}^{\llbracket{0,+\infty}\llbracket}$ and $\mathcal{M}_{\mathrm{loc}}=(\mathcal{M}_{\mathrm{loc}})^{\llbracket{0,+\infty}\llbracket}$. Subsequently, by leveraging Definition \ref{HM}, the aforementioned assertion can be straightforwardly derived.

\item [$(2)$]
    Generally, it is imperative to impose the conditions pertaining to $\mathbb{B}$-inner local martingales as stipulated in Definition \ref{HM}. Specifically, if the class $(\mathcal{M}_{\mathrm{loc}})^{i,\mathbb{B}}$ is altered to $(\mathcal{M}_{\mathrm{loc}})^\mathbb{B}$, the uniqueness of the stochastic integral $H_{\bullet}M$ may no longer be guaranteed. To illustrate this point, we provide a straightforward example. Let $\mathbb{B}^*$ and $\widehat{M}\in (\mathcal{M}_{\mathrm{loc}})^{\mathbb{B}^*}$ be as specified in Section 1.1.  It follows that $\widehat{M}\notin (\mathcal{M}_{\mathrm{loc}})^{i,{\mathbb{B}^*}}$. Now, consider the following assignments: $H=1\mathfrak{I}_{\mathbb{B}^*}$, $L_1=\widehat{M}\in (\mathcal{M}_{\mathrm{loc}})^{\mathbb{B}^*}$ and $L_2=0\mathfrak{I}_{\mathbb{B}^*}\in (\mathcal{M}_{\mathrm{loc}})^{\mathbb{B}^*}$. It can be easily verified that for every process $N\in(\mathcal{M}_{\mathrm{loc}})^{\mathbb{B}^*}$, the following relationships hold:
    \begin{align*}
    [L_1,N]&=H_{\bullet}[\widehat{M},N]=0\mathfrak{I}_{\mathbb{B}^*},\\
    [L_2,N]&=H_{\bullet}[\widehat{M},N]=0\mathfrak{I}_{\mathbb{B}^*}.
    \end{align*}
    Given that $L_1\neq L_2$, it becomes evident that the ${\mathbb{B}^*}$-stochastic integral $H_{\bullet}\widehat{M}$ can not be defined uniquely.
 \end{itemize}
\end{remark}

\subsection{Fundamental properties of $H_{\bullet}M$}
By invoking Definition \ref{HM}, in conjunction with Theorems \ref{MC} and \ref{[M]-property}, we can promptly derive the following properties associated with the $\mathbb{B}$-stochastic integral $H_{\bullet}M$. These properties align seamlessly with the well-recognized characteristics that are intrinsic to conventional stochastic integrals of predictable processes w.r.t. local martingales.

\begin{theorem}\label{HM-o-p}
Given $\widetilde{M}\in (\mathcal{M}_{\mathrm{loc}})^{i,\mathbb{B}}$ and $a,b\in\mathbf{R}$, let $H\in \mathcal{L}_m^\mathbb{B}(M)\cap\mathcal{L}_m^\mathbb{B}(\widetilde{M})$ and $K\in \mathcal{L}_m^\mathbb{B}(M)$.
\begin{itemize}
  \item [$(1)$] $aH+bK\in\mathcal{L}_m^\mathbb{B}(M)$, and in this case
  \begin{equation}\label{HM+}
      (aH+bK)_{\bullet}M=a(H_{\bullet}M)+b(K_{\bullet}M).
  \end{equation}

  \item [$(2)$] $H\in\mathcal{L}_m^\mathbb{B}(aM+b\widetilde{M})$, and in this case
  \begin{equation}\label{+HM}
      H_{\bullet}(aM+b\widetilde{M})=a(H_{\bullet}M)+b(H_{\bullet}\widetilde{M}).
  \end{equation}

  \item [$(3)$] Let $\widetilde{H}\in\mathcal{P}^\mathbb{B}$. Then  $\widetilde{H}\in\mathcal{L}_m^\mathbb{B}(H_{\bullet}M)$ if and only if
      $\widetilde{H}H\in\mathcal{L}_m^\mathbb{B}(M)$. Furthermore, if $\widetilde{H}\in\mathcal{L}_m^\mathbb{B}(H_{\bullet}M)$ (or equivalently, $\widetilde{H}H\in\mathcal{L}_m^\mathbb{B}(M)$), then
  \begin{equation}\label{hHM}
     \widetilde{H}_{\bullet}(H_{\bullet}M)=(\widetilde{H}H)_{\bullet}M.
  \end{equation}
\end{itemize}
\end{theorem}

The subsequent theorem serves as a valuable tool for our exploration of the relationship between the integrability conditions associated with $H_{\bullet}M$ and those of conventional stochastic integrals. This relationship holds significant importance in shedding light on additional properties of $H_{\bullet}M$, as defined in Definition \ref{HM}.

\begin{theorem}\label{HM=}
The following statements are equivalent:
\begin{description}
  \item[(M1)] $H\in\mathcal{L}_m^\mathbb{B}(M)$.
  \item[(M2)] There exist a coupled predictable process $\widetilde{H}$ for $H\in \mathcal{P}^\mathbb{B}$ and an inner FCS $(T_n,M^{(n)})$ for $M\in(\mathcal{M}_{\mathrm{loc}})^{i,\mathbb{B}}$ such that $\widetilde{H}\in\mathcal{L}_m(M^{(n)})$ for each $n\in \mathbf{N}^+$.
  \item[(M3)] There exist an FCS $(T_n,H^{(n)})$ for $H\in \mathcal{P}^\mathbb{B}$ and an inner FCS $(T_n,M^{(n)})$ for $M\in(\mathcal{M}_{\mathrm{loc}})^{i,\mathbb{B}}$ such that $H^{(n)}\in\mathcal{L}_m(M^{(n)})$ for each $n\in \mathbf{N}^+$.
  \item[(M4)] $\sqrt{{H^2}_{\bullet}[M]}\in (\mathcal{A}^+_{\mathrm{loc}})^\mathbb{B}$.
\end{description}
\end{theorem}

\begin{remark}\normalfont
By invoking Theorem \ref{process}(3), the assertion designated as (M3) in Theorem \ref{HM=} can be reformulated into an equivalent condition as follows:
\begin{description}
  \item[(M3$'$)] There exist an FCS $(T_n,H^{(n)})$ for $H\in \mathcal{P}^\mathbb{B}$ and an inner FCS $(S_n,M^{(n)})$ for $M\in(\mathcal{M}_{\mathrm{loc}})^{i,\mathbb{B}}$, such that $H^{(n)}\in\mathcal{L}_m(M^{(n)})$ for each $n\in \mathbf{N}^+$.
\end{description}
\end{remark}

The $\mathbb{B}$-stochastic integral $H_{\bullet}M$, as delineated in Definition \ref{HM}, can be precisely described as the aggregation of a sequence of conventional stochastic integrals of predictable processes w.r.t. local martingales. This characterization is formally established in the ensuing theorem.

\begin{theorem}\label{eq-HM}
Let $H\in \mathcal{L}_m^\mathbb{B}(M)$.
Suppose that $(T_n,H^{(n)})$ is an FCS for $H\in\mathcal{P}^\mathbb{B}$ and $(T_n,M^{(n)})$ is an inner FCS for $M\in(\mathcal{M}_{\mathrm{loc}})^{i,\mathbb{B}}$ such that for each $n\in \mathbf{N}^+$, $H^{(n)}\in\mathcal{L}_m(M^{(n)})$. Then $(T_n,H^{(n)}.M^{(n)})$ is an inner FCS for $H_{\bullet}M\in(\mathcal{M}_{\mathrm{loc}})^{i,\mathbb{B}}$, and $H_{\bullet}M$ can be expressed as
  \begin{equation}\label{HM-expression-2}
      H_{\bullet}M=\left((H_0M_0)I_{\llbracket{0}\rrbracket}+\sum\limits_{n=1}^{{+\infty}}
      (H^{(n)}.M^{(n)})I_{\rrbracket{T_{n-1},T_n}
      \rrbracket}\right)\mathfrak{I}_\mathbb{B},\quad T_0=0.
      \end{equation}
Furthermore, if $(S_n, \widetilde{H}^{(n)})$ is another FCS for $H\in\mathcal{P}^\mathbb{B}$ and $(\widetilde{S}_n,\widetilde{M}^{(n)})$ is another inner FCS for $M\in(\mathcal{M}_{\mathrm{loc}})^{i,\mathbb{B}}$ such that for each $n\in \mathbf{N}^+$, $\widetilde{H}^{(n)}\in\mathcal{L}_m(\widetilde{M}^{(n)})$, then $H_{\bullet}M=\widetilde{X}$ where $\widetilde{X}$ is given by
      \begin{equation*}
      \widetilde{X}=\left((H_0M_0)I_{\llbracket{0}\rrbracket}
      +\sum\limits_{n=1}^{{+\infty}}(\widetilde{H}^{(n)}.\widetilde{M}^{(n)})
      I_{\rrbracket{\widetilde{T}_{n-1},\widetilde{T}_n}\rrbracket}\right)\mathfrak{I}_\mathbb{B},\quad \widetilde{T}_0=0,
      \end{equation*}
and $\widetilde{T}_n=S_n\wedge \widetilde{S}_n,\;n\in \mathbf{N}^+$. In this case, we say that the expression of \eqref{HM-expression-2} is independent of the choice of the FCS $(T_n, H^{(n)})$ for $H\in\mathcal{P}^\mathbb{B}$ and the inner FCS $(T_n,M^{(n)})$ for $M\in(\mathcal{M}_{\mathrm{loc}})^{i,\mathbb{B}}$.
\end{theorem}

\begin{remark}\normalfont
Within the framework of expression \eqref{HM-expression-2}, it is feasible to substitute the FCS $(T_n, H^{(n)})$ for $H\in\mathcal{P}^\mathbb{B}$ with a coupled predictable process $\widetilde{H}$ for $H\in\mathcal{P}^\mathbb{B}$ (given that $(T_n,\widetilde{H}$) also forms an FCS for $H\in\mathcal{P}^\mathbb{B}$). Specifically, if $\widetilde{H}$ is a coupled predictable process for $H\in\mathcal{P}^\mathbb{B}$ and $(T_n,M^{(n)})$ is an inner FCS for $M\in(\mathcal{M}_{\mathrm{loc}})^{i,\mathbb{B}}$ such that $\widetilde{H}\in\mathcal{L}_m(M^{(n)})$ for each $n\in \mathbf{N}^+$, then $(T_n,\widetilde{H}.M^{(n)})$ is an inner FCS for $H_{\bullet}M\in(\mathcal{M}_{\mathrm{loc}})^{i,\mathbb{B}}$, and $H_{\bullet}M$ can be formulated as follows:
  \begin{equation}\label{HM-expression-1}
      H_{\bullet}M=\left((H_0M_0)I_{\llbracket{0}\rrbracket}+\sum\limits_{n=1}^{{+\infty}}
      (\widetilde{H}.M^{(n)})I_{\rrbracket{T_{n-1},T_n}
      \rrbracket}\right)\mathfrak{I}_\mathbb{B},\quad T_0=0.
      \end{equation}
\end{remark}

\begin{corollary}\label{bound-HM}
Assume that $H$ is a $\mathbb{B}$-locally bounded predictable process. Then
$H\in \mathcal{L}_m^\mathbb{B}(M)$. Furthermore, given that $(T_n,M^{(n)})$ is an inner FCS for $M\in(\mathcal{M}_{\mathrm{loc}})^{i,\mathbb{B}}$, if $\widetilde{H}$ is a coupled locally bounded predictable process for $H$, and if $(T_n,H^{(n)})$ is an FCS for $H$ (a $\mathbb{B}$-locally bounded predictable process), then both $(T_n,\widetilde{H}.M^{(n)})$ and $(T_n,H^{(n)}.M^{(n)})$ are inner FCSs for $H_{\bullet}M\in(\mathcal{M}_{\mathrm{loc}})^{i,\mathbb{B}}$.
\end{corollary}

\begin{corollary}\label{HMc=}
Let $H\in \mathcal{P}^\mathbb{C}$ and $M\in (\mathcal{M}_{\mathrm{loc}})^\mathbb{C}$. Then the following statements are equivalent:
\begin{itemize}
  \item[$(i)$]$H\in\mathcal{L}_m^\mathbb{C}(M)$.
  \item[$(ii)$]There exist a coupled predictable process $\widetilde{H}$ for $H\in \mathcal{P}^\mathbb{C}$ and an FCS $(T_n,M^{(n)})$ for $M\in(\mathcal{M}_{\mathrm{loc}})^\mathbb{C}$ such that $\widetilde{H}\in\mathcal{L}_m(M^{(n)})$ for each $n\in \mathbf{N}^+$.
  \item[$(iii)$]There exist an FCS $(T_n,H^{(n)})$ for $H\in \mathcal{P}^\mathbb{C}$ and an FCS $(T_n,M^{(n)})$ for $M\in(\mathcal{M}_{\mathrm{loc}})^\mathbb{C}$ satisfying $H^{(n)}\in\mathcal{L}_m(M^{(n)})$ for each $n\in \mathbf{N}^+$.
  \item[$(iv)$] There exists an FS $(\tau_n)$ for $\mathbb{C}$ such that $H^{\tau_n}\in\mathcal{L}_m(M^{\tau_n})$ for each $n\in \mathbf{N}^+$.
\end{itemize}
\end{corollary}

\begin{corollary}\label{eq-HMc}
Let $M\in (\mathcal{M}_{\mathrm{loc}})^\mathbb{C}$ and $H\in \mathcal{L}_m^\mathbb{C}(M)$.
\begin{itemize}
  \item [$(1)$] If $(\tau_n)$ is an FS for $\mathbb{C}$, then $(\tau_n,H^{\tau_n}.M^{\tau_n})$ is an FCS for $H_{\bullet}M\in(\mathcal{M}_{\mathrm{loc}})^\mathbb{C}$, and $H_{\bullet}M$ can be expressed as
      \begin{equation}\label{HM-expression0}
      H_{\bullet}M=\left((H_0M_0)I_{\llbracket{0}\rrbracket}+\sum\limits_{n=1}^{{+\infty}}
      (H^{\tau_n}.M^{\tau_n})I_{\rrbracket{\tau_{n-1},\tau_n}
      \rrbracket}\right)\mathfrak{I}_\mathbb{C},\quad \tau_0=0,
      \end{equation}
      which is independent of the choice of the FS $(\tau_n)$ for $\mathbb{C}$.

  \item [$(2)$] If $(T_n,H^{(n)})$ is an FCS for $H\in\mathcal{P}^\mathbb{C}$ and $(T_n,M^{(n)})$ is an FCS for $M\in(\mathcal{M}_{\mathrm{loc}})^\mathbb{C}$ such that for each $n\in \mathbf{N}^+$, $H^{(n)}\in\mathcal{L}_m(M^{(n)})$, then $(T_n,H^{(n)}.M^{(n)})$ is an FCS for $H_{\bullet}M\in(\mathcal{M}_{\mathrm{loc}})^\mathbb{C}$, and $H_{\bullet}M$ can be expressed as
      \begin{equation}\label{HMc-expression-2}
      H_{\bullet}M=\left((H_0M_0)I_{\llbracket{0}\rrbracket}+\sum\limits_{n=1}^{{+\infty}}
      (H^{(n)}.M^{(n)})I_{\rrbracket{T_{n-1},T_n}
      \rrbracket}\right)\mathfrak{I}_\mathbb{C},\quad T_0=0.
      \end{equation}
      which is independent of the choice of FCSs $(T_n, H^{(n)})$ for $H\in\mathcal{P}^\mathbb{C}$ and $(T_n,M^{(n)})$ for $M\in(\mathcal{M}_{\mathrm{loc}})^\mathbb{C}$.
\end{itemize}
\end{corollary}

\begin{remark}\normalfont
In the context of Corollary \ref{eq-HMc}(2), there is no need to enforce the requirement that  $(T_n,H^{(n)}.M^{(n)})$ must necessarily be an inner FCS for $H_{\bullet}M\in(\mathcal{M}_{\mathrm{loc}})^{i,\mathbb{C}}=(\mathcal{M}_{\mathrm{loc}})^{\mathbb{C}}$. This relaxation of the condition is justified by Theorem \ref{MC}(1), which asserts that $(T_n,H^{(n)}.M^{(n)})$ can always be converted into an inner FCS for $H_{\bullet}M\in(\mathcal{M}_{\mathrm{loc}})^{i,\mathbb{C}}$.
\end{remark}

Finally, the subsequent theorem elucidates the relationship between the $\mathbb{B}$-stochastic integral $H_{\bullet}M$, as defined in Definition \ref{HM}, and its continuous and purely discontinuous martingale parts.

\begin{theorem}\label{HM-property}
Let $H\in \mathcal{L}_m^\mathbb{B}(M)$, and assume that $M=M_0\mathfrak{I}_{\mathbb{B}}+M^c+M^d$ denotes the decomposition of $M$, where $M^c\in (\mathcal{M}^c_{\mathrm{loc},0})^\mathbb{B}$ and $M^d\in (\mathcal{M}^d_{\mathrm{loc}})^{i,\mathbb{B}}$.  Then:
\begin{itemize}
  \item [$(1)$]
  $H_{\bullet}M^c\in(\mathcal{M}^c_{\mathrm{loc},0})^\mathbb{B}$, and $(H_{\bullet}M)^c=H_{\bullet}M^c$,
  \item [$(2)$]
  $H_{\bullet}M^d\in(\mathcal{M}^d_{\mathrm{loc}})^{i,\mathbb{B}}$, and $(H_{\bullet}M)^d=H_{\bullet}M^d$.
  \item [$(3)$] $\mathcal{L}^\mathbb{B}_m(M)=\mathcal{L}_m^\mathbb{B}(M^c)\bigcap\mathcal{L}_m^\mathbb{B}(M^d)$.
\end{itemize}
\end{theorem}

\section{Stochastic integrals of $\mathbb{B}$-predictable processes with respect to $\mathbb{B}$-inner semimartingales}\label{section5}\setcounter{equation}{0}
In this section, we provide a constructive solution to Problem (SI) by synthesizing the $\mathbb{B}$-stochastic integrals developed in Sections \ref{section3} and \ref{section4}.
First, we formally define $\mathbb{B}$-inner semimartingales through a decomposition of $\mathbb{B}$-semimartingales relative to $\mathbb{B}$-inner local martingales.  We also extend the concept $\mathbb{B}$-quadratic covariations to the broader class of $\mathbb{B}$-semimartingales.
Next, we construct stochastic integrals of $\mathbb{B}$-predictable processes with respect to $\mathbb{B}$-inner semimartingales.
Following the construction, we examine the fundamental properties of these $\mathbb{B}$-stochastic integrals. In particular, we clarify their connections to existing stochastic integrals of predictable processes with respect to semimartingales.
Finally, we derive the It\^{o}'s formula for $\mathbb{B}$-inner semimartingales. This formula serves as a key result, emphasizing the importance of $\mathbb{B}$-stochastic integrals established in this section.

\subsection{$\mathbb{B}$-inner semimartingales and $\mathbb{B}$-quadratic covariations of $\mathbb{B}$ semimartingales}
Analogously to a semimartingale, a $\mathbb{B}$-semimartingale can be equivalently defined as the sum of a $\mathbb{B}$-local martingale and a $\mathbb{B}$-adapted process with finite variation, as is formalized in the subsequent lemma.

\begin{lemma}\label{XMA}
$X\in \mathcal{S}^\mathbb{B}$ if and only if $X$ admits a decomposition
    \begin{equation}\label{eq-XMA}
      X=M+A, \quad M\in (\mathcal{M}_{\mathrm{loc}})^\mathbb{B},\quad A\in (\mathcal{V}_0)^\mathbb{B}.
    \end{equation}
\end{lemma}

In general, the decomposition \eqref{eq-XMA} of a $\mathbb{B}$-semimartingale is not unique. To illustrate this point, consider a $\mathbb{B}$-process $M\in (\mathcal{W}_{\mathrm{loc},0})^\mathbb{B}$ such that $M\neq 0\mathfrak{I}_\mathbb{B}$.
It is evident that $M\in \mathcal{S}^\mathbb{B}$ admits two different decompositions, both conforming to the form specified by \eqref{eq-XMA}. Specifically, $M$ can be decomposed as $M=M+0\mathfrak{I}_\mathbb{B}$, where $M\in (\mathcal{M}_{\mathrm{loc}})^\mathbb{B}$ and $0\mathfrak{I}_\mathbb{B}\in (\mathcal{V}_0)^\mathbb{B}$, and alternatively as $M=0\mathfrak{I}_\mathbb{B}+M$, where $0\mathfrak{I}_\mathbb{B}\in (\mathcal{M}_{\mathrm{loc}})^\mathbb{B}$ and $M\in (\mathcal{V}_0)^\mathbb{B}$.
Furthermore, the decomposition \eqref{eq-XMA} underscores the pivotal role of $\mathbb{B}$-local martingales in the study of $\mathbb{B}$-semimartingales. Consequently, motivated by \eqref{eq-XMA} and the importance of $\mathbb{B}$-inner local martingales, we proceed to define the following concept of $\mathbb{B}$-inner semimartingales.

\begin{definition}\normalfont\label{SiB}
\begin{itemize}
  \item [$(1)$] Let $X\in \mathcal{S}^\mathbb{B}$. $X$ is called an {\bf essentially inner semimartingale} on $\mathbb{B}$ (or simply, a $\mathbb{B}$-{\bf inner semimartingale}), if $X$ admits a decomposition
  \begin{equation*}
      X=M+A, \quad M\in (\mathcal{M}_{\mathrm{loc}})^{i,\mathbb{B}},\quad A\in (\mathcal{V}_0)^\mathbb{B}.
    \end{equation*}
Such a decomposition is called an {\bf inner decomposition} of $X$, and the collection of all $\mathbb{B}$-inner semimartingales is denoted by $\mathcal{S}^{i,\mathbb{B}}$.
  \item [$(2)$] Let $X\in\mathcal{S}^{i,\mathbb{B}}$. Suppose that $(T_n,X^{(n)})$ is an FCS for $X\in \mathcal{S}^\mathbb{B}$, and that for each $n\in\mathbf{N}^+$, $X^{(n)}\in \mathcal{S}$ admits an decomposition $X^{(n)}=M^{(n)}+A^{(n)}$, where $M^{(n)}\in\mathcal{M}_{\mathrm{loc}}$ and $A^{(n)}\in \mathcal{V}_0$.
      If $(T_n,M^{(n)})$ is an inner FCS for a $\mathbb{B}$-process $M\in (\mathcal{M}_{\mathrm{loc}})^{i,\mathbb{B}}$ and $(T_n,A^{(n)})$ is an FCS for a $\mathbb{B}$-process $A\in (\mathcal{V}_0)^\mathbb{B}$ satisfying $X=M+A$, then $(T_n,X^{(n)})$ is called an {\bf inner FCS} for $X\in\mathcal{S}^{i,\mathbb{B}}$, and $(T_n,M^{(n)},A^{(n)})$ is called a {\bf decomposed inner FCS} for $X\in\mathcal{S}^{i,\mathbb{B}}$.
\end{itemize}
\end{definition}

Based on Theorem \ref{MC}(3), it is established that $\mathcal{S}^{i,\mathbb{B}}$ exhibits linearity, satisfying the property that $aX+bY\in\mathcal{S}^{i,\mathbb{B}}$ for any $a,b\in \mathbf{R}$ and $X,Y\in\mathcal{S}^{i,\mathbb{B}}$. Additionally, there exist several remarkable classes of $\mathbb{B}$-inner semimartingales within this framework:
\begin{itemize}
  \item [$(1)$] $\mathcal{S}^{i,\mathbb{C}}=\mathcal{S}^\mathbb{C}$.  In this case,
      a decomposition of $X\in\mathcal{S}^\mathbb{C}$ expressed as $X=M+A$, where $M\in (\mathcal{M}_{\mathrm{loc}})^{\mathbb{C}}$ and $A\in (\mathcal{V}_0)^\mathbb{C}$, holds true if and only if it constitutes an inner decomposition of $X\in\mathcal{S}^{i,\mathbb{C}}$.

  \item [$(2)$] If $\mathcal{M}_{\mathrm{loc}}=\mathcal{M}^c_{\mathrm{loc}}$, then  $\mathcal{S}^{i,\mathbb{B}}=\mathcal{S}^\mathbb{B}$.  In this case, the representation $X=M+A$, where $M\in (\mathcal{M}_{\mathrm{loc}})^{\mathbb{B}}$ and $A\in (\mathcal{V}_0)^\mathbb{B}$, constitutes a decomposition of $X\in\mathcal{S}^\mathbb{B}$ if and only if it simultaneously serves as an inner decomposition of $X\in\mathcal{S}^{i,\mathbb{B}}$.

  \item [$(3)$] Let $\mathbb{B}$ be given by \eqref{B}, and $\widetilde{M}\in\mathcal{M}_{\mathrm{loc},0}$ satisfying the condition $\widetilde{M}^{T_F-}\in\mathcal{M}_{\mathrm{loc},0}$. If $\widetilde{M}$ has the strong property of predictable representation, then $\mathcal{S}^{i,\mathbb{B}}=\mathcal{S}^\mathbb{B}$.  In this case, a decomposition of $X\in\mathcal{S}^\mathbb{B}$, expressed in the form $X=M+A$ where $M\in (\mathcal{M}_{\mathrm{loc}})^{\mathbb{B}}$ and $A\in (\mathcal{V}_0)^\mathbb{B}$, is valid if and only if it also constitutes an inner decomposition of $X\in\mathcal{S}^{i,\mathbb{B}}$.

  \item [$(4)$] $(\mathcal{W}_{\mathrm{loc}})^{\mathbb{B}}\subseteq \mathcal{S}^{i,\mathbb{B}}$. If $M\in(\mathcal{W}_{\mathrm{loc}})^{\mathbb{B}}$, then $M=M_0\mathfrak{I}_\mathbb{B}+(M-M_0\mathfrak{I}_\mathbb{B})$ is an inner decomposition of $M\in\mathcal{S}^{i,\mathbb{B}}$, where $M_0\mathfrak{I}_\mathbb{B}\in (\mathcal{M}_{\mathrm{loc}})^{i,\mathbb{B}}$ and $M-M_0\mathfrak{I}_\mathbb{B}\in (\mathcal{V}_0)^\mathbb{B}$.
\end{itemize}

\begin{remark}\normalfont
The question of whether the relation $\mathcal{S}^{i,\mathbb{B}}=\mathcal{S}^{\mathbb{B}}$ necessarily holds for any given $\mathbb{B}$ remains unresolved. Nonetheless, it is worth noting that even in instances where the equality $\mathcal{S}^{i,\mathbb{B}}=\mathcal{S}^{\mathbb{B}}$ does hold, the structure of $\mathcal{S}^{i,\mathbb{B}}$ may exhibit substantial discrepancies when compared to that of $\mathcal{S}^{\mathbb{B}}$.
To illustrate this point, consider the setting provided in Section 1.1, where $\mathbb{B}^*$, $\mathbb{F}$, and $\widehat{M}\in (\mathcal{W}_{\mathrm{loc}})^{\mathbb{B}^*}$ are given.
Due to the fact that $A-A^p\in\mathcal{W}_{\mathrm{loc},0}$ possesses the strong property of predictable representation (see, e.g., Proposition 2.7 in \cite{Aksamit}), it follows that all $\mathbb{F}$-local martingales are also $\mathbb{F}$-locally integrable variation martingales. This implication leads to the equality $\mathcal{S}^{i,{\mathbb{B}^*}}=\mathcal{S}^{\mathbb{B}^*}$.
From the perspective of $\mathcal{S}^{\mathbb{B}^*}$, the process $\widehat{M}$ exhibits characteristics of both a $\mathbb{B}^*$-local martingale and a $\mathbb{B}^*$-adapted process with finite variation. This is evident from the fact that $M$ can be decomposed in two ways:
\[
\widehat{M}=\widehat{M}+0\mathfrak{I}_{\mathbb{B}^*} \;(\widehat{M}\in (\mathcal{M}_{\mathrm{loc}})^{\mathbb{B}^*}, 0\mathfrak{I}_{\mathbb{B}^*}\in (\mathcal{V}_0)^{\mathbb{B}^*})\quad\text{and}\quad \widehat{M}=0\mathfrak{I}_{\mathbb{B}^*}+M\; (0\mathfrak{I}_{\mathbb{B}^*}\in (\mathcal{M}_{\mathrm{loc}})^{\mathbb{B}^*}, \widehat{M}\in (\mathcal{V}_0)^{\mathbb{B}^*}).
\]
However, from the viewpoint of $\mathcal{S}^{i,{\mathbb{B}^*}}$, the process $\widehat{M}$ is more akin to a ${\mathbb{B}^*}$-adapted process with finite variation rather than a ${\mathbb{B}^*}$-local martingale. This is because $\widehat{M}\notin(\mathcal{M}_{\mathrm{loc}})^{i,{\mathbb{B}^*}}$ and the so-called inner decomposition $\widehat{M}=\widehat{M}+0\mathfrak{I}_{\mathbb{B}^*}$ $(\widehat{M}\in (\mathcal{M}_{\mathrm{loc}})^{i,{\mathbb{B}^*}}, 0\mathfrak{I}_{\mathbb{B}^*}\in (\mathcal{V}_0)^{\mathbb{B}^*})$ is not valid.
\end{remark}

Upon substituting (\ref{con-M}) into (\ref{eq-XMA}), we derive that $X\in \mathcal{S}^\mathbb{B}$
permits a further decomposition as follows:
\begin{equation}\label{decomX}
X=X_0\mathfrak{I}_\mathbb{B}+M^c+M^d+A,
\end{equation}
where $M\in (\mathcal{M}_{\mathrm{loc}})^\mathbb{B}$ and $A\in (\mathcal{V}_0)^\mathbb{B}$.
The forthcoming lemma will elucidate that $M^c$ is uniquely determined by $X$.

\begin{lemma}\label{uXc}
Let $X\in \mathcal{S}^\mathbb{B}$, and let $(\ref{decomX})$ be a decomposition of $X$. Suppose that $X$ admits another decomposition
\begin{equation*}
X=X_0\mathfrak{I}_\mathbb{B}+N^c+N^d+V,
\end{equation*}
where $N\in (\mathcal{M}_{\mathrm{loc}})^\mathbb{B}$ and $V\in (\mathcal{V}_0)^\mathbb{B}$. It then necessarily follows that $M^c=N^c$.
\end{lemma}

In the context of the decomposition \eqref{decomX}, akin to the continuous martingale part of a semimartingale, the component $M^c$ is designated as the {\bf continuous martingale part} of $X$, alternatively denoted as $X^c$. The following theorem presents the fundamental properties of the continuous martingale part of a $\mathbb{B}$-semimartingale.

\begin{theorem}\label{XTc}
Let $X\in \mathcal{S}^\mathbb{B}$.
\begin{itemize}
  \item [$(1)$]If $(T_n,X^{(n)})$ is an FCS for $X\in \mathcal{S}^\mathbb{B}$, then $(T_n,(X^{(n)})^c)$ is an FCS for $X^c\in (\mathcal{M}^c_{\mathrm{loc},0})^\mathbb{B}$.
  \item [$(2)$]If $\tau$ is a $\mathbb{B}$-inner stopping time, then
  \begin{equation}\label{Xc-eq}
  (X^\tau)^c=(X^c)^\tau
  \end{equation}
  and
  \begin{equation}\label{XcB-eq}
  (X^\tau)^c\mathfrak{I}_\mathbb{B}=(X^c)^\tau\mathfrak{I}_\mathbb{B}=(X^\tau\mathfrak{I}_\mathbb{B})^c.
  \end{equation}
\end{itemize}
\end{theorem}

Now, given two $\mathbb{B}$-semimartingales,  we can establish their quadratic covariation by employing  their continuous martingale parts, as the subsequent definition will illustrate.

\begin{definition}\normalfont\label{[X,Y]}
Let $X,Y\in\mathcal{S}^\mathbb{B}$. The $\mathbb{B}$-{\bf quadratic covariation} of $X$ and $Y$, denoted by $[X,Y]$, is defined as follows:
\[
[X,Y]=X_0Y_0\mathfrak{I}_\mathbb{B}+\langle X^c,Y^c\rangle+\Sigma(\Delta X\Delta Y),
\]
where $X^c\in(\mathcal{M}^c_{\mathrm{loc},0})^\mathbb{B}$ and $Y^c\in(\mathcal{M}^c_{\mathrm{loc},0})^\mathbb{B}$ are the continuous martingale parts of $X$ and $Y$, respectively.
In the specific case where $X=Y$, the $\mathbb{B}$-process $[X,X]$ (or simply, $[X]$) is called the $\mathbb{B}$-{\bf quadratic variation} of $X$.
\end{definition}

The definition provided for the $\mathbb{B}$-quadratic covariation aligns with the formulation presented in Definition 8.2 of \cite{He} if $\mathbb{B}=\llbracket{0,+\infty}\llbracket$. The subsequent theorem outlines the fundamental properties of $\mathbb{B}$-quadratic covariations.

\begin{theorem}\label{[X,Y]-fcs}
Let $X,\; Y\in \mathcal{S}^\mathbb{B}$.
\begin{itemize}
  \item [$(1)$] If $Z\in \mathcal{S}^\mathbb{B}$ and $a,b\in \mathbf{R}$, then
  \[
    [X,Y]=[Y,X],\quad [aX+bY,Z]=a[X,Z]+b[Y,Z].
  \]
  \item [$(2)$] If $(T_n,X^{(n)})$ and $(T_n,Y^{(n)})$ are FCSs for $X\in \mathcal{S}^\mathbb{B}$ and $Y\in \mathcal{S}^\mathbb{B}$ respectively, then $(T_n,[X^{(n)},Y^{(n)}])$ is an FCS for $[X,Y]\in \mathcal{V}^\mathbb{B}$, and $(T_n,[X^{(n)}])$ is an FCS for $[X]\in (\mathcal{V}^+)^\mathbb{B}$.
  \item [$(3)$] If $\tau$ be a $\mathbb{B}$-inner stopping time, then
\begin{equation}\label{XY}
      [X^{\tau},Y^{\tau}]=[X,Y]^{\tau}
\end{equation}
and
\begin{equation}\label{XYB}
[X^{\tau}\mathfrak{I}_\mathbb{B},Y^{\tau}\mathfrak{I}_\mathbb{B}]
=[X,Y]^{\tau}\mathfrak{I}_\mathbb{B}=[X^{\tau},Y^{\tau}]\mathfrak{I}_\mathbb{B}=[X^{\tau}\mathfrak{I}_\mathbb{B},Y].
\end{equation}
\end{itemize}
\end{theorem}

Finally, we present an illustrative example of semimartingales on stochastic sets of interval type. This exemplification is grounded in Theorem 12.18 from \cite{He}, and underscors the practical applicability of $\mathbb{C}$-semimartingales within the framework of measure transformations.

\begin{example}\normalfont\label{ex-absolute}
Suppose that $\mathbf{Q}$ represents another probability measure on the filtered space $(\Omega,\mathcal{F},\mathbb{F})$, and that $\mathbf{Q}$ is locally absolutely continuous w.r.t. $\mathbf{P}$, denoted as
$\mathbf{Q}\overset{\mathrm{loc}}\ll\mathbf{P}$ (for details, see \cite{He,Jacod}). Define
\begin{align}
\mathbb{C}&=\bigcup_n\llbracket{0,\tau_n}\rrbracket,\label{absolute}\\
\tau_n&=\inf\left\{t: \widetilde{Z}_t\leq \frac{1}{n}\right\},\quad n\in \mathbb{N}^+,\nonumber
\end{align}
where $\widetilde{Z}$ is the density process of $\mathbf{Q}$, relative to $\mathbf{P}$. It follows that $\mathbb{C}$ is a predictable set of interval type, and $(\tau_n)$ is an FS for $\mathbb{C}$. Let $\mathcal{S}(\mathbf{Q})$ denote the classes of all $\mathbf{Q}$-semimartingales,
$\widetilde{Y}$ be an adapted c\`{a}dl\`{a}g process, and $Y=\widetilde{Y}\mathfrak{I}_\mathbb{C}$. Then according to Theorem 12.18 in \cite{He}, the following equivalence holds true:
  \[
  \widetilde{Y}\in\mathcal{S}(\mathbf{Q})\Leftrightarrow Y\in\mathcal{S}^\mathbb{C}.
   \]
 \end{example}

\subsection{Definition of $H_{\bullet}X$}
Unless otherwise explicitly specified, we shall uniformly adopt the following assumptions for the remaining content in this section: $H\in \mathcal{P}^\mathbb{B}$ and $X\in \mathcal{S}^{i,\mathbb{B}}$.

The conventional construction of the stochastic integral of a predictable process with respect to a semimartingale hinges on a decomposition of the semimartingale itself. Notably, this integral remains invariant under alternative decompositions of the semimartingale. In light of this foundational framework, we proceed by introducing the following lemma, which serves as a critical stepping stone toward the rigorous formulation of a constructive resolution to Problem (SI).

\begin{lemma}\label{HX-unique}
Assume that $X=M+A$ and $X=N+V$ are both inner decompositions of $X$, where $M\in (\mathcal{M}_{\mathrm{loc}})^{i,\mathbb{B}}$, $A\in (\mathcal{V}_0)^\mathbb{B}$, $N\in (\mathcal{M}_{\mathrm{loc}})^{i,\mathbb{B}}$ and $V\in (\mathcal{V}_0)^\mathbb{B}$. If $H\in \mathcal{L}^\mathbb{B}_m(M)\bigcap\mathcal{L}^\mathbb{B}_m(N)$ and $H\in \mathcal{L}^\mathbb{B}_s(A)\bigcap\mathcal{L}^\mathbb{B}_s(V)$, then
\begin{equation}\label{HX-unique-1}
H_{\bullet}M+H_{\bullet}A=H_{\bullet}N+H_{\bullet}V.
\end{equation}
\end{lemma}

\begin{definition}\normalfont\label{de-HX}
We say that $H$ is $\mathbb{B}$-integrable w.r.t. $X$ in the domain of $\mathbb{B}$-inner semimartingales (or simply, $H$ is $(X,\mathbb{B})$-integrable), if there exists an inner decomposition $X=M+A$ ($M\in (\mathcal{M}_{\mathrm{loc}})^{i,\mathbb{B}}$ and $A\in (\mathcal{V}_0)^\mathbb{B}$) such that $H\in \mathcal{L}_m^{\mathbb{B}}(M)\cap\mathcal{L}_s^{\mathbb{B}}(A)$.
In this case, the $\mathbb{B}$-process defined by
\begin{equation}\label{HX}
H_{\bullet}X=H_{\bullet}M+H_{\bullet}A
\end{equation}
is called the stochastic integral of $H$ w.r.t. $X$, and $X=M+A$ is an $(H,\mathbb{B})$-decomposition of $X$. The collection of all $\mathbb{B}$-predictable processes which are $(X,\mathbb{B})$-integrable is denoted by $\mathcal{L}^{\mathbb{B}}(X)$.
\end{definition}

Lemma \ref{HX-unique} ensures that the $\mathbb{B}$-stochastic integral $H_{\bullet}X$, as defined by  \eqref{HX}, remains invariant with respect to any $(H,\mathbb{B})$-decomposition of $X$. This property mirrors the behavior of the conventional stochastic integral, thereby establishing a significant analogy.

\begin{remark}\normalfont
From Corollary \ref{cD=DB}, the following relations are established:
\begin{equation*}
\mathcal{S}=\mathcal{S}^{\llbracket{0,+\infty}\llbracket}, \quad \mathcal{M}_{\mathrm{loc}}=(\mathcal{M}_{\mathrm{loc}})^{\llbracket{0,+\infty}\llbracket},
\quad \text{and}\quad \mathcal{V}_0=(\mathcal{V}_0)^{\llbracket{0,+\infty}\llbracket}.
\end{equation*}
Consequently, it becomes evident that the $\mathbb{B}$-stochastic integral $H_{\bullet}X$, as defined by \eqref{HX}, simplifies to the conventional stochastic integral $H.X$, under the condition that $\mathbb{B}=\llbracket{0,+\infty}\llbracket=\Omega\times\mathbf{R}^+$.
Specifically, such a relationship between the two integrals can be expressed as follows:
\begin{itemize}
  \item [] Let $H\in \mathcal{P}^{\llbracket{0,+\infty}\llbracket}$ and $X\in \mathcal{S}^{\llbracket{0,+\infty}\llbracket}$. If $H\in\mathcal{L}^{\llbracket{0,+\infty}\llbracket}(X)$, then $H\in\mathcal{L}(X)$ and $H_{\bullet}X=H.X$.
\end{itemize}
\end{remark}

\subsection{Fundamental properties of $H_{\bullet}X$}
Drawing upon Theorems \ref{HA-equivalent} and \ref{HM=}, the integrability within the domain of $\mathbb{B}$-inner semimartingales, as specified in Definition \ref{de-HX}, demonstrates a profound and intricate correlation with the existing integrability within the domain of semimartingales. This relationship is further expounded upon by the subsequent theorem.

\begin{theorem}\label{HX=}
The following statements are equivalent:
\begin{description}
  \item[(X1)] $H\in\mathcal{L}^\mathbb{B}(X)$.
  \item[(X2)] There exist a coupled predictable process $\widetilde{H}$ for $H\in\mathcal{P}^\mathbb{B}$ and a decomposed inner FCS $(T_n,M^{(n)},A^{(n)})$ for $X\in \mathcal{S}^{i,\mathbb{B}}$ such that $\widetilde{H}\in\mathcal{L}_m(M^{(n)})\cap\mathcal{L}_s(A^{(n)})$ for each $n\in \mathbf{N}^+$.
  \item[(X3)] There exist an FCS $(T_n,H^{(n)})$ for $H\in\mathcal{P}^\mathbb{B}$ and a decomposed inner FCS $(T_n,M^{(n)},A^{(n)})$ for $X\in \mathcal{S}^{i,\mathbb{B}}$ such that $H^{(n)}\in\mathcal{L}_m(M^{(n)})\cap\mathcal{L}_s(A^{(n)})$ for each $n\in \mathbf{N}^+$.
\end{description}
\end{theorem}

\begin{remark}\normalfont
Based on Theorem \ref{process}(3), the assertion (X3) in Theorem \ref{HX=} can be reformulated in an equivalent manner as follows:
\begin{description}
  \item[(X3$'$)] There exist an FCS $(T_n,H^{(n)})$ for $H\in\mathcal{P}^\mathbb{B}$ and a decomposed inner FCS $(S_n,M^{(n)},A^{(n)})$ for $X\in \mathcal{S}^{i,\mathbb{B}}$ such that $H^{(n)}\in\mathcal{L}_m(M^{(n)})\cap\mathcal{L}_s(A^{(n)})$ for each $n\in \mathbf{N}^+$.
\end{description}
\end{remark}

Based on Theorems \ref{HA-FCS} and \ref{eq-HM}, the stochastic integral $H_{\bullet}X$ as defined in \eqref{HX} can be characterized as the summation of a sequence of stochastic integrals of predictable processes with respect to semimartingales. This particular characterization is formally presented in the subsequent theorem.

\begin{theorem}\label{eq-HX}
Let $H\in \mathcal{L}^\mathbb{B}(X)$.
Suppose that $(T_n,H^{(n)})$ is an FCS for $H\in\mathcal{P}^\mathbb{B}$, and that $(T_n,M^{(n)},A^{(n)})$ is a decomposed inner FCS for $X\in \mathcal{S}^{i,\mathbb{B}}$ such that  $H^{(n)}\in\mathcal{L}_m(M^{(n)})\cap\mathcal{L}_s(A^{(n)})$ for each $n\in \mathbf{N}^+$. Then $(T_n,H^{(n)}.M^{(n)},H^{(n)}.A^{(n)})$ is a decomposed inner FCS for $H_{\bullet}X\in\mathcal{S}^{i,\mathbb{B}}$, and $H_{\bullet}X$ can be expressed as
\begin{equation}\label{HX-expression-2}
      H_{\bullet}X=\left((H_0X_0)I_{\llbracket{0}\rrbracket}+\sum\limits_{n=1}^{{+\infty}}
      (H^{(n)}.M^{(n)}+H^{(n)}.A^{(n)})I_{\rrbracket{T_{n-1},T_n}
      \rrbracket}\right)\mathfrak{I}_\mathbb{B},\quad T_0=0.
\end{equation}
Furthermore, if $(\widetilde{T}_n,K^{(n)})$ is another FCS for $H\in\mathcal{P}^\mathbb{B}$, and if $(\widehat{T}_n,N^{(n)},V^{(n)})$ is another decomposed inner FCS for $X\in \mathcal{S}^{i,\mathbb{B}}$ such that $K^{(n)}\in\mathcal{L}_m(N^{(n)})\cap\mathcal{L}_s(V^{(n)})$ for each $n\in \mathbf{N}^+$,
then $H_{\bullet}X=Z$, where
      \[
      Z=\left((H_0X_0)I_{\llbracket{0}\rrbracket}+\sum\limits_{n=1}^{{+\infty}}
      (K^{(n)}.N^{(n)}+K^{(n)}.V^{(n)})I_{\rrbracket{\tau_{n-1},\tau_n}
      \rrbracket}\right)\mathfrak{I}_\mathbb{B},\quad \tau_0=0,
      \]
and $(\tau_n)=(\widetilde{T}_n\wedge\widehat{T}_n)$.
In this case, we say that the expression \eqref{HX-expression-2} is independent of the choice of the FCS $(T_n,H^{(n)})$ for $H\in\mathcal{P}^\mathbb{B}$, and the decomposed inner FCS $(T_n,M^{(n)},A^{(n)})$ for $X\in \mathcal{S}^{i,\mathbb{B}}$.
\end{theorem}

\begin{remark}\normalfont
In the context of \eqref{HX-expression-2}, the FCS $(T_n, H^{(n)})$ for $H\in\mathcal{P}^\mathbb{B}$ can be substituted with a coupled predictable process $\widetilde{H}$ for $H\in\mathcal{P}^\mathbb{B}$, given that $(T_n,\widetilde{H}$) also constitutes an FCS for $H\in\mathcal{P}^\mathbb{B}$.
Specifically, if $\widetilde{H}$ is a coupled predictable process for $H\in\mathcal{P}^\mathbb{B}$ and $(T_n,M^{(n)},A^{(n)})$ represents a decomposed inner FCS for $X\in \mathcal{S}^{i,\mathbb{B}}$ such that $\widetilde{H}\in\mathcal{L}_m(M^{(n)})\cap\mathcal{L}_s(A^{(n)})$ for each $n\in \mathbf{N}^+$, then $(T_n,\widetilde{H}.M^{(n)},\widetilde{H}.A^{(n)})$ serves as a decomposed inner FCS for $H_{\bullet}X\in\mathcal{S}^{i,\mathbb{B}}$, and $H_{\bullet}X$ can be expressed as
      \begin{equation}\label{HX-expression-1}
      H_{\bullet}X=\left((H_0X_0)I_{\llbracket{0}\rrbracket}+\sum\limits_{n=1}^{{+\infty}}
      (\widetilde{H}.M^{(n)}+\widetilde{H}.A^{(n)})I_{\rrbracket{T_{n-1},T_n}
      \rrbracket}\right)\mathfrak{I}_\mathbb{B},\quad T_0=0.
      \end{equation}
\end{remark}

\begin{corollary}\label{bound-HX}
Suppose further that $H$ is a $\mathbb{B}$-locally bounded predictable process. Then
$H\in \mathcal{L}^\mathbb{B}(X)$. Furthermore, given a decomposed inner FCS $(T_n,M^{(n)},A^{(n)})$ for $X\in\mathcal{S}^{i,\mathbb{B}}$, if $\widetilde{H}$ is a coupled locally bounded predictable process for $H$ and $(T_n,H^{(n)})$ forms an FCS for $H$ (a $\mathbb{B}$-locally bounded predictable process),  then both $(T_n,\widetilde{H}.M^{(n)},\widetilde{H}.A^{(n)})$ and $(T_n,H^{(n)}.M^{(n)},H^{(n)}.A^{(n)})$ emerge as decomposed inner FCSs for $H_{\bullet}X\in\mathcal{S}^{i,\mathbb{B}}$.
\end{corollary}

\begin{corollary}\label{HX=c}
Let $H\in \mathcal{P}^\mathbb{C}$ and $X\in \mathcal{S}^\mathbb{C}$.
Then the following statements are equivalent:
\begin{itemize}
  \item[$(i)$] $H\in\mathcal{L}^\mathbb{C}(X)$.
  \item[$(ii)$] There exist a coupled predictable process $\widetilde{H}$ for $H\in\mathcal{P}^\mathbb{C}$ and an FCS $(T_n,X^{(n)})$ for $X\in \mathcal{S}^\mathbb{C}$ such that $\widetilde{H}\in\mathcal{L}(X^{(n)})$ for each $n\in \mathbf{N}^+$.
  \item[$(iii)$] There exist an FCS $(T_n,H^{(n)})$ for $H\in\mathcal{P}^\mathbb{B}$ and an FCS $(T_n,X^{(n)})$ for $X\in \mathcal{S}^\mathbb{C}$ such that  $H^{(n)}\in\mathcal{L}(X^{(n)})$ for each $n\in \mathbf{N}^+$.
  \item[$(iv)$] There exist an FS $(\tau_n)$ for $\mathbb{C}$ such that $H^{\tau_n}\in\mathcal{L}(X^{\tau_n})$ for each $n\in \mathbf{N}^+$.
\end{itemize}
\end{corollary}

\begin{corollary}\label{eq-HXc}
Let $X\in \mathcal{S}^\mathbb{C}$ and $H\in \mathcal{L}^\mathbb{C}(X)$.
\begin{itemize}
  \item [$(1)$] If $(\tau_n)$ is an FS for $\mathbb{C}$, then $(\tau_n,H^{\tau_n}.X^{\tau_n})$ is an FCS for $H_{\bullet}X\in\mathcal{S}^\mathbb{C}$, and $H_{\bullet}X$ can be expressed as
      \begin{equation}\label{HXc-expression0}
      H_{\bullet}X=\left((H_0X_0)I_{\llbracket{0}\rrbracket}+\sum\limits_{n=1}^{{+\infty}}
      (H^{\tau_n}.X^{\tau_n})I_{\rrbracket{\tau_{n-1},\tau_n}
      \rrbracket}\right)\mathfrak{I}_\mathbb{C},\quad \tau_0=0,
      \end{equation}
      which is independent of the choice of the FS $(\tau_n)$ for $\mathbb{C}$.

  \item [$(2)$] If $(T_n,H^{(n)})$ is an FCS for $H\in\mathcal{P}^\mathbb{C}$ and $(T_n,X^{(n)})$ is an FCS for $X\in\mathcal{S}^\mathbb{C}$ such that $H^{(n)}\in\mathcal{L}(X^{(n)})$ for each $n\in \mathbf{N}^+$, then $(T_n,H^{(n)}.X^{(n)})$ is an FCS for $H_{\bullet}X\in\mathcal{S}^\mathbb{C}$, and $H_{\bullet}X$ can be expressed as
      \begin{equation}\label{HXc-expression-2}
      H_{\bullet}X=\left((H_0X_0)I_{\llbracket{0}\rrbracket}+\sum\limits_{n=1}^{{+\infty}}
      (H^{(n)}.X^{(n)})I_{\rrbracket{T_{n-1},T_n}
      \rrbracket}\right)\mathfrak{I}_\mathbb{C},\quad T_0=0,
      \end{equation}
      which is independent of the choice of FCSs $(T_n, H^{(n)})$ for $H\in\mathcal{P}^\mathbb{C}$ and $(T_n,X^{(n)})$ for $X\in\mathcal{S}^\mathbb{C}$.
\end{itemize}
\end{corollary}

\begin{remark}\normalfont
It is evident that $(\tau_n,H^{\tau_n}.X^{\tau_n})$ referenced in Corollary \ref{eq-HXc}(1) also forms an inner FCS for $H_{\bullet}X\in\mathcal{S}^{i,\mathbb{C}}=\mathcal{S}^\mathbb{C}$. On the other hand, it is not a prerequisite for $(T_n,H^{(n)}.X^{(n)})$ mentioned in Corollary \ref{eq-HXc}(2) to be an inner FCS for $H_{\bullet}X\in\mathcal{S}^{i,\mathbb{C}}=\mathcal{S}^{\mathbb{C}}$. This exemption arises because, as established by Theorem \ref{MC}(1), $(T_n,H^{(n)}.X^{(n)})$ can always be converted into an inner FCS for $H_{\bullet}X\in\mathcal{S}^{i,\mathbb{C}}$.
\end{remark}

In the ensuing two theorems, we delineate the fundamental properties of the $\mathbb{B}$-stochastic integral, specifically $H_{\bullet}X$ as formalized in Definition \ref{de-HX}. These properties constitute an extension of the findings derived from conventional stochastic integrals (see, e.g., Theorems 9.15 and 9.18 of \cite{He}), particularly in the context of $\mathbb{B}$-stochastic integrals.

\begin{theorem}\label{HX-p}
Let $Y\in \mathcal{S}^{i,\mathbb{B}}$, and $a,b\in\mathbf{R}$. Suppose that $H\in \mathcal{L}^\mathbb{B}(X)\cap\mathcal{L}^\mathbb{B}(Y)$ and $K\in \mathcal{L}^\mathbb{B}(X)$.
\begin{itemize}
  \item [$(1)$] $aH+bK\in\mathcal{L}^\mathbb{B}(X)$, and in this case
  \begin{equation}\label{HX+}
      (aH+bK)_{\bullet}X=a(H_{\bullet}X)+b(K_{\bullet}X).
  \end{equation}

  \item [$(2)$] $H\in\mathcal{L}^\mathbb{B}(aX+bY)$, and in this case
  \begin{equation}\label{+HX}
      H_{\bullet}(aX+bY)=a(H_{\bullet}X)+b(H_{\bullet}Y).
  \end{equation}

  \item [$(3)$] Let $L\in\mathcal{P}^\mathbb{B}$. Then $L\in\mathcal{L}^\mathbb{B}(H_{\bullet}X)$ if and only if $LH\in\mathcal{L}^\mathbb{B}(X)$.
  Furthermore, if $L\in\mathcal{L}^\mathbb{B}(H_{\bullet}X)$ (or equivalently, $LH\in\mathcal{L}^\mathbb{B}(X)$), then
  \begin{equation}\label{hHX}
     L_{\bullet}(H_{\bullet}X)=(LH)_{\bullet}X.
  \end{equation}
\end{itemize}
\end{theorem}

\begin{theorem}\label{HX-property}
Let $H\in\mathcal{L}^\mathbb{B}(X)$. Then:
\begin{itemize}
  \item [$(1)$] $(H_{\bullet}X)^c=H_{\bullet}X^c$, $\Delta (H_{\bullet}X)=H\Delta X$, and $(H_{\bullet}X)I_{\llbracket{0}\rrbracket}=HXI_{\llbracket{0}\rrbracket}$.
  \item [$(2)$] $(H_{\bullet}X)^\tau\mathfrak{I}_{\mathbb{B}}=H_{\bullet}(X^\tau\mathfrak{I}_{\mathbb{B}})
      =(HI_{\llbracket{0,\tau}\rrbracket}\mathfrak{I}_{\mathbb{B}})_{\bullet}X=(H^\tau\mathfrak{I}_{\mathbb{B}})_{\bullet}(X^\tau\mathfrak{I}_{\mathbb{B}})$,
  where $\tau$ is a $\mathbb{B}$-inner stopping time.
  \item [$(3)$] If $L\in \mathcal{P}^\mathbb{B}$ satisfying $|L|\leq |H|$, then $L\in\mathcal{L}^\mathbb{B}(X)$.
  \item [$(4)$] For any $Y\in \mathcal{S}^\mathbb{B}$, we have
\begin{equation}\label{HXY-p}
  [H_{\bullet}X,Y]=H_{\bullet}[X,Y].
\end{equation}
\end{itemize}
\end{theorem}

Lastly, we conclude by presenting an application of aforementioned $\mathbb{C}$-stochastic integrals within the theoretical framework of measure transformations.
\begin{example}\normalfont\label{example-pq}
Let $\mathbb{C}$ and $\widetilde{Z}$ be as defined in Example \ref{ex-absolute}, and let $\widetilde{X}$ be an adapted c\`{a}dl\`{a}g process. Define $Z=\widetilde{Z}\mathfrak{I}_\mathbb{C}$ and $X=\widetilde{X}\mathfrak{I}_\mathbb{C}$. Denote by $\mathcal{M}_{\mathrm{loc}}(\mathbf{Q})$ the set of all $\mathbf{Q}$-local martingales. We then establish the following relationships:
   \begin{align}
      \widetilde{X}\in\mathcal{M}_{\mathrm{loc}}(\mathbf{Q})\Leftrightarrow&X\in\mathcal{S}^\mathbb{C}\; \text{and}\; X+ \bigg[\frac{1}{Z_{-}}{}_{\bullet}X,Z\bigg]\in (\mathcal{M}_{\mathrm{loc}})^\mathbb{C}\nonumber\\
      \Leftrightarrow&
      X\in\mathcal{S}^\mathbb{C}\; \text{and}\; X+\frac{1}{Z_{-}} {}_{\bullet}[X,Z]\in (\mathcal{M}_{\mathrm{loc}})^\mathbb{C}.\label{ex-mloc-1}
  \end{align}
Indeed, according to Theorem 12.18(4) in \cite{He}, $\widetilde{X}\in\mathcal{M}_{\mathrm{loc}}(\mathbf{Q})$ holds if and only if there exist an FS $(S_n)$ for $\mathbb{C}$ such that $X\in\mathcal{S}^\mathbb{C}$ and $Y^{(n)}\in \mathcal{M}_{\mathrm{loc}}$ for each $n\in \mathbf{N}^+$, where
\begin{equation*}
Y^{(n)}=\widetilde{X}^{S_n}+\frac{1}{(\widetilde{Z}^{S_n})_{-}} {}.[\widetilde{X}^{S_n},\widetilde{Z}^{S_n}]=X^{S_n}+\frac{1}{(Z^{S_n})_{-}} {}.[X^{S_n},Z^{S_n}].
\end{equation*}
By applying Theorem \ref{[X,Y]-fcs} and Corollary \ref{eq-HXc}(1) and observing the relations
\[
Y^{(n)}=\left(X+\frac{1}{Z_{-}} {}_{\bullet}[X,Z]\right)^{S_n},\quad n\in \mathbf{N}^+,
\]
we confirm the validity of the relationships stated in \eqref{ex-mloc-1}.
\end{example}

\subsection{It\^{o} Formula for $\mathbb{B}$-inner semimartingale}
The It\^{o} Formula, alternatively known as the change-of-variable formula, constitutes a pivotal instrument within the realm of stochastic calculus. In the subsequent theorem, we elucidate the It\^{o} Formula tailored for $\mathbb{B}$-inner semimartingales. This formula not only asserts that a ``smooth function" of a $\mathbb{B}$-inner semimartingale retains its status as a $\mathbb{B}$-inner semimartingale but also furnishes its corresponding decomposition.

\begin{theorem}\label{ito}
Fix an integer $d\in \mathbf{N}^+$, and define $Z=(X_{1},X_{2},\cdots, X_{d})$, where $X_{1},X_{2},\cdots, X_{d}$ are $\mathbb{B}$-inner semimartingales. Suppose that $F$ is a $C^2$-function on $\mathbf{R}^d$ (i.e., $F$ possesses continuous partial derivatives up to the second order). Then
\begin{align}\label{ito-eq}
F(Z)=F(Z(0))\mathfrak{I}_\mathbb{B}+\sum_{k=1}^d D_kF(Z_{-})_{\bullet}(X_k-X_k(0)\mathfrak{I}_\mathbb{B})+\eta+
\frac{1}{2}\sum_{k,l=1}^d D_{kl}F(Z_{-})_{\bullet}\langle X_k^c,X_l^c\rangle,
\end{align}
where $D_kF=\frac{\partial F}{\partial x_k}$, $D_{kl}F=\frac{\partial^2 F}{\partial x_k\partial x_l}$ for $k,l\in\{1,2,\cdots,d\}$, and
\[
\eta=\Sigma\bigg(F(Z)-F(Z_{-})-\sum_{k=1}^d D_kF(Z_{-})\Delta X_k\bigg).
\]
\end{theorem}

In Theorem \ref{ito}, the $\mathbb{B}$-process $Z=(X_{1},X_{2}\cdots X_{d})$ is denoted as a $d$-dimensional semimartingale on $\mathbb{B}$. Alternatively, we can define a $d$-dimensional semimartingale on $\mathbb{B}$ in an equivalent manner: $Z$ is termed a $d$-dimensional semimartingale on $\mathbb{B}$ if there exists a CS $(T_n,Z^{(n)})$ for $Z$ such that for each $n\in \mathbf{N}^+$, $Z^{(n)}$ is a $d$-dimensional semimartingale. Furthermore, in the case where $d=1$, the It\^{o} formula \eqref{ito-eq} simplifies to
\[
f(X)-f(X_0)\mathfrak{I}_\mathbb{B}=f'(X_{-})_{\bullet}(X-X(0)\mathfrak{I}_\mathbb{B})+\Sigma\big(f(X)-f(X_{-})-f'(X_{-})\Delta X\big)+\frac{1}{2}f''(X_{-})_{\bullet}\langle X^c\rangle,
\]
where $X$ represents a $\mathbb{B}$-inner semimartingale, and $f$ is a $C^2$-function on $\mathbf{R}$ (implying that $f$ possesses continuous derivatives of the first order $f'$ and the second order $f''$).

The subsequent corollary delineates the formula for integration by parts pertinent to two $\mathbb{B}$-inner semimartingales, constituting a pivotal application of the It\^{o} formula as articulated in \eqref{ito-eq}.

\begin{corollary}\label{IbP-p}
Let $Y\in\mathcal{S}^{i,\mathbb{B}}$. Then
\begin{equation}\label{IbP-eq}
XY={X_{-}}{}_{\bullet}Y+{Y_{-}}{}_{\bullet}X+[X,Y]-2X_0Y_0\mathfrak{I}_\mathbb{B}.
\end{equation}
\end{corollary}

\section{Applications to financial markets with sudden-stop horizons}\label{section6}\noindent
\setcounter{equation}{0}
In this section, we put forward an alternative methodology by incorporating stochastic integrals on stochastic sets of interval type into the analysis of financial markets characterized by sudden-stop horizons.
Our principal aim is to construct a financial market framework in which the investor's time-horizon is circumscribed by a stochastic set of interval type. Furthermore, we extend the dynamic price of the risky asset to an inner semimartingale on such a stochastic set of interval type.

\subsection{Essentials of mathematical finance with sudden-stop horizons}
We begin by revisiting the existing financial market model $(Y,\mathbb{F})$, wherein $Y$ represents the price of a risky asset modeled as an $\mathbb{F}$-semimartingale, and for the sake of simplicity, we assume the savings account to be an asset with a constant value. At any given time $t\in \mathbf{R}^+$, an investor possesses $\vartheta_t$ shares of the stock, while the remainder of the wealth is invested in the savings account. Consequently, the investor's wealth $X_t$ can be formulated as $X=\vartheta Y+(X-\vartheta Y)$, where $X_0=x_0$ denotes the initial wealth, and $\vartheta$ is a predictable process. Let $L(Y,\mathbb{F})$ denote the set encompassing all such $\mathbb{F}$-strategies, $\vartheta\in \mathcal{L}(Y)$. Building upon this foundation, we present the following essentials of mathematical finance (see, e.g., Subsection 1.4 in \cite{Aksamit}, noting that $[0,T\rangle=[0,T]$ when $T\in]0,+\infty[$, and $[0,T\rangle=\mathbf{R}^+$ in the case where $T=+\infty$):
\begin{itemize}
  \item A strategy $\vartheta\in L(Y,\mathbb{F})$ is called self-financing if the wealth can be expressed as $X=x_0+\vartheta.Y-\vartheta_0Y_0=x_0+\int_0^\cdot\vartheta_sdY_s$.

  \item Let $\alpha>0$ be a real number. A strategy $\vartheta\in L(Y,\mathbb{F})$ is called $\alpha$-admissible on the time horizon $[0,T\rangle$, if $(\vartheta.Y)_t\geq -\alpha$, $\mathbf{P}$-a.s. for all $t\in[0,T\rangle$.
      Denote by $l_\alpha(Y,\mathbb{F},T)$ the set of all $\alpha$-admissible strategies on $[0,T\rangle$.
      Furthermore, a strategy $\vartheta\in L(Y,\mathbb{F})$ is called admissible on $[0,T\rangle$ if $\vartheta\in\bigcup_{a>0}l_a(Y,\mathbb{F},T)$, and we denote by $l(Y,\mathbb{F},T)$ the set of all admissible strategies on $[0,T\rangle$.

  \item The financial market $(Y,\mathbb{F})$ is said to satisfy no arbitrage (in short: NA) on $[0,T\rangle$ if there does not exist any strategy  $\vartheta\in l(Y,\mathbb{F},T)$ such that
\begin{equation}\label{NA0}
\mathbf{P}((\vartheta.Y)_T\geq 0)=1, \quad\text{and}\quad \mathbf{P}((\vartheta.Y)_T>0)>0.
\end{equation}
\end{itemize}
In the case where $T=+\infty$, a strategy which is $\alpha$-admissible (resp. admissible) on $[0,+\infty[$ is also called $\alpha$-admissible (resp. admissible), and a market which satisfies NA on $[0,+\infty[$ is also said to satisfy NA.

We now proceed to construct a novel financial market framework incorporating a sudden-stop horizon. Analogous to the existing financial market setup, this market features a traded risky asset denoted by $S$, alongside a savings account which maintains a constant price. However, a distinguishing feature of this market lies in the uncertainty surrounding the investor's time-horizon, which can be characterized by a stochastic set $\mathbb{B}$ of interval type. Furthermore, the dynamic price of the risky asset can be described by a $\mathbb{B}$-inner semimartingale, as opposed to a conventional semimartingale. For clarity and precision, we denote this financial market structure by the triplet $(S,\mathbb{F},\mathbb{B})$. This market is deliberately structured in such a manner as to systematically exclude extraneous information that falls outside the purview of the pre-defined time-horizon  $\mathbb{B}$.

In the financial market $(S,\mathbb{F},\mathbb{B})$, at any given time t for which $(\omega,t)\in \mathbb{B}$ holds true, the investor possesses ${\vartheta}(\omega,t)$ shares of the risky asset, with the remaining portion of the wealth invested in the savings account. The investor's wealth $X(\omega,t)$ can subsequently be formulated as:
\begin{equation*}
{X}(\omega,t)={\vartheta}(\omega,t){S}(\omega,t)+({X}(\omega,t)-{\vartheta}(\omega,t){S}(\omega,t)), \quad (\omega,t)\in \mathbb{B},
\end{equation*}
or equivalently,
\begin{equation}\label{wealth}
X={\vartheta}{S}+({X}-{\vartheta}{S}),
\end{equation}
where $X_0=x_0$ with $x_0>0$, and ${\vartheta}\in \mathcal{P}^\mathbb{B}$. Here, $x_0$ represents a constant initial wealth, and ${\vartheta}$ is called an $(\mathbb{F},\mathbb{B})$-strategy.

Drawing parallels with the essentials of conventional mathematical finance, we can define the concepts of self-financing strategies and admissible strategies within the financial market $(S,\mathbb{F},\mathbb{B})$.

\begin{definition}\normalfont\label{essential}
In the financial market $(S,\mathbb{F},\mathbb{B})$, suppose that the investor's wealth ${X}$ and trading strategy ${\vartheta}$ are defined as per \eqref{wealth}. Let $\alpha>0$ be a real number, and denote by $L(S,\mathbb{F},\mathbb{B})$ the collection of all $(\mathbb{F},\mathbb{B})$-strategies ${\vartheta}\in \mathcal{L}^\mathbb{B}({S})$.
\begin{itemize}
  \item[$(1)$] A strategy ${\vartheta}\in L(S,\mathbb{F},\mathbb{B})$ is called {\bf self-financing} if the wealth ${X}$ can be expressed as
  \begin{equation}\label{wealth1}
      {X}=({x}_0-\vartheta_0S_0)\mathfrak{I}_\mathbb{B}+{\vartheta}_{\bullet}{S}.
  \end{equation}

  \item[$(2)$] A strategy ${\vartheta}\in L(S,\mathbb{F},\mathbb{B})$ is called {\bf $\alpha$-admissible} if
      \[
      {\vartheta}_{\bullet}{S}\geq (-\alpha)\mathfrak{I}_\mathbb{B}.
      \]
      Denote by $l_\alpha(S,\mathbb{F},\mathbb{B})$ the set of all $\alpha$-admissible strategies in the financial market $(S,\mathbb{F},\mathbb{B})$. Furthermore, a strategy ${\vartheta}\in L(S,\mathbb{F},\mathbb{B})$ is called admissible if $\vartheta\in\bigcup_{a>0}l_a(S,\mathbb{F},\mathbb{B})$, and we denote by $l(S,\mathbb{F},\mathbb{B})$ the set of all admissible strategies in the financial market $(S,\mathbb{F},\mathbb{B})$.
\end{itemize}
\end{definition}

Let $(T_n,S^{(n)})$ denote an inner FCS for $S\in\mathcal{S}^{i,\mathbb{B}}$, and fix $n\in \mathbf{N}^+$. It follows that $(S^{(n)},\mathbb{F})$ constitutes a financial market with an infinite time-horizon, which is associated with the triplet $(S,\mathbb{F},\mathbb{B})$. On $\mathbb{B}\llbracket{0,T_n}\rrbracket$, the stock price process $S^{(n)}$ within the financial market $(S^{(n)},\mathbb{F})$  mirrors that of $(S,\mathbb{F},\mathbb{B})$, thereby inducing identical portfolio strategies for the investor (specifically, the strategy $\vartheta$ and the wealth $X$ as delineated in \eqref{wealth}). Consequently, we adopt the following standing assumption:
\begin{description}
  \item [(G)]  For each $n\in \mathbf{N}^+$, the investor's strategy $\vartheta^{(n)}$ and wealth $X^{(n)}$ in $(S^{(n)},\mathbb{F})$ satisfy the conditions
\[
\vartheta^{(n)}I_{\mathbb{B}\llbracket{0,T_n}\rrbracket}
=\vartheta I_{\mathbb{B}\llbracket{0,T_n}\rrbracket},\quad X^{(n)}I_{\mathbb{B}\llbracket{0,T_n}\rrbracket}
=XI_{\mathbb{B}\llbracket{0,T_n}\rrbracket}.
\]
\end{description}
Under Assumption (G), a self-financing (resp. an admissible) strategy in $(S,\mathbb{F},\mathbb{B})$ can be characterized by a sequence of such strategies in conventional markets, as formalized by the ensuing theorem.

\begin{proposition}\label{strategy-e}
Let $\alpha>0$ be a real number.
A strategy ${\vartheta}\in L(S,\mathbb{F},\mathbb{B})$ is self-financing (resp. $\alpha$-admissible) if and only if there exist an FCS $(T_n,\vartheta^{(n)})$ for ${\vartheta}\in \mathcal{P}^\mathbb{B}$ and a decomposed inner FCS $(T_n,M^{(n)},A^{(n)})$ for $S\in \mathcal{S}^{i,\mathbb{B}}$, such that for each $n\in \mathbf{N}^+$, the strategy ${\vartheta}^{(n)}$ is self-financing (resp. $\alpha$-admissible) in the financial market $(M^{(n)}+A^{(n)},\mathbb{F})$ and satisfies ${\vartheta}^{(n)}\in\mathcal{L}_m(M^{(n)})\cap\mathcal{L}_s(A^{(n)})$.
\end{proposition}

\begin{remark}\normalfont
Based on Definition \ref{essential}(2) and Theorem \ref{strategy-e}, it is straightforward to establish the following equivalence: A strategy ${\vartheta}\in L(S,\mathbb{F},\mathbb{B})$ is admissible if and only if there exist a real number $\alpha>0$, an FCS $(T_n,\vartheta^{(n)})$ for ${\vartheta}\in \mathcal{P}^\mathbb{B}$ and a decomposed inner FCS $(T_n,M^{(n)},A^{(n)})$ for $S\in \mathcal{S}^{i,\mathbb{B}}$, such that for each $n\in \mathbf{N}^+$, the strategy ${\vartheta}^{(n)}$ is $\alpha$-admissible in the financial market $(M^{(n)}+A^{(n)},\mathbb{F})$ and satisfies ${\vartheta}^{(n)}\in\mathcal{L}_m(M^{(n)})\cap\mathcal{L}_s(A^{(n)})$.
\end{remark}

For a fixed $T>0$ and a given $\mathbb{B}$-process $Z$, the notation $Z_T$ might lack precise definition unless $T$ qualifies as a $\mathbb{B}$-inner stopping time. As a result, the concept of no arbitrage in financial markets with sudden-stop horizons cannot be straightforwardly extended from \eqref{NA0}. Drawing inspiration from Theorem \ref{strategy-e}, we propose an alternative formulation to define NA in such financial markets with sudden-stop horizons.

\begin{definition}\normalfont
The financial market $(S,\mathbb{F},\mathbb{B})$ is said to satisfy NA if there exists an inner FCS $(T_n,S^{(n)})$ for $S\in \mathcal{S}^{i,\mathbb{B}}$ such that for each $n\in \mathbf{N}^+$, the financial market $(S^{(n)},\mathbb{F})$ satisfies NA.
\end{definition}

\begin{remark}\normalfont
Consider the scenario where $\mathbb{B}=\llbracket{0,+\infty}\llbracket=\Omega\times\mathbf{R}^+$. It is evident that the financial market $(S,\mathbb{F},\llbracket{0,+\infty}\llbracket)$ reduces to the conventional financial market $(S,\mathbb{F})$. Furthermore, a strategy $\vartheta$ is deemed self-financing (resp. admissible) in the financial market $(S,\mathbb{F},\llbracket{0,+\infty}\llbracket)$ if and only if it is self-financing (resp. admissible) in the financial market $(S,\mathbb{F})$. Lastly, the financial market $({S},\mathbb{F},\llbracket{0,+\infty}\llbracket)$ satisfies NA if and only if the financial market $(S,\mathbb{F})$ also satisfies NA.
\end{remark}

Lastly, we delve into the investor's portfolio optimization problem within the financial market framework represented by $(S,\mathbb{F},\mathbb{B})$. Let $(T_n,S^{(n)})$ serve as an inner FCS for $S\in\mathcal{S}^{i,\mathbb{B}}$. Given that $\omega\in F$ as specified in equation \eqref{B}, the investor's time-horizon is restricted to the interval $[0,T(\omega)[$. Consequently, the investor is obligated to formulate portfolio strategies strictly prior to date $T(\omega)$. As elaborated in Section 1.3, the investor selects the sequence $(T_n)$ as a constituent of terminal times. This approach enables the investor to effectively eliminate extraneous information concerning asset price dynamics and portfolio strategies that extend beyond the sudden-stop horizon.

\begin{definition}\normalfont\label{portfolio-problem}
Let the investor's wealth X and strategy ${\vartheta}\in L(S,\mathbb{F},\mathbb{B})$ be defined as per equation (\ref{wealth}). Consider $(T_n,M^{(n)},A^{(n)})$ to be a decomposed inner FCS for $S\in \mathcal{S}^{i,\mathbb{B}}$, and let $\varphi$ denote a utility function, exemplified by the logarithmic utility function given by:
\[
\varphi(x)=\ln x,\quad x>0.
\]
Define ${S}^{(n)}=M^{(n)}+A^{(n)}$ for each $n\in \mathbf{N}^+$. An admissible strategy $\pi$ in the financial market $(S,\mathbb{F},\mathbb{B})$ is deemed optimal (w.r.t. $(T_n,M^{(n)},A^{(n)})$) if there exists an FCS $(T_n,\pi^{(n)})$ for ${\pi}\in \mathcal{P}^\mathbb{B}$ such that, for each $n\in \mathbf{N}^+$, $\pi^{(n)}\in\mathcal{L}_m(M^{(n)})\cap\mathcal{L}_s(A^{(n)})$ constitutes the optimal strategy for the ensuing portfolio problem in $({S}^{(n)},\mathbb{F})$:
\begin{equation}\label{optimal-problem}
\left\{
\begin{aligned}
&\pi^{(n)}=\arg\sup\left\{\mathbb{E}\left(\varphi({X}^{(n)}_{T_n})\right): {\vartheta\in l({S}^{(n)},\mathbb{F},+\infty)},\vartheta\in\mathcal{L}_m(M^{(n)})\cap\mathcal{L}_s(A^{(n)})\right\},\\
&\mathrm{s.t.}\quad {X}^{(n)}={x}_0 +\vartheta.{S}^{(n)}-\vartheta_0S_0\geq 0.
\end{aligned}
\right.
\end{equation}
\end{definition}
It is important to note that, within the framework of Definition \ref{portfolio-problem}, the optimal strategy $\pi$ typically exhibits a dependency on the sequence $(T_n,M^{(n)},A^{(n)})$. This dependence arises from the fact that distinct selections of the sequence $(T_n,M^{(n)},A^{(n)})$, particularly with regard to the terminal dates $(T_n)$, mirror the investor's diverse perspectives on the sudden-stop horizon. Consequently, these varying selections give rise to distinct portfolio strategies.

\subsection{An illustrative example}
We conduct an investigation into a simplified financial market framework characterized by a sudden-stop horizon. Within this framework, we consider a financial market with an infinite time span, in which an investor invests in a risky stock denoted as $S$. Let $\tau>0$ be an $\mathbb{F}$-stopping time, representing the uncertain date at which the stock exits the market. The investor's time-horizon can be described by a stochastic set of interval type, specifically defined as: $\mathbb{B}=\llbracket{0,\tau}\llbracket$.

The exit time $\tau$ occupies a central role within the investor's investment process. The information pertaining to $\tau$ exerts a profound impact on the investor's investment, as exemplified by the choice of $\tau$ as the terminal time. However, obtaining precise information of the exit time is challenging, particularly in scenarios where the investment duration is prolonged (i.e., $\tau(\omega)$ is sufficiently large for all $\omega\in\Omega$). To mitigate the impact of market information following the default time, it is more appropriate for the investor to select a time preceding the default time as the terminal time. Consequently, we assume that the investor addresses the default time in the following manner.
At the initial date $\tau_0=0$, the investor prudently estimates a date $\tau_1$ such that $\tau_1\leq\tau$. Subsequently, the investor determines the driving process $N^{(1)}$ that captures the stock's uncertainty, as well as the stock return $\mu_1$ and volatility $\sigma_1$ over the first time-horizon $\llbracket{0,\tau_1}\rrbracket$. For $k\in \mathbf{N}^+$, at the date $\tau_k$ when the default has not yet occurred, the investor re-estimates a date $\tau_{k+1}$ satisfying $\tau_{k+1}\leq\tau$. Simultaneously, the investor updates the driving process from $N^{(k)}$ to $N^{(k+1)}$, and adjusts the stock return from $\mu_{k}$ to $\mu_{k+1}$ and the volatility from $\sigma_{k}$ to $\sigma_{k+1}$ over the $k+1$-th time-horizon $\rrbracket{\tau_k,\tau_{k+1}}\rrbracket$, respectively.
Here, $(\tau_n)$ is an increasing sequence of stopping time satisfies the condition $\tau_n\uparrow \tau$, $(N^{(n)})$ is a sequence of continuous local martingales with null initial values, and $(\mu_n)$ and $(\sigma_{n})$ are two bounded sequences of constants with $\sigma_{n}>0$.

The sequence $(\tau_n)$ serves as a valuable tool for the investor to exclude information that lies beyond the time-horizon $\mathbb{B}$. This can be elucidated through two distinct cases. In the first case, when $\tau_n<\tau$ for every $n\in \mathbf{N}^+$ (for instance, $\tau$ is a predictable time, and $(\tau_n)$ announce $\tau$), the default time consistently occurs after the investor's newly-defined time-horizon $\llbracket{0,\tau_{n}}\rrbracket$. From the investor's standpoint, the default time and any information pertaining to stock prices beyond $\mathbb{B}$ are entirely immaterial.
Conversely, in the second case, if there exists an $m\in \mathbf{N}^+$ such that the condition $\tau_{m}<\tau$ does not hold, the investor is compelled to confront the impact of default risk. In this scenario, not all information beyond the time-horizon $\mathbb{B}$ can be disregarded. Nevertheless, by leveraging the sequence $(\tau_n)$, the investor can, with a high probability, circumvent the influence of default risk and stock-related information that extends beyond the time-horizon $\mathbb{B}$.
In reality, within the actual market, an investor can only revise the estimation of the terminal time a finite number of times, rather than an infinite number of times. Let $\alpha$ be a fixed probability; for instance, $\alpha=95\%$. Suppose $N$ is an integer such that $\mathbf{P}(\tau_N<\tau)\geq\alpha$. By updating the terminal time through the finite sequence $\{\tau_1,\tau_2,\cdots,\tau_N\}$, the investor is capable of excluding information regarding stock prices and the default time that fall beyond the time-horizon $\mathbb{B}$ with a sufficiently high probability. Consequently, from the investor's perspective, information concerning stock prices beyond $\mathbb{B}$ becomes almost irrelevant.

The sequences $(\mu_n)$, $(\sigma_n)$ and $(N^{(n)})$ can be employed, if necessary, to elucidate the effects of default risk on stock prices. To substantiate this assertion, we consider a scenario where $\tau_1<\tau$, yet there exists an $m\in \mathbf{N}^+$ such that the condition $\tau_m<\tau$ does not hold. At the initial date, the investor, by leveraging available stock information, determines the expected stock return $\mu_1$, stock volatility $\sigma_1$, and the driving process $N^{(1)}$ over the first time-horizon $\llbracket{0,\tau_{1}}\rrbracket$. However, at date $\tau_{m-1}$, there is a non-negligible probability that a default will occur over the $m$-th time-horizon $\rrbracket{\tau_{m-1},\tau_{m}}\rrbracket$. Consequently, the influence of partial default risk compels the investor to update the expected stock return $\mu_m$, stock volatility $\sigma_m$, and the driving process $N^{(m)}$ over the interval $\rrbracket{\tau_{m-1},\tau_{m}}\rrbracket$.
Empirical evidence (see, e.g., \cite{Campbell,Chava}) indicates that there exists either a positive or a negative relationship between default risk and expected stock returns. Under a positive relationship between default risk and expected stock return, the stock return $\mu_{m}$ should be higher than $\mu_{1}$. Conversely, under a negative relationship, $\mu_{m}$ should be lower than $\mu_{1}$. This phenomenon clearly demonstrates the variation within the sequence $(\mu_n)$. To formalize this, we can assume that the sequence $(\mu_n)$ satisfies the following condition:
\begin{equation}\label{mu_n}
\mu_n=f(p_n)\mu^*.
\end{equation}
Here, $p_n=\mathbf{P}(\tau_n<\tau)$ represents the probability that information beyond the time-horizon $\mathbb{B}$ is excluded. $\mu^*$ denotes the standard expected stock return when the effect of default risk is not taken into account. The function $f(x)$ $(x\in [0,1])$ is a deterministic measurable function that fulfills the condition $f(1)=1$. Moreover, when there is a positive relationship between default risk and expected stock return, $f(x)$ is a decreasing function of $x$. In contrast, when there is a negative relationship, $f(x)$ is an increasing function of $x$.

In such a financial market, the expected stock return $\mu$, stock volatility $\sigma$, and the driven process $M$ can be characterized by
\[
 \left\{
\begin{aligned}
\mu&=\left(\mu_1I_{\llbracket{0}\rrbracket}
+\sum\limits_{n=1}^{+\infty}\mu^{(n)}I_{\rrbracket{\tau_{n-1},\tau_{n}}\rrbracket}\right)\mathfrak{I}_\mathbb{B},
&\mu^{(n)}&=\mu_1I_{\llbracket{0}\rrbracket}
+\sum\limits_{k=1}^{n-1}\mu_kI_{\rrbracket{\tau_{k-1},\tau_{k}}\rrbracket}+\mu_nI_{\rrbracket{\tau_{n-1},+\infty}\llbracket};\\
\sigma&=\left(\sigma_1I_{\llbracket{0}\rrbracket}
+\sum\limits_{n=1}^{+\infty}\sigma^{(n)}I_{\rrbracket{\tau_{n-1},\tau_{n}}\rrbracket}\right)\mathfrak{I}_\mathbb{B},
&\sigma^{(n)}&=\sigma_1I_{\llbracket{0}\rrbracket}
+\sum\limits_{k=1}^{n-1}\sigma_kI_{\rrbracket{\tau_{k-1},\tau_{k}}\rrbracket}+\sigma_nI_{\rrbracket{\tau_{n-1},+\infty}\llbracket};\\
M&=\left(\sum\limits_{n=1}^{+\infty}M^{(n)}I_{\rrbracket{\tau_{n-1},\tau_{n}}\rrbracket}\right)\mathfrak{I}_\mathbb{B},
&M^{(n)}&=\sum\limits_{k=1}^{n-1}I_{\rrbracket{\tau_{k-1},\tau_{k}}\rrbracket}.N^{(k)}
+I_{\rrbracket{\tau_{n-1},+\infty}\llbracket}.N^{(n)};
\end{aligned}
\right.
\]
where $\mu\in \mathcal{P}^\mathbb{B}$, $\sigma\in \mathcal{P}^\mathbb{B}$ and $M\in (\mathcal{M}^c_{\mathrm{loc},0})^\mathbb{B}$.
Now, the stock price $S$ is modeled as a $\mathbb{B}$-inner semimartingale, expressed through the following equation:
\begin{equation}\label{price}
\left\{
\begin{aligned}
  S&=s_0\mathfrak{I}_\mathbb{B}+S_{-}{}_{\bullet}Z,\\
  Z&=\mu_{\bullet}A+\sigma_{\bullet}M,
\end{aligned}
  \right.
\end{equation}
where $s_0>0$ denotes the initial stock price, $Z\in \mathcal{S}^{i,\mathbb{B}}$ is the log-price, and $A=\widetilde{A}\mathfrak{I}_\mathbb{B}$ with $\widetilde{A}_t=t,\; t\in\mathbf{R}^+$.

Given that both $\mu$ and $\sigma$ are $\mathbb{B}$-locally bounded predictable processes, Corollary \ref{bound-HX} implies that the process $Z$ is well-defined. Subsequently, the following lemma establishes the unique existence of the stock price $S$ as specified in equation \eqref{price}.

\begin{lemma}\label{Z}
The stock price $S\in\mathcal{S}^{i,\mathbb{B}}$ uniquely exists, and it can be equivalently represented by the following expression:
  \begin{equation}\label{expZ}
     S=s_0\exp\left\{Z-\frac{1}{2}\langle Z^c\rangle\right\}.
 \end{equation}
Furthermore, $(\tau_n,S^{(n)})$ is an inner FCS for $S\in\mathcal{S}^{i,\mathbb{B}}$, and $(\tau_n,s_0+S^{(n)}\sigma^{(n)}.M^{(n)},S^{(n)}\mu^{(n)}.\widetilde{A})$ is a decomposed inner FCS for $S\in\mathcal{S}^{i,\mathbb{B}}$, where
\begin{equation}\label{Sn}
S^{(n)}=S^{(n)}_{-}=s_0\exp\left\{\mu^{(n)}.\widetilde{A}-\frac{1}{2}(\sigma^{(n)})^2.\langle M^{(n)}\rangle+\sigma^{(n)}.M^{(n)}\right\},\quad n\in \mathbf{N}^+.
\end{equation}
\end{lemma}

Given the stochastic set $\mathbb{B}$ of interval type and the $\mathbb{B}$-inner semimartingale $S$, we construct a financial market characterized by a sudden-stop horizon, denoted as $(S,\mathbb{F},\mathbb{B})$. However, the generalized conditions imposed on the sudden-stop horizon $\mathbb{B}$ and the stock price $S$ pose significant challenges in deriving the investor's optimal portfolio rule.
To address this issue, we introduce the following assumptions:
\begin{description}
  \item [(G1)] $(N^{(n)})$ is a sequence of independent standard Brownian Motions. Moreover, $\tau$ is a continuous random variable with a cumulative distribution function $F$, and $\tau$ is independent of $(N^{(n)})$.
  \item [(G2)] For $k\in \mathbf{N}^+$, the investor determines $(\tau_n)$ in the following form:
  \begin{equation*}
  \tau_n=a_n\wedge \tau,
  \end{equation*}
  where $(a_n)$ is an increasing sequence of positive constants satisfying the condition
  \[
  \left\{
  \begin{aligned}
  &a_n\uparrow a,\quad&&\text{if $\tau$ is bounded with $\tau\leq a$;}\\
  &a_n\uparrow +\infty,\quad&&\text{otherwise.}
  \end{aligned}
  \right.
  \]
  \item [(G3)]
  The expected stock return is positively correlated (resp. not correlated, resp. negatively correlated) with the default risk such that the investor determines the sequence $(\mu_n)$ as follows:
  \[
   \mu_n=[p_n+b(1-p_n)]\mu^*,\quad n\in \mathbf{N}^+,
  \]
  where $b>1$ (resp. $b=1$, resp. $0<b<1$) is a constant, and $p_n$ and $\mu^*>0$ are defined in the same manner as those in \eqref{mu_n}.
\end{description}

In line with a substantial body of existing literature, Brownian motions are employed in Assumption (G1) to model the uncertainty inherent in stock prices. As we will subsequently demonstrate in Lemma \ref{Sn2}, the independence property of the sequence $(N^{(n)})$ enables the representation of the uncertainty of stock prices as a standard Brownian motion. This representation is analogous to that in a market with an uncertain time-horizon (see, for instance, \cite{Blanchet}).
Moreover, Assumption (G1) characterizes the default time as a continuous random variable and implies the independence between default risk and stock uncertainty. This assumption is also prevalent in analogous research studies where continuous distributions are utilized to model uncertain time-horizons, and this uncertainty is assumed to be independent of the market. Examples of such studies include \cite{Blanchet,Liu,Pham,Richard}.
Assumption (G2) provides the investor with the most straightforward approach to obtaining the sequence $(\tau_n)$ of terminal times within the time-horizon $\mathbb{B}$. This method can be interpreted as representing various investment periods. Specifically, when $a_n=n$ for $n\in \mathbf{N}^+$, the first time-horizon $\llbracket{0,\tau_1}\rrbracket=\llbracket{0,1}\rrbracket\cap \llbracket{0,\tau}\rrbracket$ can be regarded as the first year of the investment, and the $n+1$-th time-horizon $\rrbracket{n\wedge\tau,n+1\wedge\tau}\rrbracket=\rrbracket{n,n+1}\rrbracket\cap \llbracket{0,\tau}\rrbracket$ can be viewed as the $n+1$-th year of the investment.
Regarding Assumption (G3), we assume that the default risk has a positive (respectively, negative) impact on the expected stock return. If the investor is compelled to consider the default time as the terminal time of the portfolio, a complete default risk scenario (i.e., $p_n=0$) results in a higher (resp. lower) expected stock return, denoted by $b\mu^*$
with $b>1$ (resp. $0<b<1$). For the sake of simplicity, we incorporate linearity in Assumption (G3) to model the effect of various default risk on the expected stock return.

Under the assumptions (G1)-(G3), the financial market $(S,\mathbb{F},\mathbb{B})$ demonstrates a strong affinity with conventional financial markets, as it has the potential to revert to an existing framework under certain specified conditions. Specifically, in the scenario where no default event occurs, i.e., when $\tau=+\infty$ and $\tau_n=+\infty$ for each $n\in \mathbf{N}^+$, the behavior of the stock price S reduces to that of a geometric Brownian motion. This simplified process is described by the following equation:
\begin{equation}\label{st}
\widetilde{S}_t=s_0\exp\left\{\left(\mu_1-\frac{\sigma_1^2}{2}\right)t+\sigma_1M_t^{(1)}\right\}, \quad t\in \mathbf{R}^+,
\end{equation}
or equivalently, by the stochastic differential equation (see, e.g., \cite{Karatzas-Shreve,Oksendal}):
\[
d\widetilde{S}_t=\widetilde{S}_t(\mu_1dt+\sigma_1dM_t^{(1)}),\quad \widetilde{S}_0=s_0,\quad t\in \mathbf{R}^+.
\]
The stock price $\widetilde{S}$ is extensively employed in financial research, particularly in well-established models such as the Black-Scholes model (see \cite{Black}). Conversely, the fundamental distinction between the financial market $(S,\mathbb{F},\mathbb{B})$ and the conventional financial market characterized by equation \eqref{st} resides in the nature of the information encapsulated within the stock price. Within the framework of $(S,\mathbb{F},\mathbb{B})$, the stock price $S$ is rigorously defined solely on the time-horizon $\mathbb{B}$, and it efficiently exclude extraneous information outside this pre-defined time-horizon. In contrast, the stock price represented by equation \eqref{st} provides information within the time-horizon $\llbracket{\tau,+\infty}\llbracket$. However, this information becomes superfluous if the stock default occurs at time $\tau$. This highlights a crucial difference in both the information content and the temporal scope between the two market frameworks.

By utilizing Assumption (G2), we can rewrite $\mu$, $\sigma$, and $M$ as follows:
\begin{equation}\label{mu-sigma}
 \left\{
\begin{aligned}
\mu&=\left(\mu_1I_{\llbracket{0}\rrbracket}
+\sum\limits_{n=1}^{+\infty}\widetilde{\mu}^{(n)}I_{\rrbracket{a_{n-1},a_{n}}\rrbracket}\right)\mathfrak{I}_\mathbb{B},
&\widetilde{\mu}^{(n)}&=\mu_1I_{\llbracket{0}\rrbracket}
+\sum\limits_{k=1}^{n-1}\mu_kI_{\rrbracket{a_{k-1},a_{k}}\rrbracket}+\mu_nI_{\rrbracket{a_{n-1},+\infty}\llbracket};\\
\sigma&=\left(\sigma_1I_{\llbracket{0}\rrbracket}
+\sum\limits_{n=1}^{+\infty}\widetilde{\sigma}^{(n)}I_{\rrbracket{a_{n-1},a_{n}}\rrbracket}\right)\mathfrak{I}_\mathbb{B},
&\widetilde{\sigma}^{(n)}&=\sigma_1I_{\llbracket{0}\rrbracket}
+\sum\limits_{k=1}^{n-1}\sigma_kI_{\rrbracket{a_{k-1},a_{k}}\rrbracket}+\sigma_nI_{\rrbracket{a_{n-1},+\infty}\llbracket};\\
M&=\left(\sum\limits_{n=1}^{+\infty}\widetilde{M}^{(n)}I_{\rrbracket{a_{n-1},a_{n}}\rrbracket}\right)\mathfrak{I}_\mathbb{B},
&\widetilde{M}^{(n)}&=\sum\limits_{k=1}^{n-1}I_{\rrbracket{a_{k-1},a_{k}}\rrbracket}.N^{(k)}
+I_{\rrbracket{a_{n-1},+\infty}\llbracket}.N^{(n)};
\end{aligned}
\right.
\end{equation}
where $(\tau_n,\widetilde{\mu}^{(n)})$, $(\tau_n,\widetilde{\sigma}^{(n)})$, and $(\tau_n,\widetilde{M}^{(n)})$ form FCSs for $\mu\in \mathcal{P}^\mathbb{B}$, $\sigma\in \mathcal{P}^\mathbb{B}$, and $M\in (\mathcal{M}^c_{\mathrm{loc},0})^\mathbb{B}$, respectively. Furthermore, we can update the results in Lemma \ref{Z} as follows:

\begin{lemma}\label{Sn2}
Suppose that Assumptions (G1)-(G3) hold true, and
for each $n\in \mathbf{N}^+$, define the process $\widetilde{S}^{(n)}$ by the following expression:
\begin{equation}\label{SSn}
\widetilde{S}^{(n)}=s_0\exp\left\{\left(\widetilde{\mu}^{(n)}-\frac{(\widetilde{\sigma}^{(n)})^2}{2}\right).\widetilde{A}+\widetilde{\sigma}^{(n)}. \widetilde{M}^{(n)}\right\}.
\end{equation}
\begin{itemize}
  \item [$(1)$] For each $n\in \mathbf{N}^+$, $\widetilde{M}^{(n)}$ is a standard Brownian motion.
  \item [$(2)$] $(\tau_n,\widetilde{S}^{(n)})$ forms
an inner FCS for $S\in\mathcal{S}^{i,\mathbb{B}}$, and $(\tau_n,s_0+\widetilde{S}^{(n)}\widetilde{\sigma}^{(n)}.\widetilde{M}^{(n)},\widetilde{S}^{(n)}\widetilde{\mu}^{(n)}
.\widetilde{A})$ serves as a decomposed inner FCS for $S\in\mathcal{S}^{i,\mathbb{B}}$.
\end{itemize}

\end{lemma}

Now, we can investigate the financial market $(S,\mathbb{F},\mathbb{B})$ and solve the investor's portfolio problem.

\begin{proposition}\label{pro}
Suppose that Assumptions (G1)-(G3) hold true, and let $\varphi$ represent the logarithmic utility function.
\begin{itemize}
  \item [$(1)$] The financial market $(S,\mathbb{F},\mathbb{B})$ satisfies NA.
  \item [$(2)$] The optimal strategy $\pi$ (w.r.t. $(\tau_n,s_0+\widetilde{S}^{(n)}\widetilde{\sigma}^{(n)}.\widetilde{M}^{(n)},\widetilde{S}^{(n)}\widetilde{\mu}^{(n)}
.\widetilde{A})$) for the problem outlined in \eqref{optimal-problem} is expressed as follows:
  \begin{equation}\label{optimal-solution}
  \pi=\frac{wX^*}{S},
  \end{equation}
  where
  \[
  w=\left(\frac{(1+(b-1)F(a_1))\mu^*}{\sigma_1^2}I_{\llbracket{0}\rrbracket}
  +\sum\limits_{n=1}^{+\infty}\frac{(1+(b-1)F(a_n))\mu^*}{\sigma_n^2}I_{\rrbracket{a_{n-1},a_{n}}\rrbracket}\right)\mathfrak{I}_\mathbb{B}
  \]
  is the optimal proportion of wealth invested in the stock, and
  \[
  X^*=x_0\exp\left(\left(\frac{\mu^2}{2\sigma^2}\right){}_{\bullet}A+\frac{\mu}{\sigma}{}_{\bullet}M\right).
\]

   is the associated optimal wealth.
\end{itemize}
\end{proposition}

The optimal strategy $\pi$ specified by \eqref{optimal-solution} excludes any information pertaining to the stock price subsequent to the default event and does not convey any details regarding portfolio strategies beyond the designated time-horizon $\mathbb{B}$. Consequently, we have ascertained that the investor's optimal portfolio rules within a market characterized by a sudden-stop horizon are independent of the stock dynamics that occur outside the prescribed time-horizon. This finding substantiates that, from the investor's standpoint, information concerning stock prices beyond $\mathbb{B}$ is immaterial. This represents a fundamental divergence between investing in a market with a sudden-stop horizon and investing in a conventional market with an uncertain time-horizon.
For example, Blanchet-Scalliet et al. \cite{Blanchet} investigated investments within an uncertain time-horizon $\llbracket{0,\tau\wedge T}\rrbracket$ (where $T$ denotes the (finite) time span of the economy, and $\tau>0$ is a random variable representing the date of the agent's death). They defined the dynamics of stock prices over the entire time span $\llbracket{0,T}\rrbracket$ to derive optimal portfolio rules. However, it is devoid of practical significance for the agent to determine optimal portfolio rules for the time interval $\llbracket{\tau\wedge T,T}\rrbracket$, which occurs after the agent's death. Therefore, our theory presents an alternative approach that focuses exclusively on optimal portfolio rules within the time-horizon.

We further propose a straightforward method for incorporating the impact of default risk into the investor's optimal portfolio rules within a market characterized by a sudden-stop horizon. In the time-horizon where default is not anticipated to occur (i.e., $F(a_n)=0$ in \eqref{optimal-solution}), the effect of default risk can be disregarded. Consequently, the market can be regarded as one with a fixed time-horizon.
Conversely, in the time-horizon where $b\neq 1$ and default is more likely to occur, as indicated by a higher value of $F(a_n)$ in \eqref{optimal-solution}, default risk exerts a more pronounced influence on the optimal portfolio rule. This observation sets our findings apart from those in markets with fixed time-horizons (e.g., \cite{Merton1969,Merton1971}), and markets with uncertain time-horizons (e.g., Propositions 3 and 4 of \cite{Blanchet}, and Subsection 3.6.2 of \cite{Pham}) where the uncertain time-horizon is independent of the driving process of stock uncertainty, and the optimal strategies are not related to the distribution of the exit time. The key difference stems from the fact that, in our market with the sudden-stop horizon $\mathbb{B}$, we estimate the terminal time using the sequence $(\tau_n)$, rather than directly considering the exit time as the terminal time.

\section*{Appendix: proofs}
\setcounter{section}{0}
\setcounter{equation}{0}
\setcounter{lemma}{0}
\renewcommand{\theequation}{A.\arabic{equation}}
\renewcommand{\thelemma}{A.\arabic{lemma}}

%%%%%%%%%%%%%%%%%%%%%%%%%%%%%%%%%%%%%%%%%%%%%%%%%%%%%%%%%%%%%%%%%%%%%%%%%%%%%%%%%%%%%%%%%%%%%%%%%%%%%%%%%%%%%%%%
\noindent{\bf Proof of Lemma \ref{th8.17}.} The detailed proof is provided in Theorem 8.17 and Theorem 8.18 of \cite{He}.
$\hfill\blacksquare$

%%%%%%%%%%%%%%%%%%%%%%%%%%%%%%%%%%%%%%%%%%%%%%%%%%%%%%%%%%%%%%%%%%%%%%%%%%%%%%%%%%%%%%%%%%%%%%%%%%%%%%%
\par\vspace{0.3cm}
%%%%%%%%%%%%%%%%%%%%%%%%%%%%%%%%%%%%%%%%%%%%%%%%%%%%%%%%%%%%%%%%%%%%%%%%%%%%%%%%%%%%%%%%%%%%%%%%%%%%%%%
\noindent {\bf Proof of Theorem \ref{process}.}
We just prove the case of FCS, and the case of CS can be proven similarly.

(1) The necessity is trivial, and we need to prove the sufficiency. Suppose $XI_{\mathbb{B}\llbracket{0,S_n}\rrbracket}=YI_{\mathbb{B}\llbracket{0,S_n}\rrbracket}$ holds for each $n\in \mathbf{N}^+$. It is easy to obtain $X_0I_{\llbracket{0}\rrbracket}=Y_0I_{\llbracket{0}\rrbracket}$ and
\[
 XI_{\mathbb{B}\rrbracket{S_{n-1},S_n}\rrbracket}=YI_{\mathbb{B}\rrbracket{S_{n-1},S_n}\rrbracket},
\quad S_0=0,\quad n\in \mathbf{N}^+.
\]
Then using the fact
\[
\mathbb{B}\subseteq\bigcup\limits_{n=1}^{{+\infty}}\llbracket{0,S_n}\rrbracket = \llbracket{0}\rrbracket\cup \left(\bigcup\limits_{n=1}^{{+\infty}}\rrbracket{S_{n-1},S_n}\rrbracket\right), \]
we deduce
\[
XI_\mathbb{B}=X_0I_{\llbracket{0}\rrbracket}+\sum\limits_{n=1}^{{+\infty}}XI_{\mathbb{B}\rrbracket{S_{n-1},S_n}
      \rrbracket}
 =Y_0I_{\llbracket{0}\rrbracket}+\sum\limits_{n=1}^{{+\infty}}YI_{\mathbb{B}\rrbracket{S_{n-1},S_n}
      \rrbracket}
 =YI_\mathbb{B},
\]
which implies that $X=Y$.

(2) The definition of $\mathbb{B}$-process yields $XI_{\mathbb{B}\llbracket{0,T_k}\rrbracket}=X^{(k)}I_{\mathbb{B}\llbracket{0,T_k}\rrbracket}$ directly. And using $T_k\leq T_l$, (\ref{xkl}) is finally obtained by
      \[
      XI_{\mathbb{B}\llbracket{0,T_k}\rrbracket}=
      (XI_{\mathbb{B}\llbracket{0,T_l}\rrbracket})I_{\llbracket{0,T_k}\rrbracket}
      =(X^{(l)}I_{\mathbb{B}\llbracket{0,T_l}\rrbracket})I_{\llbracket{0,T_k}\rrbracket}
      =X^{(l)}I_{\mathbb{B}\llbracket{0,T_k}\rrbracket}.
      \]

(3) It is easy to see
      \[
      (XI_\mathbb{B})^{\tau_n}=((XI_\mathbb{B})^{T_n})^{S_n}=((X^{(n)}I_\mathbb{B})^{T_n})^{S_n}=(X^{(n)}I_\mathbb{B})^{\tau_n},\quad n\in \mathbf{N}^+.
      \]
From the facts $\tau_n\uparrow T$ and $\bigcup\limits_{n=1}^{{+\infty}}\llbracket{0,\tau_n}\rrbracket\supseteq \mathbb{B}$, the sequence $(\tau_n,X^{(n)})$ is a CS for $X$. Since $X^{(n)}\in \mathcal{D}$ for each $n\in \mathbf{N}^+$,  $(\tau_n,X^{(n)})$ is indeed an FCS for $X\in \mathcal{D}^\mathbb{B}$.

(4) Suppose that $(\widetilde{T}_n,Y^{(n)})$ is an FCS for $Y\in \mathcal{D}^\mathbb{B}$. Put $\tau_n=\widetilde{T}_n\wedge T_n$ for each $n\in \mathbf{N}^+$. From part (3), $(\tau_n,X^{(n)})$ and $(\tau_n,Y^{(n)})$ are FCSs for $X\in \mathcal{D}^\mathbb{B}$ and $Y\in \mathcal{D}^\mathbb{B}$, respectively. Then for each $n\in \mathbf{N}^+$,
\begin{align*}
(aX+bY)I_{\mathbb{B}\llbracket{0,\tau_n}\rrbracket}
&=a(XI_{\mathbb{B}\llbracket{0,\tau_n}\rrbracket})+b(YI_{\mathbb{B}\llbracket{0,\tau_n}\rrbracket})\\
&=a(X^{(n)}I_{\mathbb{B}\llbracket{0,\tau_n}\rrbracket})+b(Y^{(n)}I_{\mathbb{B}\llbracket{0,\tau_n}\rrbracket})\\
&=(aX^{(n)}+bY^{(n)})I_{\mathbb{B}\llbracket{0,\tau_n}\rrbracket},
\end{align*}
which shows $(\tau_n,aX^{(n)}+bY^{(n)})$ is a CS for $aX+bY$. The linearity of $\mathcal{D}$ implies $aX^{(n)}+bY^{(n)}\in\mathcal{D}$ for each $n\in \mathbf{N}^+$. Hence, we conclude that $aX+bY\in \mathcal{D}^\mathbb{B}$.

(5) Using part (2), we have for each $l\in \mathbf{N}^+$,
\begin{align*}
XI_{\mathbb{B}\llbracket{0,T_l}\rrbracket}
=& X^{(l)}I_{\mathbb{B}\llbracket{0,T_l}\rrbracket}\\
=& \left(X^{(l)}_0I_{\llbracket{0}\rrbracket}+\sum\limits_{n=1}^{l}X^{(l)}I_{\mathbb{B}\rrbracket{T_{n-1},T_n}
      \rrbracket}\right)I_{\mathbb{B}\llbracket{0,T_l}\rrbracket} \\
=& \left(X_0I_{\llbracket{0}\rrbracket}+\sum\limits_{n=1}^{l}X^{(n)}I_{\mathbb{B}\rrbracket{T_{n-1},T_n}
      \rrbracket}\right)I_{\mathbb{B}\llbracket{0,T_l}\rrbracket}\\
=&\left\{\left(X_0I_{\llbracket{0}\rrbracket}+\sum\limits_{n=1}^{{+\infty}}X^{(n)}I_{\rrbracket{T_{n-1},T_n}
      \rrbracket}\right)\mathfrak{I}_\mathbb{B}\right\}I_{\mathbb{B}\llbracket{0,T_l}\rrbracket},
\end{align*}
which, by part (1), yields (\ref{x-expression}).

As for the independence of the choice of FCS, it suffices to prove that $XI_{\mathbb{B}\llbracket{0,\tau_l}\rrbracket}=\widetilde{X}I_{\mathbb{B}\llbracket{0,\tau_l}\rrbracket}$ holds for each $l\in \mathbf{N}^+$ and $\tau_l=T_l\wedge S_l$. From part (3), $(\tau_n,X^{(n)})$ and $(\tau_n,\widetilde{X}^{(n)})$ are both FCSs for $X\in \mathcal{D}^\mathbb{B}$. Then using part (2) again, we deduce for each $l\in \mathbf{N}^+$,
\begin{align*}
\widetilde{X}I_{\mathbb{B}\llbracket{0,\tau_l}\rrbracket}
&=\left(X_0I_{\llbracket{0}\rrbracket}+\sum\limits_{n=1}^{+\infty}\widetilde{X}^{(n)}I_{\mathbb{B}\rrbracket{S_{n-1},S_n}
      \rrbracket}\right)I_{\mathbb{B}\llbracket{0,S_l}\rrbracket}I_{\llbracket{0,T_l}\rrbracket}\\
&=\left(X_0I_{\llbracket{0}\rrbracket}+\sum\limits_{n=1}^{l}\widetilde{X}^{(n)}I_{\mathbb{B}\rrbracket{S_{n-1},S_n}
      \rrbracket}\right)I_{\mathbb{B}\llbracket{0,S_l}\rrbracket}I_{\llbracket{0,T_l}\rrbracket}\\
&=\left(X_0I_{\llbracket{0}\rrbracket}+\sum\limits_{n=1}^{l}\widetilde{X}^{(l)}I_{\mathbb{B}\rrbracket{S_{n-1},S_n}
      \rrbracket}\right)I_{\mathbb{B}\llbracket{0,S_l}\rrbracket}I_{\llbracket{0,T_l}\rrbracket}\\
&=\widetilde{X}^{(l)}I_{\mathbb{B}\llbracket{0,S_l}\rrbracket}I_{\llbracket{0,T_l}\rrbracket}\\
&=XI_{\mathbb{B}\llbracket{0,\tau_l}\rrbracket},
\end{align*}
which completes the proof.
$\hfill\blacksquare$

%%%%%%%%%%%%%%%%%%%%%%%%%%%%%%%%%%%%%%%%%%%%%%%%%%%%%%%%%%%%%%%%%%%%%%%%%%%%%%%%%%%%%%%%%%%%%%%%%%%%%%%
\par\vspace{0.3cm}
%%%%%%%%%%%%%%%%%%%%%%%%%%%%%%%%%%%%%%%%%%%%%%%%%%%%%%%%%%%%%%%%%%%%%%%%%%%%%%%%%%%%%%%%%%%%%%%%%%%%%%%
\noindent {\bf Proof of Corollary \ref{process-FS}.}
It is not difficult to obtain the facts
\[
\mathbb{C}=\left(\bigcup\limits_{n=1}^{{+\infty}}\llbracket{0,\tau_n}\rrbracket\right)\bigcap
\left(\bigcup\limits_{m=1}^{{+\infty}}\llbracket{0,T_m}\rrbracket\right)
=\bigcup\limits_{n=1}^{{+\infty}}\llbracket{0,S_n}\rrbracket
\]
and
\[
X^{S_n}I_{\llbracket{0,S_n}\rrbracket}=XI_{\llbracket{0,S_n}\rrbracket}
=X^{(n)}I_{\llbracket{0,S_n}\rrbracket}
=(X^{(n)})^{S_n}I_{\llbracket{0,S_n}\rrbracket}, \quad n\in \mathbf{N}^+.
\]
Then using these two established facts, all statements are directly deduced by Theorem\ref{process}.
$\hfill\blacksquare$

%%%%%%%%%%%%%%%%%%%%%%%%%%%%%%%%%%%%%%%%%%%%%%%%%%%%%%%%%%%%%%%%%%%%%%%%%%%%%%%%%%%%%%%%%%%%%%%%%%%%%%%
\par\vspace{0.3cm}
%%%%%%%%%%%%%%%%%%%%%%%%%%%%%%%%%%%%%%%%%%%%%%%%%%%%%%%%%%%%%%%%%%%%%%%%%%%%%%%%%%%%%%%%%%%%%%%%%%%%%%%
\noindent {\bf Proof of Theorem \ref{restriction}.}
(1) We just show the case that $H$ is a $\mathbb{B}$-predictable process, and the case that $H$ is a $\mathbb{B}$-locally bounded predictable process can be shown similarly.
The sufficiency is straightforward, and it remain to show the necessity.

Suppose $H\in \mathcal{P}^\mathbb{B}$ with the FCS $(T_n,H^{(n)})$.  Then \eqref{x-expression} yields
\[
      H=\left(H_0I_{\llbracket{0}\rrbracket}+\sum\limits_{n=1}^{{+\infty}}H^{(n)}I_{\rrbracket{T_{n-1},T_n}
      \rrbracket}\right)\mathfrak{I}_\mathbb{B},\quad T_0=0.
\]
Define
\[
\widetilde{H}=H_0I_{\llbracket{0}\rrbracket}+\sum\limits_{n=1}^{{+\infty}}H^{(n)}I_{\rrbracket{T_{n-1},T_n}
      \rrbracket}.
\]
Since, for each $n\in \mathbf{N}^+$, $H^{(n)}$ and $I_{\rrbracket{T_{n-1},T_n}\rrbracket}$ are predictable processes (see, e.g., Theorem 3.16 in \cite{He}), $\widetilde{H}$ is a predictable process satisfying \eqref{couple}.

(2) Note that, by \eqref{x-expression}, $XI_\mathbb{B}$ is an optional process.

{\it Necessity.} Suppose that $X$ is a $\mathbb{B}$-thin process with the FCS $(T_n,X^{(n)})$. For each $n\in \mathbf{N}^+$, $[X^{(n)}\neq 0]$ is a thin set. Then
\[
[XI_\mathbb{B}\neq 0]=\bigcup_{n\in\mathbf{N}^+}[XI_{\mathbb{B}\llbracket{0,T_n}\rrbracket}\neq 0]=\bigcup_{n\in\mathbf{N}^+}[X^{(n)}I_{\mathbb{B}\llbracket{0,T_n}\rrbracket}\neq 0]\subseteq\bigcup_{n\in\mathbf{N}^+}[X^{(n)}\neq 0],
\]
which shows that $[XI_\mathbb{B}\neq 0]$ is an optional set contained in a thin set. Hence, by Theorem 3.19 in \cite{He}, $[XI_\mathbb{B}\neq 0]$ is a thin set, which implies that $XI_\mathbb{B}$ is a thin process.

{\it Sufficiency.} Suppose that $XI_\mathbb{B}$ is a thin process. Let $\mathbb{B}$ be given by \eqref{B}. Then $X$ is a $\mathbb{B}$-thin process, because $(T_n=T,X^{(n)}=XI_\mathbb{B})$ is an FCS for $X$ (a $\mathbb{B}$-thin process).

(3) The detailed proof can be found in Theorem 8.22 of \cite{He}.

(4) For each $k\in \mathbf{N}^+$, it is straightforward to verify that
\begin{align*}
XI_{\mathbb{B}\llbracket{0,T_k}\rrbracket}
&=\left(X^{(1)}_0I_{\llbracket{0}\rrbracket}+\sum\limits_{n=1}^{k}X^{(n)}I_{\mathbb{B}\rrbracket{T_{n-1},T_n}
      \rrbracket}\right)I_{\mathbb{B}\llbracket{0,T_k}\rrbracket}\\
&=\left(X^{(k)}_0I_{\llbracket{0}\rrbracket}+\sum\limits_{n=1}^{k}X^{(k)}I_{\mathbb{B}\rrbracket{T_{n-1},T_n}
      \rrbracket}\right)I_{\mathbb{B}\llbracket{0,T_k}\rrbracket}\\
&=X^{(k)}I_{\mathbb{B}\llbracket{0,T_k}\rrbracket},
\end{align*}
which, according to the fact $X^{(k)}\in \mathcal{D}$, shows that $X\in \mathcal{D}^\mathbb{B}$ with the FCS $(T_n,X^{(n)})$.
$\hfill\blacksquare$

%%%%%%%%%%%%%%%%%%%%%%%%%%%%%%%%%%%%%%%%%%%%%%%%%%%%%%%%%%%%%%%%%%%%%%%%%%%%%%%%%%%%%%%%%%%%%%%%%%%%%%%
\par\vspace{0.3cm}
%%%%%%%%%%%%%%%%%%%%%%%%%%%%%%%%%%%%%%%%%%%%%%%%%%%%%%%%%%%%%%%%%%%%%%%%%%%%%%%%%%%%%%%%%%%%%%%%%%%%%%%
\noindent {\bf Proof of Theorem \ref{fcs}.}
The proof of $X^S\in \mathcal{D}$ can be found in Theorem 8.20 of \cite{He}, and it suffices to prove that $(T_n,(X^{(n)})^S)$ is an FCS for $X^S\mathfrak{I}_\mathbb{B}\in \mathcal{D}^\mathbb{B}$. For each $n\in \mathbf{N}^+$, observing that $S\wedge T_n$ is a $\mathbb{B}$-inner stopping time, we have
\[
X^{S\wedge T_n}I_{\llbracket{0,S\wedge T_n}\rrbracket}
=XI_{\llbracket{0,S\wedge T_n}\rrbracket}
=(XI_{\mathbb{B}\llbracket{0,T_n}\rrbracket})I_{\llbracket{0,S\wedge T_n}\rrbracket}
=(X^{(n)}I_{\mathbb{B}\llbracket{0,T_n}\rrbracket})I_{\llbracket{0,S\wedge T_n}\rrbracket}
=X^{(n)}I_{\llbracket{0,S\wedge T_n}\rrbracket},
\]
which implies that $X^{S\wedge T_n}=(X^{(n)})^{S\wedge T_n}$. Then the relations
\[
(X^S\mathfrak{I}_\mathbb{B})I_{\mathbb{B}\llbracket{0,T_n}\rrbracket}
=X^{S\wedge T_n}I_{\mathbb{B}\llbracket{0,T_n}\rrbracket}
=(X^{(n)})^{S\wedge T_n}I_{\mathbb{B}\llbracket{0,T_n}\rrbracket}
=(X^{(n)})^{S}I_{\mathbb{B}\llbracket{0,T_n}\rrbracket}, \quad n\in \mathbf{N}^+
\]
show that $(T_n,(X^{(n)})^S)$ is a CS for $X^S\mathfrak{I}_\mathbb{B}$. Since $\mathcal{D}$ is stable under stopping such that $(X^{(n)})^S\in \mathcal{D}$ for each $n\in \mathbf{N}^+$, $(T_n,(X^{(n)})^S)$ is indeed an FCS for $X^S\mathfrak{I}_\mathbb{B}\in \mathcal{D}^\mathbb{B}$.
$\hfill\blacksquare$

%%%%%%%%%%%%%%%%%%%%%%%%%%%%%%%%%%%%%%%%%%%%%%%%%%%%%%%%%%%%%%%%%%%%%%%%%%%%%%%%%%%%%%%%%%%%%%%%%%%%%%%
\par\vspace{0.3cm}
%%%%%%%%%%%%%%%%%%%%%%%%%%%%%%%%%%%%%%%%%%%%%%%%%%%%%%%%%%%%%%%%%%%%%%%%%%%%%%%%%%%%%%%%%%%%%%%%%%%%%%%
\noindent {\bf Proof of Corollary \ref{fcs-p}.}
From Theorem \ref{fcs}, the proof is trivial.
$\hfill\blacksquare$

%%%%%%%%%%%%%%%%%%%%%%%%%%%%%%%%%%%%%%%%%%%%%%%%%%%%%%%%%%%%%%%%%%%%%%%%%%%%%%%%%%%%%%%%%%%%%%%%%%%%%%%
\par\vspace{0.3cm}
%%%%%%%%%%%%%%%%%%%%%%%%%%%%%%%%%%%%%%%%%%%%%%%%%%%%%%%%%%%%%%%%%%%%%%%%%%%%%%%%%%%%%%%%%%%%%%%%%%%%%%%
\noindent {\bf Proof of Corollary \ref{cD=DB}.}
The inclusion $\mathcal{D}\subseteq\mathcal{D}^{\llbracket{0,+\infty}\llbracket}$ is straightforward, and it suffices to prove $\mathcal{D}^{\llbracket{0,+\infty}\llbracket}\subseteq\mathcal{D}$. Let $X\in \mathcal{D}^{\llbracket{0,+\infty}\llbracket}$. In fact, putting $\tau=+\infty$, then $\tau$ is a $\llbracket{0,+\infty}\llbracket$-inner stopping time. Theorem \ref{fcs} yields $X=X^\tau\in\mathcal{D}$, which finishes the proof.
$\hfill\blacksquare$

%%%%%%%%%%%%%%%%%%%%%%%%%%%%%%%%%%%%%%%%%%%%%%%%%%%%%%%%%%%%%%%%%%%%%%%%%%%%%%%%%%%%%%%%%%%%%%%%%%%%%%%
\par\vspace{0.3cm}
%%%%%%%%%%%%%%%%%%%%%%%%%%%%%%%%%%%%%%%%%%%%%%%%%%%%%%%%%%%%%%%%%%%%%%%%%%%%%%%%%%%%%%%%%%%%%%%%%%%%%%%
\noindent {\bf Proof of Theorem \ref{thin}.}
$(1)$ From Definition 7.39 in \cite{He}, $\Sigma X^{(n)}\in \mathcal{V}^\mathbb{B}$ for each $n\in \mathbf{N}^+$, and then the proof is trivial.

$(2)$ It is evident that $XI_{\llbracket{0,\tau}\rrbracket}=X^{\tau}I_{\llbracket{0,\tau}\rrbracket}$ is an optional process. According to Theorem \ref{restriction}(2), $[XI_\mathbb{B}\neq 0]$ is a thin set. Consequently, by invoking Theorem 3.19 in \cite{He} and the inclusion
\[
[XI_{\llbracket{0,\tau}\rrbracket}\neq 0]\subseteq [XI_\mathbb{B}\neq 0],
\]
it can be easily deduced that $XI_{\llbracket{0,\tau}\rrbracket}$ is a thin process.
Suppose that $\Sigma X$ is well-defined. By the definition of the summation process, it is straightforward to observe that
\begin{align*}
(\Sigma X)^\tau
&=\left(\left(\sum_{s\leq \cdot}(XI_\mathbb{B})_{s}\right)\mathfrak{I}_\mathbb{B}\right)^\tau
=\left(\sum_{s\leq \cdot}(XI_\mathbb{B})_{s}\right)^\tau\\
&=\sum_{s\leq \cdot}(XI_{\llbracket{0,\tau}\rrbracket})_s
=\Sigma (XI_{\llbracket{0,\tau}\rrbracket}),
\end{align*}
which ultimately leads to the derivation of \eqref{sigmaX}.
$\hfill\blacksquare$

%%%%%%%%%%%%%%%%%%%%%%%%%%%%%%%%%%%%%%%%%%%%%%%%%%%%%%%%%%%%%%%%%%%%%%%%%%%%%%%%%%%%%%%%%%%%%%%%%%%%%%%
\par\vspace{0.3cm}
%%%%%%%%%%%%%%%%%%%%%%%%%%%%%%%%%%%%%%%%%%%%%%%%%%%%%%%%%%%%%%%%%%%%%%%%%%%%%%%%%%%%%%%%%%%%%%%%%%%%%%%
\noindent {\bf Proof of Theorem \ref{X-left}.}
Let $(\omega,t)\in \mathbb{B}$ with $t>0$. It follows from the condition $\bigcup\limits_{n=1}^{{+\infty}}\llbracket{0,T_n}\rrbracket\supseteq \mathbb{B}$ that there exists an integer $m\in \mathbf{N}^+$ such that $(\omega,t)\in \mathbb{B}\llbracket{0,T_m}\rrbracket$. Leveraging the CS $(T_n,X^{(n)})$ for $X$, we deduce that $X(\omega,s)=X^{(m)}(\omega,s)$ for all $s\in[0,t]$. Consequently, given the existence of $X^{(m)}(\omega,t-)$, we infer the existence of $X(\omega,t-)$ and establish the relation $X(\omega,t-)=X^{(m)}(\omega,t-)$. By invoking the arbitrariness of $(\omega,t)\in \mathbb{B}$ with $t>0$, we generalize that for all such $(\omega,t)$, the left-hand limits $X(\omega,t-)$ exist. This, in turn, validates the existence of $X_{-}$. Similarly, we can prove that for each $n\in \mathbf{N}^+$,
\[
X(\omega,t-)=X^{(n)}(\omega,t-), \quad (\omega,t)\in \mathbb{B}\llbracket{0,T_n}\rrbracket,
\]
which, by Definition \ref{X+-}, implies that
\[
X_{-}I_{\mathbb{B}\llbracket{0,T_n}\rrbracket}=(X^{(n)})_{-}I_{\mathbb{B}\llbracket{0,T_n}\rrbracket}, \quad n\in \mathbf{N}^+.
\]
Therefore, $(T_n,(X^{(n)})_{-})$ is a CS for $X_{-}$.
$\hfill\blacksquare$

%%%%%%%%%%%%%%%%%%%%%%%%%%%%%%%%%%%%%%%%%%%%%%%%%%%%%%%%%%%%%%%%%%%%%%%%%%%%%%%%%%%%%%%%%%%%%%%%%%%%%%%
\par\vspace{0.3cm}
%%%%%%%%%%%%%%%%%%%%%%%%%%%%%%%%%%%%%%%%%%%%%%%%%%%%%%%%%%%%%%%%%%%%%%%%%%%%%%%%%%%%%%%%%%%%%%%%%%%%%%%
\noindent {\bf Proof of Corollary \ref{deltaX}.}
$(1)$ From the FCS $(T_n,X^{(n)})$ for $X\in \mathcal{R}^\mathbb{B}$, it is evident that $X^{(n)}\in \mathcal{R}$ for each $n\in \mathbf{N}^+$. Then $(T_n,X^{(n)})$ is a CS for $X$ such that $(X^{(n)})_{-}$ exists for each $n\in \mathbf{N}^+$. Hence, by invoking Theorem \ref{X-left}, the assertion is proven.

$(2)$ For each $n\in \mathbf{N}^+$, $X^{(n)}$ is an adapted c\`{a}dl\`{a}g process, and then Theorem 7.7 in \cite{He} shows that $(X^{(n)})_{-}$ is  a locally bounded predictable process. Since Theorem \ref{X-left} also shows that $(T_n,(X^{(n)})_{-})$ is a CS for $X_{-}$, the assertion holds true.

$(3)$ The assertions are immediate consequences of Corollary \ref{fcs-p} and parts $(1)$ and $(2)$.
$\hfill\blacksquare$

%%%%%%%%%%%%%%%%%%%%%%%%%%%%%%%%%%%%%%%%%%%%%%%%%%%%%%%%%%%%%%%%%%%%%%%%%%%%%%%%%%%%%%%%%%%%%%%%%%%%%%%
\par\vspace{0.3cm}
%%%%%%%%%%%%%%%%%%%%%%%%%%%%%%%%%%%%%%%%%%%%%%%%%%%%%%%%%%%%%%%%%%%%%%%%%%%%%%%%%%%%%%%%%%%%%%%%%%%%%%%
\noindent {\bf Proof of Theorem \ref{delta}.}
$(1)$ The statements are direct results of Corollaries \ref{fcs-p} and \ref{deltaX}.

$(2)$ From Theorem \ref{process}(3), we can assume that $(T_n,X^{(n)})$ and $(T_n,Y^{(n)})$ are FCSs for $X\in \mathcal{R}^\mathbb{B}$ and $Y\in \mathcal{R}^\mathbb{B}$, respectively. Then $(T_n,aX^{(n)})$ is an FCS for $aX\in \mathcal{R}^\mathbb{B}$, and $(T_n,X^{(n)}+Y^{(n)})$ is an FCS for $X+Y\in \mathcal{R}^\mathbb{B}$. It follows from part $(1)$ that $(T_n,\Delta X^{(n)})$,  $(T_n,\Delta Y^{(n)})$, $(T_n,\Delta(aX^{(n)}))$ and $(T_n,\Delta(X^{(n)}+Y^{(n)}))$ are CSs for $\Delta X$, $\Delta Y$, $\Delta(aX)$ and $\Delta(X+Y)$, respectively. Consequently, for each $n\in \mathbf{N}^+$, we deduce
\begin{align*}
\Delta(aX)I_{\mathbb{B}\llbracket{0,T_n}\rrbracket}
=\Delta(aX^{(n)})I_{\mathbb{B}\llbracket{0,T_n}\rrbracket}
=a\Delta(X^{(n)})I_{\mathbb{B}\llbracket{0,T_n}\rrbracket}
=(a\Delta X)I_{\mathbb{B}\llbracket{0,T_n}\rrbracket}
\end{align*}
and
\begin{align*}
\Delta(X+Y)I_{\mathbb{B}\llbracket{0,T_n}\rrbracket}
=\Delta(X^{(n)}+Y^{(n)})I_{\mathbb{B}\llbracket{0,T_n}\rrbracket}
=(\Delta X^{(n)}+\Delta Y^{(n)})I_{\mathbb{B}\llbracket{0,T_n}\rrbracket}
=(\Delta X+\Delta Y)I_{\mathbb{B}\llbracket{0,T_n}\rrbracket},
\end{align*}
which, by Theorem \ref{process}$(1)$, implies \eqref{cadlag-linear}.

$(3)$ Suppose that $(T_n,X^{(n)})$ is an FCS for $X\in \mathcal{R}^\mathbb{B}$. It is easy to see that $X\mathfrak{I}_{\widetilde{\mathbb{B}}}\in\mathcal{R}^{\widetilde{\mathbb{B}}}$ with the FCS $(S_n,X^{(n)})$, where $S_n=S\wedge T_n$ for each $n\in \mathbf{N}^+$, and $S$ is the debut of ${\widetilde{\mathbb{B}}}^c$. Part $(1)$ shows that $(T_n,\Delta X^{(n)})$ and $(S_n,\Delta X^{(n)})$ are CSs for $\Delta X$ and $\Delta (X\mathfrak{I}_{\widetilde{\mathbb{B}}})$, respectively. Then for each $n\in \mathbf{N}^+$,
\[
\Delta(X\mathfrak{I}_{\widetilde{\mathbb{B}}})I_{\widetilde{\mathbb{B}}\llbracket{0,S_n}\rrbracket}
=\Delta X^{(n)}I_{\widetilde{\mathbb{B}}\llbracket{0,S_n}\rrbracket}
=(\Delta X^{(n)}I_{\mathbb{B}\llbracket{0,T_n}\rrbracket})I_{\widetilde{\mathbb{B}}\llbracket{0,S_n}\rrbracket}
=(\Delta XI_{\mathbb{B}\llbracket{0,T_n}\rrbracket})I_{\widetilde{\mathbb{B}}\llbracket{0,S_n}\rrbracket}
=\left((\Delta X)\mathfrak{I}_{\widetilde{\mathbb{B}}}\right)I_{\widetilde{\mathbb{B}}\llbracket{0,S_n}\rrbracket},
\]
which, by Theorem \ref{process}(1), implies \eqref{eqXT0}.

$(4)$ Let $(\omega,t)\in\llbracket{0,T}\rrbracket\subseteq \mathbb{B}$. Definition \ref{X+-} guarantees the existence of $X(\omega,t-)$. It is easy to see that $X(\omega,s)=X(\omega,s\wedge T(\omega))=X^T(\omega,s)$ for $s\in[0,t]$ such that $X(\omega,t-)=X^T(\omega,t-)$. Then by Definition \ref{X+-}, we deduce that
\[
(\Delta X^T)(\omega,t)=X^T(\omega,t)-X^T(\omega,t-)=X(\omega,t)-X(\omega,t-)=(\Delta X)(\omega,t),
\]
which implies $\Delta X^T=\Delta XI_{\llbracket{0,T}\rrbracket}$ on $\llbracket{0,T}\rrbracket$. On the other hand, let $(\omega,t)\in\rrbracket{T,+\infty}\llbracket$. From the definition of $X^T$, it is evident that $X^T(\omega,s)=X(\omega,T(\omega))$ for $s\in]T(\omega),t]$ such that $X^T(\omega,t-)=X^T(\omega,t)$. Then we deduce that
\[
(\Delta X^T)(\omega,t)=X^T(\omega,t)-X^T(\omega,t-)=0,
\]
which implies $\Delta X^T=\Delta XI_{\llbracket{0,T}\rrbracket}$ on $\rrbracket{T,+\infty}\llbracket$. Thus, $\Delta X^T=\Delta XI_{\llbracket{0,T}\rrbracket}$ is derived. Using the facts $X^{T-}=X^T-\Delta X_TI_{\llbracket{T,+\infty}\llbracket}$ and $\Delta (I_{\llbracket{T,+\infty}\llbracket})=I_{\llbracket{T}\rrbracket}$, we have
\[
\Delta X^{T-}=\Delta X^T-\Delta (\Delta X_TI_{\llbracket{T,+\infty}\llbracket})
=\Delta XI_{\llbracket{0,T}\rrbracket}-\Delta X_TI_{\llbracket{T}\rrbracket}
=\Delta XI_{\llbracket{0,T}\llbracket},
\]
which yields $\Delta X^{T-}=\Delta XI_{\llbracket{0,T}\llbracket}$. Therefore, \eqref{eqXT} is valid, and we finish the proof of part (4).

$(5)$ {\it Sufficiency}. Suppose that $(T_n,X^{(n)})$ is an FCS for $X\in \mathcal{C}^\mathbb{B}$. Then $\Delta X^{(n)}=0$ for each $n\in \mathbf{N}^+$. From $\mathcal{C}^\mathbb{B}\subseteq \mathcal{R}^\mathbb{B}$, part $(1)$ shows
\[
\Delta X I_{\mathbb{B}\llbracket{0,T_n}\rrbracket}=\Delta X^{(n)}I_{\mathbb{B}\llbracket{0,T_n}\rrbracket}=0,\quad n\in \mathbf{N}^+,
\]
which, by Theorem \ref{process}(1), yields $\Delta X=0\mathfrak{I}_\mathbb{B}$.

{\it Necessity}. Suppose $\Delta X=0\mathfrak{I}_\mathbb{B}$. Let $\mathbb{B}$ be given by \eqref{B},
and $(T_n,X^{(n)})$ be an FCS for $X\in \mathcal{R}^\mathbb{B}$. For each $n\in \mathbf{N}^+$, $X^{(n)}$ is a c\`{a}dl\`{a}g process, and then $Y^{(n)}=(X^{(n)})^{{T_n}\wedge (T_F-)}$ is a c\`{a}dl\`{a}g process. For each $n\in \mathbf{N}^+$, part $(4)$ shows
\[
\Delta(Y^{(n)})=\Delta\left((X^{(n)})^{{T_n}\wedge (T_F-)}\right)=\Delta(X^{(n)})I_{\llbracket{0,T_n}\rrbracket\llbracket{0,T_F}\llbracket}
=\Delta XI_{\mathbb{B}\llbracket{0,T_n}\rrbracket}=0,
\]
which implies $Y^{(n)}\in \mathcal{C}$. It is evident that
\[
XI_{\mathbb{B}\llbracket{0,T_n}\rrbracket}=X^{(n)}I_{\mathbb{B}\llbracket{0,T_n}\rrbracket}=Y^{(n)}I_{\mathbb{B}\llbracket{0,T_n}\rrbracket},
\quad n\in \mathbf{N}^+.
\]
Hence, $X\in \mathcal{C}^\mathbb{B}$ with the FCS $(T_n,Y^{(n)})$.

$(6)$ Suppose that $(T_n,X^{(n)})$ and $(T_n,Y^{(n)})$ are FCSs for $X\in \mathcal{S}^\mathbb{B}$ and $Y\in \mathcal{S}^\mathbb{B}$, respectively. For each $n\in \mathbf{N}^+$, both $X^{(n)}$ and $Y^{(n)}$ are adapted c\`{a}dl\`{a}g processes such that both $\Delta X^{(n)}$ and $\Delta Y^{(n)}$ are thin process. Consequently, it follows that
 \[
 [\Delta X^{(n)}\Delta Y^{(n)}\neq 0]=[\Delta X^{(n)}\neq 0]\cap [\Delta Y^{(n)}\neq 0]\subseteq [\Delta X^{(n)}\neq 0], \quad n\in \mathbf{N}^+,
 \]
 which implies that the set $[\Delta X^{(n)}\Delta Y^{(n)}\neq 0]$ is an optional set contained in a thin set. By invoking Theorem 3.19 in \cite{He}, it is deduced that for each $n\in \mathbf{N}^+$, $\Delta X^{(n)}\Delta Y^{(n)}$ is a thin process. Furthermore, in light of part (1), it is evident that $(T_n,\Delta X^{(n)}\Delta Y^{(n)})$ constitutes a CS for $\Delta X\Delta Y$. This indicates that $\Delta X\Delta Y$ is a $\mathbb{B}$-thin process with the FCS $(T_n,\Delta X^{(n)}\Delta Y^{(n)})$. Utilizing Definition 8.2 in \cite{He}, we establish that for each $n\in \mathbf{N}^+$ and for all $t>0$,
\[
\sum_{s\leq t} |\Delta X_s^{(n)}\Delta Y_s^{(n)}|\leq \frac{1}{2}\left(\sum_{s\leq t} (\Delta X_s^{(n)})^2+\sum_{s\leq t} (\Delta Y_s^{(n)})^2\right)<+\infty,\quad a.s.,
\]
which demonstrates that $\Sigma(\Delta X^{(n)}\Delta Y^{(n)})$ is well-defined. Therefore, it is deduced from Theorem \ref{thin}(1) that $\Sigma (\Delta X\Delta Y)\in \mathcal{V}^\mathbb{B}$ with the FCS $(T_n,\Sigma(\Delta X^{(n)}\Delta Y^{(n)}))$.
$\hfill\blacksquare$

%%%%%%%%%%%%%%%%%%%%%%%%%%%%%%%%%%%%%%%%%%%%%%%%%%%%%%%%%%%%%%%%%%%%%%%%%%%%%%%%%%%%%%%%%%%%%%%%%%%%%%%
\par\vspace{0.3cm}
%%%%%%%%%%%%%%%%%%%%%%%%%%%%%%%%%%%%%%%%%%%%%%%%%%%%%%%%%%%%%%%%%%%%%%%%%%%%%%%%%%%%%%%%%%%%%%%%%%%%%%%
\noindent{\bf Proof of Theorem \ref{HAproperty}.}
$(1)$  By applying the conditions $H,K\in \mathcal{L}^\mathbb{B}_s(A)$, it holds that, for all $(\omega,t)\in \mathbb{B}$,
\[
\int_{[0,t]}|aH_s(\omega)+bK_s(\omega)||dA_s(\omega)|
\leq|a|\int_{[0,t]}|H_s(\omega)||dA_s(\omega)|+|b|\int_{[0,t]}|K_s(\omega)||dA_s(\omega)|<+\infty.
\]
For $(\omega,t)\in \mathbb{B}$, we define
\[
L(\omega,t)=\int_{[0,t]}(aH_s(\omega)+bK_s(\omega))dA_s(\omega).
\]
Utilizing the definitions of $H_{\bullet}A$ and $K_{\bullet}A$, it is straightforward to deduce that for all $(\omega,t)\in \mathbb{B}$,
\[
L(\omega,t)
=a\int_{[0,t]}H_s(\omega)dA_s(\omega)+b\int_{[0,t]}K_s(\omega)dA_s(\omega)
=a(H_{\bullet}A)(\omega,t)+b(K_{\bullet}A)(\omega,t).
\]
Consequently, it follows that $L\in\mathcal{V}^\mathbb{B}$ and $L=a(H_{\bullet}A)+b(K_{\bullet}A)$.
Therefore, by Definition \ref{HA}, we verify the assertion.

$(2)$ By invoking the condition $H\in \mathcal{L}^\mathbb{B}_s(A)\cap\mathcal{L}^\mathbb{B}_s(V)$, we deduce that, for all $(\omega,t)\in \mathbb{B}$,
\[
\int_{[0,t]}|H_s(\omega)||d(aA_s(\omega)+bV_s(\omega))|
\leq|a|\int_{[0,t]}|H_s(\omega)||dA_s(\omega)|+|b|\int_{[0,t]}|H_s(\omega)||dV_s(\omega)|<+\infty.
\]
For $(\omega,t)\in \mathbb{B}$, we define
\[
L(\omega,t)
=\int_{[0,t]}H_s(\omega)d(aA_s(\omega)+bV_s(\omega)).
\]
It follows from the definitions of $H_{\bullet}A$ and $H_{\bullet}V$ that, for all $(\omega,t)\in \mathbb{B}$,
\[
L(\omega,t)=a\int_{[0,t]}H_s(\omega)dA_s(\omega)+b\int_{[0,t]}H_s(\omega)dV_s(\omega)
=a(H_{\bullet}A)(\omega,t)+b(H_{\bullet}V)(\omega,t).
\]
As a result, we obtain the relations $L\in\mathcal{V}^\mathbb{B}$ and $L=a(H_{\bullet}A)+b(K_{\bullet}A)$.
Hence, the assertion is proven by Definition \ref{HA}.

$(3)$ The first assertion can be proven by the equivalence between
\[
\int_{[0,t]}|\widetilde{H}_s(\omega)H_s(\omega)||dA_s(\omega)|<+\infty, \;(\omega,t)\in \mathbb{B},\quad L_1\in\mathcal{V}^\mathbb{B}
\]
and
\[
\int_{[0,t]}|\widetilde{H}_s(\omega)||d\widetilde{A}_s(\omega)|< +\infty,\;(\omega,t)\in \mathbb{B},\quad L_2\in\mathcal{V}^\mathbb{B},
\]
where $\widetilde{A}$, $L_1$ and $L_2$ are given by for $(\omega,t)\in \mathbb{B}$
\[
\widetilde{A}_t(\omega)=\int_{[0,t]}H_s(\omega)dA_s(\omega),\quad L_1(\omega,t)=\int_{[0,t]}\widetilde{H}_s(\omega)H_s(\omega)dA_s(\omega),\quad
L_2(\omega,t)=\int_{[0,t]}\widetilde{H}_s(\omega)d\widetilde{A}_s(\omega).
\]

Suppose that $\widetilde{H}H\in \mathcal{L}^\mathbb{B}_s(A)$. By employing the equality $L_1=L_2$ and the definitions of $(\widetilde{H}H)_{\bullet}A$ and $\widetilde{H}_{\bullet}\widetilde{A}$, it is straightforward to derive \eqref{ab3}.
$\hfill\blacksquare$

%%%%%%%%%%%%%%%%%%%%%%%%%%%%%%%%%%%%%%%%%%%%%%%%%%%%%%%%%%%%%%%%%%%%%%%%%%%%%%%%%%%%%%%%%%%%%%%%%%%%%%%
\par\vspace{0.3cm}
%%%%%%%%%%%%%%%%%%%%%%%%%%%%%%%%%%%%%%%%%%%%%%%%%%%%%%%%%%%%%%%%%%%%%%%%%%%%%%%%%%%%%%%%%%%%%%%%%%%%%%%
\noindent {\bf Proof of Theorem \ref{HA-equivalent}.}
In order to establish the proof, we commence by introducing two pivotal lemmas.
\begin{lemma}\label{bound-p}
Let $H\in \mathcal{P}^\mathbb{B}$. If there exists a constant $C>0$ such that $|H|\leq C$, i.e., $|H(\omega,t)|\leq C$ for all $(\omega,t)\in \mathbb{B}$, then $H$ is a $\mathbb{B}$-locally bounded predictable process.
\end{lemma}
\begin{proof}
Let $\mathbb{B}$ be given by \eqref{B}. In fact, from \eqref{couple}, there exists a predictable process $\widetilde{H}$ such that $H=\widetilde{H}\mathfrak{I}_\mathbb{B}$. For each $n\in \mathbf{N}^+$, putting $H^{(n)}=\widetilde{H}\wedge C$, then $H^{(n)}$ is a bounded predictable process satisfying
$H^{(n)}I_{\mathbb{B}}=HI_{\mathbb{B}}$. Consequently, $H$ is a $\mathbb{B}$-locally bounded predictable process with the FCS $(T_n=T,H^{(n)})$, which proves the assertion.
\end{proof}
\begin{lemma}\label{Lemma-HA}
Let $A\in \mathcal{V}^\mathbb{B}$, and suppose that $H$ is a $\mathbb{B}$-locally bounded predictable process. Then $H\in\mathcal{L}^\mathbb{B}_s(A)$. Furthermore, if $\widetilde{H}$ is a coupled locally bounded predictable process for $H$, and if $(T_n,A^{(n)})$ is an FCS for $A\in\mathcal{V}^\mathbb{B}$, then $(T_n,\widetilde{H}.A^{(n)})$ is an FCS for $H_{\bullet}A\in\mathcal{V}^\mathbb{B}$.
\end{lemma}
\begin{proof}
Assume that $\widetilde{H}$ is a coupled locally bounded predictable process for $H$, and that $(T_n,A^{(n)})$ is an FCS for $A\in\mathcal{V}^\mathbb{B}$. It is evident that for each $n\in \mathbf{N}^+$, $\widetilde{H}\in\mathcal{L}_s(A^{(n)})$. Given any $(\omega,t)\in \mathbb{B}$, there exists an integer $m$ such that $(\omega,t)\in {\mathbb{B}\llbracket{0,T_m}\rrbracket}$. Consequently, we have $H(\omega,s)=\widetilde{H}(\omega,s)$ and $A(\omega,s)=A^{(m)}(\omega,s)$ for all $s\in [0,t]$. Hence, it follows from $\widetilde{H}\in\mathcal{L}_s(A^{(m)})$ that
\[
 \int_{[0,t]}|H_s(\omega)||dA_s(\omega)|=\int_{[0,t]}|\widetilde{H}_s(\omega)||dA^{(m)}_s(\omega)|
 <+\infty.
\]
For $(\omega,t)\in \mathbb{B}$, we define
\[
 L(\omega,t)=\int_{[0,t]}H_s(\omega)dA_s(\omega).
\]
For each $n\in \mathbf{N}^+$ and for all $(\omega,t)\in \mathbb{B}\llbracket{0,T_n}\rrbracket$, by invoking the facts  $H(\omega,s)=\widetilde{H}(\omega,s)$ and $A(\omega,s)=A^{(n)}(\omega,s)$ for all $s\in [0,t]$,
it can be demonstrated that
\[
 L(\omega,t)=\int_{[0,t]}\widetilde{H}_s(\omega)dA^{(n)}_s(\omega).
\]
This naturally implies $LI_{\mathbb{B}\llbracket{0,T_n}\rrbracket}
=(\widetilde{H}.A^{(n)})I_{\mathbb{B}\llbracket{0,T_n}\rrbracket}$. Since $\widetilde{H}.A^{(n)}\in \mathcal{V}$ holds (see, e.g., Theorem 3.46 in \cite{He}) for each $n\in \mathbf{N}^+$, we establish that $(T_n,\widetilde{H}.A^{(n)})$ is an FCS for $L\in\mathcal{V}^\mathbb{B}$. Therefore, by utilizing Definition \ref{HA}, it holds that $H\in\mathcal{L}^\mathbb{B}_s(A)$ and $H_{\bullet}A=L\in\mathcal{V}^\mathbb{B}$ with the FCS $(T_n,\widetilde{H}.A^{(n)})$.
\end{proof}

Now we can show the proof of Theorem \ref{HA-equivalent}.

(A1) $\Rightarrow$ (A2).
Let $L=H_{\bullet}A$, and fix a constant $\varepsilon>0$. Suppose that $\widetilde{H}$ is a coupled predictable process for $H\in\mathcal{P}^\mathbb{B}$, and that $(T_n,\widetilde{A}^{(n)})$ and $(T_n,L^{(n)})$ are FCSs for $A\in\mathcal{V}^\mathbb{B}$ and $L\in \mathcal{V}^\mathbb{B}$, respectively. By invoking Theorem \ref{HAproperty} and Lemma \ref{Lemma-HA}, we can express the $\mathbb{B}$-process $A$ as follows:
\[
A=K_{\bullet}L+X_{\bullet}A,\quad K=\left(\frac{1}{H}I_{[|H|>\varepsilon]}\right)\mathfrak{I}_\mathbb{B},\;X=I_{[|H|\leq\varepsilon]}\mathfrak{I}_\mathbb{B},
\]
Here, according to Lemma \ref{bound-p}, $K$ and $X$ are $\mathbb{B}$-locally bounded predictable processes, ensuring that $K_{\bullet}L$ and $X_{\bullet}A$ are well-defined.
Next, define the processes:
\[
\widetilde{K}=\frac{1}{\widetilde{H}}I_{[|\widetilde{H}|>\varepsilon]},\quad \widetilde{X}=I_{[|\widetilde{H}|\leq\varepsilon]}.
\]
These are coupled locally bounded predictable processes for $K$ and $X$, respectively.
According to Lemma \ref{Lemma-HA}, $(T_n,\widetilde{K}.L^{(n)})$ and $(T_n,\widetilde{X}.\widetilde{A}^{(n)})$ are FCSs for $K_{\bullet}L\in\mathcal{V}^\mathbb{B}$ and $X_{\bullet}A\in\mathcal{V}^\mathbb{B}$, respectively.
For each $n\in\mathbf{N}^+$, we define
\[
A^{(n)}=\widetilde{K}.L^{(n)}+\widetilde{X}.\widetilde{A}^{(n)},
\]
and it follows that  $A^{(n)}\in\mathcal{V}^\mathbb{B}$ and
\begin{align*}
A^{(n)}I_{\mathbb{B}\llbracket{0,T_n}\rrbracket}
=(\widetilde{K}.L^{(n)}+\widetilde{X}.\widetilde{A}^{(n)})I_{\mathbb{B}\llbracket{0,T_n}\rrbracket}
=(K_{\bullet}L+X_{\bullet}A)I_{\mathbb{B}\llbracket{0,T_n}\rrbracket}
=AI_{\mathbb{B}\llbracket{0,T_n}\rrbracket}.
\end{align*}
Consequently, $(T_n,A^{(n)})$ is an FCS for $A\in \mathcal{V}^\mathbb{B}$. For each $n\in\mathbf{N}^+$, we infer that $\widetilde{H}\in \mathcal{L}_s(A^{(n)})$ from the identity
\[
(\widetilde{H}\widetilde{K}).L^{(n)}
+(\widetilde{H}\widetilde{X}).\widetilde{A}^{(n)}=\widetilde{H}.(\widetilde{K}.L^{(n)}
+\widetilde{X}.\widetilde{A}^{(n)})=\widetilde{H}.A^{(n)},
\]
where we use the fact that $\widetilde{H}\widetilde{K}$ and $\widetilde{H}\widetilde{X}$ are locally bounded predictable processes satisfying $\widetilde{H}\widetilde{K}\in\mathcal{L}_s(L^{(n)})$ and  $\widetilde{H}\widetilde{X}\in\mathcal{L}_s(\widetilde{A}^{(n)})$.
In conclusion, $\widetilde{H}$ is a coupled predictable process for $H\in\mathcal{P}^\mathbb{B}$, and $(T_n,A^{(n)})$ is an FCS for $A\in\mathcal{V}^\mathbb{B}$ such that for each $n\in \mathbf{N}^+$, $\widetilde{H}\in \mathcal{L}_s(A^{(n)})$, thereby establishing (A2).

(A2) $\Rightarrow$ (A3). Suppose (A2) holds. For each $n\in \mathbf{N}^+$, put $H^{(n)}=\widetilde{H}$. Then $(T_n,H^{(n)})$ is an FCS for $H\in \mathcal{P}^\mathbb{B}$ such that for each $n\in \mathbf{N}^+$, $H^{(n)}\in\mathcal{L}_s(A^{(n)})$, which yields (A3).

(A3) $\Rightarrow$ (A1). Suppose (A3) holds. Firstly, we fix any $(\omega,t)\in \mathbb{B}$. There exists an integer $m\in \mathbf{N}^+$ such that $(\omega,t)\in {\mathbb{B}\llbracket{0,T_m}\rrbracket}$. Then according to FCSs $(T_n, H^{(n)})$ for $H\in\mathcal{P}^\mathbb{B}$ and $(T_n,A^{(n)})$ for $A\in\mathcal{V}^\mathbb{B}$, we infer that
\[
H(\omega,s)=H^{(m)}(\omega,s),\quad A(\omega,s)=A^{(m)}(\omega,s),\quad s\in[0,t].
\]
By utilizing the relation $H^{(m)}\in\mathcal{L}_s(A^{(m)})$, it is established that
\[
\int_{[0,t]}|H_s(\omega)||dA_s(\omega)|
=\int_{[0,t]}|H^{(m)}_s(\omega)||dA^{(m)}_s(\omega)|<{+\infty}.
 \]
Subsequently, for all $(\omega,t)\in \mathbb{B}$, we define
\[
L(\omega,t)=\int_{[0,t]}H_s(\omega)dA_s(\omega).
\]
For each $n\in \mathbf{N}^+$ and for all $(\omega,t)\in \mathbb{B}\llbracket{0,T_n}\rrbracket$, by utilizing the relations  $H(\omega,s)=H^{(n)}(\omega,s)$ and $A(\omega,s)=A^{(n)}(\omega,s)$ for $s\in [0,t]$,
it can be demonstrated that
\[
 L(\omega,t)=\int_{[0,t]}H^{(n)}_s(\omega)dA^{(n)}_s(\omega).
\]
This indicates $LI_{\mathbb{B}\llbracket{0,T_n}\rrbracket}
=(H^{(n)}.A^{(n)})I_{\mathbb{B}\llbracket{0,T_n}\rrbracket}$. Since $H^{(n)}.A^{(n)}\in \mathcal{V}$ holds for each $n\in \mathbf{N}^+$, we obtain that $L\in\mathcal{V}^\mathbb{B}$ with the FCS $(T_n,H^{(n)}.A^{(n)})$. Therefore, according to Definition \ref{HA}, we verify the validity of $H\in\mathcal{L}^\mathbb{B}_s(A)$, and establish (A1). $\hfill\blacksquare$

%%%%%%%%%%%%%%%%%%%%%%%%%%%%%%%%%%%%%%%%%%%%%%%%%%%%%%%%%%%%%%%%%%%%%%%%%%%%%%%%%%%%%%%%%%%%%%%%%%%%%%%
\par\vspace{0.3cm}
%%%%%%%%%%%%%%%%%%%%%%%%%%%%%%%%%%%%%%%%%%%%%%%%%%%%%%%%%%%%%%%%%%%%%%%%%%%%%%%%%%%%%%%%%%%%%%%%%%%%%%%
\noindent {\bf Proof of Theorem \ref{HA-FCS}.}
The assertion that $H_{\bullet}A\in\mathcal{V}^\mathbb{B}$ with the FCS $(T_n,H^{(n)}.A^{(n)})$ has been established within the proof of the implication (A3) $\Rightarrow$ (A1) in Theorem \ref{HA-equivalent}.
The expression \eqref{HA-expression2} can be easily derived from \eqref{x-expression}.
Suppose that $(S_n, \widetilde{H}^{(n)})$ and $(\widetilde{S}_n,\widetilde{A}^{(n)})$ are also FCSs for $H\in\mathcal{P}^\mathbb{B}$ and $A\in\mathcal{V}^\mathbb{B}$ respectively such that for each $n\in \mathbf{N}^+$, $\widetilde{H}^{(n)}\in \mathcal{L}_s(\widetilde{A}^{(n)})$.
Similarly, it can be shown that $(\widetilde{T}_n,\widetilde{H}^{(n)}.\widetilde{A}^{(n)})$ is an FCS for $H_{\bullet}A\in\mathcal{V}^{\mathbb{B}}$. Consequently, $H_{\bullet}A=X$ is deduced from the independence property of \eqref{x-expression}.$\hfill\blacksquare$

%%%%%%%%%%%%%%%%%%%%%%%%%%%%%%%%%%%%%%%%%%%%%%%%%%%%%%%%%%%%%%%%%%%%%%%%%%%%%%%%%%%%%%%%%%%%%%%%%%%%%%%
\par\vspace{0.3cm}
%%%%%%%%%%%%%%%%%%%%%%%%%%%%%%%%%%%%%%%%%%%%%%%%%%%%%%%%%%%%%%%%%%%%%%%%%%%%%%%%%%%%%%%%%%%%%%%%%%%%%%%
\noindent {\bf Proof of Corollary \ref{bound-HA}.}
For each $n\in \mathbf{N}^+$, it holds that $H^{(n)}\in \mathcal{L}_s(A^{(n)})$, and $H^{(n)}.A^{(n)}\in \mathcal{V}$ (see, e.g., Theorem 3.46 in \cite{He}). Consequently, it follows from Theorem \ref{HA-FCS} that $H_{\bullet}A\in\mathcal{V}^\mathbb{B}$ with the FCS $(T_n,H^{(n)}.A^{(n)})$. $\hfill\blacksquare$

%%%%%%%%%%%%%%%%%%%%%%%%%%%%%%%%%%%%%%%%%%%%%%%%%%%%%%%%%%%%%%%%%%%%%%%%%%%%%%%%%%%%%%%%%%%%%%%%%%%%%%%
\par\vspace{0.3cm}
%%%%%%%%%%%%%%%%%%%%%%%%%%%%%%%%%%%%%%%%%%%%%%%%%%%%%%%%%%%%%%%%%%%%%%%%%%%%%%%%%%%%%%%%%%%%%%%%%%%%%%%
\noindent {\bf Proof of Corollary \ref{HAp-equivalent}.}
$(1)$ The sufficiency can be straightforwardly established by invoking Corollary \ref{fcs-p}(2) and statement (A3). Now, we proceed to demonstrate the necessity. Assume that $H\in\mathcal{L}^\mathbb{C}_s(A)$.
Let $(S_n)$ be an FS for $\mathbb{C}$, and suppose that $(T_n, H^{(n)})$ and $(T_n,A^{(n)})$ are FCSs for $H\in\mathcal{P}^\mathbb{C}$ and $A\in\mathcal{V}^\mathbb{C}$ respectively such that for each $n\in \mathbf{N}^+$, $H^{(n)}\in\mathcal{L}_s(A^{(n)})$ (see statement (A3)).
By defining $(\tau_n)=(S_n\wedge T_n)$, Corollary \ref{process-FS}(2) asserts that $(\tau_n,H^{(n)})$ and $(\tau_n,A^{(n)})$ constitute FCSs for $H\in\mathcal{P}^\mathbb{B}$ and $A\in\mathcal{V}^\mathbb{B}$, respectively, and they satisfy the relations:
\[
H^{\tau_n}=(H^{(n)})^{\tau_n},\quad A^{\tau_n}=(A^{(n)})^{\tau_n},\quad n\in\mathbf{N}^+.
\]
Furthermore, for each $n\in\mathbf{N}^+$, the identity
\[
(H^{(n)}.A^{(n)})^{\tau_n}=(H^{(n)})^{\tau_n}.(A^{(n)})^{\tau_n}
\]
implies $H^{\tau_n}\in\mathcal{L}_s(A^{\tau_n})$. consequently, the FS $(\tau_n)$ for $\mathbb{C}$ aligns with with our expectations.

$(2)$ According to Corollary \ref{fcs-p}(2), $(\tau_n, H^{\tau_n})$ and $(\tau_n,A^{\tau_n})$ are FCSs for $H\in\mathcal{P}^\mathbb{C}$ and $A\in\mathcal{V}^\mathbb{C}$, respectively.
For each $n\in \mathbf{N}^+$ and for all $(\omega,t)\in \llbracket{0,\tau_n}\rrbracket$, it necessarily follows that $H(\omega,s)=H^{\tau_n}(\omega,s)$ and $A(\omega,s)=A^{\tau_n}(\omega,s)$ for $s\in [0,t]$.
Subsequently, by applying the condition $H\in\mathcal{L}^\mathbb{C}_s(A)$, it can be demonstrated that
\[
 \int_{[0,t]}|H^{\tau_n}_s(\omega)||dA^{\tau_n}_s(\omega)|
 =\int_{[0,t]}|H_s(\omega)||dA_s(\omega)|<+\infty.
\]
This finding entails that $H^{\tau_n}\in\mathcal{L}_s(A^{\tau_n})$ for each $n\in \mathbf{N}^+$. Hence, we can establish the proof by leveraging Theorem \ref{HA-FCS}.
$\hfill\blacksquare$

%%%%%%%%%%%%%%%%%%%%%%%%%%%%%%%%%%%%%%%%%%%%%%%%%%%%%%%%%%%%%%%%%%%%%%%%%%%%%%%%%%%%%%%%%%%%%%%%%%%%%%%
\par\vspace{0.3cm}
%%%%%%%%%%%%%%%%%%%%%%%%%%%%%%%%%%%%%%%%%%%%%%%%%%%%%%%%%%%%%%%%%%%%%%%%%%%%%%%%%%%%%%%%%%%%%%%%%%%%%%%
\noindent {\bf Proof of Lemma \ref{ex-quad-p}.}
In order to establish the proof, we commence by introducing a pivotal lemma.
\begin{lemma}\label{LM-A2}
If $M\in(\mathcal{M}^c_{\mathrm{loc}})^\mathbb{B}\cap\mathcal{V}^\mathbb{B}$, then $M=M_0\mathfrak{I}_\mathbb{B}$.
\end{lemma}
\begin{proof}
Let $\mathbb{B}$ be given by \eqref{B}. If both $(T_n,M^{(n)})$ and $(T_n,N^{(n)})$ are CSs for $M$, then  it is easy to see
\begin{equation}\label{PQ-1}
M^{(n)}I_{\mathbb{B}\llbracket{0,T_n}\rrbracket}=MI_{\mathbb{B}\llbracket{0,T_n}\rrbracket}
=N^{(n)}I_{\mathbb{B}\llbracket{0,T_n}\rrbracket},\quad n\in \mathbf{N}^+.
\end{equation}

Suppose that $(T_n,M^{(n)})$ and $(T_n,N^{(n)})$ are FCSs for $M\in (\mathcal{M}^c_{\mathrm{loc}})^\mathbb{B}$ and $M\in \mathcal{V}^\mathbb{B}$, respectively.
Fix $k\in \mathbf{N}^+$. Using $M^{(k)}\in \mathcal{M}^c_{\mathrm{loc}}$ and \eqref{PQ-1}, we deduce
\[
(M^{(k)})^{T_k}=(M^{(k)})^{T_k\wedge (T_F-)}
=(N^{(k)})^{T_k\wedge (T_F-)}.
\]
which, by $(M^{(k)})^{T_k}\in\mathcal{M}^c_{\mathrm{loc}}$ and $(N^{(k)})^{T_k\wedge (T_F-)}\in\mathcal{V}$, implies that $(M^{(k)})^{T_k}\in \mathcal{M}^c_{\mathrm{loc}}\cap\mathcal{V}$.
Lemmas I.4.13 and I.4.14 in \cite{Jacod} show $(M^{(k)})^{T_k}=M^{(k)}_0=M_0$. Then \eqref{PQ-1} yields $M=M_0\mathfrak{I}_\mathbb{B}$.
\end{proof}

Now we present the proof. Assume that $\mathbb{B}$ is given by \eqref{B}. Without loss of generalization, let $(T_n,M^{(n)})$ and $(T_n,N^{(n)})$ be FCSs for $M\in (\mathcal{M}^c_{\mathrm{loc}})^\mathbb{B}$ and $N\in (\mathcal{M}^c_{\mathrm{loc}})^\mathbb{B}$, respectively.

Firstly, we show the $\mathbb{B}$-process, defined by
\begin{equation}\label{L=<M,N>}
V=\left(M_0N_0I_{\llbracket{0}\rrbracket}+\sum\limits_{n=1}^{{+\infty}}\langle M^{(n)},N^{(n)}\rangle I_{\rrbracket{T_{n-1},T_n}
      \rrbracket}\right)\mathfrak{I}_\mathbb{B}, \quad T_0=0,
\end{equation}
satisfies $V\in (\mathcal{A}_{\mathrm{loc}}\cap \mathcal{C})^\mathbb{B}$ with the FCS $(T_n,\langle M^{(n)},N^{(n)}\rangle)$.
For any $l,k\in\mathbf{N}^+$ with $k\leq l$, Theorem \ref{process}(2) yields
$M^{(k)}I_{\mathbb{B}\llbracket{0,T_k}\rrbracket}=M^{(l)}I_{\mathbb{B}\llbracket{0,T_k}\rrbracket}$ and $N^{(k)}I_{\mathbb{B}\llbracket{0,T_k}\rrbracket}=N^{(l)}I_{\mathbb{B}\llbracket{0,T_k}\rrbracket}$, or equivalently,
\begin{equation*}
\left\{
\begin{aligned}
(M^{(k)})^{T_k}&=(M^{(k)})^{T_k\wedge(T_F-)}=(M^{(l)})^{T_k\wedge(T_F-)}=(M^{(l)})^{T_k},\\ (N^{(k)})^{T_k}&=(N^{(k)})^{T_k\wedge (T_F-)}=(N^{(l)})^{T_k\wedge (T_F-)}=(N^{(l)})^{T_k},
\end{aligned}
\right.
\end{equation*}
which, by Definition 7.29 in \cite{He}, implies that
\begin{align}\label{eq14}
 \langle M^{(k)},N^{(k)}\rangle I_{\mathbb{B}\llbracket{0,T_k}\rrbracket}
=&\langle M^{(k)},N^{(k)}\rangle^{T_k}I_{\mathbb{B}\llbracket{0,T_k}\rrbracket}\nonumber\\
=&\langle(M^{(k)})^{T_k},(N^{(k)})^{T_k}\rangle I_{\mathbb{B}\llbracket{0,T_k}\rrbracket}\nonumber\\
=&\langle(M^{(l)})^{T_k},(N^{(l)})^{T_k}\rangle I_{\mathbb{B}\llbracket{0,T_k}\rrbracket}\nonumber\\
=&\langle M^{(l)},N^{(l)}\rangle^{T_k}I_{\mathbb{B}\llbracket{0,T_k}\rrbracket}\nonumber\\
=&\langle M^{(l)},N^{(l)}\rangle I_{\mathbb{B}\llbracket{0,T_k}\rrbracket}.
\end{align}
Then Theorem \ref{restriction}(4) shows that $(T_n,\langle M^{(n)},N^{(n)}\rangle)$ is a CS for $V$. Since $\langle M^{(n)},N^{(n)}\rangle\in \mathcal{A}_{\mathrm{loc}}\cap \mathcal{C}$ (see, e.g., Lemma 7.28 and its Remark in \cite{He}) for each $n\in \mathbf{N}^+$, we deduce that $(T_n,\langle M^{(n)},N^{(n)}\rangle)$ is an FCS for $V\in (\mathcal{A}_{\mathrm{loc}}\cap \mathcal{C})^\mathbb{B}$.

Next, we show $MN-V\in (\mathcal{M}^c_{\mathrm{loc},0})^\mathbb{B}$ with the FCS $(T_n,M^{(n)}N^{(n)}-\langle M^{(n)},N^{(n)}\rangle)$, thereby proving the existence of $V$.
For each $n\in \mathbf{N}^+$, by (\ref{eq14}), we have
\begin{align*}
(MN-V)I_{\mathbb{B}\llbracket{0,T_n}\rrbracket}&=\sum\limits_{k=1}^{n}(M^{(k)}N^{(k)}-\langle M^{(k)},N^{(k)}\rangle)I_{\mathbb{B}\rrbracket{T_{k-1},T_k}
      \rrbracket}\\
      &=\sum\limits_{k=1}^{n}(M^{(n)}N^{(n)}-\langle M^{(n)},N^{(n)}\rangle)I_{\mathbb{B}\rrbracket{T_{k-1},T_k}
      \rrbracket}\\
      &=(M^{(n)}N^{(n)}-\langle M^{(n)},N^{(n)}\rangle)I_{\mathbb{B}\llbracket{0,T_n}\rrbracket},
\end{align*}
which implies $(T_n,M^{(n)}N^{(n)}-\langle M^{(n)},N^{(n)}\rangle)$ is a CS for $MN-V$.
Since $M^{(n)}N^{(n)}-\langle M^{(n)},N^{(n)}\rangle\in \mathcal{M}^c_{\mathrm{loc},0}$ (see, e.g., Lemma 7.28 and its Remark in \cite{He}) for each $n\in \mathbf{N}^+$, the relation $MN-V\in(\mathcal{M}^c_{\mathrm{loc},0})^\mathbb{B}$ holds true, and $(T_n,M^{(n)}N^{(n)}-\langle M^{(n)},N^{(n)}\rangle)$ is an FCS for $MN-V\in (\mathcal{M}^c_{\mathrm{loc},0})^\mathbb{B}$.

Finally, we show the uniqueness of $V\in (\mathcal{A}_{\mathrm{loc}}\cap \mathcal{C})^\mathbb{B}$. Suppose that there exists another $\mathbb{B}$-process $A\in (\mathcal{A}_{\mathrm{loc}}\cap \mathcal{C})^\mathbb{B}$ such that $MN-A\in (\mathcal{M}^c_{\mathrm{loc},0})^\mathbb{B}$. Put $L=(MN-A)-(MN-V)=V-A$. Then
\[
L\in (\mathcal{M}^c_{\mathrm{loc},0})^\mathbb{B}\cap(\mathcal{A}_{\mathrm{loc}}\cap \mathcal{C})^\mathbb{B}\subseteq (\mathcal{M}^c_{\mathrm{loc,0}})^\mathbb{B}\cap\mathcal{V}^\mathbb{B}.
\]
Lemma \ref{LM-A2} shows $L=0\mathfrak{I}_\mathbb{B}$, i.e., $V=A$, which yields the uniqueness of $V\in (\mathcal{A}_{\mathrm{loc}}\cap \mathcal{C})^\mathbb{B}$.

Summarizing, the process $V$ defined by \eqref{L=<M,N>} is what we need, and we finish the proof.
$\hfill\blacksquare$

%%%%%%%%%%%%%%%%%%%%%%%%%%%%%%%%%%%%%%%%%%%%%%%%%%%%%%%%%%%%%%%%%%%%%%%%%%%%%%%%%%%%%%%%%%%%%%%%%%%%%%%
\par\vspace{0.3cm}
%%%%%%%%%%%%%%%%%%%%%%%%%%%%%%%%%%%%%%%%%%%%%%%%%%%%%%%%%%%%%%%%%%%%%%%%%%%%%%%%%%%%%%%%%%%%%%%%%%%%%%%
\noindent {\bf Proof of Theorem \ref{property-qr-p}.}
$(1)$ Since $\langle M,N\rangle$ can be expressed as \eqref{L=<M,N>}, the assertion that $\langle M,N\rangle\in (\mathcal{A}_{\mathrm{loc}}\cap \mathcal{C})^\mathbb{B}$ with the FCS $(T_n,\langle M^{(n)},N^{(n)}\rangle)$ has been shown in the proof of Lemma \ref{ex-quad-p}. It remains to prove $\langle M\rangle\in (\mathcal{A}^{+}_{\mathrm{loc}}\cap \mathcal{C})^\mathbb{B}$ with the FCS $(T_n,\langle M^{(n)}\rangle)$.
The fact $\langle M\rangle=\langle M,M\rangle$ indicates that $(T_n,\langle M^{(n)}\rangle)$ is a CS for $\langle M\rangle$, and Lemma 7.28 and its Remark in \cite{He} yield $\langle M^{(n)}\rangle\in \mathcal{A}^{+}_{\mathrm{loc}}\cap \mathcal{C}$ for each $n\in \mathbf{N}^+$. Therefore, we deduce $\langle M\rangle\in (\mathcal{A}^{+}_{\mathrm{loc}}\cap \mathcal{C})^\mathbb{B}$ with the FCS $(T_n,\langle M^{(n)}\rangle)$.

$(2)$ The proof of $\langle M,N\rangle=\langle N,M\rangle$ is trivial, and it suffices to prove $\langle aM+b\widetilde{M},N\rangle=a\langle M,N\rangle+b\langle\widetilde{M},N\rangle$. Suppose that $(T_n,M^{(n)})$, $(T_n,\widetilde{M}^{(n)})$ and $(T_n,N^{(n)})$ are FCSs for $M\in (\mathcal{M}^c_{\mathrm{loc}})^\mathbb{B}$, $\widetilde{M}\in (\mathcal{M}^c_{\mathrm{loc}})^\mathbb{B}$ and $N\in(\mathcal{M}^c_{\mathrm{loc}})^\mathbb{B}$, respectively.
From part $(1)$, $(T_n,\langle M^{(n)},N^{(n)}\rangle)$, $(T_n,\langle \widetilde{M}^{(n)},N^{(n)}\rangle)$ and $(T_n,\langle aM^{(n)}+b\widetilde{M}^{(n)},N^{(n)}\rangle )$ are FCSs for $\langle M,N\rangle\in (\mathcal{A}_{\mathrm{loc}}\cap \mathcal{C})^\mathbb{B}$, $\langle \widetilde{M},N\rangle\in (\mathcal{A}_{\mathrm{loc}}\cap \mathcal{C})^\mathbb{B}$ and $\langle aM+b\widetilde{M},N\rangle\in(\mathcal{A}_{\mathrm{loc}}\cap \mathcal{C})^\mathbb{B}$, respectively. Then for each $n\in \mathbf{N}^{+}$,
\begin{align*}
\langle aM+b\widetilde{M},N\rangle I_{\mathbb{B}\llbracket{0,T_n}\rrbracket}
=&\langle aM^{(n)}+b\widetilde{M}^{(n)},N^{(n)}\rangle I_{\mathbb{B}\llbracket{0,T_n}\rrbracket}\\
=&a\langle M^{(n)},N^{(n)}\rangle I_{\mathbb{B}\llbracket{0,T_n}\rrbracket}+b\langle \widetilde{M}^{(n)},N^{(n)}\rangle I_{\mathbb{B}\llbracket{0,T_n}\rrbracket}\\
=&(a\langle M,N\rangle +b\langle \widetilde{M},N\rangle )I_{\mathbb{B}\llbracket{0,T_n}\rrbracket},
\end{align*}
which, by Theorem \ref{process}(1), finishes the proof.

$(3)$ From Lemma 7.28 and its Remark in \cite{He}, $\langle M^{\tau},N^{\tau}\rangle$ is the unique process $A\in \mathcal{A}_{\mathrm{loc}}\cap \mathcal{C}$ such that $M^{\tau}N^{\tau}-A\in \mathcal{M}^c_{\mathrm{loc},0}$. Definition \ref{de-quad-p} shows that $\langle M,N\rangle$ is the unique process $V\in (\mathcal{A}_{\mathrm{loc}}\cap \mathcal{C})^\mathbb{B}$ such that $MN-V\in (\mathcal{M}^c_{\mathrm{loc},0})^\mathbb{B}$, which, by Theorem \ref{fcs}, implies $\langle M,N\rangle^{\tau}\in\mathcal{A}_{\mathrm{loc}}\cap \mathcal{C}$ satisfying
\[
M^{\tau}N^{\tau}-\langle M,N\rangle^{\tau}=(MN-\langle M,N\rangle)^{\tau}\in \mathcal{M}^c_{\mathrm{loc},0}.
\]
Hence, by the uniqueness, \eqref{Mc-tau1} is deduced.

Suppose that $(T_n,M^{(n)})$ and $(T_n,N^{(n)})$ are FCSs for $M\in (\mathcal{M}^c_{\mathrm{loc}})^\mathbb{B}$ and $N\in(\mathcal{M}^c_{\mathrm{loc}})^\mathbb{B}$, respectively.
From part $(1)$ and Theorem \ref{fcs}, $(T_n,\langle (M^{(n)})^{\tau},(N^{(n)})^{\tau}\rangle)$, $(T_n,\langle M^{(n)},N^{(n)}\rangle^{\tau})$ and $(T_n,\langle(M^{(n)})^{\tau},N^{(n)}\rangle)$
are FCSs for $\langle M^{\tau}\mathfrak{I}_\mathbb{B},N^{\tau}\mathfrak{I}_\mathbb{B}\rangle\in (\mathcal{A}_{\mathrm{loc}}\cap \mathcal{C})^\mathbb{B}$, $\langle M,N\rangle^{\tau}\mathfrak{I}_\mathbb{B}\in (\mathcal{A}_{\mathrm{loc}}\cap \mathcal{C})^\mathbb{B}$ and $\langle M^{\tau}\mathfrak{I}_\mathbb{B},N\rangle\in(\mathcal{A}_{\mathrm{loc}}\cap \mathcal{C})^\mathbb{B}$, respectively. Consequently, the equality $\langle M^{\tau}\mathfrak{I}_\mathbb{B},N^{\tau}\mathfrak{I}_\mathbb{B}\rangle=\langle M,N\rangle^{\tau}\mathfrak{I}_\mathbb{B}$ is obtained by
\[
\langle M^{\tau}\mathfrak{I}_\mathbb{B},N^{\tau}\mathfrak{I}_\mathbb{B}\rangle I_{\mathbb{B}\llbracket{0,T_n}\rrbracket}
=\langle (M^{(n)})^{\tau},(N^{(n)})^{\tau}\rangle I_{\mathbb{B}\llbracket{0,T_n}\rrbracket}
=\langle M^{(n)},N^{(n)}\rangle^{\tau}I_{\mathbb{B}\llbracket{0,T_n}\rrbracket}
=\left(\langle M,N\rangle^{\tau}\mathfrak{I}_\mathbb{B}\right)I_{\mathbb{B}\llbracket{0,T_n}\rrbracket}
\]
for each $n\in \mathbf{N}^{+}$. Similarly, the equality $\langle M^{\tau}\mathfrak{I}_\mathbb{B},N^{\tau}\mathfrak{I}_\mathbb{B}\rangle=\langle M^{\tau}\mathfrak{I}_\mathbb{B},N\rangle$ is derived from
\[
\langle M^{\tau}\mathfrak{I}_\mathbb{B},N^{\tau}\mathfrak{I}_\mathbb{B}\rangle I_{\mathbb{B}\llbracket{0,T_n}\rrbracket}
=\langle (M^{(n)})^{\tau},(N^{(n)})^{\tau}\rangle I_{\mathbb{B}\llbracket{0,T_n}\rrbracket}
=\langle (M^{(n)})^{\tau},N^{(n)}\rangle I_{\mathbb{B}\llbracket{0,T_n}\rrbracket}
=\langle M^{\tau}\mathfrak{I}_\mathbb{B},N\rangle I_{\mathbb{B}\llbracket{0,T_n}\rrbracket}
\]
for each $n\in \mathbf{N}^{+}$. The equality $\langle M^{\tau},N^{\tau}\rangle\mathfrak{I}_\mathbb{B}=\langle M,N\rangle^{\tau}\mathfrak{I}_\mathbb{B}$ is easily deduced by \eqref{Mc-tau1}. Therefore, we establish the validity of \eqref{Mc-tau2}.
$\hfill\blacksquare$

%%%%%%%%%%%%%%%%%%%%%%%%%%%%%%%%%%%%%%%%%%%%%%%%%%%%%%%%%%%%%%%%%%%%%%%%%%%%%%%%%%%%%%%%%%%%%%%%%%%%%%%
\par\vspace{0.3cm}
%%%%%%%%%%%%%%%%%%%%%%%%%%%%%%%%%%%%%%%%%%%%%%%%%%%%%%%%%%%%%%%%%%%%%%%%%%%%%%%%%%%%%%%%%%%%%%%%%%%%%%%
\noindent {\bf Proof of Theorem \ref{Lem-M}.}
In order to establish the proof, we commence by introducing a pivotal lemma.
\begin{lemma}\label{LM-A3}
If $M\in (\mathcal{M}^c_{\mathrm{loc},0})^\mathbb{B}\cap(\mathcal{M}^d_{\mathrm{loc}})^\mathbb{B}$, then $M=0\mathfrak{I}_\mathbb{B}$.
\end{lemma}
\begin{proof}
Suppose that $(T_n,M^{(n)})$ and $(T_n,N^{(n)})$ are FCSs for $M\in (\mathcal{M}^c_{\mathrm{loc},0})^\mathbb{B}$ and $M\in (\mathcal{M}^d_{\mathrm{loc}})^\mathbb{B}$, respectively. Fix $k\in \mathbf{N}^+$. Using $M^{(k)}\in \mathcal{M}^c_{\mathrm{loc},0}$ and \eqref{PQ-1}, we deduce
\[
(N^{(k)})^{T_k\wedge (T_F-)}=(M^{(k)})^{T_k\wedge (T_F-)}=(M^{(k)})^{T_k}\in\mathcal{M}^c_{\mathrm{loc},0}.
\]
The fact $(N^{(k)})^{T_k\wedge (T_F-)}=(N^{(k)})^{T_k}-\Delta(N^{(k)})^{T_k}_{T_F}I_{\llbracket{T_F,+\infty}\llbracket}$ implies $\Delta(N^{(k)})^{T_k}_{T_F}I_{\llbracket{T_F,+\infty}\llbracket}\in\mathcal{M}_{\mathrm{loc}}$. Then it follows from Theorem 7.19 and Definition 7.21 in \cite{He} that $\Delta(N^{(k)})^{T_k}_{T_F}I_{\llbracket{T_F,+\infty}\llbracket}\in\mathcal{M}^d_{\mathrm{loc}}$. Hence, we have
\[
(N^{(k)})^{T_k\wedge (T_F-)}\in\mathcal{M}^c_{\mathrm{loc},0}\cap\mathcal{M}^d_{\mathrm{loc}},
\]
which, by Lemma 7.22 in \cite{He}, implies $(N^{(k)})^{T_k\wedge (T_F-)}=0$. Consequently, $M=0\mathfrak{I}_\mathbb{B}$ is easily deduced.
\end{proof}

$(1)$ The existence of the decomposition \eqref{con-M} has been shown in Theorem 8.23 of \cite{He}, and we just prove its uniqueness. Let $M=M_0\mathfrak{I}_\mathbb{B}+N+L$ be another decomposition of $M$,
where $N\in (\mathcal{M}^c_{\mathrm{loc},0})^\mathbb{B}$ and $L\in (\mathcal{M}^d_{\mathrm{loc}})^\mathbb{B}$. Then the fact
$M^c-N=L-M^d$ yields
\[
M^c-N\in(\mathcal{M}^c_{\mathrm{loc},0})^\mathbb{B}\cap(\mathcal{M}^d_{\mathrm{loc}})^\mathbb{B},
\]
which, by Lemma \ref{LM-A3}, implies $M^c=N$ and $M^d=L$. Hence, the decomposition \eqref{con-M} is unique.

$(2)$ Let $\mathbb{B}$ be given by \eqref{B}, and assume that $(T_n,M^{(n)})$ is an FCS for $M\in (\mathcal{M}_{\mathrm{loc}})^\mathbb{B}$.

Fix $m\in \mathbf{N}^+$, and define $\mathbb{B}_m=\mathbb{B}\llbracket{0,T_m}\rrbracket$. Using the fact
\[
\mathbb{B}_m=\llbracket{0,T_F}\llbracket\;\cap\;\llbracket{0,T_m}\rrbracket,
\]
it is easy to see that $\mathbb{B}_m$ is an optional set of interval type.  Since both $M\mathfrak{I}_{\mathbb{B}_m}$ and $M^{(m)}\mathfrak{I}_{\mathbb{B}_m}$ are $\mathbb{B}_m$-local martingales, \eqref{con-M} yields the following decompositions:
\[
\left\{
\begin{aligned}
M\mathfrak{I}_{\mathbb{B}_m}&=(M_0\mathfrak{I}_\mathbb{B}+M^c+M^d)\mathfrak{I}_{\mathbb{B}_m}=M_0\mathfrak{I}_{\mathbb{B}_m}+M^c\mathfrak{I}_{\mathbb{B}_m}+M^d\mathfrak{I}_{\mathbb{B}_m},\\
M^{(m)}\mathfrak{I}_{\mathbb{B}_m}&=\left(M_0+(M^{(m)})^c+(M^{(m)})^d\right)\mathfrak{I}_{\mathbb{B}_m}=M_0\mathfrak{I}_{\mathbb{B}_m}+(M^{(m)})^c\mathfrak{I}_{\mathbb{B}_m}+(M^{(m)})^d\mathfrak{I}_{\mathbb{B}_m}.
\end{aligned}
\right.
\]
From $M\mathfrak{I}_{\mathbb{B}_m}=M^{(m)}\mathfrak{I}_{\mathbb{B}_m}$,  the uniqueness of above decompositions deduces
\begin{equation}\label{tm11}
M^c\mathfrak{I}_{\mathbb{B}_m}=(M^{(m)})^c\mathfrak{I}_{\mathbb{B}_m},
\quad M^d\mathfrak{I}_{\mathbb{B}_m}=(M^{(m)})^d\mathfrak{I}_{\mathbb{B}_m}.
\end{equation}

Since (\ref{tm11}) holds for any $m\in \mathbf{N}^+$, we have
\[
M^cI_{\mathbb{B}\llbracket{0,T_n}\rrbracket}=(M^{(n)})^cI_{\mathbb{B}\llbracket{0,T_n}\rrbracket},\quad M^dI_{\mathbb{B}\llbracket{0,T_n}\rrbracket}=(M^{(n)})^dI_{\mathbb{B}\llbracket{0,T_n}\rrbracket},\quad n\in \mathbf{N}^+.
\]
Then $(T_n,(M^{(n)})^c)$ and $(T_n,(M^{(n)})^d)$ are CSs for $M^c$ and $M^d$, respectively. Finally, the assertion is proven by the facts $(M^{(n)})^c\in\mathcal{M}^c_{\mathrm{loc},0}$ and $(M^{(n)})^d\in\mathcal{M}^d_{\mathrm{loc}}$ for each $n\in \mathbf{N}^+$.

$(3)$ From Theorem \ref{fcs}, $M^\tau\in\mathcal{M}_{\mathrm{loc}}$ and it admits a unique decomposition
$M^\tau=M_0+(M^\tau)^c+(M^\tau)^d$. Using $M=M_0\mathfrak{I}_\mathbb{B}+M^c+M^d$, we also have another decomposition $M^\tau=M_0+(M^c)^\tau+(M^d)^\tau$. Hence, (\ref{McS}) is obtained by the uniqueness of the decomposition of $M^\tau$.

From part (1), $M^\tau\mathfrak{I}_\mathbb{B}\in (\mathcal{M}_{\mathrm{loc}})^\mathbb{B}$ admits a unique decomposition $M^\tau\mathfrak{I}_\mathbb{B}=M_0\mathfrak{I}_\mathbb{B}+(M^\tau\mathfrak{I}_\mathbb{B})^c+(M^\tau\mathfrak{I}_\mathbb{B})^d$. Using (\ref{McS}) and the fact $M^\tau=M_0+(M^\tau)^c+(M^\tau)^d$, we also deduce that $M^\tau\mathfrak{I}_\mathbb{B}=M_0\mathfrak{I}_\mathbb{B}+(M^\tau)^c\mathfrak{I}_\mathbb{B}+(M^\tau)^d\mathfrak{I}_\mathbb{B}$ and
$M^\tau\mathfrak{I}_\mathbb{B}=M_0\mathfrak{I}_\mathbb{B}+(M^c)^\tau\mathfrak{I}_\mathbb{B}+(M^d)^\tau\mathfrak{I}_\mathbb{B}$. Hence, (\ref{McSB}) is derived through the uniqueness of the decomposition of $M^\tau\mathfrak{I}_\mathbb{B}\in (\mathcal{M}_{\mathrm{loc}})^\mathbb{B}$.
$\hfill\blacksquare$

%%%%%%%%%%%%%%%%%%%%%%%%%%%%%%%%%%%%%%%%%%%%%%%%%%%%%%%%%%%%%%%%%%%%%%%%%%%%%%%%%%%%%%%%%%%%%%%%%%%%%%%
\par\vspace{0.3cm}
%%%%%%%%%%%%%%%%%%%%%%%%%%%%%%%%%%%%%%%%%%%%%%%%%%%%%%%%%%%%%%%%%%%%%%%%%%%%%%%%%%%%%%%%%%%%%%%%%%%%%%%
\noindent {\bf Proof of Theorem \ref{MC}.}
Let $\mathbb{B}$ be given by \eqref{B}.

$(1)$ From Theorem \ref{process}(3), both $(\tau_n,N^{(n)})$ and $(\tau_n,M^{(n)})$ are FCSs for $M\in(\mathcal{M}_{\mathrm{loc}})^{\mathbb{B}}$. Then it follows that
\[
(N^{(n)})^{\tau_n\wedge (T_F-)}=(M^{(n)})^{\tau_n\wedge (T_F-)}=\left((M^{(n)})^{T_n\wedge (T_F-)}\right)^{S_n}
\in\mathcal{M}_{\mathrm{loc}},\quad n\in \mathbf{N}^{+}.
\]
Hence, $(\tau_n,N^{(n)})$ is indeed an inner FCS for $M\in(\mathcal{M}_{\mathrm{loc}})^{i,\mathbb{B}}$.

$(2)$ It suffices to prove $(\mathcal{M}_{\mathrm{loc}})^{i,\mathbb{B}}\cap\mathcal{C}^\mathbb{B}\subseteq(\mathcal{M}^c_{\mathrm{loc}})^\mathbb{B}$. Let us further suppose $M\in(\mathcal{M}_{\mathrm{loc}})^{i,\mathbb{B}}\cap\mathcal{C}^\mathbb{B}$. Assume that $(T_n,M^{(n)})$ is an inner FCS for $M\in (\mathcal{M}_{\mathrm{loc}})^{i,\mathbb{B}}$, and that $(T_n,N^{(n)})$ is an FCS for $M\in \mathcal{C}^\mathbb{B}$. For each $n\in \mathbf{N}^+$, it holds that
\[
M^{(n)}I_{\mathbb{B}\llbracket{0,T_n}\rrbracket}=MI_{\mathbb{B}\llbracket{0,T_n}\rrbracket}
=N^{(n)}I_{\mathbb{B}\llbracket{0,T_n}\rrbracket},
\]
which is equivalent to the statement
\[
(M^{(n)})^{T_n\wedge (T_F-)}=(N^{(n)})^{T_n\wedge (T_F-)}.
\]
Given that $(M^{(n)})^{T_n\wedge (T_F-)}\in\mathcal{M}_{\mathrm{loc}}$ and $(N^{(n)})^{T_n\wedge (T_F-)}\in\mathcal{C}$, it follows that
\[
(M^{(n)})^{T_n\wedge (T_F-)}\in \mathcal{M}_{\mathrm{loc}}\cap\mathcal{C}=\mathcal{M}^c_{\mathrm{loc}}.
\]
Since $(T_n,(M^{(n)})^{T_n\wedge (T_F-)})$ constitutes a CS for $M$, we deduce that $M\in(\mathcal{M}^c_{\mathrm{loc}})^\mathbb{B}$, and $(\mathcal{M}_{\mathrm{loc}})^{i,\mathbb{B}}\cap\mathcal{C}^\mathbb{B}\subseteq(\mathcal{M}^c_{\mathrm{loc}})^\mathbb{B}$.

$(3)$ Suppose that $(T_n,M^{(n)})$ and $(T_n,N^{(n)})$ are inner FCSs for $M\in (\mathcal{M}_{\mathrm{loc}})^{i,\mathbb{B}}$ and $N\in (\mathcal{M}_{\mathrm{loc}})^{i,\mathbb{B}}$, respectively.
It is easy to deduce that for each $n\in \mathbf{N}^{+}$, $aM^{(n)}+bN^{(n)}\in\mathcal{M}_{\mathrm{loc}}$ and
\begin{align*}
(aM+bN)I_{\mathbb{B}\llbracket{0,T_n}\rrbracket}
&=aMI_{\mathbb{B}\llbracket{0,T_n}\rrbracket}+bNI_{\mathbb{B}\llbracket{0,T_n}\rrbracket}\\
&=aM^{(n)}I_{\mathbb{B}\llbracket{0,T_n}\rrbracket}+bN^{(n)}I_{\mathbb{B}\llbracket{0,T_n}\rrbracket}\\
&=(aM^{(n)}+bN^{(n)})I_{\mathbb{B}\llbracket{0,T_n}\rrbracket}.
\end{align*}
This implies that $(T_n,aM^{(n)}+bN^{(n)})$ is an FCS for $aM+bN\in (\mathcal{M}_{\mathrm{loc}})^\mathbb{B}$. Consequently, by invoking the facts
\begin{align*}
(aM^{(n)}+bN^{(n)})^{T_n\wedge (T_F-)}
=a(M^{(n)})^{T_n\wedge (T_F-)}+b(N^{(n)})^{T_n\wedge (T_F-)}\in\mathcal{M}_{\mathrm{loc}},\quad n\in \mathbf{N}^{+},
\end{align*}
we can establish the assertions.

$(4)$ Observe that $M^d=M-M_0\mathfrak{I}_\mathbb{B}-M^c$ with $M^c\in(\mathcal{M}^{c}_{\mathrm{loc,0}})^\mathbb{B}\subseteq(\mathcal{M}_{\mathrm{loc}})^{i,\mathbb{B}}$. Then the assertion follows directly from part (3) and Theorem \ref{Lem-M}(2).

$(5)$ By employing induction, the assertion can be easily derived from \eqref{continuation}. $\hfill\blacksquare$

%%%%%%%%%%%%%%%%%%%%%%%%%%%%%%%%%%%%%%%%%%%%%%%%%%%%%%%%%%%%%%%%%%%%%%%%%%%%%%%%%%%%%%%%%%%%%%%%%%%%%%%
\par\vspace{0.3cm}
%%%%%%%%%%%%%%%%%%%%%%%%%%%%%%%%%%%%%%%%%%%%%%%%%%%%%%%%%%%%%%%%%%%%%%%%%%%%%%%%%%%%%%%%%%%%%%%%%%%%%%%
\noindent {\bf Proof of Theorem \ref{[M]-property}.}
Let $\mathbb{B}$ be given by \eqref{B}.

$(1)$ The assertion that $[M,N]\in \mathcal{V}^\mathbb{B}$ with the FCS $(T_n,[M^{(n)},N^{(n)}])$ is a direct consequence of Theorems \ref{Lem-M}(2), \ref{property-qr-p}(1) and \ref{delta}(6). The relation $[M]=[M,M]$ indicates that $(T_n,[M^{(n)}])$ is a CS for $[M]$, which in turn implies that $(T_n,\sqrt{[M^{(n)}]})$ is a CS for $\sqrt{[M]}$.
For each $n\in \mathbf{N}^{+}$, Definition 7.29 in \cite{He} asserts that $[M^{(n)}]\in\mathcal{V}^+$, and Theorem 7.30 in \cite{He} establishes that $\sqrt{[M^{(n)}]}\in \mathcal{A}^+_{\mathrm{loc}}$. Consequently, it follows that $[M]\in (\mathcal{V}^+)^\mathbb{B}$ with the FCS $(T_n,[M^{(n)}])$, and that $\sqrt{[M]}\in (\mathcal{A}^+_{\mathrm{loc}})^\mathbb{B}$ with the FCS $(T_n,\sqrt{[M^{(n)}]})$.

$(2)$ The initial equality is straightforward, and by appealing to part (1), the demonstration of the second equality proceeds analogously to the proof of Theorem \ref{property-qr-p}(2).

$(3)$ Using \eqref{sigmaX}, \eqref{Mc-tau1}, \eqref{McS} and \eqref{[M,N]}, we deduce
\begin{align*}
[M^{\tau},N^{\tau}]
&=M^{\tau}_0N^{\tau}_0+\langle (M^{\tau})^c, (N^{\tau})^c \rangle+\Sigma (\Delta M^{\tau}\Delta N^{\tau})\\
&=M^{\tau}_0N^{\tau}_0+\langle (M^c)^{\tau}, (N^c)^{\tau} \rangle+\Sigma (\Delta M\Delta NI_{\llbracket{0,\tau}\rrbracket})\\
&=\left(M_0N_0\mathfrak{I}_\mathbb{B}+\langle M^c,N^c\rangle+\Sigma (\Delta M\Delta N)\right)^{\tau}\\
&=[M,N]^{\tau},
\end{align*}
which yields \eqref{MN}. By part (1), the proof of \eqref{MNB} is analogous to that of \eqref{Mc-tau2}.

$(4)$ Suppose that $(T_n,M^{(n)})$ and $(T_n,N^{(n)})$ are inner FCSs for $M\in (\mathcal{M}_{\mathrm{loc}})^{i,\mathbb{B}}$ and $N\in(\mathcal{M}_{\mathrm{loc}})^{i,\mathbb{B}}$, respectively. According to part $(1)$, $(T_n,[M^{(n)},N^{(n)}])$ constitutes an FCS for $[M,N]\in \mathcal{V}^\mathbb{B}$. Then it follows that
\[
(M^{(n)}N^{(n)}-[M^{(n)},N^{(n)}])I_{\mathbb{B}\llbracket{0,T_n}\rrbracket}
=(MN-[M,N])I_{\mathbb{B}\llbracket{0,T_n}\rrbracket}, \quad n\in \mathbf{N}^{+},
\]
which implies that $(T_n,M^{(n)}N^{(n)}-[M^{(n)},N^{(n)}])$ is a CS for $MN-[M,N]$.
For each $n\in \mathbf{N}^{+}$, Theorem 7.31 in \cite{He} demonstrates that $M^{(n)}N^{(n)}-[M^{(n)},N^{(n)}]\in\mathcal{M}_{\mathrm{loc},0}$.  And it can be also verified that, for each $n\in \mathbf{N}^{+}$, $(M^{(n)}N^{(n)}-[M^{(n)},N^{(n)}])^{T_n\wedge (T_F-)}\in\mathcal{M}_{\mathrm{loc},0}$ by considering the relations $(M^{(n)})^{T_n\wedge (T_F-)}\in\mathcal{M}_{\mathrm{loc}}$ and $(N^{(n)})^{T_n\wedge (T_F-)}\in\mathcal{M}_{\mathrm{loc}}$, alongside the identity
\[
(M^{(n)}N^{(n)}-[M^{(n)},N^{(n)}])^{T_n\wedge (T_F-)}
=(M^{(n)})^{T_n\wedge (T_F-)}(N^{(n)})^{T_n\wedge (T_F-)}-[(M^{(n)})^{T_n\wedge (T_F-)},(N^{(n)})^{T_n\wedge (T_F-)}].
\]
Consequently, it holds that $MN-[M,N]\in(\mathcal{M}_{\mathrm{loc,0}})^{i,\mathbb{B}}$ with the inner FCS $(T_n,M^{(n)}N^{(n)}-[M^{(n)},N^{(n)}])$.
By applying Theorem \ref{delta}(1), we deduce
\[
\Delta [M,N]I_{\mathbb{B}\llbracket{0,T_n}\rrbracket}=\Delta[M^{(n)},N^{(n)}]I_{\mathbb{B}\llbracket{0,T_n}\rrbracket}
=\Delta M^{(n)}\Delta N^{(n)}I_{\mathbb{B}\llbracket{0,T_n}\rrbracket}=\Delta M\Delta NI_{\mathbb{B}\llbracket{0,T_n}\rrbracket}, \quad n\in \mathbf{N}^{+},
\]
which, by Theorem \ref{process}(1), yields $\Delta [M,N]=\Delta M\Delta N$.

Next, we show the uniqueness of $[M,N]$. Suppose that $A$ is another $\mathbb{B}$-process in $\mathcal{V}^\mathbb{B}$ such that $MN-A\in (\mathcal{M}_{\mathrm{loc,0}})^{i,\mathbb{B}}$ and $\Delta A=\Delta M\Delta N$. Put $L=[M,N]-A$. Then $L\in\mathcal{V}^\mathbb{B}\cap(\mathcal{M}_{\mathrm{loc,0}})^{i,\mathbb{B}}$. From Theorem \ref{delta}(5), the fact $\Delta L=\Delta [M,N]-\Delta A=0\mathfrak{I}_\mathbb{B}$ implies $L\in \mathcal{C}^\mathbb{B}$. Theorem \ref{MC}(2) shows $L\in\mathcal{V}^\mathbb{B}\cap(\mathcal{M}^c_{\mathrm{loc,0}})^{\mathbb{B}}$, and Lemma \ref{LM-A2} then yields $L=0\mathfrak{I}_\mathbb{B}$. Hence, the proof is concluded.

$(5)$ {\it Sufficiency}. Let $M=0\mathfrak{I}_\mathbb{B}$. It is evident that $(T_n=T,M^{(n)}=0)$ is an FCS for $M\in (\mathcal{M}_{\mathrm{loc}})^\mathbb{B}$. From part $(1)$, $(T_n,[M^{(n)}]=0)$ is an FCS for $[M]\in (\mathcal{V}^+)^\mathbb{B}$. Consequently, it follows that
\[
[M]I_{\mathbb{B}\llbracket{0,T_n}\rrbracket}=0I_{\mathbb{B}\llbracket{0,T_n}\rrbracket}=0, \quad n\in \mathbf{N}^{+},
\]
which, by Theorem \ref{process}(1), yields $[M]=0\mathfrak{I}_\mathbb{B}$.

{\it Necessity}. Let $[M]=0\mathfrak{I}_\mathbb{B}$. Suppose that $(T_n,M^{(n)})$ is an inner FCS for $M\in (\mathcal{M}_{\mathrm{loc}})^{i,\mathbb{B}}$. According to part $(1)$, we obtain the equalities
\[
[M^{(n)}]I_{\mathbb{B}\llbracket{0,T_n}\rrbracket}=[M]I_{\mathbb{B}\llbracket{0,T_n}\rrbracket}=0, \quad n\in \mathbf{N}^{+}.
\]
By Definition 8.2 in \cite{He}, we further deduce
\[
[(M^{(n)})^{T_n\wedge (T_F-)}]=[M^{(n)}]^{T_n\wedge (T_F-)}=0, \quad n\in \mathbf{N}^{+},
\]
Then for each $n\in \mathbf{N}^{+}$, the assertion $(M^{(n)})^{T_n\wedge (T_F-)}\in\mathcal{M}_{\mathrm{loc}}$ yields $(M^{(n)})^{T_n\wedge (T_F-)}=0$ (see, e.g., Definition 7.29 in \cite{He}). Consequently, it follows that the CS $(T_n,M^{(n)})$ for $M$ deduces
\[
MI_{\mathbb{B}\llbracket{0,T_n}\rrbracket}=M^{(n)}I_{\mathbb{B}\llbracket{0,T_n}\rrbracket}
=(M^{(n)})^{T_n\wedge (T_F-)}I_{\mathbb{B}\llbracket{0,T_n}\rrbracket}=0, \quad n\in \mathbf{N}^{+}.
\]
This ultimately leads to the conclusion $M=0\mathfrak{I}_\mathbb{B}$ by Theorem \ref{process}(1).
$\hfill\blacksquare$

%%%%%%%%%%%%%%%%%%%%%%%%%%%%%%%%%%%%%%%%%%%%%%%%%%%%%%%%%%%%%%%%%%%%%%%%%%%%%%%%%%%%%%%%%%%%%%%%%%%%%%%
\par\vspace{0.3cm}
%%%%%%%%%%%%%%%%%%%%%%%%%%%%%%%%%%%%%%%%%%%%%%%%%%%%%%%%%%%%%%%%%%%%%%%%%%%%%%%%%%%%%%%%%%%%%%%%%%%%%%%
\noindent {\bf Proof of Theorem \ref{HM-o-p}.}
Suppose that $N$ is an arbitrary $\mathbb{B}$-local martingale.

$(1)$ According to $H,\;K\in \mathcal{L}_m^\mathbb{B}(M)$, it follows that
\[
[H_{\bullet}M,N]=H_{\bullet}[M,N],\quad [K_{\bullet}M,N]=K_{\bullet}[M,N],
\]
which, by invoking Theorem \ref{HAproperty}(1) and Theorem \ref{[M]-property}(2), leads to the conclusion that
\[
[a(H_{\bullet}M)+b(K_{\bullet}M),N]=(aH+bK)_{\bullet}[M,N].
\]
Subsequently, the validity of the assertion is established by appealing to Definition \ref{HM} and the fact $a(H_{\bullet}M)+b(K_{\bullet}M)\in(\mathcal{M}_{\mathrm{loc}})^{i,\mathbb{B}}$.

$(2)$ From the relation $H\in \mathcal{L}_m^\mathbb{B}(M)\cap \mathcal{L}_m^\mathbb{B}(\widetilde{M})$, we have
\[
[H_{\bullet}M,N]=H_{\bullet}[M,N],\quad [H_{\bullet}\widetilde{M},N]=H_{\bullet}[\widetilde{M},N],
\]
which, by application of Theorem \ref{HAproperty}(2) and Theorem \ref{[M]-property}(2), implies
\[
[a(H_{\bullet}M)+b(H_{\bullet}\widetilde{M}),N]=H_{\bullet}[aM+b\widetilde{M},N].
\]
Consequently, invoking Definition \ref{HM} and the fact $a(H_{\bullet}M)+b(H_{\bullet}\widetilde{M})\in(\mathcal{M}_{\mathrm{loc}})^{i,\mathbb{B}}$, we establish the validity of the assertion.

$(3)$ By leveraging Theorem \ref{HAproperty}(3) and Definition \ref{HM}, the necessity can be verified by the following relations:
\begin{align}
&\widetilde{H}\in\mathcal{L}_m^\mathbb{B}(H_{\bullet}M)\nonumber\\
\Rightarrow\quad& [\widetilde{H}_{\bullet}(H_{\bullet}M),N]
=\widetilde{H}_{\bullet}[H_{\bullet}M,N]=\widetilde{H}_{\bullet}(H_{\bullet}[M,N])
=(\widetilde{H}H)_{\bullet}[M,N]\label{equi}\\
\Rightarrow\quad&\widetilde{H}H\in\mathcal{L}_m^\mathbb{B}(M);\nonumber
\end{align}
and the sufficiency can be established by the following relations:
\begin{align*}
&\widetilde{H}H\in\mathcal{L}_m^\mathbb{B}(M)\\
\Rightarrow\quad& [(\widetilde{H}H)_{\bullet}M,N]=(\widetilde{H}H)_{\bullet}[M,N]
=\widetilde{H}_{\bullet}(H_{\bullet}[M,N])=\widetilde{H}_{\bullet}[H_{\bullet}M,N]\\
\Rightarrow\quad&\widetilde{H}\in\mathcal{L}_m^\mathbb{B}(H_{\bullet}M).
\end{align*}
Furthermore, \eqref{hHM} is straightforwardly obtained as a consequence of the relation \eqref{equi}.
$\hfill\blacksquare$

%%%%%%%%%%%%%%%%%%%%%%%%%%%%%%%%%%%%%%%%%%%%%%%%%%%%%%%%%%%%%%%%%%%%%%%%%%%%%%%%%%%%%%%%%%%%%%%%%%%%%%%
\par\vspace{0.3cm}
%%%%%%%%%%%%%%%%%%%%%%%%%%%%%%%%%%%%%%%%%%%%%%%%%%%%%%%%%%%%%%%%%%%%%%%%%%%%%%%%%%%%%%%%%%%%%%%%%%%%%%%
\noindent {\bf Proof of Theorem \ref{HM=}.}
To prove the theorem, we initially present the following lemma.
\begin{lemma}\label{bounded}
Let $M\in (\mathcal{M}_{\mathrm{loc}})^{i,\mathbb{B}}$, and suppose that $H$ is a $\mathbb{B}$-locally bounded predictable process. Then $H\in\mathcal{L}_m^\mathbb{B}(M)$. Furthermore, if $\widetilde{H}$ is a coupled locally bounded predictable process for $H$, and if $(T_n,M^{(n)},\widetilde{M})$ is an inner continuation for $M\in (\mathcal{M}_{\mathrm{loc}})^{i,\mathbb{B}}$, then $(T_n,\widetilde{H}.\widetilde{M}^{T_n})$ forms an inner FCS for $H_{\bullet}M\in (\mathcal{M}_{\mathrm{loc}})^{i,\mathbb{B}}$.
\end{lemma}
\begin{proof}
Let $N$ be an arbitrary $\mathbb{B}$-local martingale, and let $\widetilde{H}$ be a coupled locally bounded predictable process for $H$. Suppose that $(T_n,N^{(n)},\widetilde{N})$ constitutes a continuation for $N\in (\mathcal{M}_{\mathrm{loc}})^{\mathbb{B}}$, and that $(T_n,M^{(n)},\widetilde{M})$ is an inner continuation for $M\in (\mathcal{M}_{\mathrm{loc}})^{i,\mathbb{B}}$. Given that $\widetilde{H}$ is a locally bounded predictable process, it follows from Theorem 9.2 in \cite{He} that $\widetilde{H}\in \mathcal{L}_m(\widetilde{M}^{T_n})$ for each $n\in\mathbf{N}^+$.
Define $\mathbb{C}=\bigcup\limits_{n}\llbracket{0,T_n}\rrbracket$ and $L=\widetilde{L}\mathfrak{I}_\mathbb{B}$, where
\[
\widetilde{L}=\left(\widetilde{H}_0\widetilde{M}_0I_{\llbracket{0}\rrbracket}
+\sum\limits_{n=1}^{{+\infty}}(\widetilde{H}.\widetilde{M}^{T_n})I_{\rrbracket{T_{n-1},T_n}
      \rrbracket}\right)\mathfrak{I}_{\mathbb{C}},\quad T_0=0.
\]

We begin by establishing that $L\in(\mathcal{M}_{\mathrm{loc}})^{i,\mathbb{B}}$ with the inner FCS $(T_n,\widetilde{H}.\widetilde{M}^{T_n})$. For $k,n\in \mathbf{N}^+$ with $n\leq k$, it is straightforward to observe that
\[
(\widetilde{H}.\widetilde{M}^{T_n})I_{\mathbb{B}\llbracket{0,T_n}\rrbracket}
=(\widetilde{H}.\widetilde{M}^{T_k})^{T_n}I_{\mathbb{B}\llbracket{0,T_n}\rrbracket}
=(\widetilde{H}.\widetilde{M}^{T_k})I_{\mathbb{B}\llbracket{0,T_n}\rrbracket}.
\]
By invoking Theorem \ref{restriction}(4), we deduce that $L\in(\mathcal{M}_{\mathrm{loc}})^\mathbb{B}$ with the FCS $(T_n,\widetilde{H}.\widetilde{M}^{T_n})$. According to Theorem \ref{MC}(5), $(T_n,\widetilde{M}^{T_n})$ constitutes an inner FCS for $M\in (\mathcal{M}_{\mathrm{loc}})^{i,\mathbb{B}}$. For each $n\in \mathbf{N}^+$, it holds that $\widetilde{M}^{T_n\wedge (T_F-)}\in\mathcal{M}_{\mathrm{loc}}$ and $\widetilde{H}\in \mathcal{L}_m(\widetilde{M}^{T_n\wedge (T_F-)})$. Consequently, it follows that
\[
(\widetilde{H}.\widetilde{M}^{T_n})^{T_n\wedge (T_F-)}
=\widetilde{H}.\widetilde{M}^{T_n\wedge (T_F-)}\in\mathcal{M}_{\mathrm{loc}}, \quad n\in \mathbf{N}^+.
\]
This further confirms that  $L\in(\mathcal{M}_{\mathrm{loc}})^{i,\mathbb{B}}$ with the inner FCS $(T_n,\widetilde{H}.\widetilde{M}^{T_n})$.

Subsequently, we demonstrate that $[L,N]=H_{\bullet}[M,N]$. For each $n\in \mathbf{N}^+$, we deduce
\begin{align*}
[L,N]I_{\mathbb{B}\llbracket{0,T_n}\rrbracket}
=&[\widetilde{H}.\widetilde{M}^{T_n},\widetilde{N}^{T_n}]I_{\mathbb{B}\llbracket{0,T_n}\rrbracket}\\
=&(\widetilde{H}.[\widetilde{M}^{T_n},\widetilde{N}^{T_n}])I_{\mathbb{B}\llbracket{0,T_n}\rrbracket}\\
=&(H_{\bullet}[M,N])I_{\mathbb{B}\llbracket{0,T_n}\rrbracket},
\end{align*}
where the first equality comes from Theorem \ref{[M]-property}(1), the second equality from the fact $\widetilde{H}\in\mathcal{L}_m(\widetilde{M}^{T_n})$, and the last equality from Theorem \ref{HA-FCS}.
Consequently, Theorem \ref{process}(1) implies that $[L,N]=H_{\bullet}[M,N]$.

Lastly, we directly infer the assertions based on Definition \ref{HM}.
\end{proof}

Let $\mathbb{B}$ be given by \eqref{B}. We present the proof.

(M1) $\Rightarrow$ (M2). Denote $L=H_{\bullet}M$, and fix a constant $\varepsilon>0$. Suppose that $\widetilde{H}$ is a coupled predictable process for $H\in\mathcal{P}^\mathbb{B}$, and that $(T_n,\widetilde{M}^{(n)},\widetilde{M})$ and $(T_n,\widetilde{L}^{(n)},\widetilde{L})$ are inner continuations for $M\in (\mathcal{M}_{\mathrm{loc}})^{i,\mathbb{B}}$ and $L\in (\mathcal{M}_{\mathrm{loc}})^{i,\mathbb{B}}$, respectively. By invoking Theorem \ref{HM-o-p} and Lemma \ref{bounded}, the $\mathbb{B}$-process $M$ can be represented as
\[
M=K_{\bullet}L+X_{\bullet}M,\quad K=\left(\frac{1}{H}I_{[|H|>\varepsilon]}\right)\mathfrak{I}_\mathbb{B},\;X=I_{[|H|\leq\varepsilon]}\mathfrak{I}_\mathbb{B},
\]
with $K$ and $X$ being $\mathbb{B}$-locally bounded predictable processes (according to Lemma \ref{bound-p}), ensuring the well-definedness of $K_{\bullet}L$ and $X_{\bullet}M$.
Next, define the processes:
\[
\widetilde{K}=\frac{1}{\widetilde{H}}I_{[|\widetilde{H}|>\varepsilon]},\quad \widetilde{X}=I_{[|\widetilde{H}|\leq\varepsilon]}.
\]
These are coupled locally bounded predictable processes for $K$ and $X$, respectively.
According to Lemma \ref{bounded}, $(T_n,\widetilde{K}.\widetilde{L}^{T_n})$ and $(T_n,\widetilde{X}.\widetilde{M}^{T_n})$ are inner FCSs for $K_{\bullet}L\in(\mathcal{M}_{\mathrm{loc}})^{i,\mathbb{B}}$ and $X_{\bullet}M\in(\mathcal{M}_{\mathrm{loc}})^{i,\mathbb{B}}$, respectively.
For each $n\in\mathbf{N}^+$, by setting
\[
M^{(n)}=\widetilde{K}.\widetilde{L}^{T_n}
+\widetilde{X}.\widetilde{M}^{T_n},
\]
it follows that $(M^{(n)})^{T_n\wedge(T_F-)}\in \mathcal{M}_{\mathrm{loc}}$ and
\begin{align*}
M^{(n)}I_{\mathbb{B}\llbracket{0,T_n}\rrbracket}
&=(\widetilde{K}.\widetilde{L}^{T_n}+\widetilde{X}.\widetilde{M}^{T_n})I_{\mathbb{B}\llbracket{0,T_n}\rrbracket}\\
&=(K_{\bullet}L+X_{\bullet}M)I_{\mathbb{B}\llbracket{0,T_n}\rrbracket}\\
&=MI_{\mathbb{B}\llbracket{0,T_n}\rrbracket},
\end{align*}
indicating that $(T_n,M^{(n)})$ is an inner FCS for $M\in (\mathcal{M}_{\mathrm{loc}})^{i,\mathbb{B}}$. For each $n\in\mathbf{N}^+$,  we infer $\widetilde{H}\in \mathcal{L}_m(M^{(n)})$ from
\[
(\widetilde{H}\widetilde{K}).\widetilde{L}^{T_n}
+(\widetilde{H}\widetilde{X}).\widetilde{M}^{T_n}=\widetilde{H}.(\widetilde{K}.\widetilde{L}^{T_n}
+\widetilde{X}.\widetilde{M}^{T_n})=\widetilde{H}.M^{(n)},
\]
where we use the fact that $\widetilde{H}\widetilde{K}$ and $\widetilde{H}\widetilde{X}$ are locally bounded predictable processes satisfying $\widetilde{H}\widetilde{K}\in\mathcal{L}_m(\widetilde{L}^{T_n})$ and  $\widetilde{H}\widetilde{X}\in\mathcal{L}_m(\widetilde{M}^{T_n})$.
In conclusion, $\widetilde{H}$ is a coupled predictable process for $H\in\mathcal{P}^\mathbb{B}$ and $(T_n,M^{(n)})$ is an inner FCS for $M\in (\mathcal{M}_{\mathrm{loc}})^{i,\mathbb{B}}$, such that $\widetilde{H}\in\mathcal{L}_m(M^{(n)})$ for each $n\in \mathbf{N}^+$, thereby establishing (M2).

(M2) $\Rightarrow$ (M1). Suppose that $(M2)$ holds. Define the following $\mathbb{B}$-process:
\begin{equation}\label{L=HM}
  L=\left((H_0M_0)I_{\llbracket{0}\rrbracket}
  +\sum\limits_{n=1}^{{+\infty}}(\widetilde{H}.{M}^{(n)})I_{\rrbracket{T_{n-1},T_n}
      \rrbracket}\right)\mathfrak{I}_\mathbb{B},\quad T_0=0.
\end{equation}

Our initial objective is to demonstrate that $L\in (\mathcal{M}_{\mathrm{loc}})^{i,\mathbb{B}}$ with the inner FCS $(T_n,\widetilde{H}.{M}^{(n)})$.
According to Theorem \ref{process}(2), for $k,\;l\in \mathbf{N}^+$ with $k\leq l$, the inner FCS $(T_n,M^{(n)})$ for $M\in (\mathcal{M}_{\mathrm{loc}})^{i,\mathbb{B}}$ implies that
$M^{(k)}I_{\mathbb{B}\llbracket{0,T_k}\rrbracket}=M^{(l)}I_{\mathbb{B}\llbracket{0,T_k}\rrbracket}$,
or equivalently,
\[
(M^{(k)})^{T_k\wedge(T_F-)}=(M^{(l)})^{T_k\wedge(T_F-)},
\]
which subsequently entails
\begin{align*}
(\widetilde{H}.M^{(k)})I_{\mathbb{B}\llbracket{0,T_k}\rrbracket}
&=(\widetilde{H}.M^{(k)})^{T_k\wedge(T_F-)}I_{\mathbb{B}\llbracket{0,T_k}\rrbracket}
=\left(\widetilde{H}.(M^{(k)})^{T_k\wedge(T_F-)}\right)I_{\mathbb{B}\llbracket{0,T_k}\rrbracket}\\
&=\left(\widetilde{H}.(M^{(l)})^{T_k\wedge(T_F-)}\right)I_{\mathbb{B}\llbracket{0,T_k}\rrbracket}
=(\widetilde{H}.M^{(l)})^{T_k\wedge(T_F-)}I_{\mathbb{B}\llbracket{0,T_k}\rrbracket}\\
&=(\widetilde{H}.M^{(l)})I_{\mathbb{B}\llbracket{0,T_k}\rrbracket}.
\end{align*}
Theorem \ref{restriction}(4) confirms that  $(T_n,\widetilde{H}.{M}^{(n)})$ is a CS for $L$. Furthermore, it is evident that for $n\in \mathbf{N}^+$, $\widetilde{H}.{M}^{(n)}\in \mathcal{M}_{\mathrm{loc}}$, $\widetilde{H}\in\mathcal{L}_m((M^{(n)})^{T_n\wedge(T_F-)})$, and
\[
(\widetilde{H}.M^{(n)})^{T_n\wedge(T_F-)}=\widetilde{H}.(M^{(n)})^{T_n\wedge(T_F-)}\in\mathcal{M}_{\mathrm{loc}}.
\]
Therefore, $(T_n,\widetilde{H}.{M}^{(n)})$ forms an inner FCS for $L\in (\mathcal{M}_{\mathrm{loc}})^{i,\mathbb{B}}$.

Next, we proceed to prove that $H\in\mathcal{L}^\mathbb{B}_m(M)$. Let $N\in (\mathcal{M}_{\mathrm{loc}})^{\mathbb{B}}$ with the FCS $(T_n, N^{(n)})$. By invoking Theorems \ref{HA-FCS} and \ref{[M]-property}(1), we establish that $(T_n,[\widetilde{H}.{M}^{(n)},N^{(n)}])$ and $(T_n,\widetilde{H}.[M^{(n)},N^{(n)}])$ are FCSs for $[L,N]\in \mathcal{V}^\mathbb{B}$ and $H_{\bullet}[M,N]\in \mathcal{V}^\mathbb{B}$, respectively. Consequently, for each $n\in \mathbf{N}^+$, we deduce that
\begin{align*}
[L,N]I_{\mathbb{B}\llbracket{0,T_n}\rrbracket}
=[\widetilde{H}.M^{(n)},N^{(n)}]I_{\mathbb{B}\llbracket{0,T_n}\rrbracket}
=(\widetilde{H}.[M^{(n)},N^{(n)}])I_{\mathbb{B}\llbracket{0,T_n}\rrbracket}
=(H_{\bullet}[M,N])I_{\mathbb{B}\llbracket{0,T_n}\rrbracket},
\end{align*}
which, by Theorem \ref{process}(1), implies $[L,N]=H_{\bullet}[M,N]$. Hence, we conclude that $H\in\mathcal{L}^\mathbb{B}_m(M)$.

(M2) $\Rightarrow$ (M3). Suppose that (M2) holds. For each $n\in \mathbf{N}^+$, define $H^{(n)}=\widetilde{H}$. It follows that $(T_n,H^{(n)})$ is an FCS for $H\in \mathcal{P}^\mathbb{B}$ such that $H^{(n)}\in\mathcal{L}_m(M^{(n)})$ for each $n\in \mathbf{N}^+$. Hence, the assertion (M3) is established.

(M3) $\Rightarrow$ (M4). Suppose that (M3) holds. For each $n\in \mathbf{N}^+$, according to the fact $H^{(n)}\in\mathcal{L}_m(M^{(n)})$, it is established that $(H^{(n)})^2\in\mathcal{L}_s([M^{(n)}])$. Theorems \ref{HA-FCS} and \ref{[M]-property}(1) show that $(T_n,(H^{(n)})^2.[{M}^{(n)}])$ is an FCS for ${H^2}_{\bullet}[M]\in \mathcal{V}^\mathbb{B}$, which gives
\[
\left((H^{(n)})^2.[{M}^{(n)}]\right)I_{\mathbb{B}\llbracket{0,T_n}\rrbracket}
=({H^2}_{\bullet}[M])I_{\mathbb{B}\llbracket{0,T_n}\rrbracket}, \quad n\in \mathbf{N}^+.
\]
This naturally yields
\[
\sqrt{(H^{(n)})^2.[{M}^{(n)}]}I_{\mathbb{B}\llbracket{0,T_n}\rrbracket}
=\sqrt{{H^2}_{\bullet}[M]}I_{\mathbb{B}\llbracket{0,T_n}\rrbracket}, \quad n\in \mathbf{N}^+,
\]
which, by the fact $\sqrt{(H^{(n)})^2.[{M}^{(n)}]}\in\mathcal{A}^+_{\mathrm{loc}}$ (see, e.g., Theorem 9.2 in \cite{He}), indicates $\sqrt{{H^2}_{\bullet}[M]}\in (\mathcal{A}^+_{\mathrm{loc}})^\mathbb{B}$.

(M4) $\Rightarrow$ (M2). Suppose that $\sqrt{{H^2}_{\bullet}[M]}\in (\mathcal{A}^+_{\mathrm{loc}})^\mathbb{B}$. Let $\widehat{H}$ denote a coupled predictable process for $H\in\mathcal{P}^\mathbb{B}$, and let $(T_n,N^{(n)})$ be an inner FCS for $M\in (\mathcal{M}_{\mathrm{loc}})^{i,\mathbb{B}}$. According to statement (A2), there exist a coupled predictable process $\widetilde{K}$ for $H^2\in\mathcal{P}^\mathbb{B}$ and an FCS $(T_n,V^{(n)})$ for $[M]\in \mathcal{V}^\mathbb{B}$, such that $(T_n,\sqrt{|\widetilde{K}|.V^{(n)}})$ is an FCS for $\sqrt{{H^2}_{\bullet}[M]}\in (\mathcal{A}^+_{\mathrm{loc}})^\mathbb{B}$, where $|\widetilde{K}|\in\mathcal{L}_s(V^{(n)})$ for each $n\in \mathbf{N}^+$. Define
\[
\widetilde{H}=\sqrt{|\widetilde{K}|}I_{[\widehat{H}\geq 0]}-\sqrt{|\widetilde{K}|}I_{[\widehat{H}< 0]},\quad
M^{(n)}=(N^{(n)})^{T_n\wedge(T_F-)},\quad n\in \mathbf{N}^+.
\]
It follows that $\widetilde{H}$ is a coupled predictable process for $H\in\mathcal{P}^\mathbb{B}$ satisfying $|\widetilde{K}|=\widetilde{H}^2$, and $(T_n,M^{(n)})$ is an inner FCS for $M\in (\mathcal{M}_{\mathrm{loc}})^{i,\mathbb{B}}$.
The FCSs $(T_n,V^{(n)})$ and $(T_n,[N^{(n)}])$ for $[M]\in \mathcal{V}^\mathbb{B}$ lead to the equality
\[
[N^{(n)}]I_{\mathbb{B}\llbracket{0,T_n}\rrbracket}=[M]I_{\mathbb{B}\llbracket{0,T_n}\rrbracket}
=V^{(n)}I_{\mathbb{B}\llbracket{0,T_n}\rrbracket}, \quad n\in \mathbf{N}^+.
\]
Then we infer that
\[
[M^{(n)}]=[(N^{(n)})^{T_n\wedge(T_F-)}]
=[N^{(n)}]^{T_n\wedge(T_F-)}=(V^{(n)})^{T_n\wedge(T_F-)}, \quad n\in \mathbf{N}^+.
\]
For each $n\in \mathbf{N}^+$, the relation
\[
(|\widetilde{K}|.V^{(n)})^{T_n\wedge(T_F-)}=\widetilde{H}^2.(V^{(n)})^{T_n\wedge(T_F-)}=\widetilde{H}^2.[M^{(n)}]
\]
demonstrates that $\widetilde{H}^2\in\mathcal{L}_s([M^{(n)}])$ and
\[
\sqrt{\widetilde{H}^2.[M^{(n)}]}=\left(\sqrt{|\widetilde{K}|.V^{(n)}}\right)^{T_n\wedge(T_F-)}\in\mathcal{A}^+_{\mathrm{loc}}.
\]
 Invoking Theorem 9.2 in \cite{He}, we conclude that $\widetilde{H}\in\mathcal{L}_m(M^{(n)})$ for each $n\in \mathbf{N}^+$, thereby establishing (M2).
$\hfill\blacksquare$

%%%%%%%%%%%%%%%%%%%%%%%%%%%%%%%%%%%%%%%%%%%%%%%%%%%%%%%%%%%%%%%%%%%%%%%%%%%%%%%%%%%%%%%%%%%%%%%%%%%%%%%
\par\vspace{0.3cm}
%%%%%%%%%%%%%%%%%%%%%%%%%%%%%%%%%%%%%%%%%%%%%%%%%%%%%%%%%%%%%%%%%%%%%%%%%%%%%%%%%%%%%%%%%%%%%%%%%%%%%%%
\noindent {\bf Proof of Theorem \ref{eq-HM}.}
Let $\mathbb{B}$ be given by \eqref{B}. Suppose that $(T_n,H^{(n)})$ is an FCS for $H\in\mathcal{P}^\mathbb{B}$, and $(T_n,M^{(n)})$ is an inner FCS for $M\in(\mathcal{M}_{\mathrm{loc}})^{i,\mathbb{B}}$, such that for each $n\in \mathbf{N}^+$, $H^{(n)}\in\mathcal{L}_m(M^{(n)})$. Define
\begin{equation}\label{HMn}
      L=\left((H_0M_0)I_{\llbracket{0}\rrbracket}+\sum\limits_{n=1}^{{+\infty}}
      (H^{(n)}.M^{(n)})I_{\rrbracket{T_{n-1},T_n}
      \rrbracket}\right)\mathfrak{I}_\mathbb{B},\quad T_0=0.
      \end{equation}
Utilizing the expression of \eqref{x-expression} and its independence property, it suffices to demonstrate that $L=H_{\bullet}M$ and that $(T_n,H^{(n)}.M^{(n)})$ is an inner FCS for $L\in(\mathcal{M}_{\mathrm{loc}})^{i,\mathbb{B}}$.

First, for any $k,\;l\in \mathbf{N}^+$ with $k\leq l$, we establish that
\begin{equation}\label{HMn-1}
H^{(l)}.(M^{(l)})^{T_k\wedge(T_F-)}=H^{(k)}.(M^{(k)})^{T_k\wedge(T_F-)}.
\end{equation}
By Theorem \ref{process}(2), the FCS $(T_n,H^{(n)})$ for $H\in\mathcal{P}^\mathbb{B}$ implies
\begin{equation}\label{HMn-2}
(H^{(l)}-H^{(k)})I_{\mathbb{B}\llbracket{0,T_k}\rrbracket}=0,
\end{equation}
while the inner FCS $(T_n,M^{(n)})$ for $M\in(\mathcal{M}_{\mathrm{loc}})^{i,\mathbb{B}}$ yields
\begin{equation}\label{HMn-3}
(M^{(k)})^{T_k\wedge(T_F-)}=(M^{(l)})^{T_k\wedge(T_F-)}\in\mathcal{M}_{\mathrm{loc}}.
\end{equation}
Given $H^{(k)}\in\mathcal{L}_m(M^{(k)})$ and $H^{(l)}\in\mathcal{L}_m(M^{(l)})$, it follows that
\[
H^{(k)}\in\mathcal{L}_m((M^{(k)})^{T_k\wedge(T_F-)})\quad \text{and}\quad H^{(l)}\in\mathcal{L}_m((M^{(l)})^{T_k\wedge(T_F-)})=\mathcal{L}_m((M^{(k)})^{T_k\wedge(T_F-)}).
\]
Combining these with \eqref{HMn-2}, we obtain
\[
[(H^{(l)}-H^{(k)}).(M^{(k)})^{T_k\wedge(T_F-)}]I_{\mathbb{B}\llbracket{0,T_k}\rrbracket}
=\left((H^{(l)}-H^{(k)})^2.[(M^{(k)})^{T_k\wedge(T_F-)}]\right)I_{\mathbb{B}\llbracket{0,T_k}\rrbracket}=0.
\]
Now we deduce that
\[
[(H^{(l)}-H^{(k)}).(M^{(k)})^{T_k\wedge(T_F-)}]=0,
\]
which, since $(H^{(l)}-H^{(k)}).(M^{(k)})^{T_k\wedge(T_F-)}\in \mathcal{M}_{\mathrm{loc}}$, implies
\[
(H^{(l)}-H^{(k)}).(M^{(k)})^{T_k\wedge(T_F-)}=0.
\]
Thus, \eqref{HMn-1} is derived from \eqref{HMn-3}.

Next, we show that $L\in(\mathcal{M}_{\mathrm{loc}})^{i,\mathbb{B}}$ with the inner FCS $(T_n,H^{(n)}.M^{(n)})$.
For $k,\;l\in \mathbf{N}^+$ with $k\leq l$, \eqref{HMn-1} gives
\[
(H^{(l)}.M^{(l)})I_{\mathbb{B}\llbracket{0,T_k}\rrbracket}=(H^{(k)}.M^{(k)})I_{\mathbb{B}\llbracket{0,T_k}\rrbracket},
\]
which, by Theorem \ref{restriction}(4), indicates that $(T_n,H^{(n)}.M^{(n)})$ is a CS for $L$. Since $(T_n,M^{(n)})$ is an inner FCS for $M\in(\mathcal{M}_{\mathrm{loc}})^{i,\mathbb{B}}$, for each $n\in \mathbf{N}^+$, it follows from $H^{(n)}\in\mathcal{L}_m(M^{(n)})$ that
$H^{(n)}\in\mathcal{L}_m((M^{(n)})^{T_n\wedge(T_F-)})$ and
\[
(H^{(n)}.M^{(n)})^{T_n\wedge(T_F-)}
=H^{(n)}.(M^{(n)})^{T_n\wedge(T_F-)}\in\mathcal{M}_{\mathrm{loc}}, \quad n\in \mathbf{N}^+.
\]
Consequently, it holds that $(T_n,H^{(n)}.M^{(n)})$ is an inner FCS for $L\in(\mathcal{M}_{\mathrm{loc}})^{i,\mathbb{B}}$.

Finally, we prove that $L=H_{\bullet}M$. Let $N\in (\mathcal{M}_{\mathrm{loc}})^\mathbb{B}$ with the FCS $(T_n, N^{(n)})$. According to Theorem \ref{HA-FCS} and Theorem \ref{[M]-property}(1), $(T_n,[H^{(n)}.{M}^{(n)},N^{(n)}])$ and $(T_n,H^{(n)}.[M^{(n)},N^{(n)}])$ are FCSs for $[L,N]\in \mathcal{V}^\mathbb{B}$ and $H_{\bullet}[M,N]\in \mathcal{V}^\mathbb{B}$, respectively. Therefore, for each $n\in \mathbf{N}^+$, we have
\begin{align*}
[L,N]I_{\mathbb{B}\llbracket{0,T_n}\rrbracket}
=[H^{(n)}.M^{(n)},N^{(n)}]I_{\mathbb{B}\llbracket{0,T_n}\rrbracket}
=(H^{(n)}.[M^{(n)},N^{(n)}])I_{\mathbb{B}\llbracket{0,T_n}\rrbracket}
=(H_{\bullet}[M,N])I_{\mathbb{B}\llbracket{0,T_n}\rrbracket},
\end{align*}
which, by Theorem \ref{process}(1), implies $[L,N]=H_{\bullet}[M,N]$. Hence, we establish the equality $L=H_{\bullet}M$, and finish the proof.
$\hfill\blacksquare$

%%%%%%%%%%%%%%%%%%%%%%%%%%%%%%%%%%%%%%%%%%%%%%%%%%%%%%%%%%%%%%%%%%%%%%%%%%%%%%%%%%%%%%%%%%%%%%%%%%%%%%%
\par\vspace{0.3cm}
%%%%%%%%%%%%%%%%%%%%%%%%%%%%%%%%%%%%%%%%%%%%%%%%%%%%%%%%%%%%%%%%%%%%%%%%%%%%%%%%%%%%%%%%%%%%%%%%%%%%%%%
\noindent {\bf Proof of Corollary \ref{bound-HM}.}
Let $(T_n,M^{(n)})$ denote an inner FCS for $M\in(\mathcal{M}_{\mathrm{loc}})^{i,\mathbb{B}}$. Suppose that $\widetilde{H}$ is a coupled locally bounded predictable process for $H$, and that $(T_n,H^{(n)})$ is an FCS for $H$ (a $\mathbb{B}$-locally bounded predictable process). For each $n\in \mathbf{N}^+$, it holds that $\widetilde{H}\in\mathcal{L}_m(M^{(n)})$ and $H^{(n)}\in\mathcal{L}_m(M^{(n)})$ (see, e.g., Theorem 9.2 in \cite{He}). Then the statements are inferred from Theorems \ref{HM=} and \ref{eq-HM}.
$\hfill\blacksquare$

%%%%%%%%%%%%%%%%%%%%%%%%%%%%%%%%%%%%%%%%%%%%%%%%%%%%%%%%%%%%%%%%%%%%%%%%%%%%%%%%%%%%%%%%%%%%%%%%%%%%%%%
\par\vspace{0.3cm}
%%%%%%%%%%%%%%%%%%%%%%%%%%%%%%%%%%%%%%%%%%%%%%%%%%%%%%%%%%%%%%%%%%%%%%%%%%%%%%%%%%%%%%%%%%%%%%%%%%%%%%%
\noindent {\bf Proof of Corollary \ref{HMc=}.}
The implications $(i)\Rightarrow (ii)\Rightarrow (iii)$ follow straightforwardly from Theorem \ref{HM=} and the fact $(\mathcal{M}_{\mathrm{loc}})^\mathbb{C}=(\mathcal{M}_{\mathrm{loc}})^{i,\mathbb{C}}$; furthermore, the implication
$(iv)\Rightarrow (i)$ is a direct result of Theorem \ref{HM=} and Corollary \ref{fcs-p}(2). Consequently, it remains to demonstrate the validity of the implication $(iii)\Rightarrow (iv)$. Suppose that $(iii)$ holds. Let $(S_n)$ be an FS for $\mathbb{C}$, and define $(\tau_n)=(S_n\wedge T_n)$. According to Corollary \ref{process-FS}(2), $(\tau_n)$ is an FS for $\mathbb{C}$ satisfying
\[
H^{\tau_n}=(H^{(n)})^{\tau_n},\quad M^{\tau_n}=(M^{(n)})^{\tau_n},\quad n\in \mathbf{N}^+.
\]
Subsequently, for each $n\in \mathbf{N}^+$, by invoking the equality
\[
(H^{(n)}.M^{(n)})^{\tau_n}=(H^{(n)})^{\tau_n}.(M^{(n)})^{\tau_n},
\]
we establish that $H^{\tau_n}\in\mathcal{L}_m(M^{\tau_n})$.
This derivation confirms the validity of statement $(iv)$.
$\hfill\blacksquare$

%%%%%%%%%%%%%%%%%%%%%%%%%%%%%%%%%%%%%%%%%%%%%%%%%%%%%%%%%%%%%%%%%%%%%%%%%%%%%%%%%%%%%%%%%%%%%%%%%%%%%%%
\par\vspace{0.3cm}
%%%%%%%%%%%%%%%%%%%%%%%%%%%%%%%%%%%%%%%%%%%%%%%%%%%%%%%%%%%%%%%%%%%%%%%%%%%%%%%%%%%%%%%%%%%%%%%%%%%%%%%
\noindent {\bf Proof of Corollary \ref{eq-HMc}.}
$(1)$ Let $(\tau_n)$ be an FS for $\mathbb{C}$. By virtue of Corollary \ref{fcs-p}(2), $(\tau_n, H^{\tau_n})$ is an FCS for $H\in\mathcal{P}^\mathbb{C}$, and $(\tau_n, M^{\tau_n})$ forms an inner FCS for $M\in(\mathcal{M}_{\mathrm{loc}})^{i,\mathbb{C}}=(\mathcal{M}_{\mathrm{loc}})^\mathbb{C}$.
Subsequently, it can be established that for each $n\in \mathbf{N}^+$,
\[
\sqrt{(H^{\tau_n})^2.[M^{\tau_n}]}=\sqrt{(H^2)^{\tau_n}.[M]^{\tau_n}}
=\left(\sqrt{{H^2}_{\bullet}[M]}\right)^{\tau_n}\in\mathcal{A}^+_{\mathrm{loc}},
\]
where the first equality is derived from \eqref{MN}, the second equality stems from Corollary \ref{HAp-equivalent}, and the last relation is a consequence of statement (M4) and Corollary \ref{fcs-p}(2). By invoking Theorem 9.2 of \cite{He}, we deduce that $H^{\tau_n}\in\mathcal{L}_m(M^{\tau_n})$ for each $n\in \mathbf{N}^+$.
As a result, the assertion can be easily verified by appealing to Theorem \ref{eq-HM}.

$(2)$ Let $(\tau_n)$ be an FS for $\mathbb{C}$, and set $\tau_0=0$. Define $(S_n)=(\tau_n\wedge T_n)$.
According to Corollary \ref{process-FS}(2), it follows that $(S_n)$ is also an FS for $\mathbb{C}$, fulfilling the conditions:
\[
H^{S_n}=(H^{(n)})^{S_n},\quad M^{S_n}=(M^{(n)})^{S_n},\quad n\in \mathbf{N}^+.
\]
Based on part (1), we can conclude that $(S_n,H^{S_n}.M^{S_n})$ constitutes an FCS for $H_{\bullet}M\in(\mathcal{M}_{\mathrm{loc}})^\mathbb{C}$.

For every $n\in \mathbf{N}^+$ and $i\in \mathbf{N}^+$, the FCS $(T_n,H^{(n)})$ for $H\in\mathcal{P}^\mathbb{C}$ gives rise to the following equality:
\begin{align*}
H^{(n)}I_{\llbracket{0,T_n}\rrbracket}I_{\llbracket{0,\tau_i}\rrbracket}
=HI_{\llbracket{0,T_n}\rrbracket}I_{\llbracket{0,\tau_i}\rrbracket},
\end{align*}
which entails that $(H^{(n)})^{T_n\wedge \tau_i}=H^{T_n\wedge \tau_i}$. Analogously, it holds that $(M^{(n)})^{T_n\wedge \tau_i}=M^{T_n\wedge \tau_i}$ for every $n\in \mathbb{N}^+$ and $i\in \mathbb{N}^+$.
Utilizing \eqref{HM-expression0} and the fact
\[
\mathbb{C}=\bigcup\limits_{i=1}^{{+\infty}}\llbracket{0,\tau_i}\rrbracket = \llbracket{0}\rrbracket\cup \left(\bigcup\limits_{i=1}^{{+\infty}}\rrbracket{\tau_{i-1},\tau_i}\rrbracket\right), \]
we can infer that, for each $n\in\mathbf{N}^+$,
\begin{align*}
&(H^{(n)}.M^{(n)})I_{\mathbb{C}\llbracket{0,T_n}\rrbracket}\\
=&H_0M_0I_{\llbracket{0}\rrbracket}+\sum_{i=1}^{\infty}(H^{(n)}.M^{(n)})I_{\llbracket{0,T_n}\rrbracket}I_{\rrbracket{\tau_{i-1},\tau_i}\rrbracket}\\
=&H_0M_0I_{\llbracket{0}\rrbracket}+\sum_{i=1}^{\infty}\left((H^{(n)})^{T_n\wedge \tau_i}.(M^{(n)})^{T_n\wedge \tau_i}\right)I_{\llbracket{0,T_n}\rrbracket}I_{\rrbracket{\tau_{i-1},\tau_i}\rrbracket}\\
=&H_0M_0I_{\llbracket{0}\rrbracket}+\sum_{i=1}^{\infty}\left(H^{T_n\wedge \tau_i}.M^{T_n\wedge \tau_i}\right)I_{\llbracket{0,T_n}\rrbracket}I_{\rrbracket{\tau_{i-1},,\tau_i}\rrbracket}\\
=&\left((H_0M_0)I_{\llbracket{0}\rrbracket}+\sum\limits_{i=1}^{{+\infty}}
      (H^{\tau_i}.M^{\tau_i})I_{\rrbracket{\tau_{i-1},\tau_i}
      \rrbracket}\right)I_{\mathbb{C}\llbracket{0,T_n}\rrbracket}\\
=&(H_{\bullet}M)I_{\mathbb{C}\llbracket{0,T_n}\rrbracket}.
\end{align*}
Consequently, $(T_n,H^{(n)}.M^{(n)})$ is an FCS for $H_{\bullet}M\in(\mathcal{M}_{\mathrm{loc}})^\mathbb{C}$, and the expression
\eqref{HMc-expression-2} can be straightforwardly derived from \eqref{x-expression}.
$\hfill\blacksquare$

%%%%%%%%%%%%%%%%%%%%%%%%%%%%%%%%%%%%%%%%%%%%%%%%%%%%%%%%%%%%%%%%%%%%%%%%%%%%%%%%%%%%%%%%%%%%%%%%%%%%%%%
\par\vspace{0.3cm}
%%%%%%%%%%%%%%%%%%%%%%%%%%%%%%%%%%%%%%%%%%%%%%%%%%%%%%%%%%%%%%%%%%%%%%%%%%%%%%%%%%%%%%%%%%%%%%%%%%%%%%%
\noindent {\bf Proof of Theorem \ref{HM-property}.}
Suppose that $(T_n,H^{(n)})$ is an FCS for $H\in\mathcal{P}^\mathbb{B}$, and that $(T_n,M^{(n)})$ is an inner FCS for $M\in(\mathcal{M}_{\mathrm{loc}})^{i,\mathbb{B}}$ such that $H^{(n)}\in\mathcal{L}_m(M^{(n)})$ for each $n\in \mathbf{N}^+$.

$(1)$ According to Theorem \ref{Lem-M}(2), it is established that $(T_n,(M^{(n)})^c)$ is an (inner) FCS for $M^c\in(\mathcal{M}^c_{\mathrm{loc},0})^{i,\mathbb{B}}=(\mathcal{M}^c_{\mathrm{loc},0})^\mathbb{B}$. For each $n\in \mathbf{N}^+$,  it can be directly inferred that
\[
H^{(n)}.(M^{(n)})^c=(H^{(n)}.M^{(n)})^c\in\mathcal{M}^c_{\mathrm{loc},0}.
\]
Subsequently, by appealing to Theorems \ref{HM=} and \ref{eq-HM}, we confirm that $H\in\mathcal{L}_m^\mathbb{B}(M^c)$ and $H_{\bullet}M^c\in(\mathcal{M}^c_{\mathrm{loc},0})^\mathbb{B}$. Furthermore, the equality $(H_{\bullet}M)^c=H_{\bullet}M^c$  is established through the following derivation:
\begin{equation*}
(H_{\bullet}M^c)I_{\mathbb{B}\llbracket{0,T_n}\rrbracket}
=(H^{(n)}.(M^{(n)})^c)I_{\mathbb{B}\llbracket{0,T_n}\rrbracket}
=(H^{(n)}.M^{(n)})^cI_{\mathbb{B}\llbracket{0,T_n}\rrbracket}
=(H_{\bullet}M)^cI_{\mathbb{B}\llbracket{0,T_n}\rrbracket},\quad n\in \mathbf{N}^+.
\end{equation*}

$(2)$ By leveraging Theorem \ref{MC}(4), the proof proceeds in a manner analogous to the demonstration of part (1).

$(3)$ Based on the results established in parts (1) and (2), we have demonstrated the inclusion $\mathcal{L}^\mathbb{B}_m(M)\subseteq\mathcal{L}_m^\mathbb{B}(M^c)\bigcap\mathcal{L}_m^\mathbb{B}(M^d)$. Conversely, the inclusion $\mathcal{L}_m^\mathbb{B}(M^c)\bigcap\mathcal{L}_m^\mathbb{B}(M^d)\subseteq\mathcal{L}^\mathbb{B}_m(M)$ can be straightforwardly inferred from \eqref{+HM}. Hence, the proof is now complete.
$\hfill\blacksquare$

%%%%%%%%%%%%%%%%%%%%%%%%%%%%%%%%%%%%%%%%%%%%%%%%%%%%%%%%%%%%%%%%%%%%%%%%%%%%%%%%%%%%%%%%%%%%%%%%%%%%%%%
\par\vspace{0.3cm}
%%%%%%%%%%%%%%%%%%%%%%%%%%%%%%%%%%%%%%%%%%%%%%%%%%%%%%%%%%%%%%%%%%%%%%%%%%%%%%%%%%%%%%%%%%%%%%%%%%%%%%%
\noindent {\bf Proof of Lemma \ref{XMA}.}
{\it Sufficiency}. Suppose that $X$ admits the decomposition as specified in \eqref{eq-XMA}. Let $(T_n, M^{(n)})$ and $(T_n, A^{(n)})$ be FCSs for $M\in (\mathcal{M}_{\mathrm{loc}})^\mathbb{B}$ and $A\in (\mathcal{V}_0)^\mathbb{B}$, respectively. For each $n\in \mathbf{N}^+$, by defining $X^{(n)}=M^{(n)}+A^{(n)}$, it is evident that $X^{(n)}\in \mathcal{S}$. Subsequently, we derive
\[
X^{(n)}I_{\mathbb{B}\llbracket{0,T_n}\rrbracket}
=M^{(n)}I_{\mathbb{B}\llbracket{0,T_n}\rrbracket}+A^{(n)}I_{\mathbb{B}\llbracket{0,T_n}\rrbracket}
=(M+A)I_{\mathbb{B}\llbracket{0,T_n}\rrbracket}
=XI_{\mathbb{B}\llbracket{0,T_n}\rrbracket},\quad n\in \mathbf{N}^+.
\]
This derivation consequently verifies that $X\in \mathcal{S}^\mathbb{B}$ with the FCS $(T_n, X^{(n)})$.

{\it Necessity}. Suppose $X\in \mathcal{S}^\mathbb{B}$.
We initially demonstrate the validity of \eqref{eq-XMA} for the specific case where $\mathbb{B}=\mathbb{C}$, i.e., $\mathbb{B}$ is a predictable set of interval type.
Let $(\tau_n)$ denote an FS for $\mathbb{C}$. According to Corollary \ref{fcs-p}(2), it follows that $(\tau_n,X^{\tau_n})$ is an FCS for $X\in \mathcal{S}^\mathbb{C}$. For each $n\in \mathbf{N}^+$, assume that $X^{\tau_n}=M^{(n)}+A^{(n)}$ is a decomposition of $X^{\tau_n}\in \mathcal{S}$, where $M^{(n)}\in \mathcal{M}_{\mathrm{loc}}$ and $A^{(n)}\in \mathcal{V}_0$.
Define the following processes: $\widetilde{M}^{(1)}=(M^{(1)})^{\tau_1}, \widetilde{A}^{(1)}=(A^{(1)})^{\tau_1}$, and for $n\in \mathbf{N}^+$,
\begin{equation}\label{MMAA}
\left\{
\begin{aligned}
\widetilde{M}^{(n+1)}&=\widetilde{M}^{(n)}+(M^{(n+1)})^{\tau_{n+1}}
-(M^{(n+1)})^{\tau_{n}},\\
\widetilde{A}^{(n+1)}&=\widetilde{A}^{(n)}+(A^{(n+1)})^{\tau_{n+1}}
-(A^{(n+1)})^{\tau_{n}}.
\end{aligned}
\right.
\end{equation}
For any $n,k\in \mathbf{N}^+$ with $n\leq k$, through induction, we deduce that   $X^{\tau_n}=\widetilde{M}^{(n)}+\widetilde{A}^{(n)}$ ($\widetilde{M}^{(n)}\in\mathcal{M}_{\mathrm{loc}}$ and $\widetilde{A}^{(n)}\in\mathcal{V}_0$) and
\begin{equation}\label{tm2}
(\widetilde{M}^{(k)})^{\tau_n}=(\widetilde{M}^{(n)})^{\tau_n},\quad (\widetilde{A}^{(k)})^{\tau_n}=(\widetilde{A}^{(n)})^{\tau_n}.
\end{equation}
Utilizing Theorem \ref{restriction}(4) and (\ref{tm2}), we establish that $(\tau_n,\widetilde{M}^{(n)})$ and $(\tau_n,\widetilde{A}^{(n)})$ are FCSs for $M\in (\mathcal{M}_{\mathrm{loc}})^\mathbb{C}$ and $A\in (\mathcal{V}_0)^\mathbb{C}$ respectively, where $M$ and $A$ are defined as
\begin{equation}\label{MA-p}
\left\{
\begin{aligned}
 M&=\left(X_0I_{\llbracket{0}\rrbracket}+\sum\limits_{n=1}^{{+\infty}}\widetilde{M}^{(n)}I_{\rrbracket{\tau_{n-1},\tau_n}
      \rrbracket}\right)\mathfrak{I}_\mathbb{C},\\
 A&=\left(\sum\limits_{n=1}^{{+\infty}}\widetilde{A}^{(n)}I_{\rrbracket{\tau_{n-1},\tau_n}
      \rrbracket}\right)\mathfrak{I}_\mathbb{C},\quad \tau_0=0.
\end{aligned}
\right.
\end{equation}
Consequently, for each $k\in \mathbf{N}^+$, we derive
\begin{align*}
(M+A)I_{\llbracket{0,\tau_k}\rrbracket}
&=X_0I_{\llbracket{0}\rrbracket}+\sum\limits_{n=1}^{k}(\widetilde{M}^{(n)}+\widetilde{A}^{(n)})I_{\rrbracket{\tau_{n-1},\tau_n}
      \rrbracket}\\
&=X_0I_{\llbracket{0}\rrbracket}+\sum\limits_{n=1}^{k}X^{\tau_n}I_{\rrbracket{\tau_{n-1},\tau_n}
      \rrbracket}\\
&=X_0I_{\llbracket{0}\rrbracket}+\sum\limits_{n=1}^{k}XI_{\rrbracket{\tau_{n-1},\tau_n}
      \rrbracket}\\
&=XI_{\llbracket{0,\tau_k}\rrbracket},
\end{align*}
which, by Corollary \ref{process-FS}(1), yields $X=M+A$ with $M\in (\mathcal{M}_{\mathrm{loc}})^\mathbb{C}$ and $A\in (\mathcal{V}_0)^\mathbb{C}$.

Subsequently, we proceed to prove the decomposition outlined in \eqref{eq-XMA} under the assumption that $\mathbb{B}$ constitutes an optional set of interval type.
Let $(T_n,X^{(n)},\widetilde{X})$ represent a continuation for $X\in \mathcal{S}^\mathbb{B}$, with $\mathbb{C}$ defined as
$\mathbb{C}=\bigcup\limits_{n}\llbracket{0,T_n}\rrbracket$. Consequently, $\widetilde{X}\in \mathcal{S}^\mathbb{C}$ admits a decomposition given by
\[
 \widetilde{X}=\widetilde{M}+\widetilde{A}, \quad \widetilde{M}\in (\mathcal{M}_{\mathrm{loc}})^\mathbb{C},\quad \widetilde{A}\in (\mathcal{V}_0)^\mathbb{C}.
\]
By defining $M=\widetilde{M}\mathfrak{I}_\mathbb{B}$ and $A=\widetilde{A}\mathfrak{I}_\mathbb{B}$, it follows that $M\in (\mathcal{M}_{\mathrm{loc}})^\mathbb{B}$ and $A\in (\mathcal{V}_0)^\mathbb{B}$ satisfying the conditions
\[
X=\widetilde{X}\mathfrak{I}_\mathbb{B}=(\widetilde{M}+\widetilde{A})\mathfrak{I}_\mathbb{B}
=M+A,
\]
which ultimately establishes the validity of \eqref{eq-XMA}.
$\hfill\blacksquare$

%%%%%%%%%%%%%%%%%%%%%%%%%%%%%%%%%%%%%%%%%%%%%%%%%%%%%%%%%%%%%%%%%%%%%%%%%%%%%%%%%%%%%%%%%%%%%%%%%%%%%%%
\par\vspace{0.3cm}
%%%%%%%%%%%%%%%%%%%%%%%%%%%%%%%%%%%%%%%%%%%%%%%%%%%%%%%%%%%%%%%%%%%%%%%%%%%%%%%%%%%%%%%%%%%%%%%%%%%%%%%
\noindent {\bf Proof of Lemma \ref{uXc}.}
In order to establish the proof, we commence by introducing a pivotal lemma.
\begin{lemma}\label{LM-A4}
If $M\in (\mathcal{M}_{\mathrm{loc}})^\mathbb{B}\cap\mathcal{V}^\mathbb{B}$, then $M\in(\mathcal{W}_{\mathrm{loc}})^\mathbb{B}$.
\end{lemma}
\begin{proof}
Suppose that $(T_n,M^{(n)})$ and $(T_n,N^{(n)})$ are FCSs for $M\in (\mathcal{M}_{\mathrm{loc}})^\mathbb{B}$ and $M\in\mathcal{V}^\mathbb{B}$, respectively. Fix $k\in \mathbf{N}^+$. \eqref{PQ-1} yields
\[
(M^{(k)})^{T_k\wedge (T_F-)}=(N^{(k)})^{T_k\wedge (T_F-)}.
\]
By utilizing the fact
\[
(M^{(k)})^{T_k}=(M^{(k)})^{T_k\wedge (T_F-)}+\Delta(M^{(k)})^{T_k}_{T_F}I_{\llbracket{T_F,+\infty}\llbracket}
=(N^{(k)})^{T_k\wedge (T_F-)}+\Delta(M^{(k)})^{T_k}_{T_F}I_{\llbracket{T_F,+\infty}\llbracket},
\]
we can establish that $(M^{(k)})^{T_k}\in\mathcal{V}$. Then it follows that $(M^{(k)})^{T_k}\in\mathcal{M}_{\mathrm{loc}}\cap\mathcal{V}$, which, by Theorem 7.19 in \cite{He}, implies that $(M^{(k)})^{T_k}\in\mathcal{W}_{\mathrm{loc}}$. Hence, $M\in(\mathcal{W}_{\mathrm{loc}})^\mathbb{B}$ with the FCS $(T_n,(M^{(n)})^{T_n})$.
\end{proof}

Now we are ready to prove the result.
Given two decompositions of $X$ resulting in the equality $M^c+M^d+A=N^c+N^d+V$, we can deduce that $A-V=N^c+N^d-M^c-M^d$, which subsequently implies
\[
A-V\in (\mathcal{M}_{\mathrm{loc}})^\mathbb{B}\cap(\mathcal{V}_0)^\mathbb{B}.
\]
By invoking Lemma \ref{LM-A4}, it is established that $A-V\in(\mathcal{W}_{\mathrm{loc},0})^\mathbb{B}$. Given $\mathcal{W}_{\mathrm{loc},0}\subseteq\mathcal{M}^d_{\mathrm{loc}}$, it follows that $A-V\in(\mathcal{M}^d_{\mathrm{loc}})^\mathbb{B}$. Furthermore, since
$N^c-M^c=(M^d-N^d)+(A-V)$, it becomes evident that
\[
N^c-M^c\in(\mathcal{M}^c_{\mathrm{loc},0})^\mathbb{B}\cap(\mathcal{M}^d_{\mathrm{loc}})^\mathbb{B}.
\]
Hence, by applying Lemma \ref{LM-A2}, we conclude that $N^c=M^c$, thereby completing the proof.
$\hfill\blacksquare$

%%%%%%%%%%%%%%%%%%%%%%%%%%%%%%%%%%%%%%%%%%%%%%%%%%%%%%%%%%%%%%%%%%%%%%%%%%%%%%%%%%%%%%%%%%%%%%%%%%%%%%%
\par\vspace{0.3cm}
%%%%%%%%%%%%%%%%%%%%%%%%%%%%%%%%%%%%%%%%%%%%%%%%%%%%%%%%%%%%%%%%%%%%%%%%%%%%%%%%%%%%%%%%%%%%%%%%%%%%%%%
\noindent {\bf Proof of Theorem \ref{XTc}.}
$(1)$
The proof proceeds analogously to that of Theorem \ref{Lem-M}(1). Fix $n\in \mathbf{N}^+$ and denote $\mathbb{B}_n=\mathbb{B}\llbracket{0,T_n}\rrbracket$.  Both $X\mathfrak{I}_{\mathbb{B}_n}$ and $X^{(n)}\mathfrak{I}_{\mathbb{B}_n}$ are $\mathbb{B}_n$-semimartingales. Assume that $X$ admits the decomposition \eqref{decomX}, and that $X^{(n)}$ admits the decomposition
$X^{(n)}=X^{(n)}_0+N^c+N^d+V$, where $N^c\in \mathcal{M}^c_{\mathrm{loc},0}$, $N^d\in \mathcal{M}^d_{\mathrm{loc}}$, and $V\in \mathcal{V}_0$.
Then $X^c=M^c$ and $(X^{(n)})^c=N^c$, and the following relations hold:
\[
\left\{
\begin{aligned}
X\mathfrak{I}_{\mathbb{B}_n}&=(X_0\mathfrak{I}_\mathbb{B}+M^c+M^d+A)\mathfrak{I}_{\mathbb{B}_n}
=X_0\mathfrak{I}_{\mathbb{B}_n}
    +M^c\mathfrak{I}_{\mathbb{B}_n}+M^d\mathfrak{I}_{\mathbb{B}_n}+A\mathfrak{I}_{\mathbb{B}_n},\\
X^{(n)}\mathfrak{I}_{\mathbb{B}_n}&=(X^{(n)}_0+N^c+N^d+V)\mathfrak{I}_{\mathbb{B}_n}=X^{(n)}_0\mathfrak{I}_{\mathbb{B}_n}
+N^c\mathfrak{I}_{\mathbb{B}_n}+N^d\mathfrak{I}_{\mathbb{B}_n}+V\mathfrak{I}_{\mathbb{B}_n},\\
\end{aligned}
\right.
\]
where $M^c\mathfrak{I}_{\mathbb{B}_n},N^c\mathfrak{I}_{\mathbb{B}_n}\in (\mathcal{M}^c_{\mathrm{loc},0})^{\mathbb{B}_n}$; $M^d\mathfrak{I}_{\mathbb{B}_n},N^d\mathfrak{I}_{\mathbb{B}_n}\in (\mathcal{M}^d_{\mathrm{loc}})^{\mathbb{B}_n}$; and $A\mathfrak{I}_{\mathbb{B}_n}, V\mathfrak{I}_{\mathbb{B}_n}\in (\mathcal{V}_0)^{\mathbb{B}_n}$.
Observing that $X\mathfrak{I}_{\mathbb{B}_n}=X^{(n)}\mathfrak{I}_{\mathbb{B}_n}$ and invoking the uniqueness of their continuous martingale parts, we deduce:
\begin{equation}\label{Xc-eq2}
X^c\mathfrak{I}_{\mathbb{B}_n}=M^c\mathfrak{I}_{\mathbb{B}_n}=(X\mathfrak{I}_{\mathbb{B}_n})^c=(X^{(n)}\mathfrak{I}_{\mathbb{B}_n})^c
=N^c\mathfrak{I}_{\mathbb{B}_n}=(X^{(n)})^c\mathfrak{I}_{\mathbb{B}_n}.
\end{equation}
Since (\ref{Xc-eq2}) holds for each $n\in \mathbf{N}^+$, we obtain $X^cI_{\mathbb{B}\llbracket{0,T_n}\rrbracket}=(X^{(n)})^cI_{\mathbb{B}\llbracket{0,T_n}\rrbracket}$, demonstrating that $(T_n,(X^{(n)})^c)$ is a CS for $X^c$. Given that $(X^{(n)})^c\in \mathcal{M}^c_{\mathrm{loc},0}$ for each $n\in \mathbf{N}^+$, the sequence $(T_n,(X^{(n)})^c)$ indeed constitutes an FCS for $X^c\in (\mathcal{M}^c_{\mathrm{loc},0})^\mathbb{B}$.

$(2)$ Let $\eqref{decomX}$ denote a decomposition of $X$. According to Theorem \ref{fcs}, it holds that $X^\tau\in \mathcal{S}$. Then $X^\tau$ possesses a unique continuous martingale part, denoted as $(X^\tau)^c$. Utilizing the decomposition given by \eqref{decomX}, we can deduce that
\[
X^\tau=(X_0\mathfrak{I}_\mathbb{B}+M^c+M^d+A)^{\tau}=X_0+(M^c)^{\tau}+(M^d)^{\tau}+A^{\tau},
\]
which, by Lemma \ref{uXc}, implies that the continuous martingale part of $X^\tau$ is given by $(X^\tau)^c=(M^c)^{\tau}$. Consequently, equation \eqref{Xc-eq} is established by the equality $X^c=M^c$. Regarding equation \eqref{XcB-eq}, the first equality has already been proven by \eqref{Xc-eq}, and it remains to demonstrate the second equality. The continuous martingale part of $X^\tau\mathfrak{I}_\mathbb{B}$ is denoted as $(X^\tau\mathfrak{I}_\mathbb{B})^c$. The relation
\[
 X^\tau\mathfrak{I}_\mathbb{B}=X_0\mathfrak{I}_\mathbb{B}
 +(M^c)^{\tau}\mathfrak{I}_\mathbb{B}+(M^d)^{\tau}\mathfrak{I}_\mathbb{B}+A^{\tau}\mathfrak{I}_\mathbb{B}
\]
reveals that the continuous martingale part of $X^\tau\mathfrak{I}_\mathbb{B}$ can alternatively be expressed as $(M^c)^{\tau}\mathfrak{I}_\mathbb{B}=(X^c)^{\tau}\mathfrak{I}_\mathbb{B}$. Therefore, the second equality of \eqref{XcB-eq} is derived from the uniqueness of the continuous martingale part of $X^\tau\mathfrak{I}_\mathbb{B}$.
$\hfill\blacksquare$

%%%%%%%%%%%%%%%%%%%%%%%%%%%%%%%%%%%%%%%%%%%%%%%%%%%%%%%%%%%%%%%%%%%%%%%%%%%%%%%%%%%%%%%%%%%%%%%%%%%%%%%
\par\vspace{0.3cm}
%%%%%%%%%%%%%%%%%%%%%%%%%%%%%%%%%%%%%%%%%%%%%%%%%%%%%%%%%%%%%%%%%%%%%%%%%%%%%%%%%%%%%%%%%%%%%%%%%%%%%%%
\noindent {\bf Proof of Theorem \ref{[X,Y]-fcs}.}
$(1)$ According to Definition \ref{[X,Y]} and Theorem \ref{property-qr-p}(2), it is demonstrated that
\[
[X,Y]=X_0Y_0\mathfrak{I}_\mathbb{B}+\langle X^c,Y^c\rangle+\Sigma(\Delta X\Delta Y)
=Y_0X_0\mathfrak{I}_\mathbb{B}+\langle Y^c,X^c\rangle+\Sigma(\Delta Y\Delta X)=[Y,X]
\]
and
\begin{align*}
[aX+bY,Z]
&=(aX+bY)_0Z_0\mathfrak{I}_\mathbb{B}+\langle (aX+bY)^c,Z^c\rangle+\Sigma(\Delta (aX+bY)\Delta Z)\\
&=(aX_0+bY_0)Z_0\mathfrak{I}_\mathbb{B}+\langle aX^c+bY^c,Z^c\rangle+\Sigma(a\Delta X+b\Delta Y)\Delta Z)\\
&=a\left(X_0Z_0\mathfrak{I}_\mathbb{B}+\langle X^c,Z^c\rangle+\Sigma(\Delta X\Delta Z)\right)+b\left(Y_0Z_0\mathfrak{I}_\mathbb{B}+\langle Y^c,Z^c\rangle+\Sigma(\Delta Y\Delta Z)\right)\\
&=a[X,Z]+b[Y,Z].
\end{align*}
Consequently, the assertion is established.

$(2)$ By invoking Theorems \ref{delta}(6), \ref{property-qr-p}(1) and \ref{XTc}(1), we deduce that for each $n\in \mathbf{N}^+$, $[X^{(n)},Y^{(n)}]\in \mathcal{V}$ and
\begin{align*}
&[X^{(n)},Y^{(n)}]I_{\mathbb{B}\llbracket{0,T_n}\rrbracket}\\
=&\left(X_0^{(n)}Y_0^{(n)}+\langle(X^{(n)})^c,(Y^{(n)})^c\rangle+\Sigma(\Delta X^{(n)}\Delta Y^{(n)})\right)I_{\mathbb{B}\llbracket{0,T_n}\rrbracket}\\
=&\left(X_0Y_0\mathfrak{I}_\mathbb{B}+\langle X^c,Y^c\rangle+\Sigma(\Delta X\Delta Y)\right)I_{\mathbb{B}\llbracket{0,T_n}\rrbracket}\\
=&[X,Y]I_{\mathbb{B}\llbracket{0,T_n}\rrbracket}
\end{align*}
Consequently, we establish that $[X,Y]\in \mathcal{V}^\mathbb{B}$ with the FCS $(T_n,[X^{(n)},Y^{(n)}])$. Additionally, noting that $[X]=[X,X]$, it follows that $(T_n,[X^{(n)}])$ is a CS for $[X]$. Since Definition 8.2 in \cite{He} yields $[X^{(n)}]\in \mathcal{V}^{+}$ for each $n\in \mathbf{N}^+$, we deduce that $[X]\in (\mathcal{V}^+)^\mathbb{B}$ with the FCS $(T_n,[X^{(n)}])$.

$(3)$ To begin, we establish the proof of \eqref{XY}. Theorems \ref{property-qr-p}(3) and \ref{XTc}(2) show
\begin{equation*}
      \langle (X^{\tau})^c,(Y^{\tau})^c\rangle=\langle (X^c)^{\tau},(Y^c)^{\tau}\rangle=\langle X^c,Y^c\rangle^{\tau},
\end{equation*}
and the relations \eqref{sigmaX} and \eqref{eqXT} yield
\[
\Sigma(\Delta (X^{\tau})\Delta (Y^{\tau}))
=\Sigma((\Delta X\Delta Y)I_{\llbracket{0,\tau}\rrbracket})
=(\Sigma(\Delta X\Delta Y))^\tau.
\]
Consequently, \eqref{XY} can be derived through
\begin{align*}
[X^{\tau},Y^{\tau}]
&=X_0Y_0+\langle (X^{\tau})^c,(Y^{\tau})^c\rangle+\Sigma(\Delta (X^{\tau})\Delta (Y^{\tau}))\\
&=X_0Y_0+\langle X^c,Y^c\rangle^{\tau}+(\Sigma(\Delta X\Delta Y))^\tau\\
&=\left(X_0Y_0\mathfrak{I}_\mathbb{B}+\langle X^c,Y^c\rangle+\Sigma(\Delta X\Delta Y)\right)^\tau\\
&=[X,Y]^{\tau}.
\end{align*}

Proceeding further, we direct our attention to the verification of \eqref{XYB}. Let us assume that $(T_n,X^{(n)})$ and $(T_n,Y^{(n)})$ are FCSs for $X\in\mathcal{S}^\mathbb{B}$ and $Y\in\mathcal{S}^\mathbb{B}$, respectively. According to Theorem \ref{fcs},
it follows that $(T_n,(X^{(n)})^\tau)$ and $(T_n,(Y^{(n)})^\tau)$ are FCSs for $X^\tau\mathfrak{I}_\mathbb{B}\in\mathcal{S}^\mathbb{B}$ and $Y^\tau\mathfrak{I}_\mathbb{B}\in\mathcal{S}^\mathbb{B}$, respectively. By invoking part (2) and Theorem \ref{fcs}, we can deduce that for each $n\in \mathbf{N}^+$,
\begin{align*}
[X^{\tau}\mathfrak{I}_\mathbb{B},Y^{\tau}\mathfrak{I}_\mathbb{B}]I_{\mathbb{B}\llbracket{0,T_n}\rrbracket}
&=[(X^{(n)})^\tau,(Y^{(n)})^\tau]I_{\mathbb{B}\llbracket{0,T_n}\rrbracket}
=[X^{(n)},Y^{(n)}]^\tau I_{\mathbb{B}\llbracket{0,T_n}\rrbracket}\\
&=[X,Y]^\tau I_{\mathbb{B}\llbracket{0,T_n}\rrbracket}
=([X,Y]^\tau \mathfrak{I}_\mathbb{B}) I_{\mathbb{B}\llbracket{0,T_n}\rrbracket},
\end{align*}
which establishes the first equality of \eqref{XYB}.
The second equality of \eqref{XYB} can be straightforwardly derived from \eqref{XY}. To obtain the last equality of \eqref{XYB}, we once again apply part (2) and Theorem \ref{fcs}, yielding:
\begin{align*}
[X^{\tau}\mathfrak{I}_\mathbb{B},Y]I_{\mathbb{B}\llbracket{0,T_n}\rrbracket}
&=[(X^{(n)})^\tau,Y^{(n)}]I_{\mathbb{B}\llbracket{0,T_n}\rrbracket}
=[X^{(n)},Y^{(n)}]^\tau I_{\mathbb{B}\llbracket{0,T_n}\rrbracket}\\
&=[X,Y]^\tau I_{\mathbb{B}\llbracket{0,T_n}\rrbracket}
=([X,Y]^\tau \mathfrak{I}_\mathbb{B}) I_{\mathbb{B}\llbracket{0,T_n}\rrbracket}
\end{align*}
for each $n\in \mathbf{N}^+$. Thus, the validity of \eqref{XYB} is firmly established.
$\hfill\blacksquare$

%%%%%%%%%%%%%%%%%%%%%%%%%%%%%%%%%%%%%%%%%%%%%%%%%%%%%%%%%%%%%%%%%%%%%%%%%%%%%%%%%%%%%%%%%%%%%%%%%%%%%%%
\par\vspace{0.3cm}
%%%%%%%%%%%%%%%%%%%%%%%%%%%%%%%%%%%%%%%%%%%%%%%%%%%%%%%%%%%%%%%%%%%%%%%%%%%%%%%%%%%%%%%%%%%%%%%%%%%%%%%
\noindent {\bf Proof of Lemma \ref{HX-unique}.}
Let $\mathbb{B}$ be given by \eqref{B}.
From statement (M2) and the relation $H\in \mathcal{L}^\mathbb{B}_m(M)\bigcap\mathcal{L}^\mathbb{B}_m(N)$, there exist two coupled predictable processes $\widetilde{H}$ and $\widetilde{K}$ for $H\in \mathcal{P}^\mathbb{B}$, and two inner FCSs $(T_n,M^{(n)})$ and $(T_n,N^{(n)})$ for $M\in (\mathcal{M}_{\mathrm{loc}})^{i,\mathbb{B}}$ and $N\in (\mathcal{M}_{\mathrm{loc}})^{i,\mathbb{B}}$ respectively, such that $\widetilde{H}\in \mathcal{L}_m(M^{(n)})$ and $\widetilde{K}\in \mathcal{L}_m(N^{(n)})$ for each $n\in \mathbf{N}^+$. Furthermore, according to statement (A2) and the relation $H\in \mathcal{L}^\mathbb{B}_s(A)\bigcap\mathcal{L}^\mathbb{B}_s(V)$, there exist two coupled predictable processes $\widehat{H}$ and $\widehat{K}$ for $H\in \mathcal{P}^\mathbb{B}$, and two FCSs $(T_n,A^{(n)})$ and $(T_n,V^{(n)})$ for $A\in (\mathcal{V}_0)^\mathbb{B}$ and $V\in (\mathcal{V}_0)^\mathbb{B}$ respectively, such that $\widehat{H}\in\mathcal{L}_s(A^{(n)})$ and $\widehat{K}\in\mathcal{L}_s(V^{(n)})$ for each $n\in \mathbf{N}^+$.

Given that both $(T_n,M^{(n)},A^{(n)})$ and $(T_n,N^{(n)},V^{(n)})$ constitute decomposed inner FCSs for $X\in \mathcal{S}^{i,\mathbb{B}}$, it is straightforward to observe the following equalities:
\[
(M^{(n)}+A^{(n)})^{T_n\wedge (T_F-)}=(N^{(n)}+V^{(n)})^{T_n\wedge (T_F-)}, \quad n\in \mathbf{N}^+.
\]
We define $L=\widetilde{L}-\widehat{L}$, where
\[
\widetilde{L}=\min\{\widetilde{H}^+,\widehat{H}^+,\widetilde{K}^+,\widehat{K}^+\},\quad
\widehat{L}=\min\{\widetilde{H}^-,\widehat{H}^-,\widetilde{K}^-,\widehat{K}^-\}.
\]
Consequently, $L$ emerges as a coupled predictable process for $H\in \mathcal{P}^\mathbb{B}$, satisfying the constraint
\[
|L|\leq \min\{|\widetilde{H}|,|\widehat{H}|,|\widetilde{K}|,|\widehat{K}|\}.
\]
Then it holds that $L\in \mathcal{L}_m(M^{(n)})\bigcap \mathcal{L}_m(N^{(n)})$ and $L\in \mathcal{L}_s(V^{(n)})\bigcap \mathcal{L}_s(A^{(n)})$ for each $n\in \mathbf{N}^+$.
Now, invoking Theorems \ref{HA-FCS} and \ref{eq-HM}, we deduce that for each $n\in \mathbf{N}^+$,
\begin{align*}
(H_{\bullet}M+H_{\bullet}A)I_{\mathbb{B}\llbracket{0,T_n}\rrbracket}
&=(L.M^{(n)}+L.A^{(n)})I_{\mathbb{B}\llbracket{0,T_n}\rrbracket}\\
&=\left(L.(M^{(n)}+A^{(n)})^{T_n\wedge (T_F-)}\right)I_{\mathbb{B}\llbracket{0,T_n}\rrbracket}\\
&=\left(L.(N^{(n)}+V^{(n)})^{T_n\wedge (T_F-)}\right)I_{\mathbb{B}\llbracket{0,T_n}\rrbracket}\\
&=(L.N^{(n)}+L.V^{(n)})I_{\mathbb{B}\llbracket{0,T_n}\rrbracket}\\
&=(H_{\bullet}N+H_{\bullet}V)I_{\mathbb{B}\llbracket{0,T_n}\rrbracket},
\end{align*}
which, by virtue of Theorem \ref{process}(1), leads to the confirmation of \eqref{HX-unique-1}.
$\hfill\blacksquare$

%%%%%%%%%%%%%%%%%%%%%%%%%%%%%%%%%%%%%%%%%%%%%%%%%%%%%%%%%%%%%%%%%%%%%%%%%%%%%%%%%%%%%%%%%%%%%%%%%%%%%%%
\par\vspace{0.3cm}
%%%%%%%%%%%%%%%%%%%%%%%%%%%%%%%%%%%%%%%%%%%%%%%%%%%%%%%%%%%%%%%%%%%%%%%%%%%%%%%%%%%%%%%%%%%%%%%%%%%%%%%
\noindent {\bf Proof of Theorem \ref{HX=}.}
The implication (X2) $\Rightarrow$ (X3) is straightforward, whereas the implication (X3) $\Rightarrow$ (X1) can be effortlessly established by invoking Theorems \ref{HA-equivalent} and \ref{HM=}. What remains to be demonstrated is the implication (X1) $\Rightarrow$ (X2).

Suppose that $X=M+A$ is an inner decomposition of $X$ such that $H\in \mathcal{L}_m^{\mathbb{B}}(M)\cap \mathcal{L}_s^{\mathbb{B}}(A)$, where $M\in (\mathcal{M}_{\mathrm{loc}})^{i,\mathbb{B}}$ and $A\in (\mathcal{V}_0)^\mathbb{B}$. By invoking Theorem \ref{HM=}, we can ascertain the existence of a coupled predictable process $\widetilde{K}$ for $H\in \mathcal{P}^\mathbb{B}$ and an inner FCS $(\widetilde{T}_n,M^{(n)})$ for $M\in (\mathcal{M}_{\mathrm{loc}})^{i,\mathbb{B}}$, satisfying $\widetilde{K}\in\mathcal{L}_m(M^{(n)})$ for each $n\in \mathbf{N}^+$. Furthermore, from Theorem \ref{HA-equivalent},
there exist a coupled predictable process $\widehat{K}$ for $H\in \mathcal{P}^\mathbb{B}$ and an FCS $(\widehat{T}_n,A^{(n)})$ for $A\in (\mathcal{V}_0)^\mathbb{B}$ such that $\widehat{K}\in\mathcal{L}_s(A^{(n)})$ for each $n\in \mathbf{N}^+$. We then define:
\[
(T_n)=(\widetilde{T}_n\wedge \widehat{T}_n),\quad \widetilde{H}=\min\{\widetilde{K}^+,\widehat{K}^+\}-\min\{\widetilde{K}^-,\widehat{K}^-\}.
\]
It is straightforward to verify that $\widetilde{H}$ is also a coupled predictable process for $H\in \mathcal{P}^\mathbb{B}$, fulfilling the condition $|\widetilde{H}|\leq \min\{|\widetilde{K}|,|\widehat{K}|\}$. According to Theorem \ref{process}(3) and Theorem \ref{MC}(1), $(T_n,M^{(n)})$ serves as an inner FCS for $M\in (\mathcal{M}_{\mathrm{loc}})^{i,\mathbb{B}}$, and $(T_n,A^{(n)})$ as an FCS for $A\in (\mathcal{V}_0)^\mathbb{B}$. Consequently, $(T_n,M^{(n)},A^{(n)})$ forms a decomposed inner FCS for $X\in \mathcal{S}^{i,\mathbb{B}}$.
Lastly, it is established that $\widetilde{H}\in\mathcal{L}_m(M^{(n)})\cap\mathcal{L}_s(A^{(n)})$ for each $n\in \mathbf{N}^+$, thereby confirming the assertion (X2).
$\hfill\blacksquare$

%%%%%%%%%%%%%%%%%%%%%%%%%%%%%%%%%%%%%%%%%%%%%%%%%%%%%%%%%%%%%%%%%%%%%%%%%%%%%%%%%%%%%%%%%%%%%%%%%%%%%%%
\par\vspace{0.3cm}
%%%%%%%%%%%%%%%%%%%%%%%%%%%%%%%%%%%%%%%%%%%%%%%%%%%%%%%%%%%%%%%%%%%%%%%%%%%%%%%%%%%%%%%%%%%%%%%%%%%%%%%
\noindent {\bf Proof of Theorem \ref{eq-HX}.}
By virtue of Theorem \ref{process}(3), it is established that $(\tau_n,K^{(n)})$ constitutes an FCSs for $H\in\mathcal{P}^\mathbb{B}$, while $(\tau_n,N^{(n)},V^{(n)})$ forms a decomposed inner FCS for $X\in \mathcal{S}^{i,\mathbb{B}}$. Leveraging Definition \ref{de-HX} alongside Theorems \ref{HA-FCS} and \ref{eq-HM}, it becomes evident that both $(T_n,H^{(n)}.M^{(n)},H^{(n)}.A^{(n)})$ and $(\tau_n,K^{(n)}.N^{(n)},K^{(n)}.V^{(n)})$ serve as decomposed inner FCSs for $H_{\bullet}X\in\mathcal{S}^{i,\mathbb{B}}$. Therefore, according to \eqref{x-expression}, we can establish the expression given by \eqref{HX-expression-2} and its associated independence property.
$\hfill\blacksquare$

%%%%%%%%%%%%%%%%%%%%%%%%%%%%%%%%%%%%%%%%%%%%%%%%%%%%%%%%%%%%%%%%%%%%%%%%%%%%%%%%%%%%%%%%%%%%%%%%%%%%%%%
\par\vspace{0.3cm}
%%%%%%%%%%%%%%%%%%%%%%%%%%%%%%%%%%%%%%%%%%%%%%%%%%%%%%%%%%%%%%%%%%%%%%%%%%%%%%%%%%%%%%%%%%%%%%%%%%%%%%%
\noindent {\bf Proof of Corollary \ref{bound-HX}.}
Suppose that $\widetilde{H}$ is a coupled locally bounded predictable process for $H$, $(T_n, H^{(n)})$ is an FCS for $H$ (a $\mathbb{B}$-locally bounded predictable process), and $(T_n,M^{(n)},A^{(n)})$ is  a decomposed inner FCS for $X\in\mathcal{S}^{i,\mathbb{B}}$. By invoking Theorem 9.2 in \cite{He}, it becomes evident that $\widetilde{H}\in\mathcal{L}_m(M^{(n)})\cap\mathcal{L}_s(A^{(n)})$ and $H^{(n)}\in\mathcal{L}_m(M^{(n)})\cap\mathcal{L}_s(A^{(n)})$ for each $n\in \mathbf{N}^+$. Subsequently, the assertions can be deduced by applying Theorems \ref{HX=} and \ref{eq-HX}.
$\hfill\blacksquare$

%%%%%%%%%%%%%%%%%%%%%%%%%%%%%%%%%%%%%%%%%%%%%%%%%%%%%%%%%%%%%%%%%%%%%%%%%%%%%%%%%%%%%%%%%%%%%%%%%%%%%%%
\par\vspace{0.3cm}
%%%%%%%%%%%%%%%%%%%%%%%%%%%%%%%%%%%%%%%%%%%%%%%%%%%%%%%%%%%%%%%%%%%%%%%%%%%%%%%%%%%%%%%%%%%%%%%%%%%%%%%
\noindent {\bf Proof of Corollary \ref{HX=c}.}
The implication $(i)\Rightarrow (ii)$ can be straightforwardly inferred by leveraging Theorem \ref{HX=} in conjunction with the established equality $\mathcal{S}^\mathbb{C}=\mathcal{S}^{i,\mathbb{C}}$. Furthermore, the implication $(ii)\Rightarrow (iii)$ is evident and requires no further elaboration. To complete the proof, we proceed to demonstrate the implications $(iii)\Rightarrow (iv)$ and $(iv)\Rightarrow (i)$.

$(iii)\Rightarrow (iv)$. Assuming that $(iii)$ is satisfied, let $(S_n)$ denote an FS for $\mathbb{C}$, and define $(\tau_n)=(T_n\wedge S_n)$. According to Corollary \ref{process-FS}(2), it follows that $(\tau_n)$ is an FS for $\mathbb{C}$ satisfying
\[
H^{\tau_n}=(H^{(n)})^{\tau_n},\quad X^{\tau_n}=(X^{(n)})^{\tau_n},\quad n\in \mathbf{N}^+.
\]
Subsequently, for each $n\in \mathbf{N}^+$, we can deduce that $H^{\tau_n}\in\mathcal{L}(X^{\tau_n})$ by applying the relation
\[
(H^{(n)}.X^{(n)})^{\tau_n}=(H^{(n)})^{\tau_n}.(X^{(n)})^{\tau_n}.
\]
This derivation establishes the validity of the assertion $(iv)$.

$(iv)\Rightarrow (i)$. Assume that condition $(iv)$ holds. For each $n\in \mathbb{N}^+$, we consider a decomposition $X^{\tau_n}=M^{(n)}+A^{(n)}$ of $X^{\tau_n}$ such that $H^{\tau_n}\in\mathcal{L}_m(M^{(n)})\cap\mathcal{L}_s(A^{(n)})$, where $M^{(n)}\in \mathcal{M}_{\mathrm{loc}}$ and $A^{(n)}\in \mathcal{V}_0$.
Consider the decomposition $X=M+A$, as specified in the proof of Lemma \ref{XMA}, where $M$ and $A$ are defined by \eqref{MA-p}. Additionally, it has been established that $(\tau_n,\widetilde{M}^{(n)})$ forms an inner FCS for $M\in (\mathcal{M}_{\mathrm{loc}})^\mathbb{C}$, and $(\tau_n,\widetilde{A}^{(n)})$ serves as an FCS for $A\in (\mathcal{V}_0)^\mathbb{C}$, where, for each $n\in \mathbf{N}^+$, $\widetilde{M}^{(n)}$ and $\widetilde{A}^{(n)}$ are defined by \eqref{MMAA}.

Let $\tau_0=0$. For each $n,k\in \mathbb{N}^+$ with $k\leq n$, the identity
\[
H^{\tau_n}=H^{\tau_k}I_{\llbracket{0,T_k}\rrbracket}+H^{\tau_n}I_{\rrbracket{\tau_k,+\infty}\llbracket}
\]
documents that $H^{\tau_n}\in\mathcal{L}_m((M^{(k)})^{\tau_k})\cap\mathcal{L}_m((M^{(k)})^{\tau_{k-1}})$ and $H^{\tau_n}\in\mathcal{L}_s((A^{(k)})^{\tau_k})\cap\mathcal{L}_s((A^{(k)})^{\tau_{k-1}})$.
Consequently, by applying \eqref{MMAA}, we infer that $H^{\tau_n}\in\mathcal{L}_m(\widetilde{M}^{(n)})\cap\mathcal{L}_s(\widetilde{A}^{(n)})$ for each $n\in \mathbb{N}^+$. Hence, $(\tau_n,H^{\tau_n})$ constitutes an FCS for $H\in\mathcal{P}^\mathbb{C}$, and $(\tau_n,\widetilde{M}^{(n)},\widetilde{A}^{(n)})$ forms a decomposed inner FCS for $X\in \mathcal{S}^{i,\mathbb{C}}$ such that for each $n\in \mathbf{N}^+$, $H^{\tau_n}\in\mathcal{L}_m(\widetilde{M}^{(n)})\cap\mathcal{L}_s(\widetilde{A}^{(n)})$. By virtue of statement (X3), this implies $H\in\mathcal{L}^\mathbb{C}(X)$.
$\hfill\blacksquare$

%%%%%%%%%%%%%%%%%%%%%%%%%%%%%%%%%%%%%%%%%%%%%%%%%%%%%%%%%%%%%%%%%%%%%%%%%%%%%%%%%%%%%%%%%%%%%%%%%%%%%%%
\par\vspace{0.3cm}
%%%%%%%%%%%%%%%%%%%%%%%%%%%%%%%%%%%%%%%%%%%%%%%%%%%%%%%%%%%%%%%%%%%%%%%%%%%%%%%%%%%%%%%%%%%%%%%%%%%%%%%
\noindent {\bf Proof of Corollary \ref{eq-HXc}.}
$(1)$ Based on Definition \ref{de-HX} and the relation $H\in \mathcal{L}^\mathbb{C}(X)$,  there exists a decomposition $X=M+A$ ($M\in (\mathcal{M}_{\mathrm{loc}})^\mathbb{C}$ and $A\in (\mathcal{V}_0)^\mathbb{C}$) such that $H\in \mathcal{L}_m^{\mathbb{C}}(M)\cap\mathcal{L}_s^{\mathbb{C}}(A)$.
Building upon this, Corollary \ref{HAp-equivalent}(2) establishes that $(\tau_n,H^{\tau_n}.A^{\tau_n})$ constitutes an FCS for $H_{\bullet}A\in\mathcal{V}^\mathbb{C}$. Similarly, Corollary \ref{eq-HMc}(1) asserts that $(\tau_n,H^{\tau_n}.M^{\tau_n})$ forms an FCS for $H_{\bullet}M\in(\mathcal{M}_{\mathrm{loc}})^\mathbb{C}$.
Since $X^{\tau_n}=M^{\tau_n}+A^{\tau_n}$ for each $n\in\mathbf{N}^+$, it becomes evident that $(\tau_n,H^{\tau_n}.X^{\tau_n})$ is an FCS for $H_{\bullet}X\in\mathcal{S}^\mathbb{C}$. Finally, the expression \eqref{HXc-expression0} is derived directly from \eqref{x-expression}.

$(2)$ Using part $(1)$, the proof is analogous with that of Corollary \ref{eq-HMc}$(2)$.
$\hfill\blacksquare$

%%%%%%%%%%%%%%%%%%%%%%%%%%%%%%%%%%%%%%%%%%%%%%%%%%%%%%%%%%%%%%%%%%%%%%%%%%%%%%%%%%%%%%%%%%%%%%%%%%%%%%%
\par\vspace{0.3cm}
%%%%%%%%%%%%%%%%%%%%%%%%%%%%%%%%%%%%%%%%%%%%%%%%%%%%%%%%%%%%%%%%%%%%%%%%%%%%%%%%%%%%%%%%%%%%%%%%%%%%%%%
\noindent {\bf Proof of Theorem \ref{HX-p}.}
Let $\mathbb{B}$ be given by \eqref{B}, and fix a constant $\varepsilon>0$. From Theorem \ref{HX=} and the relation $H\in \mathcal{L}^\mathbb{B}(X)$, there exist a coupled predictable process $\widetilde{H}$ for $H\in\mathcal{P}^\mathbb{B}$ and a decomposed inner FCSs $(T_n,\widetilde{M}^{(n)},\widetilde{A}^{(n)})$ for $X\in \mathcal{S}^{i,\mathbb{B}}$, such that $\widetilde{H}\in\mathcal{L}_m(\widetilde{M}^{(n)})\cap\mathcal{L}_s(\widetilde{A}^{(n)})$ for each $n\in \mathbf{N}^+$. Similarly, according to the relation $K\in \mathcal{L}^\mathbb{B}(X)$, there exist a coupled predictable process $\widetilde{K}$ for $K\in\mathcal{P}^\mathbb{B}$ and a decomposed inner FCSs $(T_n,\widehat{M}^{(n)},\widehat{A}^{(n)})$ for $X\in \mathcal{S}^{i,\mathbb{B}}$, such that $\widetilde{K}\in\mathcal{L}_m(\widehat{M}^{(n)})\cap\mathcal{L}_s(\widehat{A}^{(n)})$ for each $n\in \mathbf{N}^+$.

$(1)$ It is straightforward to observe that $aH\in\mathcal{L}^\mathbb{B}(X)$, with the property $(aH)_{\bullet}X=a(H_{\bullet}X)$. Consequently, our focus shifts to establishing the following assertions:
\[
H+K\in\mathcal{L}^\mathbb{B}(X),\quad  (H+K)_{\bullet}X=H_{\bullet}X+K_{\bullet}X.
\]
This demonstration is sufficient to complete our proof.

For each $n\in \mathbf{N}^+$, we define
\begin{align*}
M^{(n)}&=\left(\frac{1}{\widetilde{H}+\widetilde{K}}I_{[|\widetilde{H}+\widetilde{K}|>\varepsilon]}\right).
\left(\widetilde{H}.\widetilde{M}^{(n)}+\widetilde{K}.\widehat{M}^{(n)}\right)
+{I_{[|\widetilde{H}+\widetilde{K}|\leq\varepsilon]}}.\widetilde{M}^{(n)},\\
A^{(n)}&=\left(\frac{1}{\widetilde{H}+\widetilde{K}}I_{[|\widetilde{H}+\widetilde{K}|>\varepsilon]}\right).
\left(\widetilde{H}.\widetilde{A}^{(n)}+\widetilde{K}.\widehat{A}^{(n)}\right)
+{I_{[|\widetilde{H}+\widetilde{K}|\leq\varepsilon]}}.\widetilde{A}^{(n)},
\end{align*}
where all integrals involved are well-defined. For any $k,n\in \mathbf{N}^+$ with $n\leq k$, leveraging the properties that $(\widetilde{M}^{(n)})^{T_n\wedge (T_F-)}\in\mathcal{M}_{\mathrm{loc}}$ and  $(\widehat{M}^{(n)})^{T_n\wedge (T_F-)}\in\mathcal{M}_{\mathrm{loc}}$, it can be inferred that
\[
M^{(n)}I_{\mathbb{B}\llbracket{0,T_n}\rrbracket}=M^{(k)}I_{\mathbb{B}\llbracket{0,T_n}\rrbracket},\quad
(M^{(n)})^{T_n\wedge (T_F-)}\in\mathcal{M}_{\mathrm{loc}},
\]
Consequently, by Theorem \ref{restriction}(4), we establish that $(T_n,M^{(n)})$ is an inner FCS for a $\mathbb{B}$-process $M\in (\mathcal{M}_{\mathrm{loc}})^{i,\mathbb{B}}$, satisfying that $\widetilde{H}+\widetilde{K}\in\mathcal{L}_m(M^{(n)})$ for each $n\in \mathbf{N}^+$. Analogously, it holds that $(T_n,A^{(n)})$ is an FCS for a $\mathbb{B}$-process $A\in (\mathcal{V}_0)^\mathbb{B}$, fulfilling that $\widetilde{H}+\widetilde{K}\in\mathcal{L}_s(A^{(n)})$ for each $n\in \mathbf{N}^+$. Since both $(T_n,\widetilde{M}^{(n)}+\widetilde{A}^{(n)})$ and $(T_n,\widehat{M}^{(n)}+\widehat{A}^{(n)})$ are inner FCSs for $X\in \mathcal{S}^{i,\mathbb{B}}$, we deduce
\[
(\widetilde{M}^{(n)}+\widetilde{A}^{(n)})^{T_n\wedge (T_F-)}=(\widehat{M}^{(n)}+\widehat{A}^{(n)})^{T_n\wedge (T_F-)},\quad n\in \mathbf{N}^+.
\]
Then it follows that for each $n\in \mathbf{N}^+$,
\begin{align*}
&(M+A)I_{\mathbb{B}\llbracket{0,T_n}\rrbracket}\\
=&(M^{(n)}+A^{(n)})I_{\mathbb{B}\llbracket{0,T_n}\rrbracket}\\
=&\left\{\left(\frac{1}{\widetilde{H}+\widetilde{K}}I_{[|\widetilde{H}+\widetilde{K}|>\varepsilon]}\right).
\left(\widetilde{H}.(\widetilde{M}^{(n)}+\widetilde{A}^{(n)})+\widetilde{K}.(\widehat{M}^{(n)}+\widehat{A}^{(n)})\right)
+{I_{[|\widetilde{H}+\widetilde{K}|\leq\varepsilon]}}.(\widetilde{M}^{(n)}+\widetilde{A}^{(n)})\right\}I_{\mathbb{B}\llbracket{0,T_n}\rrbracket}\\
=&\left\{\left(\frac{1}{\widetilde{H}+\widetilde{K}}I_{[|\widetilde{H}+\widetilde{K}|>\varepsilon]}\right).
\left((\widetilde{H}+\widetilde{K}).(\widetilde{M}^{(n)}+\widetilde{A}^{(n)})^{T_n\wedge (T_F-)}\right)
+{I_{[|\widetilde{H}+\widetilde{K}|\leq\varepsilon]}}.(\widetilde{M}^{(n)}+\widetilde{A}^{(n)})\right\}I_{\mathbb{B}\llbracket{0,T_n}\rrbracket}\\
=&\left\{{I_{[|\widetilde{H}+\widetilde{K}|>\varepsilon]}}.(\widetilde{M}^{(n)}+\widetilde{A}^{(n)})+
{I_{[|\widetilde{H}+\widetilde{K}|\leq\varepsilon]}}.(\widetilde{M}^{(n)}+\widetilde{A}^{(n)})\right\}I_{\mathbb{B}\llbracket{0,T_n}\rrbracket}\\
=&(\widetilde{M}^{(n)}+\widetilde{A}^{(n)})I_{\mathbb{B}\llbracket{0,T_n}\rrbracket}\\
=&XI_{\mathbb{B}\llbracket{0,T_n}\rrbracket},
\end{align*}
which, by Theorem \ref{process}(1), implies $M+A=X\in \mathcal{S}^{i,\mathbb{B}}$ with the decomposed inner FCS $(T_n,M^{(n)},A^{(n)})$.
Hence, $\widetilde{H}+\widetilde{K}$ is a coupled predictable process for $H+K\in\mathcal{P}^\mathbb{B}$, and $(T_n,M^{(n)},A^{(n)})$ forms a decomposed inner FCS for $X\in \mathcal{S}^{i,\mathbb{B}}$ such that $\widetilde{H}+\widetilde{K}\in\mathcal{L}_m(M^{(n)})\cap\mathcal{L}_s(A^{(n)})$ for each $n\in \mathbf{N}^+$. As a result, we derive the assertion $H+K\in\mathcal{L}^\mathbb{B}(X)$ from Theorem \ref{HX=}.

Applying the inner FCS $(T_n,M^{(n)}+A^{(n)})$ for $X\in \mathcal{S}^{i,\mathbb{B}}$, it becomes evident that
\[
(M^{(n)}+A^{(n)})^{T_n\wedge (T_F-)}=(\widetilde{M}^{(n)}+\widetilde{A}^{(n)})^{T_n\wedge (T_F-)}=(\widehat{M}^{(n)}+\widehat{A}^{(n)})^{T_n\wedge (T_F-)},\quad n\in \mathbf{N}^+.
\]
Then it follows from Theorem \ref{eq-HX} that for each $n\in \mathbf{N}^+$,
\begin{align*}
((H+K)_{\bullet}X)I_{\mathbb{B}\llbracket{0,T_n}\rrbracket}
&=\left((\widetilde{H}+\widetilde{K}).(M^{(n)}+A^{(n)})\right)I_{\mathbb{B}\llbracket{0,T_n}\rrbracket}\\
&=\left(\widetilde{H}.(M^{(n)}+A^{(n)})^{T_n\wedge (T_F-)}+\widetilde{K}.(M^{(n)}+A^{(n)})^{T_n\wedge (T_F-)}\right)I_{\mathbb{B}\llbracket{0,T_n}\rrbracket}\\
&=\left(\widetilde{H}.(\widetilde{M}^{(n)}+\widetilde{A}^{(n)})^{T_n\wedge (T_F-)}\right)I_{\mathbb{B}\llbracket{0,T_n}\rrbracket}+\left(\widetilde{K}.(\widehat{M}^{(n)}+\widehat{A}^{(n)})^{T_n\wedge (T_F-)}\right)I_{\mathbb{B}\llbracket{0,T_n}\rrbracket}\\
&=\left(\widetilde{H}.(\widetilde{M}^{(n)}+\widetilde{A}^{(n)})\right)I_{\mathbb{B}\llbracket{0,T_n}\rrbracket}
+\left(\widetilde{K}.(\widehat{M}^{(n)}+\widehat{A}^{(n)})\right)I_{\mathbb{B}\llbracket{0,T_n}\rrbracket}\\
&=(H_{\bullet}X+K_{\bullet}X)I_{\mathbb{B}\llbracket{0,T_n}\rrbracket}.
\end{align*}
Thus, by Theorem \ref{process}(1), we establish the assertion $(H+K)_{\bullet}X=H_{\bullet}X+K_{\bullet}X$.

$(2)$ Suppose that $X=M+A$ ($M\in (\mathcal{M}_{\mathrm{loc}})^{i,\mathbb{B}}$ and $A\in (\mathcal{V}_0)^\mathbb{B}$) and $Y=N+V$ ($N\in (\mathcal{M}_{\mathrm{loc}})^{i,\mathbb{B}}$ and $V\in (\mathcal{V}_0)^\mathbb{B}$) are inner decompositions such that $H\in \mathcal{L}_m^{\mathbb{B}}(M)\bigcap\mathcal{L}_m^{\mathbb{B}}(N)$ and $H\in \mathcal{L}_s^{\mathbb{B}}(A)\bigcap\mathcal{L}_s^{\mathbb{B}}(V)$. By invoking Theorem \ref{HM-o-p}(2), it follows that $H\in \mathcal{L}_m^{\mathbb{B}}(aM+bN)$. Additionally, Theorem \ref{HAproperty}(2) establishes that $H\in \mathcal{L}_s^{\mathbb{B}}(aA+bV)$.
Observing $aX+bY=(aM+bN)+(aA+bV)$, where $aM+bN\in(\mathcal{M}_{\mathrm{loc}})^{i,\mathbb{B}}$ and $aA+bV\in(\mathcal{V}_0)^\mathbb{B}$, we can infer from Definition \ref{de-HX} that
$H\in\mathcal{L}^\mathbb{B}(aX+bY)$. Furthermore, \eqref{+HX} can be deduced as a result of employing \eqref{ab2} and \eqref{+HM}:
\begin{align*}
H_{\bullet}(aX+bY)
&=H_{\bullet}(aM+bN)+H_{\bullet}(aA+bV)\\
&=a(H_{\bullet}M)+b(H_{\bullet}M)+a(H_{\bullet}N)+b(H_{\bullet}V)\\
&=a(H_{\bullet}M+H_{\bullet}A)+b(H_{\bullet}N+H_{\bullet}V)\\
&=a(H_{\bullet}X)+b(H_{\bullet}Y).
\end{align*}

$(3)$ {\it Sufficiency.} Suppose that $LH\in\mathcal{L}^\mathbb{B}(X)$. Let $\widetilde{J}$ and $\widetilde{L}$ denote coupled predictable processes for $LH\in\mathcal{P}^\mathbb{B}$ and $L\in\mathcal{P}^\mathbb{B}$, respectively. We further assume that $(T_n,M^{(n)},A^{(n)})$ forms a decomposed inner FCS for $X\in \mathcal{S}^{i,\mathbb{B}}$ such that $\widetilde{J}\in\mathcal{L}_m(M^{(n)})\cap\mathcal{L}_s(A^{(n)})$ for each $n\in \mathbf{N}^+$.

For each $n\in \mathbf{N}^+$, we define
\begin{align*}
N^{(n)}&
=\left(\frac{1}{\widetilde{L}}I_{[|\widetilde{L}|>\varepsilon]}\right).
(\widetilde{J}.M^{(n)})+{I_{[|\widetilde{L}|\leq\varepsilon]}}.(\widetilde{H}.\widetilde{M}^{(n)})
=\left(\frac{\widetilde{J}}{\widetilde{L}}I_{[|\widetilde{L}|>\varepsilon]}\right)
.M^{(n)}+(\widetilde{H}I_{[|\widetilde{L}|\leq\varepsilon]}).\widetilde{M}^{(n)},\\
V^{(n)}&=\left(\frac{1}{\widetilde{L}}I_{[|\widetilde{L}|>\varepsilon]}\right).
(\widetilde{J}.A^{(n)})+{I_{[|\widetilde{L}|\leq\varepsilon]}}.(\widetilde{H}.\widetilde{A}^{(n)})
=\left(\frac{\widetilde{J}}{\widetilde{L}}I_{[|\widetilde{L}|>\varepsilon]}\right)
.A^{(n)}+(\widetilde{H}I_{[|\widetilde{L}|\leq\varepsilon]}).\widetilde{A}^{(n)},
\end{align*}
where all integrals involved are well-defined. For any $k,n\in \mathbf{N}^+$ with $n\leq k$, employing the properties that $(\widetilde{M}^{(n)})^{T_n\wedge (T_F-)}\in\mathcal{M}_{\mathrm{loc}}$ and  $(M^{(n)})^{T_n\wedge (T_F-)}\in\mathcal{M}_{\mathrm{loc}}$, it can be inferred that
\[
N^{(n)}I_{\mathbb{B}\llbracket{0,T_n}\rrbracket}=N^{(k)}I_{\mathbb{B}\llbracket{0,T_n}\rrbracket},\quad
(N^{(n)})^{T_n\wedge (T_F-)}\in\mathcal{M}_{\mathrm{loc}}.
\]
By virtue of Theorem \ref{restriction}(4), we establish that $(T_n,N^{(n)})$ is an inner FCS for a $\mathbb{B}$-process $N\in (\mathcal{M}_{\mathrm{loc}})^{i,\mathbb{B}}$ satisfying $\widetilde{L}\in\mathcal{L}_m(N^{(n)})$ for each $n\in \mathbf{N}^+$.
Similarly, it holds that $(T_n,V^{(n)})$ is an FCS for a $\mathbb{B}$-process $V\in (\mathcal{V}_0)^\mathbb{B}$ such that $\widetilde{L}\in\mathcal{L}_s(V^{(n)})$ for each $n\in \mathbf{N}^+$.
Furthermore, it is easy to verify that $(HI_{[|L|>\varepsilon]}){\mathfrak{I}_\mathbb{B}}_{\bullet} X\in\mathcal{S}^{i,\mathbb{B}}$ with the decomposed inner FCS
\[
\left(T_n,\left(\frac{\widetilde{J}}{\widetilde{L}}I_{[|\widetilde{L}|>\varepsilon]}\right)
.M^{(n)},\left(\frac{\widetilde{J}}{\widetilde{L}}I_{[|\widetilde{L}|>\varepsilon]}\right)
.A^{(n)}\right),
\]
and similarly, $(HI_{[|L|\leq\varepsilon]}){\mathfrak{I}_\mathbb{B}}_{\bullet} X\in\mathcal{S}^{i,\mathbb{B}}$ with the decomposed inner FCS
\[
\left(T_n,(\widetilde{H}I_{[|\widetilde{L}|\leq\varepsilon]}).\widetilde{M}^{(n)},
(\widetilde{H}I_{[|\widetilde{L}|\leq\varepsilon]}).\widetilde{A}^{(n)}\right).
\]
Hence, by invoking part $(1)$, we deduce that $(T_n,N^{(n)},V^{(n)})$ is a decomposed inner FCS for
\[
H_{\bullet} X=(HI_{[|L|>\varepsilon]}){\mathfrak{I}_\mathbb{B}}_{\bullet} X+(HI_{[|L|\leq\varepsilon]}){\mathfrak{I}_\mathbb{B}}_{\bullet} X\in\mathcal{S}^{i,\mathbb{B}},
\]
which implies $H_{\bullet} X=N+V$.
Consequently, given that $\widetilde{L}$ is identified as  a coupled predictable process for $L\in\mathcal{P}^\mathbb{B}$, $(T_n,N^{(n)},V^{(n)})$ forms a decomposed inner FCS for $H_{\bullet} X\in \mathcal{S}^{i,\mathbb{B}}$ such that $\widetilde{L}\in\mathcal{L}_m(N^{(n)})\cap\mathcal{L}_s(V^{(n)})$ for each $n\in \mathbf{N}^+$. Finally, the assertion $L\in\mathcal{L}^\mathbb{B}(H_{\bullet}X)$ is derived from Theorem \ref{HX=}.

In this case, we demonstrate the validity of \eqref{hHX}. For each $n\in \mathbf{N}^+$, the following equality holds:
\begin{align*}
\widetilde{L}.(N^{(n)}+V^{(n)})
=\left(\widetilde{J}I_{[|\widetilde{L}|>\varepsilon]}\right)
.(M^{(n)}+A^{(n)})+\left(\widetilde{L}\widetilde{H}I_{[|\widetilde{L}|\leq\varepsilon]}\right).(\widetilde{M}^{(n)}+\widetilde{A}^{(n)}),
\end{align*}
where the sequence $(T_n,(\widetilde{J}I_{[|\widetilde{L}|>\varepsilon]}).(M^{(n)}+A^{(n)}))$ is an inner FCS for $(LHI_{[|L|>\varepsilon]}){\mathfrak{I}_\mathbb{B}}_{\bullet}X\in \mathcal{S}^{i,\mathbb{B}}$, and the sequence $(T_n,(\widetilde{L}\widetilde{H}I_{[|\widetilde{L}|\leq\varepsilon]}).(\widetilde{M}^{(n)}+\widetilde{A}^{(n)}))$ is an inner FCS for $(LHI_{[|L|\leq \varepsilon]}){\mathfrak{I}_\mathbb{B}}_{\bullet}X\in \mathcal{S}^{i,\mathbb{B}}$.
By invoking part (1), it follows that for each $n\in \mathbf{N}^+$,
\begin{align*}
&(L_{\bullet}(H_{\bullet}X))I_{\mathbb{B}\llbracket{0,T_n}\rrbracket}\\
=&\left\{\widetilde{L}.(N^{(n)}+V^{(n)})\right\}I_{\mathbb{B}\llbracket{0,T_n}\rrbracket}\\
=&\left\{\left(\widetilde{J}I_{[|\widetilde{L}|>\varepsilon]}\right)
.(M^{(n)}+A^{(n)})+\left(\widetilde{L}\widetilde{H}I_{[|\widetilde{L}|\leq\varepsilon]}\right).(\widetilde{M}^{(n)}+\widetilde{A}^{(n)})\right\}
I_{\mathbb{B}\llbracket{0,T_n}\rrbracket}\\
=&\left\{(LHI_{[|L|>\varepsilon]}){\mathfrak{I}_\mathbb{B}}_{\bullet}X+(LHI_{[|L|\leq \varepsilon]}){\mathfrak{I}_\mathbb{B}}_{\bullet}X\right\}I_{\mathbb{B}\llbracket{0,T_n}\rrbracket}\\
=&((LH)_{\bullet}X)I_{\mathbb{B}\llbracket{0,T_n}\rrbracket}.
\end{align*}
Consequently, by applying Theorem \ref{process}(1), we obtain the desired result \eqref{hHX}.

{\it Necessity.} Suppose that $L\in\mathcal{L}^\mathbb{B}(H_{\bullet}X)$, and let $(T_n,Z^{(n)},U^{(n)})$ denote a decomposed inner FCS for  $L_{\bullet}(H_{\bullet}X)\in\mathcal{S}^{i,\mathbb{B}}$.
We further assume that $\widetilde{L}$ is a coupled predictable process for $L\in\mathcal{P}^\mathbb{B}$, and that $(T_n,N^{(n)},V^{(n)})$ forms a decomposed inner FCS for $H_{\bullet}X\in \mathcal{S}^{i,\mathbb{B}}$ such that $\widetilde{L}\in\mathcal{L}_m(N^{(n)})\cap\mathcal{L}_s(V^{(n)})$ for each $n\in \mathbf{N}^+$.

For each $n\in \mathbf{N}^+$, we define
\begin{align*}
M^{(n)}&
=\left(\frac{1}{\widetilde{H}}I_{[|\widetilde{H}|>\varepsilon]}\right).N^{(n)}
+{I_{[|\widetilde{H}|\leq\varepsilon,|\widetilde{L}\widetilde{H}|\leq\varepsilon]}}.\widetilde{M}^{(n)}
+\left(\frac{1}{\widetilde{L}\widetilde{H}}I_{[|\widetilde{H}|\leq\varepsilon,|\widetilde{L}\widetilde{H}|>\varepsilon]}\right).Z^{(n)},\\
A^{(n)}&
=\left(\frac{1}{\widetilde{H}}I_{[|\widetilde{H}|>\varepsilon]}\right).V^{(n)}
+{I_{[|\widetilde{H}|\leq\varepsilon,|\widetilde{L}\widetilde{H}|\leq\varepsilon]}}.\widetilde{A}^{(n)}
+\left(\frac{1}{\widetilde{L}\widetilde{H}}I_{[|\widetilde{H}|\leq\varepsilon,|\widetilde{L}\widetilde{H}|>\varepsilon]}\right).U^{(n)},
\end{align*}
where all integrals involved are well-defined.  For any $k,n\in \mathbf{N}^+$ with $n\leq k$, utilizing the properties that $(N^{(n)})^{T_n\wedge (T_F-)}\in\mathcal{M}_{\mathrm{loc}}$, $(\widetilde{M}^{(n)})^{T_n\wedge (T_F-)}\in\mathcal{M}_{\mathrm{loc}}$ and  $(Z^{(n)})^{T_n\wedge (T_F-)}\in\mathcal{M}_{\mathrm{loc}}$, it can be easily verified that
\[
M^{(n)}I_{\mathbb{B}\llbracket{0,T_n}\rrbracket}=M^{(k)}I_{\mathbb{B}\llbracket{0,T_n}\rrbracket},\quad
(M^{(n)})^{T_n\wedge (T_F-)}\in\mathcal{M}_{\mathrm{loc}}.
\]
By invoking Theorem \ref{restriction}(4), it follows that $(T_n,M^{(n)})$ is an inner FCS for a $\mathbb{B}$-process $M\in (\mathcal{M}_{\mathrm{loc}})^{i,\mathbb{B}}$ satisfying $\widetilde{L}\widetilde{H}\in\mathcal{L}_m(M^{(n)})$.
Similarly, it holds that $(T_n,A^{(n)})$ is an FCS for a $\mathbb{B}$-process $A\in (\mathcal{V}_0)^\mathbb{B}$ such that  $\widetilde{L}\widetilde{H}\in\mathcal{L}_s(A^{(n)})$ for each $n\in \mathbf{N}^+$.
Considering the two inner FCSs $(T_n,N^{(n)}+V^{(n)})$ and $(T_n,\widetilde{H}.(\widetilde{M}^{(n)}+\widetilde{A}^{(n)}))$  for $H_{\bullet}X\in \mathcal{S}^{i,\mathbb{B}}$, we obtain:
\[
(N^{(n)}+V^{(n)})^{T_n\wedge (T_F-)}=(\widetilde{H}.(\widetilde{M}^{(n)}+\widetilde{A}^{(n)}))^{T_n\wedge (T_F-)}=\widetilde{H}.(\widetilde{M}^{(n)}+\widetilde{A}^{(n)})^{T_n\wedge (T_F-)},\quad n\in \mathbf{N}^+.
\]
Analogously, for the inner FCSs $(T_n,Z^{(n)}+U^{(n)})$ and $(T_n,\widetilde{L}.(N^{(n)}+V^{(n)}))$  for $L_{\bullet}(H_{\bullet}X)\in\mathcal{S}^{i,\mathbb{B}}$, we deduce:
\[
(Z^{(n)}+U^{(n)})^{T_n\wedge (T_F-)}=(\widetilde{L}.(N^{(n)}+V^{(n)}))^{T_n\wedge (T_F-)}=(\widetilde{L}\widetilde{H}).(\widetilde{M}^{(n)}+\widetilde{A}^{(n)})^{T_n\wedge (T_F-)},\quad n\in \mathbf{N}^+.
\]
Thus, for each $n\in \mathbf{N}^+$, we have
\begin{align*}
&(M+A)I_{\mathbb{B}\llbracket{0,T_n}\rrbracket}\\
=&(M^{(n)}+A^{(n)})I_{\mathbb{B}\llbracket{0,T_n}\rrbracket}\\
=&\bigg\{\left(\frac{1}{\widetilde{H}}I_{[|\widetilde{H}|>\varepsilon]}\right).(N^{(n)}+V^{(n)})^{T_n\wedge(T_F-)}
+{I_{[|\widetilde{H}|\leq\varepsilon,|\widetilde{L}\widetilde{H}|\leq\varepsilon]}}.(\widetilde{M}^{(n)}+\widetilde{A}^{(n)})\\
&+\left(\frac{1}{\widetilde{L}\widetilde{H}}I_{[|\widetilde{H}|\leq\varepsilon,|\widetilde{L}\widetilde{H}|>\varepsilon]}\right).(Z^{(n)}+U^{(n)})^{T_n\wedge(T_F-)}\bigg\}
I_{\mathbb{B}\llbracket{0,T_n}\rrbracket}\\
=&\bigg\{\left(\frac{1}{\widetilde{H}}I_{[|\widetilde{H}|>\varepsilon]}\right).
\left(\widetilde{H}.(\widetilde{M}^{(n)}+\widetilde{A}^{(n)})^{T_n\wedge (T_F-)}\right)+{I_{[|\widetilde{H}|\leq\varepsilon,|\widetilde{L}\widetilde{H}|\leq\varepsilon]}}.(\widetilde{M}^{(n)}+\widetilde{A}^{(n)})\\
&+\left(\frac{1}{\widetilde{L}\widetilde{H}}I_{[|\widetilde{H}|\leq\varepsilon,|\widetilde{L}\widetilde{H}|>\varepsilon]}\right).\left((\widetilde{L}\widetilde{H}).(\widetilde{M}^{(n)}+\widetilde{A}^{(n)})^{T_n\wedge (T_F-)}\right)\bigg\}
I_{\mathbb{B}\llbracket{0,T_n}\rrbracket}\\
=&\left\{{I_{[|\widetilde{H}|>\varepsilon]}}.(\widetilde{M}^{(n)}+\widetilde{A}^{(n)})
+{I_{[|\widetilde{H}|\leq\varepsilon,|\widetilde{L}\widetilde{H}|\leq\varepsilon]}}.(\widetilde{M}^{(n)}+\widetilde{A}^{(n)})
+{I_{[|\widetilde{H}|\leq\varepsilon,|\widetilde{L}\widetilde{H}|>\varepsilon]}}.(\widetilde{M}^{(n)}+\widetilde{A}^{(n)})\right\}
I_{\mathbb{B}\llbracket{0,T_n}\rrbracket}\\
=&(\widetilde{M}^{(n)}+\widetilde{A}^{(n)})I_{\mathbb{B}\llbracket{0,T_n}\rrbracket}\\
=&XI_{\mathbb{B}\llbracket{0,T_n}\rrbracket}.
\end{align*}
By Theorem \ref{process}(1), we conclude that $M+A=X$.
Therefore, $\widetilde{L}\widetilde{H}$ is a coupled predictable process for $LH\in\mathcal{P}^\mathbb{B}$, and $(T_n,M^{(n)},A^{(n)})$ forms a decomposed inner FCS for $X\in \mathcal{S}^{i,\mathbb{B}}$ such that $\widetilde{L}\widetilde{H}\in\mathcal{L}_m(M^{(n)})\cap\mathcal{L}_s(A^{(n)})$ for each $n\in \mathbf{N}^+$. Finally, the assertion $LH\in\mathcal{L}^\mathbb{B}(X)$ is derived from Theorem \ref{HX=}.
$\hfill\blacksquare$

%%%%%%%%%%%%%%%%%%%%%%%%%%%%%%%%%%%%%%%%%%%%%%%%%%%%%%%%%%%%%%%%%%%%%%%%%%%%%%%%%%%%%%%%%%%%%%%%%%%%%%%
\par\vspace{0.3cm}
%%%%%%%%%%%%%%%%%%%%%%%%%%%%%%%%%%%%%%%%%%%%%%%%%%%%%%%%%%%%%%%%%%%%%%%%%%%%%%%%%%%%%%%%%%%%%%%%%%%%%%%
\noindent {\bf Proof of Theorem \ref{HX-property}.}
From $H\in\mathcal{L}^\mathbb{B}(X)$, Definition \ref{de-HX} yields an inner decomposition $X=M+A$, where $M\in (\mathcal{M}_{\mathrm{loc}})^{i,\mathbb{B}}$ and $A\in (\mathcal{V}_0)^\mathbb{B}$, ensuring that $H\in \mathcal{L}_m^{\mathbb{B}}(M)\cap\mathcal{L}_s^{\mathbb{B}}(A)$.
Assume that $(T_n,H^{(n)})$ is an FCS for $H\in\mathcal{P}^\mathbb{B}$, and that $(T_n,M^{(n)},A^{(n)})$ forms a decomposed inner FCS for $X\in \mathcal{S}^{i,\mathbb{B}}$, with the property that $H^{(n)}\in\mathcal{L}_m(M^{(n)})\cap\mathcal{L}_s(A^{(n)})$ for each $n\in \mathbf{N}^+$.
Subsequently, Theorem \ref{eq-HX} elucidates that $(T_n,H^{(n)}.M^{(n)},H^{(n)}.A^{(n)})$ serves as a decomposed inner FCS for $H_{\bullet}X\in\mathcal{S}^{i,\mathbb{B}}$.
Furthermore, Theorem \ref{fcs} establishes that $X^\tau \mathfrak{I}_\mathbb{B}\in\mathcal{S}^{i,\mathbb{B}}$ with the decomposed inner FCS $(T_n,(M^{(n)})^\tau,(A^{(n)})^\tau)$, and $H^\tau \mathfrak{I}_\mathbb{B}\in\mathcal{P}^\mathbb{B}$ with the FCS $(T_n,(H^{(n)})^\tau)$.
For the purpose of clarity, let $X^{(n)}=M^{(n)}+A^{(n)}$ for each $n\in \mathbf{N}^+$.

$(1)$ By applying Theorem \ref{HM-property}, it is straightforward to infer that $(H_{\bullet}X)^c=H_{\bullet}X^c$ through the following equalities:
\[
(H_{\bullet}X)^c=(H_{\bullet}M+H_{\bullet}A)^c=(H_{\bullet}M)^c
=H_{\bullet}M^c=H_{\bullet}X^c.
\]
The identity $(H_{\bullet}X)I_{\llbracket{0}\rrbracket}=HXI_{\llbracket{0}\rrbracket}$ is a direct consequence of equation \eqref{HX-expression-2}. Furthermore, invoking Theorem \ref{delta}(1), we obtain that for each $n\in \mathbf{N}^+$,
\begin{align*}
\Delta(H_{\bullet}X)I_{\mathbb{B}\llbracket{0,T_n}\rrbracket}
=\Delta(H^{(n)}.X^{(n)})I_{\mathbb{B}\llbracket{0,T_n}\rrbracket}
=(H^{(n)}\Delta X^{(n)})I_{\mathbb{B}\llbracket{0,T_n}\rrbracket}
=(H\Delta X)I_{\mathbb{B}\llbracket{0,T_n}\rrbracket},
\end{align*}
which, according to Theorem \ref{process}(1), implies $\Delta (H_{\bullet}X)=H\Delta X$.

$(2)$ Firstly, for each $n\in \mathbf{N}^+$, we establish that
$(H^{(n)}.M^{(n)})^\tau=H^{(n)}.(M^{(n)})^\tau$ and $(H^{(n)}.A^{(n)})^\tau=H^{(n)}.(A^{(n)})^\tau$, demonstrating that $H^{(n)}\in\mathcal{L}_m((M^{(n)})^\tau)\cap\mathcal{L}_s((A^{(n)})^\tau)$.
Consequently, Theorem \ref{HX=} confirms that $H\in \mathcal{L}^\mathbb{B}(X^\tau \mathfrak{I}_\mathbb{B})$. It is deduced from Theorem \ref{eq-HX} that $(T_n,H^{(n)}.(X^{(n)})^\tau)$ forms an inner FCS for $H_{\bullet}(X^\tau \mathfrak{I}_\mathbb{B})\in\mathcal{S}^{i,\mathbb{B}}$.
By leveraging Theorems \ref{fcs} and \ref{eq-HX}, it is straightforward to derive the following for $n\in\mathbf{N}^+$:
\begin{align*}
(H_{\bullet}X)^\tau I_{\mathbb{B}\llbracket{0,T_n}\rrbracket}
=(H^{(n)}.X^{(n)})^\tau I_{\mathbb{B}\llbracket{0,T_n}\rrbracket}
=(H^{(n)}.(X^{(n)})^\tau)I_{\mathbb{B}\llbracket{0,T_n}\rrbracket}
=(H_{\bullet}(X^\tau \mathfrak{I}_\mathbb{B}))I_{\mathbb{B}\llbracket{0,T_n}\rrbracket},
\end{align*}
which, by Theorem \ref{process}(1), leads to the equality $(H_{\bullet}X)^\tau\mathfrak{I}_{B}=H_{\bullet}(X^\tau\mathfrak{I}_{B})$.

Secondly, for each $n\in \mathbf{N}^+$, we establish the relations
$(H^{(n)}.M^{(n)})^\tau=(H^{(n)})^\tau.(M^{(n)})^\tau$ and $(H^{(n)}.A^{(n)})^\tau=(H^{(n)})^\tau.(A^{(n)})^\tau$, signifying that $(H^{(n)})^\tau\in\mathcal{L}_m((M^{(n)})^\tau)\cap\mathcal{L}_s((A^{(n)})^\tau)$.
Therefore, Theorem \ref{HX=} indicates that $H^\tau\mathfrak{I}_\mathbb{B}\in \mathcal{L}^\mathbb{B}(X^\tau \mathfrak{I}_\mathbb{B})$. From Theorem \ref{eq-HX}, $(T_n,(H^{(n)})^\tau.(X^{(n)})^\tau)$ forms an inner FCS for $(H^\tau \mathfrak{I}_\mathbb{B})_{\bullet}(X^\tau \mathfrak{I}_\mathbb{B})\in\mathcal{S}^{i,\mathbb{B}}$.
Utilizing Theorems \ref{fcs} and \ref{eq-HX}, we can deduce the following for $n\in\mathbf{N}^+$:
\begin{align*}
(H_{\bullet}X)^\tau I_{\mathbb{B}\llbracket{0,T_n}\rrbracket}
=(H^{(n)}.X^{(n)})^\tau I_{\mathbb{B}\llbracket{0,T_n}\rrbracket}
=((H^{(n)})^\tau.(X^{(n)})^\tau)I_{\mathbb{B}\llbracket{0,T_n}\rrbracket}
=((H^\tau \mathfrak{I}_\mathbb{B})_{\bullet}(X^\tau \mathfrak{I}_\mathbb{B}))I_{\mathbb{B}\llbracket{0,T_n}\rrbracket},
\end{align*}
which, by Theorem \ref{process}(1), indicates the equality $(H_{\bullet}X)^\tau\mathfrak{I}_{\mathbb{B}}=(H^\tau\mathfrak{I}_{\mathbb{B}})_{\bullet}(X^\tau\mathfrak{I}_{\mathbb{B}})$.

Finally, to conclude the proof, we demonstrate that $HI_{\llbracket{0,\tau}\rrbracket}\mathfrak{I}_\mathbb{B}\in \mathcal{L}^\mathbb{B}(X)$ and $(H_{\bullet}X)^\tau\mathfrak{I}_\mathbb{B}
      =(HI_{\llbracket{0,\tau}\rrbracket}\mathfrak{I}_\mathbb{B})_{\bullet}X$.
It is evident that
\begin{equation*}
(HI_{\llbracket{0,\tau}\rrbracket}\mathfrak{I}_\mathbb{B})I_{\mathbb{B}\llbracket{0,T_n}\rrbracket}
=HI_{\mathbb{B}\llbracket{0,T_n}\rrbracket}I_{\llbracket{0,\tau}\rrbracket}
=(H^{(n)}I_{\llbracket{0,\tau}\rrbracket})I_{\mathbb{B}\llbracket{0,T_n}\rrbracket}, \quad n\in \mathbf{N}^+,
\end{equation*}
which, by the relation $H^{(n)}I_{\llbracket{0,\tau}\rrbracket}\in \mathcal{P}$, implies that  $(T_n,H^{(n)}I_{\llbracket{0,\tau}\rrbracket})$ constitutes an FCS for $HI_{\llbracket{0,\tau}\rrbracket}\mathfrak{I}_\mathbb{B}\in \mathcal{P}^\mathbb{B}$.
For each $n\in \mathbf{N}^+$, it verifies that
$(H^{(n)}.M^{(n)})^\tau=(H^{(n)}I_{\llbracket{0,\tau}\rrbracket}).M^{(n)}$ and
$(H^{(n)}.A^{(n)})^\tau=(H^{(n)}I_{\llbracket{0,\tau}\rrbracket}).A^{(n)}$, meaning that $H^{(n)}I_{\llbracket{0,\tau}\rrbracket}\in\mathcal{L}_m(M^{(n)})\cap\mathcal{L}_s(A^{(n)})$.
Hence, Theorem \ref{HX=} confirms that $HI_{\llbracket{0,\tau}\rrbracket}\mathfrak{I}_\mathbb{B}\in \mathcal{L}^\mathbb{B}(X)$.
Theorem \ref{eq-HX} shows that $(T_n,(H^{(n)}I_{\llbracket{0,\tau}\rrbracket}).X^{(n)})$ forms an inner FCS for $(HI_{\llbracket{0,\tau}\rrbracket}\mathfrak{I}_\mathbb{B})_{\bullet}X\in\mathcal{S}^{i,\mathbb{B}}$.
By applying Theorems \ref{fcs} and \ref{eq-HX}, we obtain the following for $n\in\mathbf{N}^+$:
\begin{align*}
(H_{\bullet}X)^\tau I_{\mathbb{B}\llbracket{0,T_n}\rrbracket}
=(H^{(n)}.X^{(n)})^\tau I_{\mathbb{B}\llbracket{0,T_n}\rrbracket}
=((H^{(n)}I_{\llbracket{0,\tau}\rrbracket}).X^{(n)})I_{\mathbb{B}\llbracket{0,T_n}\rrbracket}
=((HI_{\llbracket{0,\tau}\rrbracket}\mathfrak{I}_\mathbb{B})_{\bullet}X)I_{\mathbb{B}\llbracket{0,T_n}\rrbracket},
\end{align*}
which, according to Theorem \ref{process}(1), establishes the equality $(H_{\bullet}X)^\tau\mathfrak{I}_\mathbb{B}
      =(HI_{\llbracket{0,\tau}\rrbracket}\mathfrak{I}_\mathbb{B})_{\bullet}X$.

$(3)$ Let $(T_n,\widetilde{L}^{(n)})$ denote an FCS for $L\in\mathcal{P}^\mathbb{B}$. For each $n\in \mathbf{N}^+$, we define
\[
L^{(n)}=\min\{(\widetilde{L}^{(n)})^+, |H^{(n)}|\}-\min\{(\widetilde{L}^{(n)})^-, |H^{(n)}|\}.
\]
Given that the inequality $|L|\leq |H|$ implies $|\widetilde{L}^{(n)}|\leq |H^{(n)}|$ on $\mathbb{B}\llbracket{0,T_n}\rrbracket$, it follows that $L^{(n)}I_{\mathbb{B}\llbracket{0,T_n}\rrbracket}=\widetilde{L}^{(n)}I_{\mathbb{B}\llbracket{0,T_n}\rrbracket}$ for each $n\in \mathbf{N}^+$.
Subsequently, $(T_n,L^{(n)})$ forms an FCS for $L\in \mathcal{P}^\mathbb{B}$, satisfying the condition $|L^{(n)}|\leq 2|H^{(n)}|$ for each $n\in \mathbf{N}^+$.
For each $n\in \mathbf{N}^+$, by invoking the facts $|L^{(n)}|\leq 2|H^{(n)}|$ and $2H^{(n)}\in\mathcal{L}_m(M^{(n)})\cap\mathcal{L}_s(A^{(n)})$, it follows that $L^{(n)}\in\mathcal{L}_m(M^{(n)})\cap\mathcal{L}_s(A^{(n)})$. Consequently,
$(T_n,M^{(n)},A^{(n)})$ forms a decomposed inner FCS for $X\in \mathcal{S}^{i,\mathbb{B}}$ such that $L^{(n)}\in\mathcal{L}_m(M^{(n)})\cap\mathcal{L}_s(A^{(n)})$ for each $n\in \mathbf{N}^+$.
By utilizing Theorem \ref{HX=}, we conclude that $L\in\mathcal{L}^\mathbb{B}(X)$.

$(4)$ Suppose that $(T_n,Y^{(n)})$ is an FCS for $Y\in \mathcal{S}^\mathbb{B}$.
For each $n\in \mathbf{N}^+$, we have the following sequence of equalities:
\begin{equation*}
[H_{\bullet}X,Y]I_{\mathbb{B}\llbracket{0,T_n}\rrbracket}
=[H^{(n)}.X^{(n)},Y^{(n)}]I_{\mathbb{B}\llbracket{0,T_n}\rrbracket}
=\left(H^{(n)}.[X^{(n)},Y^{(n)}]\right)I_{\mathbb{B}\llbracket{0,T_n}\rrbracket}
=\left(H_{\bullet}[X,Y]\right)I_{\mathbb{B}\llbracket{0,T_n}\rrbracket},
\end{equation*}
where the first equality comes from Theorem \ref{[X,Y]-fcs}(2) and $H_{\bullet}X\in\mathcal{S}^{i,\mathbb{B}}$ with the inner FCS $(T_n,H^{(n)}.X^{(n)})$, the second equality from Theorem 9.15 of \cite{He}, and the last equality from Theorem \ref{HA-FCS} and $[X,Y]\in \mathcal{V}^\mathbb{B}$ with the FCS $(T_n,[X^{(n)},Y^{(n)}])$.
Consequently, by applying Theorem \ref{process}(1), we obtain the desired result \eqref{HXY-p}.
$\hfill\blacksquare$

%%%%%%%%%%%%%%%%%%%%%%%%%%%%%%%%%%%%%%%%%%%%%%%%%%%%%%%%%%%%%%%%%%%%%%%%%%%%%%%%%%%%%%%%%%%%%%%%%%%%%%%
\par\vspace{0.3cm}
%%%%%%%%%%%%%%%%%%%%%%%%%%%%%%%%%%%%%%%%%%%%%%%%%%%%%%%%%%%%%%%%%%%%%%%%%%%%%%%%%%%%%%%%%%%%%%%%%%%%%%%
\noindent {\bf Proof of Theorem \ref{ito}.}
Based on Corollary \ref{deltaX}(2) and Corollary \ref{bound-HX}, the stochastic integrals featured in \eqref{ito-eq} are rigorously defined.
Consider an inner FCS $(T_n,X_k^{(n)})$ for $X_k\in\mathcal{S}^{i,\mathbb{B}}$, where $k=1,2,\cdots,d$.
For each $n\in \mathbf{N}^+$, define
$Z^{(n)}=(X_{1}^{(n)},X_{2}^{(n)},\cdots, X_{d}^{(n)})$ and introduce $\eta^{(n)}=\Sigma \alpha^{(n)}$ with
\[
\alpha^{(n)}=F(Z^{(n)})-F(Z^{(n)}_{-})-\sum_{k=1}^d D_kF(Z^{(n)}_{-})\Delta X^{(n)}_k.
\]
It is derived from Theorem \ref{delta}(1) and Corollary \ref{deltaX}(2) that $(T_n,\alpha^{(n)})$ is an FCS for $\alpha$ (a $\mathbb{B}$-thin process), where $\alpha=F(Z)-F(Z_{-})-\sum_{k=1}^d D_kF(Z_{-})\Delta X_k$,
and according to Theorem 9.35 in \cite{He}, $\eta^{(n)}$ is well-defined.
Consequently, Theorem \ref{thin}(1) establishes that $\eta$ is well-defined and that $(T_n,\eta^{(n)})$ constitutes an FCS for $\eta\in \mathcal{V}^\mathbb{B}$.
Subsequently, we infer that for each $n\in \mathbf{N}^+$, the following equalities hold:
\begin{align*}
&\big(F(Z)-F(Z(0))\mathfrak{I}_\mathbb{B}\big)I_{\mathbb{B}\llbracket{0,T_n}\rrbracket}\\
=&\big(F(Z^{(n)})-F(Z^{(n)}(0))\big)I_{\mathbb{B}\llbracket{0,T_n}\rrbracket}\\
=&\bigg(\sum_{k=1}^d D_kF(Z^{(n)}_{-}).(X_k^{(n)}-X_k^{(n)}(0))+\eta^{(n)}+
\frac{1}{2}\sum_{k,l=1}^d D_{kl}F(Z^{(n)}_{-}).\langle(X_k^{(n)})^c,(X_l^{(n)})^c\rangle\bigg)I_{\mathbb{B}\llbracket{0,T_n}\rrbracket}\\
=&\bigg(\sum_{k=1}^d D_kF(Z_{-})_{\bullet}(X_k-X_k(0)\mathfrak{I}_\mathbb{B})+\eta+
\frac{1}{2}\sum_{k,l=1}^d D_{kl}F(Z_{-})_{\bullet}\langle X_k^c,X_l^c\rangle\bigg)I_{\mathbb{B}\llbracket{0,T_n}\rrbracket}.
\end{align*}
Here, the second equality is derived from the It\^{o} formula for semimartingales (see, e.g., Theorem 9.35 in \cite{He}), while the third equality stems from Corollary \ref{bound-HX} and the properties $\langle X_k^c,X_l^c\rangle\in(\mathcal{A}_{\mathrm{loc}}\cap \mathcal{C})^\mathbb{B}$ with the FCS $(T_n,\langle(X_k^{(n)})^c,(X_l^{(n)})^c\rangle)$ (as per Theorem \ref{property-qr-p}(1) and Theorem \ref{XTc}(1)).
Therefore, Theorem \ref{process}(1) leads to the confirmation of \eqref{ito}.
$\hfill\blacksquare$

%%%%%%%%%%%%%%%%%%%%%%%%%%%%%%%%%%%%%%%%%%%%%%%%%%%%%%%%%%%%%%%%%%%%%%%%%%%%%%%%%%%%%%%%%%%%%%%%%%%%%%%
\par\vspace{0.3cm}
%%%%%%%%%%%%%%%%%%%%%%%%%%%%%%%%%%%%%%%%%%%%%%%%%%%%%%%%%%%%%%%%%%%%%%%%%%%%%%%%%%%%%%%%%%%%%%%%%%%%%%%
\noindent {\bf Proof of Corollary \ref{IbP-p}.}
By applying Theorem \ref{ito} with $d=2$, where $Z=(X,Y)$ and $F(x,y)=xy$, we derive the following expression:
\begin{equation}\label{IP-1}
XY-X_0Y_0\mathfrak{I}_\mathbb{B}={X_{-}}{}_{\bullet}(Y-Y_0\mathfrak{I}_\mathbb{B})+{Y_{-}}{}_{\bullet}(X-X_0\mathfrak{I}_\mathbb{B})
+\langle X^c,Y^c\rangle+\Sigma(XY-X_{-}Y_{-}-X_{-}\Delta Y-Y_{-}\Delta X).
\end{equation}
From \eqref{HX-expression-2}, it is straightforward to observe that
\[
{X_{-}}{}_{\bullet}(Y_0\mathfrak{I}_\mathbb{B})={Y_{-}}{}_{\bullet}(X_0\mathfrak{I}_\mathbb{B})
=X_0Y_0\mathfrak{I}_\mathbb{B}.
\]
Furthermore, utilizing the relation $\Delta X\Delta Y=XY-X_{-}Y_{-}-X_{-}\Delta Y-Y_{-}\Delta X$, we obtain
\begin{equation}\label{IP-2}
\langle X^c,Y^c\rangle+\Sigma(XY-X_{-}Y_{-}-X_{-}\Delta Y-Y_{-}\Delta X)
=[X,Y]-X_0Y_0\mathfrak{I}_\mathbb{B}.
\end{equation}
Finally, \eqref{IbP-eq} is established by invoking Theorem \ref{HX-p}(2) and substituting \eqref{IP-2} into \eqref{IP-1}.
$\hfill\blacksquare$

%%%%%%%%%%%%%%%%%%%%%%%%%%%%%%%%%%%%%%%%%%%%%%%%%%%%%%%%%%%%%%%%%%%%%%%%%%%%%%%%%%%%%%%%%%%%%%%%%%%%%%%
\par\vspace{0.3cm}
%%%%%%%%%%%%%%%%%%%%%%%%%%%%%%%%%%%%%%%%%%%%%%%%%%%%%%%%%%%%%%%%%%%%%%%%%%%%%%%%%%%%%%%%%%%%%%%%%%%%%%%
\noindent {\bf Proof of Proposition \ref{strategy-e}.}
Let $\mathbb{B}$ be given by \eqref{B}.
We just provide a proof for the case of the self-financing strategy, and the demonstration for the case of the $\alpha$-admissible strategy can be carried out in a similar manner.

{\it Necessity.} Assume that ${\vartheta}\in L(S,\mathbb{F},\mathbb{B})$ is a self-financing strategy. By leveraging equation \eqref{wealth1} and the results from Theorems \ref{HX=} and \ref{eq-HX}, we can infer the existence of an FCS $(T_n,\vartheta^{(n)})$ for ${\vartheta}\in \mathcal{P}^\mathbb{B}$ and a decomposed inner FCS $(T_n,\widetilde{M}^{(n)},\widetilde{A}^{(n)})$ for $S\in \mathcal{S}^{i,\mathbb{B}}$. These satisfy the condition that for every $n\in \mathbf{N}^+$, ${\vartheta}^{(n)}\in\mathcal{L}_m(\widetilde{M}^{(n)})\cap\mathcal{L}_s(\widetilde{A}^{(n)})$ and that $(T_n,{\vartheta}^{(n)}.\widetilde{M}^{(n)},{\vartheta}^{(n)}.\widetilde{M}^{(n)})$ forms a decomposed inner FCS for $\vartheta_{\bullet}S\in\mathcal{S}^{i,\mathbb{B}}$. Define the following:
\[
M^{(n)}=(\widetilde{M}^{(n)})^{T_n\wedge (T_F-)},\quad A^{(n)}=(\widetilde{A}^{(n)})^{T_n\wedge (T_F-)},\quad
\widetilde{S}^{(n)}=\widetilde{M}^{(n)}+\widetilde{A}^{(n)},\quad S^{(n)}=M^{(n)}+A^{(n)},\quad n\in \mathbf{N}^+.
\]
It follows that $(T_n,M^{(n)},A^{(n)})$ also constitutes a decomposed inner FCS for $S\in \mathcal{S}^{i,\mathbb{B}}$ such that ${\vartheta}^{(n)}\in\mathcal{L}_m(M^{(n)})\cap\mathcal{L}_s(A^{(n)})$ for each $n\in \mathbf{N}^+$.
For each $n\in \mathbf{N}^+$, denote $\widetilde{X}^{(n)}$ as the wealth in the financial market $(\widetilde{S}^{(n)},\mathbb{F})$. Consequently, $X^{(n)}=(\widetilde{X}^{(n)})^{T_n\wedge (T_F-)}$ represents the wealth in $(S^{(n)},\mathbb{F})$. By applying \eqref{wealth1} and (G), it becomes evident that for each $n\in \mathbf{N}^+$,
\[
\widetilde{X}^{(n)}I_{\mathbb{B}\llbracket{0,T_n}\rrbracket}
=XI_{\mathbb{B}\llbracket{0,T_n}\rrbracket}
=(({x}_0-\vartheta_0S_0)\mathfrak{I}_\mathbb{B}+{\vartheta}_{\bullet}{S})I_{\mathbb{B}\llbracket{0,T_n}\rrbracket}
=(x_0-{\vartheta}^{(n)}_0\widetilde{S}^{(n)}_0+{\vartheta}^{(n)}.\widetilde{S}^{(n)})I_{\mathbb{B}\llbracket{0,T_n}\rrbracket},
\]
which further implies that $X^{(n)}=x_0-{\vartheta}^{(n)}_0S^{(n)}_0+{\vartheta}^{(n)}.S^{(n)}$. Hence, for each $n\in \mathbf{N}^+$, the strategy ${\vartheta}^{(n)}$ is self-financing in the financial market $(S^{(n)},\mathbb{F})$.

{\it Sufficiency.} Let $(T_n,\vartheta^{(n)})$ denote the FCS for ${\vartheta}\in \mathcal{P}^\mathbb{B}$, and $(T_n,M^{(n)},A^{(n)})$ represent the decomposed inner FCS for $S\in \mathcal{S}^{i,\mathbb{B}}$. For each $n\in \mathbf{N}^+$, the strategy ${\vartheta}^{(n)}$ is assumed to be self-financing within the financial market defined by $({S}^{(n)}=M^{(n)}+A^{(n)},\mathbb{F})$. Consequently, it can be inferred that for each $n\in \mathbf{N}^+$, ${\vartheta}^{(n)}\in L({S}^{(n)},\mathbb{F})$ satisfying the equation:
\begin{equation}\label{wealth2}
{X}^{(n)}={x}_0-{\vartheta}^{(n)}_0S^{(n)}_0+{\vartheta}^{(n)}.{S}^{(n)},
\end{equation}
where ${X}^{(n)}$ signifies the investor's wealth in the financial market $({S}^{(n)},\mathbb{F})$. By employing (G), \eqref{wealth2}, and Theorem \ref{eq-HX}, we can effortlessly derive \eqref{wealth1} as follows:
\begin{align*}
{X}I_{\mathbb{B}\llbracket{0,T_n}\rrbracket}
&={X}^{(n)}I_{\mathbb{B}\llbracket{0,T_n}\rrbracket}\\
&=({x}_0-{\vartheta}^{(n)}_0S^{(n)}_0+{\vartheta}^{(n)}.{S}^{(n)})I_{\mathbb{B}\llbracket{0,T_n}\rrbracket}\\
&=(({x}_0-\vartheta_0S_0)\mathfrak{I}_\mathbb{B}+{\vartheta}_{\bullet}{S})I_{\mathbb{B}\llbracket{0,T_n}\rrbracket}
\end{align*}
for each $n\in \mathbf{N}^+$. Therefore, ${\vartheta}$ is self-financing in $(S,\mathbb{F},\mathbb{B})$.
$\hfill\blacksquare$

%%%%%%%%%%%%%%%%%%%%%%%%%%%%%%%%%%%%%%%%%%%%%%%%%%%%%%%%%%%%%%%%%%%%%%%%%%%%%%%%%%%%%%%%%%%%%%%%%%%%%%%
\par\vspace{0.3cm}
%%%%%%%%%%%%%%%%%%%%%%%%%%%%%%%%%%%%%%%%%%%%%%%%%%%%%%%%%%%%%%%%%%%%%%%%%%%%%%%%%%%%%%%%%%%%%%%%%%%%%%%
\noindent {\bf Proof of Lemma \ref{Z}.}
Let $S$ be specified by \eqref{expZ}.
It suffices to prove that $S$ is the unique $\mathbb{B}$-inner semimartingale satisfying \eqref{price}, and the rest of the proof is trivial.

Define $X=Z-\frac{1}{2}\langle Z^c\rangle$. It is evident that $X\in \mathcal{S}^{i,\mathbb{B}}$, satisfying the conditions $\Delta X=\Delta Z=0$ (as established by Theorem \ref{HX-property}(1)), $X^c=Z^c$ and $X_0=0$. Applying Theorem \ref{ito} with $d=1$ and $f(x)=s_0e^x$, we derive the following:
\begin{align*}
S=&f(X)\\
=&f(X_0)\mathfrak{I}_\mathbb{B}+f'(X_{-})_{\bullet}X+\frac{1}{2}f''(X_{-})_{\bullet}\langle X^c\rangle\\
=&s_0\mathfrak{I}_\mathbb{B}+{S_{-}}{}_{\bullet}\left(Z-\frac{1}{2}\langle Z^c\rangle\right)+\frac{1}{2}{S_{-}}{}_{\bullet}\langle Z^c\rangle\\
=&s_0\mathfrak{I}_\mathbb{B}+{S_{-}}{}_{\bullet}Z,
\end{align*}
where the last equality is justified by Theorem \ref{HX-p}. Consequently, $S$ emerges as a $\mathbb{B}$-inner semimartingale that fulfills the condition given by \eqref{price}.

Suppose that $U$ is another $\mathbb{B}$-inner semimartingale which satisfies the condition $U=s_0\mathfrak{I}_\mathbb{B}+{U_{-}}{}_{\bullet}Z$. By defining $Y=S-U$, we can invoke Theorem \ref{HX-p}(1) to establish that $Y={Y_{-}}{}_{\bullet}Z$.
Next, let $(T_n,\widetilde{Y}^{(n)})$ and $(T_n,Z^{(n)})$ be inner FCSs for $Y\in \mathcal{S}^{i,\mathbb{B}}$ and $Z\in \mathcal{S}^{i,\mathbb{B}}$, respectively. For each $n\in \mathbf{N}^+$, we set $Y^{(n)}=(\widetilde{Y}^{(n)})^{S_n\wedge (\tau-)}$. It is straightforward to verify that $(S_n,Y^{(n)})$ constitutes an inner FCS for $Y\in \mathcal{S}^{i,\mathbb{B}}$.
Furthermore, Corollaries \ref{deltaX}(2) and \ref{bound-HX} indicate that $(S_n,Y^{(n)}_{-}.Z^{(n)})$ is an inner FCS for ${Y_{-}}{}_{\bullet}Z\in \mathcal{S}^{i,\mathbb{B}}$. Given that both $(T_n,Y^{(n)})$ and $(T_n,Y^{(n)}_{-}.Z^{(n)})$ serve as inner FCSs for $Y\in \mathcal{S}^{i,\mathbb{B}}$, we can infer the following equality:
\[
Y^{(n)}=(\widetilde{Y}^{(n)})^{S_n\wedge (\tau-)}=(Y^{(n)}_{-}.Z^{(n)})^{S_n\wedge (\tau-)}
=Y^{(n)}_{-}.(Z^{(n)})^{S_n\wedge (\tau-)},\quad n\in \mathbf{N}^+,
\]
where $(Z^{(n)})^{S_n\wedge (\tau-)}$ is a semimartingale.
The Dol\'{e}an-Dade exponential formula (see, e.g., Theorem 9.39 in \cite{He}) then implies that $Y^{(n)}=0$ for each $n\in \mathbf{N}^+$. Consequently, we arrive at the conclusion that $Y=0\mathfrak{I}_\mathbb{B}$ (i.e., $U=S$), thereby completing the proof.
$\hfill\blacksquare$

%%%%%%%%%%%%%%%%%%%%%%%%%%%%%%%%%%%%%%%%%%%%%%%%%%%%%%%%%%%%%%%%%%%%%%%%%%%%%%%%%%%%%%%%%%%%%%%%%%%%%%%
\par\vspace{0.3cm}
%%%%%%%%%%%%%%%%%%%%%%%%%%%%%%%%%%%%%%%%%%%%%%%%%%%%%%%%%%%%%%%%%%%%%%%%%%%%%%%%%%%%%%%%%%%%%%%%%%%%%%%
\noindent {\bf Proof of Lemma \ref{Sn2}.}
(1) From the definition, it is straightforward to see that $\widetilde{M}^{(n)}\in \mathcal{M}_{\mathrm{loc},0}$. Given that $N^{(k)},k=1,2,\cdots,n$ are independent standard Brownian motions, it follows that
\[
\langle \widetilde{M}^{(n)}\rangle=\sum\limits_{k=1}^{n-1}I_{\rrbracket{a_{k-1},a_k}\rrbracket}.\widetilde{A}
+I_{\rrbracket{a_{n-1},+\infty}\llbracket}.\widetilde{A}
=\widetilde{A}.
\]
By invoking L\'{e}vy theorem (see, e.g., Theorem 3.16 in \cite{Karatzas-Shreve}), we conclude that $\widetilde{M}^{(n)}$ is indeed a standard Brownian motion.

(2) By using \eqref{expZ} and part (1), the proof is straightforward.
$\hfill\blacksquare$

%%%%%%%%%%%%%%%%%%%%%%%%%%%%%%%%%%%%%%%%%%%%%%%%%%%%%%%%%%%%%%%%%%%%%%%%%%%%%%%%%%%%%%%%%%%%%%%%%%%%%%%
\par\vspace{0.3cm}
%%%%%%%%%%%%%%%%%%%%%%%%%%%%%%%%%%%%%%%%%%%%%%%%%%%%%%%%%%%%%%%%%%%%%%%%%%%%%%%%%%%%%%%%%%%%%%%%%%%%%%%
\noindent {\bf Proof of Proposition \ref{pro}.}
According to Lemma \ref{Sn2},  $\widetilde{S}^{(n)}$ can alternatively be represented as the stochastic differential equation
\[
d\widetilde{S}^{(n)}_t=\widetilde{S}^{(n)}_t(\widetilde{\mu}^{(n)}_t dt+\widetilde{\sigma}^{(n)}_t d\widetilde{M}^{(n)}_t), \quad \widetilde{S}^{(n)}_0=s_0,\; t\in \mathbf{R}^+,
\]
and $(\tau_n,(\widetilde{S}^{(n)})^{a_n})$ forms an
inner FCS for $S\in\mathcal{S}^{i,\mathbb{B}}$.

(1) It is straightforward to verify that the financial market $((\widetilde{S}^{(n)})^{a_n},\mathbb{F})$ satisfies the NA condition for each $n\in \mathbf{N}^+$ (see, e.g., Theorem 12.1.8 in \cite{Oksendal}). Consequently, the financial market $(S,\mathbb{F},\mathbb{B})$ also satisfies NA.

(2) Fix $n\in \mathbf{N}^+$, and consider the financial market $(\widetilde{S}^{(n)},\mathbb{F})$ over the time span $\llbracket{0,a_n}\rrbracket$ and the investor's portfolio problem with uncertain horizon $\tau$.
Define $J$ as the inverse of the first derivative of $\varphi$, i.e., $J(y)=\frac{1}{y},y>0$, and put
\[
H_t=\exp\left\{-\int_0^t\frac{\widetilde{\mu}_s^{(n)}}{\widetilde{\sigma}_s^{(n)}}d\widetilde{M}_s^{(n)}
-\frac{1}{2}\int_0^t\left(\frac{\widetilde{\mu}_s^{(n)}}{\widetilde{\sigma}_s^{(n)}}\right)^2ds
\right\},\; t\geq 0,
\]
and $v_t=\frac{1}{x_0}$, where $x_0$ is the investor's initial wealth.
Then it can be verified that $J(v_0)=x_0$ holds and $(H_tJ(v_tH_t))_{t\geq 0}$ is a martingale.
Consequently, by applying Theorem 2 from \cite{Blanchet}, the optimal proportion $w^{(n)}$ of wealth invested in the stock, and the optimal wealth $X^{(n)}$ for the problem \eqref{optimal-problem} within the market $(\widetilde{S}^{(n)},\mathbb{F})$ is determined as
\[
w^{(n)}_t=-\frac{J'(v_tH_t)v_tH_t\mu^{(n)}_t}{J(v_tH_t)(\sigma^{(n)}_t)^2}
=\frac{\mu^{(n)}_t}{(\sigma^{(n)}_t)^2}
\]
and
\[
X^{(n)}_t=J(v_tH_t)=\frac{x_0}{H_t}.
\]
By utilizing (G2), (G3) and \eqref{mu-sigma},
we establish that $(\tau_n,w^{(n)})$ forms an FCS for $w\in \mathcal{P}^\mathbb{B}$, and $(\tau_n,X^{(n)})$ serves as a CS for $X^*$. It follows that the optimal strategy $\pi^{(n)}$ in $(\widetilde{S}^{(n)},\mathbb{F})$ can be established by $\pi^{(n)}=\frac{w^{(n)}X^{(n)}}{\widetilde{S}^{(n)}}$, and $(\tau_n,\pi^{(n)})$ forms an FCS for $\pi\in \mathcal{P}^\mathbb{B}$, where $\pi$ is given by \eqref{optimal-solution}.
Therefore, we conclude that $\pi=\frac{wX^*}{S}$ is the optimal strategy (w.r.t. $(\tau_n,s_0+\widetilde{S}^{(n)}\widetilde{\sigma}^{(n)}.\widetilde{M}^{(n)},\widetilde{S}^{(n)}\widetilde{\mu}^{(n)}
.\widetilde{A})$) in the financial market $(S,\mathbb{F},\mathbb{B})$.
$\hfill\blacksquare$

\end{document}